\documentclass[reqno,twoside]{amsart} 

\usepackage{amsmath,amsthm,amssymb,amsfonts,mathrsfs,color}
\usepackage{stmaryrd}
\usepackage{empheq}
\usepackage{cleveref}
\usepackage{dsfont}

\usepackage{enumitem}

\usepackage{pgfplots}
\pgfplotsset{compat=1.15}
\usepackage{mathrsfs}
\usetikzlibrary{arrows}
\usetikzlibrary{patterns}

\definecolor{uuuuuu}{rgb}{0.26666666666666666,0.26666666666666666,0.26666666666666666}
\definecolor{wqwqwq}{rgb}{0.3764705882352941,0.3764705882352941,0.3764705882352941}
\definecolor{uququq}{rgb}{0.25098039215686274,0.25098039215686274,0.25098039215686274}
\usepackage{xcolor} 
\definecolor{B1}{RGB}{49,62,72} 
\definecolor{C1}{RGB}{124,135,143}
\definecolor{D1}{RGB}{213,218,223}

\usepackage{caption}
\usepackage{tikz}
\usepackage{graphicx}
\usepackage[absolute]{textpos} 
\usepackage{colortbl}
\usepackage{array}
\usepackage{geometry}

\definecolor{A2}{RGB}{198,11,70}
\definecolor{B2}{RGB}{237,20,91}
\definecolor{C2}{RGB}{238,52,35}
\definecolor{D2}{RGB}{243,115,32}
\definecolor{A3}{RGB}{124,42,144}
\definecolor{B3}{RGB}{125,106,175}
\definecolor{C3}{RGB}{198,103,29} 
\definecolor{D3}{RGB}{254,188,24}
\definecolor{A4}{RGB}{0,78,125}
\definecolor{B4}{RGB}{14,135,201}
\definecolor{C4}{RGB}{0,148,181}
\definecolor{D4}{RGB}{70,195,210}
\definecolor{A5}{RGB}{0,128,122}
\definecolor{B5}{RGB}{64,183,105}
\definecolor{C5}{RGB}{140,198,62}
\definecolor{D5}{RGB}{213,223,61}
\theoremstyle{plain}
\begingroup
\newtheorem{thm}{Theorem}[section]
\newtheorem{lem}[thm]{Lemma}
\newtheorem{prop}[thm]{Proposition}

\endgroup

\theoremstyle{definition}
\begingroup
\newtheorem{defn}[thm]{Definition}
\newtheorem{rem}[thm]{Remark}
\endgroup

\numberwithin{equation}{section}

\usepackage{environ}
\NewEnviron{myequation}{%
\begin{equation}
\scalebox{1.5}{$\BODY$}
\end{equation}
}

\newcommand{\res}{\mathop{\hbox{\vrule height 7pt width .5pt depth 0pt
\vrule height .5pt width 6pt depth 0pt}}\nolimits}

\newcommand{\ds}{\displaystyle}

\newcommand{\N}{\mathbb N} 
\newcommand{\R}{\mathbb R} 
\newcommand{\Sn}{\mathbb S}

\newcommand{\Ms}{{\mathbb M}^{N{\times}N}_{\rm sym}}
\newcommand{\Msd}{{\mathbb M}^{2{\times}2}_{\rm sym}}

\newcommand{\E}{{\mathcal E}}
\newcommand{\F}{{\mathcal F}}
\newcommand{\G}{{\mathcal G}}

\newcommand{\wto}{\rightharpoonup}
\newcommand{\e}{\varepsilon}

\newcommand{\LL}{{\mathcal L}}
\newcommand{\HH}{{\mathcal H}}
\newcommand{\KK}{ K}
\newcommand{\M}{{\mathcal M}}

\newcommand{\T}{{\mathcal T}}

\newcommand{\infconv}{ \, {\scriptstyle \Box}  \, }

\renewcommand{\bar}[1]{\overline{#1}}
\newcommand{\norme}[1]{\Vert #1 \Vert}

\let\O=\Omega

\DeclareSymbolFont{yhlargesymbols}{OMX}{yhex}{m}{n} 
\DeclareMathAccent{\yhwidehat}{\mathord}{yhlargesymbols}{"62}

\begin{document}
 
\title[Effects of oscillation scales in discrete brittle damage models]{Effects of oscillation scales in discrete brittle damage models}
\author[E. Bonhomme]{\'Elise Bonhomme}

\address[E. Bonhomme]{Universit\'e Libre de Bruxelles, Département de math\'ematique, 1050 Brussels, Belgium.}
\email{elise.bonhomme@ulb.be}

\begin{abstract} 
This paper is concerned with the asymptotic analysis of a sequence of variational models of brittle damage in the context of linearized elasticity in the two-dimensional discrete setting. 
We consider a discrete version of Francfort and Marigo's brittle damage model, where the total energy is restricted to continuous and piecewise affine vectorial displacements; within different regimes where the damaged regions concentrate on vanishingly small sets while the stiffness of the damaged material degenerates to $0$. 
In this setting, the convergence of the space discretization, the concentration of the damaged regions, and the decay of the elastic properties of the damaged phase all compete simultaneously in non-trivial ways according to the scaling law under consideration.
The mesh size turns out to be a crucial feature of the analysis, as it induces a
minimal scale of spatial oscillations for admissible displacements.
This study was motivated by the numerical investigations performed in \cite{AJVG} on the one hand, where Allaire-Jouve-Van Goethem have shown that forcing the stiffness decay on sets of
arbitrarily small measure seems to lead to concentrations phenomena such as in brittle fracture; and by the static analysis performed in \cite{BIR} on the other hand, where Babadjian-Iurlano-Rindler addressed the rigorous asymptotic analysis of this observation in terms of the $\Gamma$-convergence of the total energies, in the continuous setting in space. Surprisingly, they showed that fracture-type models were not obtained asymptotically, thus raising the question of the dependence of the effective models with respect to the scaling of a spatial discretization in Francfort and Marigo's model. This is the content of the present work, where we show that concentrations phenomena are only captured asymptotically in regimes where the mesh size and the concentration of damaged regions are of same order. 

\end{abstract}

\maketitle

\tableofcontents

\section{Introduction}
\subsection{The variational brittle damage model of Francfort and Marigo}
A process of damage is a dissipative phenomenon during which the rigidity tensor of the medium, a linearly elastic material whose reference configuration is a bounded Lipschitz open set $\O \subset \R^N$, is irreversibly weakening under the effect of the loading to which it is subjected during a time. General damage models are classically based on the introduction of a time-space dependent internal variable $\chi(t,x)$ describing the evolution of the rigidity tensor $\mathbf A(t,x) = \tilde{ \mathbf A}( \chi(t,x)) \in \mathcal L (\Ms;\Ms)$. For instance, a natural and simple choice consists in taking 
$$\mathbf A(t,x) := \mathbf A_0 \chi(t,x) + \mathbf A_1 \left(1 - \chi(t,x) \right),$$
where $\chi(t,x) \in [0,1]$ is non-decreasing in time, thus forcing the irreversibility of damage, and $\mathbf{A}_0$ and $\mathbf{A}_1$ are two symmetric and coercive fourth-order tensors representing the elastic properties of the damaged and healthy phases respectively, satisfying the ordering property
$$
\mathbf{A}_0 < \mathbf{A}_1
$$
as quadratic forms acting on real symmetric matrices $\Ms$.
Thus, $\chi = 1$ corresponds to fully damaged regions while $\chi = 0$ corresponds to sound regions. The particular case where $\chi(t,x) \in \{0,1\}$ only takes the values $0$ and $1$ corresponds to brittle damage, while the general case $\chi(t,x) \in [0,1]$ corresponds to progressive damage. The evolution of the material is then governed, on the one hand, by the Hooke Law which states that the Cauchy stress is linearly proportional to the strain 
$$
\sigma(t,x) = \mathbf A(t,x) Du(t,x),
$$
and by a Griffith's type maximal stress constraint on the other hand, which stipulates that elastic deformations will only yield to damage when the Cauchy stress exceeds a given threshold. In other words, the rigidity tensor of the material will drop from the value $\mathbf A_1$, the rigidity of the sound state, to the value $\mathbf A_0$, the rigidity of the damaged state, if the Cauchy stress exceeds a given threshold (see \cite{Griffith,Suquet,FM1}). Following the model introduced by Francfort and Marigo (\cite{FM1,FM}), the total energy associated to a displacement $u \in H^1(\O;\R^N)$ and a characteristic function of the damaged zones $\chi \in L^\infty(\O;\{0,1\})$ then puts in competition the elastic energy stored in the body and a dissipative term due to damage:
\begin{equation}\label{eq:E(u,chi)}
\E(u,\chi) = \frac12 \int_\O \left( \chi \mathbf{A}_0 + (1-\chi) \mathbf{A}_1 \right) e(u):e(u) \, dx + \kappa \int_\O \chi \, dx,
\end{equation}
where the symmetric gradient $e(u) = \frac{\nabla u + \nabla u^T}{2}$ is the linearized elastic strain and $\kappa > 0$ is the material's toughness of the damaged regions. In other words, the cost for damaging a part of the material is taken as proportional to the volume of the damaged zone. 
In \cite{FM1}, Francfort and Marigo introduced a variational approach to quasi-static brittle damage evolution ({\it i.e.} without inertia effects), based on the premise that at every instant, the material seeks to minimize a total energy of the form \eqref{eq:E(u,chi)} while accounting for the loading history and the irreversibility of damage.

More specifically, given a loading $f \in AC([0,T];L^2(\O))$, a Dirichlet boundary condition $w \in AC([0,T];H^1(\R^N))$, and a subdivision of the time interval $[0,T]$, $0 = t_0 < \dots < t_I = T$, Francfort and Marigo seek the pairs $(u_i,\chi_i) \in H^1(\O;\R^N) \times L^\infty(\O;\{0,1\})$ that minimize the total energy at time $t_i$
$$
\E_i(u,\chi) := \E(u,\chi) - \int_\O f(t_i) \cdot u \, dx
$$
among all pairs $(u,\chi) \in H^1(\O;\R^N) \times L^\infty(\O;\{0,1\})$ such that $u = w(t_i)$ on $\partial \O$ and $\chi \geq \chi_{i-1}$. Unfortunately, it is by now well known that this minimization problem is ill-posed due to the lack of convexity of its density, so that the energy must be relaxed. By doing so, because of the energetic interest that minimizing sequences have in forming microstructures converging to homogenized states (fine mixtures of the healthy and damaged phases), the brittle character of damage disappears in favor of progressive damage. The model finally proposed by Francfort and Marigo then consists of seeking the pairs $(u_i, \theta_i) \in H^1(\O;\R^N) \times L^\infty(\O;[0,1])$ that minimize
$$
\E^*_i(u,\theta) = \int_\O \underset{\mathbf{A}_* \in \G_\theta(\mathbf{A}_0,\mathbf{A}_1)}{\min} \left( \frac12 \mathbf{A}_* e(u):e(u) \right) \, dx + \kappa \int_\O \theta \, dx - \int_\O f(t_i) \cdot u \, dx
$$
among all pairs $(u,\theta) \in H^1(\O;\R^N) \times L^\infty(\O;[0,1])$ such that $u = w(t_i)$ on $\partial \O$ and $\theta \geq \theta_{i-1}$, where the class of admissible states is extended to the set of homogenized composite materials $\G_\theta(\mathbf{A}_0, \mathbf{A}_1)$ obtained by increasingly fine mixtures between the damaged and healthy phases in volume fractions $\theta \in [0,1]$ and $1 - \theta$, respectively (see \cite{FM1, FG, Allaire, AL}). This problem is solved in \cite{FG}, where Francfort and Garroni define and prove the well-posedness of this relaxed variational formulation in continuous time for the quasi-static case. A quasi-static evolution of brittle damage is then given by a triplet
\begin{equation*}
(u,\Theta, \mathbf{A}) : [0,T] \to H^1(\O;\R^N) \times L^\infty(\O;[0,1]) \times \G(\mathbf{A}_0,\mathbf{A}_1)
\end{equation*}
such that $\mathbf{A}(t) \in \G_{1 - \Theta(t)}(\mathbf{A}_0, \mathbf{A}_1)$ at all times $t \in [0,T]$; the stiffness $\mathbf{A}$ and the volume fraction of healthy material $\Theta$ decrease over time; $(u(t), \mathbf{A}(t), 0)$ satisfies a "unilateral" stability property (due to the irreversibility of damage) by minimizing at all times $t \in [0,T]$
$$
 \int_\O \frac12 \mathbf{A}_* e(v):e(v) \, dx + \kappa \int_\O \Theta(t) \theta \, dx - \int_\O f(t) \cdot v \, dx
$$
among all triplets $v \in w(t) + H^1_0(\O;\R^N)$, $\theta \in L^\infty(\O;[0,1])$, and $\mathbf{A}_* \in \G_{\theta}(\mathbf{A}_0, \mathbf{A}(t))$; and finally, the total energy
$$
\E := \int_\O \frac12 \mathbf{A} e(u):e(u) \, dx + \kappa \int_\O (1 - \Theta) \, dx - \int_\O f \cdot u \, dx
$$
satisfies the energy balance
$$
\E(t) = \E(0) + \int_0^t \int_\O \mathbf{A} e(u) : e( \dot{w}) \, dx \, ds - \int_0^t \int_\O \left( \dot{f} \cdot u + f \cdot \dot{w} \right) \, dx \, ds.
$$

\subsection{Motivation and results}

The present paper consists of a first step in the understanding of the asymptotic behavior of such evolutions under scaling laws, where the damaged zone concentrates on vanishingly small sets while damage intensifies until it becomes total. In other words, we consider brittle damage energies of the form:
\begin{equation}\label{eq:E_e(u,chi)}
\E_\e(u,\chi) =\frac12 \int_\O  \big( \eta_\e \chi \mathbf A_0 + (1-\chi) \mathbf A_1 \big) e(u):e(u) \, dx + \frac{\kappa}{\e} \int_\O \chi \, dx
\end{equation}
where $\eta_\varepsilon \searrow 0$ as $\varepsilon \searrow 0$. Thus, the stiffness $\eta_\varepsilon \mathbf{A}_0$ of the damaged state degenerates to $0$, while the diverging toughness term $\kappa / \varepsilon \to +\infty$ forces the damaged zone to concentrate on Lebesgue-negligible sets. 

In the numerical implementation \cite{AJVG}, Allaire-Jouve-Van Goethem proposed a method to simulate quasi-static crack growth in brittle materials ({\it i.e.} with constant toughness), based on the intuition that, at each time step, Francfort-Marigo's total energy \eqref{eq:E_e(u,chi)} $\Gamma$-converges to a Griffith-type fracture energy as $\varepsilon \searrow 0$. Yet, Babadjian-Iurlano-Rindler showed in \cite{BIR} that the limiting energy does not correspond to a fracture model, regardless of the convergence rate $\lim_\varepsilon \eta_\varepsilon / \varepsilon \in [0, +\infty]$. However, this does not completely invalidate Allaire-Jouve-Van Goethem's conjecture, as they actually minimize a discretized version of \eqref{eq:E_e(u,chi)} by restricting its domain to piecewise affine displacements $u$ and piecewise constant characteristic functions $\chi$ (adapted to a common triangulation of $\Omega$). Hence, a third infinitesimal parameter is introduced, the mesh size $0 < h_\varepsilon \ll 1$, which induces a minimal spatial oscillation scale for admissible displacements.

These two works hence raised the question of the way spatial discretization in the static Francfort-Marigo model \eqref{eq:E_e(u,chi)} will influence the nature of the effective models. This is the objective of the present paper, which extends the static analysis \cite{BIR} to the discrete setting, in dimension two.

\medskip

More specifically, let $\O$ be a bounded open set of $\R^2$ with Lipschitz boundary. As in \cite{CDM}, we introduce the following class of admissible meshes.
\begin{defn}\label{def:triangulation Static}
A triangulation of $\O$ is a finite family of closed triangles intersecting $\O$, whose union contains $\O$, and such that, given any two triangles of this family, their intersection, if not empty, is exactly a vertex or an edge common to both triangles. Given some angle $\theta_0>0$ and a function $h \mapsto \omega(h)$ with $\omega(h) \geq 6h$ for any $h>0$ and $\lim_{h \to 0^+} \omega(h)=0$, we define
$$\mathcal T_{h}(\O):=\mathcal T_{h}(\O,\omega,\theta_0)$$
as the set of all triangulations of $\O$ made of triangles whose edges have length between $h$ and $\omega(h)$, and whose angles are all greater than or equal to $\theta_0$. Then we consider the finite element space $X_{h}(\O)$ of all couples $(u,\chi) \in  C^0(\O;\R^2) \times L^\infty(\O;\{0,1\})$ for which there exists $\mathbf T \in \mathcal T_{h}(\O)$ such that $u$ is affine and $\chi$ is constant on each triangle $T \in \mathbf T$.
\end{defn}
\begin{rem}
Imposing $\theta_0 > 0$ and $\omega(h) \geq h$ corresponds to a non-flatness condition that ensures the existence of a radius $\varrho(\theta_0) > 0$ such that for all triangle $T \in \mathbf T $, one can find a point $x \in T$ such that
$$
\bar{ B_{\varrho h} (x)} \subset T.
$$ 
In particular, this enforces a minimal scale of spatial oscillations for admissible displacements, as their gradient is constant on each triangle.
\end{rem}

Given $\e >0$, $\eta_\e >0$ and $h_\e > 0$, we introduce the following converging rates
\begin{equation}\label{eq:converging rates}
\alpha=\lim_{\e \to 0}\frac{\eta_\e}{\e} \in [0,\infty], \quad \beta=\lim_{\e \to 0}\frac{h_\e}{\e} \in [0,\infty]
\end{equation} 
and consider the brittle damage functionals $\F_\e:L^1(\O;\R^2) \times L^1(\O) \to [0,\infty]$ defined by:
\begin{equation}\label{eq:F_eps}
\F_\e(u,\chi)=
\begin{cases}
{\displaystyle \frac12 \int_\O  \big( \eta_\e \chi \mathbf A_0 + (1-\chi) \mathbf A_1 \big) e(u):e(u) \, dx + \frac{\kappa}{\e} \int_\O \chi \, dx }& \text{ if } (u,\chi) \in X_{h_\e}(\O) ,\\
+\infty & \text{ otherwise},
\end{cases}
\end{equation}
where $\kappa \in (0,\infty)$ and $\mathbf A_0$, $\mathbf A_1 \in \mathcal L(\Msd;\Msd)$ are symmetric fourth order tensors satisfying
\begin{equation}\label{eq:rigidity tensors}
a_i\, {\rm Id} \leq \mathbf A_i \leq a'_i\, {\rm Id} \quad \text{ for } i\in \{0,1\}
\end{equation}
as quadratic forms over $\Msd$, for some constants $a_0,a_1,a'_0,a'_1 \in (0,\infty)$. To state precisely our main result, Theorem \ref{thm:1} below, we need to introduce some notation (we refer to section \ref{subsec:notation} regarding functional spaces). 
For all unit vector $k \in \mathbb S^1$, let $V_k \subset \Msd$ be the two-dimensional vectorial space defined by 
$$
V_k := \{ x \odot k \, : \, x \in \R^2 \}
$$
as well as the convex, continuous, non-negative and $2$-homogeneous function $g : \Msd \to [0,\infty)$ defined for $\xi \in \Msd$ by
\begin{equation*}
g(\xi) := \sup_{ k \in \mathbb S^1} \left\lvert \ds \Pi_k \left( \mathbf A_0^{\scriptstyle{-\frac12}} \xi \right) \right\rvert^2,
\end{equation*}
where $\Pi_k (\xi) $ is the orthogonal projection of $\xi \in \Msd$ onto $\mathbf A_0^{\scriptstyle{\frac12}} V_k$.
\begin{rem}
Although the function $g(\xi)$ has been extensively studied (see \cite{AK93,K91,KM87} and references therein), its computation remains a difficult task. Nonetheless, explicit formulas have been obtained in the particular case of an isotropic stiffness tensor (see \cite{AK93,AL,FM86,Allaire} for instance) where, given a bulk and shear moduli $\gamma >0$ and $\mu >0$ respectively,
$$
\mathbf A_0 \xi = \gamma \text{\rm Tr}(\xi) I_2 +  2 \mu \left(  \xi - \frac12 \text{\rm Tr}(\xi) I_2 \right) \, \text{ for all } \xi \in \Msd.
$$
As most elastic materials also satisfy $\lambda := \gamma - \mu >0$, we choose to express this law in terms of the Lam\'e coefficients $\mu,\lambda >0$ and say that the stiffness tensor $\mathbf A_0$ is isotropic if
\begin{equation}\label{eq:isotropic Hooke Law A_0}
\mathbf A_0 \xi = 2 \mu \xi + \lambda \text{\rm Tr}(\xi) I_2.
\end{equation}
In this case (see \cite[Proposition 7.4]{AK93} specified to the two dimensional setting), $g$ is explicitely given in terms of the ordered eigenvalues $\xi_1 \leq \xi_2$ of $\xi \in \Msd$ by
\begin{equation}\label{eq:g isotropic case}
g(\xi) = \begin{cases}
							 \vspace*{0.2cm} \frac{\xi_1^2}{\lambda + 2 \mu}  &  \text{if } \frac{\lambda + 2 \mu}{2(\lambda + \mu)} (\xi_1 + \xi_2) < \xi_1, \\
							\vspace*{0.2cm} \frac{\left( \xi_1 - \xi_2 \right)^2}{4 \mu} + \frac{\left( \xi_1 + \xi_2 \right)^2}{4( \lambda +  \mu)}  &  \text{if } \xi_1 \leq \frac{\lambda + 2 \mu}{2(\lambda + \mu)} (\xi_1 + \xi_2) \leq \xi_2,  \\
							 \frac{\xi_2^2}{\lambda_0 + 2 \mu} & \if \text{if } \xi_2 < \frac{\lambda + 2 \mu}{2(\lambda + \mu)} (\xi_1 + \xi_2).
						 \end{cases}
\end{equation}
\end{rem}
Let us now introduce the following densities, that will be encountered in the asymptotic analysis later on.
\begin{itemize}[leftmargin=*,label=-]
\item Provided $\alpha \in (0,\infty)$, let $\KK = \KK(\alpha) \subset \Msd$ be the closed convexe set, containging the origin, defined by
\begin{equation*}
\KK := \left\{ \xi \in \Msd \, : \, g(\xi) \leq 2  \alpha \kappa \right\}.
\end{equation*}
We then define the density $\bar W_\alpha : \Msd \to [0,\infty)$ given by
\begin{equation}\label{eq:W alpha}
\bar W_\alpha := \left( f^* + I_\KK \right)^* = f \infconv \sqrt{2 \alpha \kappa h}
\end{equation}
where, for all $\xi \in \Msd$, 
\begin{equation*}
f(\xi) = \frac12 \mathbf A_1 \xi:\xi, \quad I_\KK(\xi) = \begin{cases} 0 & \text{if } \xi \in \KK , \\ \infty & \text{otherwise,} \end{cases}  \quad \text{and} \quad h(\xi) = g^*(2 \xi)=\frac{\left( I_\KK^*(\xi) \right)^2}{2\alpha\kappa},
\end{equation*} 
and let $\bar W_\alpha^\infty : \Msd \to [0,\infty)$ be its recession function, which is well defined and satisfies
\begin{equation}\label{eq:recession}
\bar W_\alpha^\infty(\xi) := \lim_{t \to \infty} \frac{\bar W_\alpha(t\xi)}{t} = \sqrt{2\alpha \kappa h(\xi)} \quad \text{for all } \xi \in \Msd.
\end{equation}
The density $\bar W_\alpha$ turns out to be quadratic close to the origin and grows linearly at infinity, as typically encountered in models of Hencky plasticity (see \cite{M2016}). \vspace{1mm}

\item Provided $\alpha \in (0,\infty)$ and $\beta \in (0,\infty)$, we introduce the density $\phi_{\alpha,\beta} : \R \to [0,\infty)$ defined by 
\begin{equation}\label{eq:density intermediate}
\phi_{\alpha,\beta} : t \in \mathbb{R} \mapsto \begin{cases}
						\ds \frac{ \alpha}{2 \beta \sin \theta_0} t^2 + \beta \kappa \sin \theta_0 & \quad \text{if } \ds \left\lvert t \right\rvert \leq \beta \sin \theta_0 \sqrt{\frac{2 \kappa}{ \alpha}}, \\
						\sqrt{2 \kappa  \alpha} \left\lvert t \right\rvert & \quad \text{otherwise,}
						\end{cases}
\end{equation}
which has linear growth at infinity (as in Hencky plasticity) and behaves as a positive constante close to the origin (as in brittle fracture).
\end{itemize}

\begin{rem}
In the case where $\alpha \in (0,\infty)$ and $\mathbf A_0$ is isotropic and given by \eqref{eq:isotropic Hooke Law A_0}, the function $h(\xi)$ is explicitly known and given by 
\begin{equation}\label{eq:h isotropic case}
h(\xi) = \mathbf A_0 \xi:\xi + 4 \mu \left( \text{\rm det} \xi \right)_+ = \begin{cases} 
																				(\lambda + 2 \mu) \lvert \xi_1 + \xi_2 \rvert^2 & \text{if } \xi_1 \xi_2 \geq 0, \\
																				\mathbf A_0 \xi:\xi & \text{otherwise,}
																			\end{cases}
\end{equation}
where $\xi_1 \leq \xi_2$ are the ordered eigenvalues of $\xi \in \Msd$ (see \cite[Theorem 5.3]{AL}). Moreover, the convex set $\KK$ is bounded (see Figure \ref{fig:K isotropic}) and explicitly determined by the Lam\'e coefficients of $\mathbf A_0$.
\end{rem}
\begin{figure}[hbtp]
\centering
\begin{minipage}{0.35\textwidth}
        \includegraphics[width=\textwidth, trim=1cm 2cm 2cm 1.5cm, clip]{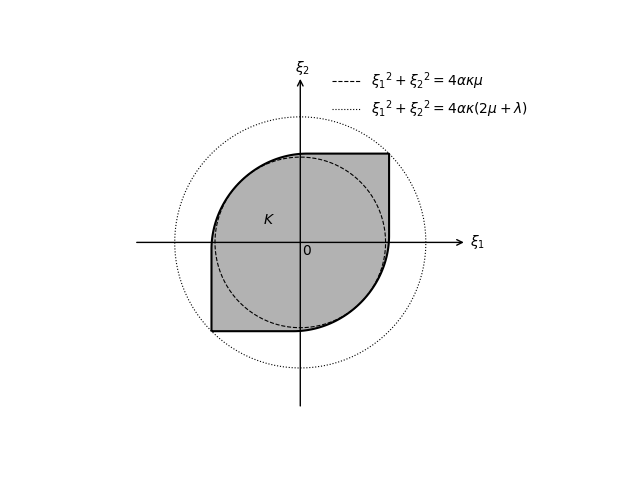}
\end{minipage}
\begin{minipage}{0.3\textwidth}
\includegraphics[width=\textwidth, trim=1cm 1.4cm 2cm 1cm, clip]{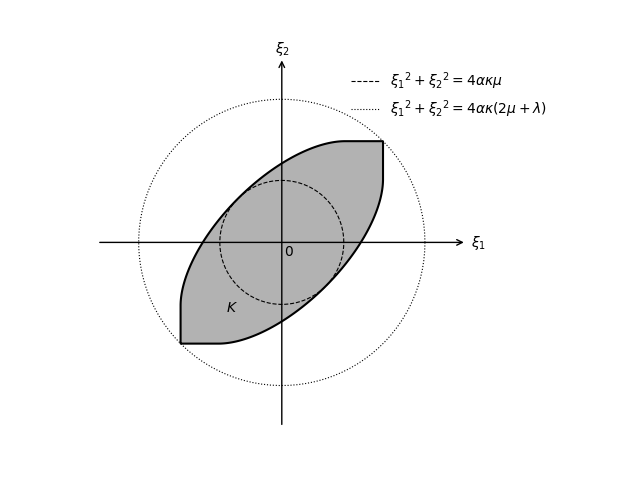}
\end{minipage}
\caption{Representation of the convex set $\KK$ in the isotropic case, in terms of the ordered eigenvalues $\xi_1 \leq \xi_2$ of $\xi \in \Msd$, for different values of Lam\'e coefficients.}
\label{fig:K isotropic}
\end{figure}

We now have everything at hand to sate our main result, that is the following $\Gamma$-convergence analysis of the discrete brittle damage energies $\F_\e$ given by \eqref{eq:F_eps}.

\begin{thm}\label{thm:1}
The functionals $\F_\e$ $\Gamma$-converge as $\e \to 0$ with respect to the strong $L^1(\O;\mathbb{R}^2) \times L^1(\O)$-topology to the functional $\F_{\alpha,\beta}: L^1(\O; \mathbb{R}^2) \times L^1(\O) \to [0, \infty]$ defined, according to the convergence rates $\alpha, \beta $, by:
\begin{itemize}[leftmargin=*,label=-]
\item {\it Linear elasticity:} if $\alpha = \infty$ or $\beta=\infty$, then
$$
\F_{\alpha,\beta}(u,\chi) = \begin{cases}
							 \frac12 \ds \int_\O  {\bf A}_1 e(u):e(u) \, dx  & \text{if } \chi=0 \text{ a.e.} \text{ and } u \in H^1(\O; \mathbb{R}^2) , \\
							+ \infty & \text{otherwise;}
							 \end{cases}
$$
	 							 
\item {\it Trivial regime:} if $\alpha=\beta=0$ and $\theta_0 < 45^\circ$, then
$$
\F_{\alpha,\beta}(u,\chi) = \begin{cases}
							 0 & \text{if } \chi=0 \text{ a.e.} \text{ and } u \in L^1(\O; \mathbb{R}^2) ,\\
							+ \infty & \text{otherwise;}
							 \end{cases}
$$

\item {\it Hencky plasticity:} if $\beta=0<\alpha<\infty$, $\theta_0< \arctan (1/4)$, and $\mathbf A_0$ is isotropic, then
$$
\F_{\alpha,\beta}(u,\chi) = \begin{cases}
							  \ds \int_\O  \bar W_\alpha(e(u))\, dx + \ds \int_{\bar{\O}} \bar W_\alpha^\infty \left( \frac{d E^s(u)}{d |E^s(u)|} \right) \, d |E^s(u)|  & \text{if }  \chi=0 \text{ a.e. and } u \in BD(\O) , \\
							 +\infty & \text{otherwise;}
							 \end{cases}
$$

\item {\it Brittle fracture:} if $\alpha=0<\beta<\infty$ and $\theta_0 < 45^\circ - \text{\rm arctan}(1/2)$, then
$$
\F_{\alpha,\beta}(u,\chi) = \begin{cases}
							 \frac12 \ds \int_\O  {\bf A}_1 e(u):e(u) \, dx +  \beta \kappa \sin \theta_0 \HH^1(J_u)  & \text{if }  \begin{cases} \chi=0 \text{ a.e.} ,\\  u \in GSBD^2(\O)\cap L^1(\O;\R^2), \end{cases} \\
							+ \infty & \text{otherwise;}
							 \end{cases}
$$

\item {\it Intermediate between brittle fracture and plasticity: } if $\alpha,\beta \in (0,\infty)$ and $\theta_0 < \arctan (1/4)$, then
$$
\F_{\alpha,\beta}(u,\chi) = \begin{cases}
							 \frac12 \ds \int_\O  {\bf A}_1 e(u):e(u) \, dx +  \ds \int_{J_u} \phi_{\alpha,\beta} \left( \left\lvert \sqrt{\mathbf A_0} (\left[ u \right] \odot \nu_u) \right\rvert \right)  & \text{if } \begin{cases} \chi=0 \text{ a.e.,} \\ u \in SBD^2(\O), \end{cases} \\
							 +\infty & \text{otherwise.}
							 \end{cases}
$$
\end{itemize}
\end{thm}

\begin{rem} 
\begin{itemize}[leftmargin=*,label=-]
\item In the last four regimes, one has to further assume that the minimal angle $\theta_0>0$ introduced in the definition of admissible meshes is actually smaller than an ad hoc angle $\Theta_0\in \{45^\circ, 45^\circ - \text{\rm arctan}(1/2), \arctan (1/4) \}$, as the optimal triangulations we (explicitly) construct make use of triangles with angles equal to $\Theta_0$.
\item Let us also point out that the choice of the $L^1(\O;\R^2) \times L^1(\O)$-topology might sometimes be improved (see the regime of brittle fracture \cite{BB}), in the sense that the natural topology should be the one for which sequences of displacements with uniformly bounded energies prove to be compact. Here, for simplicity, we confine ourselves to the strong $L^1(\O;\R^2) \times L^1(\O)$-topolgy which in particular provides compactness (see Proposition \ref{prop:domains}), whatever the regime under consideration.
\end{itemize}
\end{rem}

Let us briefly comment the results of Theorem \ref{thm:1}. First note that in all the regimes where $\beta = 0$ ({\it i.e.} $h_\e \ll \e$), we recover the same three effective models as \cite{BIR}, according to the respective regimes in $\alpha$: a model of linear elasticity when $\eta_\e \gg \e$, a trivial model when $\eta_\e \ll \e$, and a model of Hencky plasticity when $\eta_\e \sim \e$. Formally, when the size of the mesh is negligible compared to $\e$, the convergence of the space discretization towards the continuous model occurs infinitly more rapidly than the other phenomena involved, so that the discrete models behave essentially as the continuous ones. In order to motivate the interpretation of the limit model as one of Hencky plasticity when $\eta_\e \sim \e$, let us reformulate the expression of the limit energy $\E(u) := \F_{\alpha,\beta}(u,0)$ for a given displacement $u \in BD(\O)$. Let us write the Radon-Nikod\'ym decomposition of $E(u) \in \M(\O;\Msd)$ with respect
to the Lebesgue measure, $Eu = e(u) \LL^2 \res \O + E^s(u)$, and use the definition of the infimal convolution to split the absolutely continuous part of the strain as $e(u) = e + p^a$ with $e\in L^2(\O;\Msd)$ and $p^a \in L^1(\O;\Msd)$ such that 
$$
\bar W_\alpha (e(u)) = \frac12 \mathbf A_1 e:e + I_\KK^*(p^a) \quad \LL^2 \text{-a.e. in } \O.
$$ 
Then, defining $p = p^a \LL^2 \res \O + E^s(u) \in \M(\O;\Msd)$, we get that $E(u) = e \LL^2 \res \O + p$ and
$$
\E(u) = \frac12 \int_\O \mathbf A_1 e:e \, dx + \int_\O I_\KK^* \left( \frac{d p }{d \lvert p \rvert} \right) \, d \lvert p \rvert.
$$
In other words, the strain $E(u)$ is additively decomposed in an elastic strain $e$ accounting for the reversible deformations and a plastic strain $p$ accounting for the permanent ones, which respectively contributes in the elastic and dissipative terms of the energy, where the support function of $\KK$
$$I_\KK^* (\xi) = \sup_{\tau \in \KK} \xi : \tau$$
stands for the plastic dissipation potential (see \cite{Suquet}). Moreover, this decomposition is optimal in the sense that it minimizes the balance of the total energy between the elastic energy stored in the medium and the cost due to plastic deformations. Therefore, this effective model can indeed be interpreted as one of Hencky plasticity, and we refer to \cite{M2016,Suquet} for more details on this subject.

\medskip
On the other hand, when the size of the mesh is of order $\e$, we obtain two additional effective models where concentration phenoma occur. In other words, the limiting energies put in competition the elastic energy stored in the material and a dissipative fracture-type term, which is supported on the discontinuity set of the displacement. Then, depending on the scaling of $\eta_\e$ with respect to $\e$, we observe two distinct behaviors in the surface densities of these dissipative terms.
\begin{itemize}[leftmargin=*,label=-]
\item When $\alpha = 0$, we obtain a constant surface density $\beta \kappa \sin \theta_0$. Thus, the surface energy is proportional to the length of the fracture, given by its one-dimensional Hausdorff measure, which corresponds to a model of brittle fracture as conjectured by Allaire-Jouve-Van Goethem in \cite{AJVG}.
\item Finally, when $\alpha \in (0, +\infty)$, the convergence of the space discretization, the concentration of damaged zones, and the degeneration of the elastic properties of the damaged phase all compete simultaneously in a non-trivial way. Heuristically, we are halfway in between the Hencky plasticity regime ($\eta_\e \sim \e$) and the fracture regime ($h_\e \sim \e$). We obtain a surface density $\phi_{\alpha,\beta} \left( \lvert \sqrt{\mathbf A_0} \left[ u \right] \odot \nu_u \rvert \right)$ of the form $1 + \left\lvert u^+ - u^- \right\rvert$. This variational model is qualitatively similar to those introduced in \cite{DMI,I} and \cite{ALRC}, where the total energy is given by
$$
\int_\O \frac12 \left\lvert \nabla u \right\rvert^2 \, dx + \gamma \HH^1(J_u) + \sigma_0 \int_{J_u} \left\lvert u^+ - u^- \right\rvert \, d \HH^1
$$
for a scalar displacement $u: \O \to \mathbb{R}$. The cost of anelastic deformations differs depending on the magnitude of the jump. For small amplitude of jumps, the cost to pay is given by the constant surface density equal to the toughness $\gamma > 0$ ({\it i.e.}, the energy per unit area to locally fracture the material), as in brittle fracture. While for large jumps, the surface density spent grows linearly with the amplitude of the jump: $\sigma_0 \left\lvert u^+ - u^- \right\rvert$, as in plasticity, where $\sigma_0 > 0$ represents the threshold of stresses above which elastic deformations give way to plastic deformations. Here, the surface density exhibits the same behavior: for small jumps, the dominant dissipative term corresponds to the fracture energy ($\gamma = \beta \kappa \sin \theta_0$), while it becomes negligible compared to the cost of plastic deformation ($\sigma_0 = \sqrt{2 \kappa \alpha}$) for large jumps. Said differently, the anelastic energy can be interpreted as intermediate between brittle fracture and plasticity, as it can be decomposed as the sum of the length of the crack weighted by the effective material's toughness and a surface energy accounting for plastic deformations,
$$
\beta \kappa \sin \theta_0 \HH^1(J_u)
+
 \int_{J_u}  \psi_{\alpha,\beta} \left(\left\lvert \sqrt{\mathbf A_0} \left[ u \right] \odot \nu_u \right\rvert \right) \, d \HH^1,
$$ 
whose density $
\psi_{\alpha,\beta}(t) := \phi_{\alpha,\beta}(t) - \beta \kappa \sin \theta_0 \geq 0$
has quadratic growth near the origin and linear growth at infinity in terms of the amplitude of the jump, with a slope given by the threshold of stresses $\sqrt{2 \kappa \alpha}$ above which elastic deformations yield to plastic ones (see Figure \ref{fig:plastic surface energy density}).
\end{itemize}

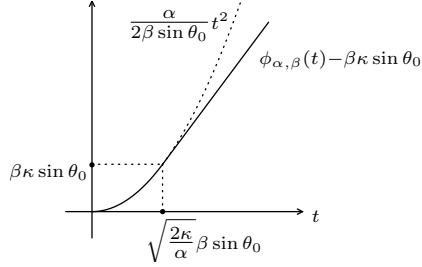
\begin{figure}[hbtp]
\begin{tikzpicture}[line cap=round,line join=round,>=triangle 45,x=0.35cm,y=0.35cm]
\clip(-7.5,-1.9) rectangle (13,8.);
\draw[line width=.5pt,smooth,samples=100,domain=0:2*sqrt(1.8)] plot(\x,{(\x)^(2)/4});
\draw[line width=.5pt,smooth,samples=100,dotted,domain=2*sqrt(1.8):4.3*sqrt(1.8)] plot(\x,{(\x)^(2)/4});

\draw[line width=.5pt,smooth,samples=100,domain=2*sqrt(1.8):5*sqrt(1.8)] plot(\x,{(sqrt(1.8)*(\x)-1.8)});
\draw [line width=.5pt] (0.,8.)-- (0.,-1.);
\draw [line width=.5pt] (0.,8.)-- (-0.1,7.8);
\draw [line width=.5pt] (0.,8.)-- (0.1,7.8);
\draw [line width=.5pt] (-1.,0.)-- (8.,0.);
\draw [line width=.5pt] (8.,0.)-- (7.8,0.1);
\draw [line width=.5pt] (8.,0.)-- (7.8,-0.1);
\draw [line width=0.5pt,dotted] (2.68,0.)--(2.68,1.8);
\draw [line width=0.5pt,dotted] (0.,1.8)--(2.68,1.8);

\draw (6,6.5) node[anchor=north west] {$\scriptstyle{\phi_{\alpha,\beta}(t) - \beta \kappa \sin \theta_0}$};
\draw (8.,0.4) node[anchor=north west] {$\scriptstyle{t}$};
\draw (1.6889277940965626,0.1) node[anchor=north west] {$\scriptstyle{\sqrt{\tfrac{2 \kappa}{\alpha}} \beta \sin \theta_0}$};
\draw (-3.5,2.152688762251922) node[anchor=north west] {$\scriptstyle{\beta \kappa \sin \theta_0}$};
\draw (1.,8) node[anchor=north west] {$\scriptstyle{\tfrac{\alpha}{2\beta \sin \theta_0}t^2}$};

\fill (2.6832815634816676,0.) circle (1.0pt);
\fill (0.,1.7772368605626672) circle (1.pt);
\end{tikzpicture}
\caption{Plastic surface energy density.}
\label{fig:plastic surface energy density}
\end{figure}

In conclusion, the present analysis seems to fully unify the works of \cite{BIR} and \cite{AJVG} under a common reading, stipulating that fracture-type models are only asymptotically captured by Francfort-Marigo's model \eqref{eq:E_e(u,chi)} in regimes where the minimal scale of spatial oscillations $h_\varepsilon$ is of order $ \varepsilon$.

\subsection{Strategy of proof}
Let $\F'_{\alpha,\beta}$ and $\F''_{\alpha,\beta}:L^1(\O;\R^2) \times L^1(\O) \to [0,+\infty]$ be the $\Gamma$-lower and $\Gamma$-upper limits respectively, {\it i.e.},  for all $(u,\chi) \in L^1(\O;\R^2) \times L^1(\O)$,
$$\F_{\alpha,\beta}'(u,\chi):=\inf\left\{\liminf_{\e \to 0}\F_\e(u_\e,\chi_\e) : (u_\e,\chi_\e) \to (u,\chi) \text{ in }L^1(\O;\R^2) \times L^1(\O)\right\},$$
and
$$\F_{\alpha,\beta}''(u,\chi):=\inf\left\{\limsup_{\e \to 0}\F_\e(u_\e,\chi_\e) : (u_\e,\chi_\e) \to (u,\chi) \text{ in }L^1(\O;\R^2) \times L^1(\O)\right\}.$$
As usual in $\Gamma$-convergence, the proof is achieved by combining a compactness result to identify the domain of finiteness of the $\Gamma$-limit, a lower bound $(a)$ and an upper bound $(b)$ inequality, {\it i.e.}
$$
\F_{\alpha,\beta} \underset{(a)}{\leq } \F'_{\alpha,\beta} \leq \F''_{\alpha,\beta} \underset{(b)}{\leq} \F_{\alpha,\beta}. 
$$

\medskip
Our compactness result, Proposition \ref{prop:domains}, rests on both compactness results in $BD(\O)$ if $\alpha\in (0,\infty]$ and in $GSBD$ (see \cite{DM2}) if $\beta \in (0,\infty]$. Given a sequence $\{ (u_\e,\chi_\e)  \}_{\e>0}$ with uniformly bounded energy such that $u_\e \to u$ strongly in $L^1(\O;\R^2)$, one can directly infer that
$$
\int_\O \chi_\e \lesssim \e \F_\e(u_\e,\chi_\e) \to 0 \quad \text{as } \e \searrow 0
$$
and, provided $\alpha >0$, $\norme{ e(u_\e) }_{L^1(\O;\Msd)} \lesssim_\alpha \F_\e(u_\e,\chi_\e)$. As when $\beta >0$, one can apply \cite[Theorem 11.3]{DM2} to the modified function $v_\e :=(1 - \chi_\e) u_\e$ which formally consists of putting the value zero inside each triangle $T$ where the symmetric gradient of $u_\e$ is "large". Although this might create a jump on the boundary of $T$, its perimeter remains controled by the energy as it can be estimated by $\LL^2(T)/\e$. It leads to compactness (in measure) for the modified displacements $v_\e$ and to the fact that $u \in L^1(\O;\R^2) \cap GSBD^2(\O)$.

\medskip
In order to describe our arguments of proof for $(a)$ and $(b)$, let us assume for simplicity that $\kappa=1$, and $\O = (0,1)^2$.
It turns out that the most relevant regimes are $\eta_\e \ll \e \sim h_\e$ and $\eta_\e \sim \e \sim h_\e$. Indeed, if $\e \ll h_\e$ or $\e \ll \eta_\e$, formally the
damaged set is so small that the limit model happens to be of pure elasticity type (see Proposition \ref{prop:linear elasticity}). This can be easily seen in the one-dimensional case, when $\e \ll h_\e$, as 
$$h_\e^{-1} \norme{\chi_\e}_{L^1(\O)} = \# \{ T \, : \, \chi_\e = 1 \text{ in } T \} \leq \e h_\e^{-1} \F_\e(u_\e,\chi_\e) \to 0,
$$
so that $\chi_\e \equiv 0$ for $\e>0$ sufficiently small.  
On the other hand, if $\eta_\e \ll \e$ and $h_\e \ll \e$, the elastic energy associated with
the damaged material is so negligible that one can approximate a jump with zero energy and we obtain a trivial $\Gamma$-limit (see Proposition \ref{prop:trivial regime}), as illustrated in Figure \ref{fig:jump zero energy} in dimension one. While the proof of the regime of brittle fracture $\eta_\e \ll h_\e \sim \e$ actually consists of a trivial generalization of the approximation of the Griffith functional performed in \cite{BB} (see Proposition \ref{prop:brittle fracture}).
\begin{figure}[hbtp]
\begin{tikzpicture}[line cap=round,line join=round,>=triangle 45,x=3.cm,y=3.cm]
\clip(-1.1,-0.25) rectangle (3,0.75);
\draw [line width=.5pt] (0.,0.)-- (2.,0.);
\draw [line width=.5pt] (0.,0.2)-- (1.,0.2);
\draw [line width=.5pt] (1.,0.6)-- (2.,0.6);
\draw [line width=.5pt,dotted] (1.,0.6)-- (1.,0.);
\draw [line width=.5pt,dotted,color=red] (0.8,0.2)-- (1.2,0.6);
\draw [line width=.5pt,color=red] (1.2,0.6)-- (2.,0.6);
\draw [line width=.5pt,color=red] (0.8,0.2)-- (0.,0.2);
\draw [line width=.5pt] (0.9,0.3)-- (0.8980290203197886,0.3606616523045977);
\draw [line width=.5pt] (0.8980290203197886,0.3606616523045977)-- (0.9601692632349021,0.36016926323490217);
\draw [line width=.5pt,dotted] (0.8,0.2)-- (0.8,0.);
\draw [line width=.5pt,dotted] (1.2,0.6)-- (1.2,0.);
\draw (1.,0.8) node[anchor=north west] {$u = u^- \mathds{1}_{(0,1/2)} + u^+ \mathds{1}_{(1/2,1)}$};
\draw [color=red](0.31215110601187396,0.35) node[anchor=north west] {$u_\varepsilon$};
\draw [color=red](1.3549951147700228,0.32274734793937315) node[anchor=north west] {$\chi_\varepsilon = \mathds{1}_{\big(\tfrac{1-\delta_\varepsilon}{2}, \tfrac{1+\delta_\varepsilon}{2} \big)}$};
\draw [line width=.5pt] (0.8,-0.05)-- (1.2,-0.05);
\draw [line width=.5pt] (0.8,-0.05)-- (0.84,-0.03);
\draw [line width=.5pt] (0.8,-0.05)-- (0.84,-0.07);
\draw [line width=.5pt] (1.2,-0.05)-- (1.16,-0.03);
\draw [line width=.5pt] (1.2,-0.05)-- (1.16,-0.07);
\draw (0.2,-0.07) node[anchor=north west] {$\eta_\varepsilon + h_\varepsilon \ll \delta_\varepsilon = \sqrt{\varepsilon(h_\varepsilon + \eta_\varepsilon)} \ll \varepsilon$};
\draw (0.49,0.4638509770313844) node[anchor=north west] {$\tfrac{u^+ - u^-}{\delta_\varepsilon}$};
\draw (-0.9,0.7) node[anchor=north west] {${\mathcal F}_\varepsilon (u_\varepsilon,\chi_\varepsilon) = \delta_\varepsilon \left(  \tfrac{ \eta_\varepsilon \lvert \left[ u \right] \rvert^2}{2\delta_\varepsilon^2} + \tfrac{1}{\varepsilon} \right) \to 0$};
\end{tikzpicture}
\caption{Approximation of a step function by finite elements with zero energy}
\label{fig:jump zero energy}
\end{figure}
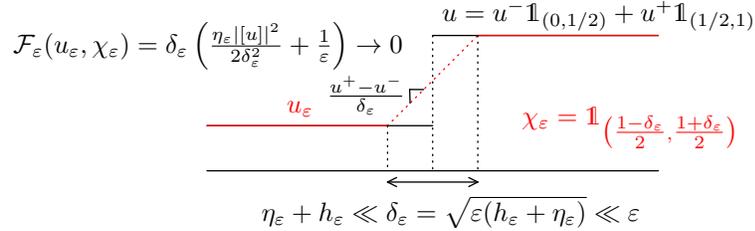

\medskip

Let us now pass to the case $\beta=0<\alpha$. Let us assume for simplicity that $\eta_\e = \e \gg h_\e$, and consider the case where $u $ is affine with $e(u) = \xi \in \Msd$. In this regime, the construction of a recovery sequence is strongly inspired by the connection between the undergoing homogenization process and the reachability of the Hashin-Shtrikman bound by laminated composite materials. More precisely, introducing the density
$$
W_\e(\xi) = \min \left( \frac{1}{\e} + \frac{\e}{2} \mathbf A_0 \xi : \xi \, ; \, \frac12 \mathbf A_1 \xi:\xi \right) = \min_{\chi \in \{0,1\}} \left\{ \chi \left( \frac{1}{\e} + \frac{\e}{2} \mathbf A_0 \xi : \xi  \right) + (1 - \chi)  \frac12 \mathbf A_1 \xi:\xi \right\}
$$
and its symetric quasiconvex envelope
$$
SQW_\e(\xi) = \inf_{\phi \in C^\infty_c ((0,1)^2;\R^2)} \int_{(0,1)^2} W_\e(\xi + e(\phi) ) \, dx,
$$
arguing as in the proof of \cite[Proposition 3.3]{BIR} (which can be easily adapted to the non-isotropic case), one can check that $$
SQW_\e(\xi) \to \bar W_\alpha(\xi) \quad \text{as } \e \searrow 0.
$$
As shown in \cite[Propositions 2.1 and 3.2]{AK93}, the computation of $SQW_\e(\xi)$ is in direct connection with the so-called Hashin-Shtrikman bound 
\begin{multline*}
\mathscr H_\e(\theta,\xi) := \inf_{\mathbf A_\e^* \in \G_\theta (\e \mathbf A_0;\mathbf A_1)} \mathbf A_\e^* \xi:\xi  \\
= \e \mathbf A_0 \xi:\xi + (1-\theta) \sup_{\tau \in \Msd} \left\{ 2 \xi:\tau - \left( \mathbf A_1 - \e \mathbf A_0 \right)^{-1} \tau:\tau - \frac{\theta}{\e} g(\tau) \right\}, 
\end{multline*}
where the $\G$-closure set $ \G_\theta (\e \mathbf A_0;\mathbf A_1)$ is the set of all composite materials resulting (in terms of $H$-convergence) from fine mixtures of weak and strong materials $\e \mathbf A_0$ and $\mathbf A_1$ in volume proportion $\theta$ and $1-\theta$ respectively. Indeed,
$$
SQW_\e(\xi) = \inf_{\theta \in [0,1]} \left\{ \frac{\theta}{\e} + \frac12 \mathscr H_\e (\theta,\xi) \right\}
$$
and, by strict convexity, the above minimization is actually reached at a unique $\theta_\e \in [0,1]$. Next using the reachability of the Hashin-Shtrikman bound, one can find a laminated material $\mathbf A_\e^* \in \G_{\theta_\e}(\e \mathbf A_0;\mathbf A_1)$ such that
$$
SQW_\e(\xi) = \frac{\theta_\e}{\e} + \frac12  \mathbf A_\e^* \xi;\xi.
$$
More precisely, one can find $p_\e \in \llbracket 1, 4 \rrbracket$ directions of lamination $b_{i,\e} \in \mathbb S^1$ and volume proportions $\theta_{i,\e} \in [0,1]$, such that $\theta_\e = 1 - \prod_{i=1}^{p_\e} (1 - \theta_{i,\e})$, so that the material $\mathbf A_\e^*$ is obtained by iteratively homogenizing the strong phase $\mathbf A_1$ with the previously homogenized phase $\mathbf A_{i,\e}$ in the direction $b_{i+1,\e}$, in proportions $1 - \theta_{i+1,\e}$ and $\theta_{i+1,\e}$ respectively, where $\mathbf A_{0,\e} = \e \mathbf A_0$. In other words, after an extraction procedure, one can find a sequence of characteristic functions $\chi_{k,\e} \in L^\infty(\O;\{ 0, 1\})$ such that $\{\chi_{k,\e} = 1 \} $ has a laminated structure, in the sense that it is the union of layers of width $\theta_{i,\e}/k$ periodically distributed along the direction $b_{i,\e}$,
$
\chi_{k,\e} \overset{*}{\rightharpoonup} \theta_\e $ weakly* in $ L^\infty(\O)$, and
$$
\mathbf A_{k,\e} := \e \chi_{k,\e}   \mathbf A_0 + (1 - \chi_{k,\e}  ) \mathbf A_1  \quad \text{$H$-converges to $\mathbf A_\e^* $ as $ k \to \infty$.}
$$
In particular, the solutions $u_{k,\e} \in u + H^1_0(\O;\R^2)$ of
$
 {\rm div} \left( \mathbf A_{k,\e} e(u_{k,\e}) \right) =  {\rm div } \left( \mathbf A_\e^* \xi \right)
$
converge weakly to $ u$ in $H^1(\O;\R^2)$ and 
$$
\E_{k,\e} := \int_\O \left(   \tfrac{1}{\e}\chi_{k,\e}  + \tfrac{\e}{2} \mathbf A_{k,\e} e(u_{k,\e}):e(u_{k,\e}) \right)  \, dx \to \LL^2(\O) SQW_\e (\xi) \quad \text{as } k \to \infty.
$$
A naive and tempting idea then consists of a diagonal extraction $k_\e \to \infty$ as $\e \searrow 0$, such that $u_\e := u_{k_\e,\e}$ and $\chi_\e := \chi_{k_\e,\e}$ satisfy $(u_\e,\chi_\e) \to (u,0)$ strongly in $L^1(\O;\R^2) \times L^1(\O)$ and 
$$
\E_{k_\e,\e} \to \F_{\alpha,\beta}(u,0)
$$
as $\e \searrow 0$. The drawbacks of this approach are twofold. First, although $\chi_\e$ is piecewise constant on a laminated structure, we have no control on the widths of the layers $\theta_\e /k_\e$ which could be too small with respect to $h_\e$. Second, the displacement $u_\e$ cannot be piecewise affine in dimension two (due to the boundary conditions and continuity accross the interfaces $\partial \{ \chi_\e = 1 \}$), so that $(u_\e,\chi_\e)$ is not admissible. In dimension one, the first issue is easy to by-pass, by directly injecting the right scales in the laminated structure. Indeed, in this case we have $\mathscr H_\e(\theta,\xi) = \left( {\theta}{(\e  a_0)^{-1}} + (1 - \theta) a_1^{-1} \right)^{-1} \xi^2$. Hence,
$$ SQW_\e(\xi) = \frac{ \theta_\e}{\e} + \frac12 \left( \frac{\theta_\e}{\e  a_0} + \frac{1 - \theta_\e}{ a_1} \right)^{-1} \xi^2 
$$
where $\theta_\e = 0$ if $a_1 \lvert \xi \rvert \leq \sqrt{2  a_0}$ and 
$$
\theta_\e =
\e \sqrt{  \frac{a_1 a_0}{2 (a_1 - \e a_0)}} \left( \lvert \xi \rvert - \sqrt{   \frac{2  a_0}{a_1 (a_1 - \e a_0)}} \right) \gg h_\e 
$$
otherwise. Setting $m_\e = \left\lfloor \theta_\e/ h_\e \right\rfloor $ and $ l_\e = m_\e^{-1} >0$, we define the characteristic function $\chi_\e=\varrho_\e({\cdot}/{l_\e})$, where $\varrho_\e \in L^\infty(\R)$ is the $1$-periodic function equal to $ 1$ in $( 0, \theta_\e)$ and $ 0$ in $(\theta_\e,1)$, as well as the laminated material $a_\e := \e a_0 \chi_\e + a_1 (1 - \chi_\e)$ and the displacement $u_\e \in u + H^1_0((0,1))$ solution of $ (a_\e u'_\e)' = 0$ in $(0,1)$ (see Figure \ref{fig:subdivision h_eps}). Then, one can check that:
$$
u_\e  = u + \frac{(a_1 - \e a_0) \xi}{\theta_\e a_1 + (1 - \theta_\e)\e a_0}  \sum_{i=0}^{m_\e - 1} \left(  (1 - \theta_\e) \left(  \cdot - {\scriptstyle{i  l_\e}} \right) \mathds{1}_{( 0,  \theta_\e    ) }(\tfrac{\cdot }{l_\e}- i )  
  + \theta_\e \left({\scriptstyle{ (i+1) l_\e}} - \cdot \right) \mathds{1}_{ (\theta_\e, 1 )  } (\tfrac{\cdot }{l_\e}- i ) \right)
$$
is piecewise affine, $\norme{\chi_\e}_{L^1((0,1))} = \theta_\e \to 0$, $\norme{u_\e - u}_{L^1((0,1))} \leq \lvert \xi \rvert l_\e \to 0$ and $(u_\e, \chi_\e) \in X_{h_\e}((0,1))$. 
Moreover, $a_\e u'_\e \equiv \left( {\theta_\e}{(\e a_0)^{-1}} + (1 - \theta_\e)a_1^{-1} \right)^{-1} \xi$ is constant, so that
$
\F_\e(u_\e,\chi_\e) =  SQW_\e(\xi )\to \F_{\alpha,\beta}(u,0)
$ when $\e \searrow 0$.

\begin{figure}[hbtp]
\begin{tikzpicture}[line cap=round,line join=round,>=triangle 45,x=0.6cm,y=0.6cm]
\clip(-2,-0.9) rectangle (10,3.8);
\draw [line width=.5pt] (0.,0.)-- (8.,0.);
\draw [line width=.5pt] (0.5,1.)-- (1.998370714816295,1.2055416207163243);
\draw [line width=.5pt] (0.,0.5)-- (0.5,1.);
\draw [line width=.5pt] (2.4951926242495035,1.6954179582001478)-- (1.998370714816295,1.2055416207163243);
\draw [line width=.5pt] (2.4951926242495035,1.6954179582001478)-- (3.99983648789484,1.9111042114121408);
\draw [line width=.5pt] (3.99983648789484,1.9111042114121408)-- (4.493128443476388,2.3992362581550726);
\draw [line width=.5pt] (4.493128443476388,2.3992362581550726)-- (5.992096353089828,2.5922186952394872);
\draw [line width=.5pt] (5.992096353089828,2.5922186952394872)-- (6.488144594342397,3.0982042347982035);
\draw [line width=.5pt] (6.488144594342397,3.0982042347982035)-- (7.988616812086046,3.318088238697907);
\draw [line width=.5pt,dashed,domain=-1:9] plot(\x,{(--3.982179958524007--2.8145335556187643*\x)/7.964359917048014});
\draw (-0.2,-0.1) node[anchor=north west] {$0$};
\draw (7.8,-0.1) node[anchor=north west] {$1$};
\draw [line width=.5pt, dotted, opacity=0.3] (2.4951926242495035,0) -- (2.4951926242495035,1.6954179582001478);
\draw [line width=.5pt, dotted, opacity = 0.3] (4.493128443476388,-0.) -- (4.493128443476388,0.15);
\draw [line width=.5pt, dotted, opacity = 0.3] (4.493128443476388,0.5) -- (4.493128443476388,2.3992362581550726);

\draw [line width=.5pt, dotted, opacity = 0.3] (6.488144594342397,-0.) -- (6.488144594342397,3.0982042347982035);
\draw [line width=.5pt, dotted, opacity = 0.3] (0.5,-0.) -- (0.5,1.);
\draw (8.4,3.6) node[anchor=north west] {$u$};
\draw (-1,3.) node[anchor=north west] {${{u_{\varepsilon} \in u + H^1_0((0,1))}}$};
\draw (1.6,0.85) node[anchor=north west] {$i l_\varepsilon$};
\draw (1.7,-0.15) node[anchor=north west] {$(i + \theta_\varepsilon) l_\varepsilon$};
\draw (3.2,0.85) node[anchor=north west] {$(i+1) l_\varepsilon$};
\begin{scriptsize}
\draw [fill] (2.,0.) circle (1.5pt);
\draw [fill] (4.,0.) circle (1.5pt);
\draw [fill] (6.,0.) circle (1.5pt);
\draw [fill] (8.,0.) circle (1.5pt);
\draw [fill] (0.,0.) circle (1.5pt);
\draw [fill, opacity=0.3] (2.4951926242495035,0) circle (1.5 pt);
\draw [fill, opacity=0.3] (4.493128443476388,0) circle (1.5 pt);
\draw [fill, opacity=0.3] (6.488144594342397,0) circle (1.5 pt);
\draw [fill, opacity=0.3] (0.5,0) circle (1.5 pt);

\end{scriptsize}
\end{tikzpicture}
\caption{Recovery sequence in the case of an affine displacement, in dimension one.}
\label{fig:subdivision h_eps}
\end{figure}
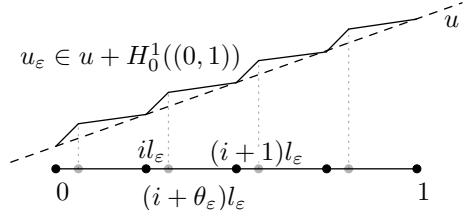
\noindent

While the issues mentionned above are trivially by-passed in dimension one, overcoming them in the vectorial setting is much more technical. Essentially, we cannot directly use the laminated structures and displacements inherited from the $H$-convergence described above. Instead, we will explicitly construct a laminated structure directly at the right scales, that mimics the properties of the optimal (but non admissible) structure $\mathbf A_{k_\e,\e}$ that are effectively needed to pass to the limit as $\e \searrow 0$. Unfortunatly, with this constructive proof, we are not able to deal with a general stiffness $\mathbf A_0$, and we concede to the case where $\mathbf A_0$ is isotropic (see Proposition \ref{prop:hencky plasticity}). 

\medskip

It finally remains the regime where $\alpha , \beta \in (0,+\infty)$. Let us assume that $h_\e = \e = \eta_\e$ for simplicity, $\mathbf A_0 = \mathbf A_1 = \text{\rm id}$, and fix $u \in SBD^2(\O)$. Here, the most difficult part lies in the derivation of the upper bound. In order to recover the surface energy, whose density is 
$$ \phi(\left[ u \right]) = \begin{cases} 
								\sin \theta_0 + \tfrac{1}{2 \sin \theta_0} \lvert \left[ u \right] \odot \nu_u \rvert^2 & \text{if } \lvert \left[u\right] \odot \nu_u \rvert \leq \sqrt{2} \sin \theta_0, \\
								\sqrt{2} \lvert \left[ u \right] \odot \nu_u \rvert & \text{otherwise, }
							\end{cases} $$
one has to take the amplitude of the effective jump into account, when it comes to the construction of the optimal triangulation. To illustrate this, let us consider the step function $u \in SBV^2(\O)$ defined by $u \equiv u^-$ in $(0,1/2)\times(0,1)$ and $u \equiv u^+ \neq u^- $ in $(1/2,1) \times (0,1)$, and let us consider the triangulation $\widetilde{\mathbf T}_\e$ introduced in \cite{CDM,BC} for the Mumford-Shah functional (also used here and in \cite{BB} for the regime of brittle fracture), as well as the Lagrange interpolation $\tilde u_\e$ of the values of $u$ at the vertices of $\widetilde{\mathbf T}_\e$ and $\tilde \chi_\e$ the characteristic function of the neigbhorhood of $J_u $ as illustrated in Figure \ref{fig:intro wrong recovery}. Then, one gets 
$$
\F_\e(\tilde u_\e, \tilde \chi_\e) = \int_\O \left(  \frac{\e \tilde \chi_\e }{2} \lvert e(\tilde u_\e ) \rvert^2 + \frac{\tilde \chi_\e}{\e} \right)\, dx =  \left(\sin \theta_0 + \frac{\lvert \left[ u \right] \odot \nu \rvert^2}{2 \sin \theta_0} \right)\HH^1(J_u) .
$$
If $\lvert \left[ u \right] \odot \nu \rvert \leq \sqrt{2} \sin \theta_0$, this construction is giving the right upper bound. Yet, for a large amplitude of jump $\lvert \left[ u \right] \odot \nu \rvert > \sqrt{2} \sin \theta_0$, this construction yields a too large upper bound. While if instead of $\widetilde{\mathbf T}_\e$, we had used $\widehat{\mathbf T}_\e$ together with the associated Lagrange interpolation $\hat u_\e$ and characteristic function $\hat \chi_\e$ as illustrated in Figure \ref{fig:intro wrong recovery}, then we would have get
$$
\F_\e(\hat u_\e, \hat \chi_\e) = \int_\O \left(   \frac{\e \hat \chi_\e }{2} \lvert e(\hat u_\e ) \rvert^2 + \frac{\hat \chi_\e}{\e} \right)\, dx =  \left(l_\e + \frac{\lvert \left[ u \right] \odot \nu \rvert^2}{2l_\e} \right)\HH^1(J_u) = \sqrt{2} \lvert \left[ u \right] \odot \nu \rvert
$$
as $l_\e = \e \sqrt{2}^{-1} \lvert \left[ u \right] \odot \nu \rvert$. In other words, in the general situation where $\left[ u \right]$ might not be constant along its jumpset, one has to adapt the construction of the triangulation according to the amplitude of the effective jump: greater the amplitude of the jump is, further away from the jump set must the interpolation points be (see Figure \ref{fig:recovery cohesive fracture}).

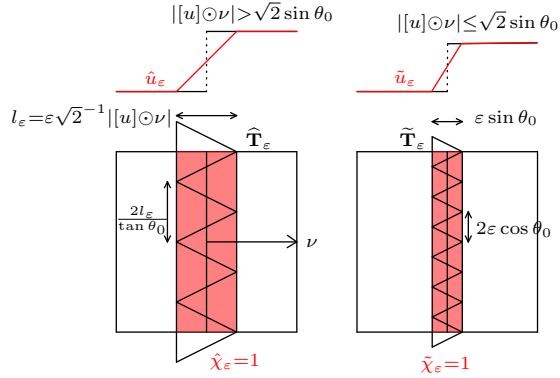
\begin{figure}[hbtp]
\centering
\begin{tikzpicture}[line cap=round,line join=round,>=triangle 45,x=.4cm,y=0.4cm]
\clip(-11.5996567352469,-1.8) rectangle (11.29908840217944,11.3);

\draw [line width=.5pt] (-7.,6.)-- (-1.,6.);
\draw [line width=.5pt] (-7.,0.)-- (-1.,0.);
\draw [line width=.5pt] (-7.,0.)-- (-7.,6.);
\draw [line width=.5pt] (-1.,6.)-- (-1.,0.);
\fill [color=red,opacity=0.5] (-5,0)--(-5,6)--(-3,6)--(-3,0)--(-5,0)--cycle;

\draw [line width=.5pt] (1.,6.)-- (7.,6.);
\draw [line width=.5pt] (1.,0.)-- (7.,0.);
\draw [line width=.5pt] (1.,0.)-- (1.,6.);
\draw [line width=.5pt] (7.,6.)-- (7.,0.);

\fill [color=red,opacity=0.5] (3.5,0)--(3.5,6)--(4.5,6)--(4.5,0)--(3.5,0)--cycle;

\draw [line width=.5pt] (-4.,6.)-- (-4.,0.);
\draw [line width=.5pt] (4.,6.)-- (4.,0.);

\draw [line width=0.5pt] (-5,-1)--(-5,7);
\draw [line width=0.5pt] (-3,0)--(-3,6);

\draw [line width=.5pt] (-5.,5.)-- (-3.,6.);
\draw [line width=.5pt] (-3.,6.)-- (-5.,7.);
\draw [line width=.5pt] (-5.,5.)-- (-3.,4.);
\draw [line width=.5pt] (-5.,3.)-- (-3.,4.);
\draw [line width=.5pt] (-3.,2.)-- (-5.,3.);
\draw [line width=.5pt] (-5.,1.)-- (-3.,2.);
\draw [line width=.5pt] (-3.,0.)-- (-5.,1.);
\draw [line width=.5pt] (-5.,-1.)-- (-3.,0.);

\draw [line width=0.5pt] (-5.3,5)--(-5.3,3);
\draw [line width=0.5pt] (-5.3,5)--(-5.4,4.8);
\draw [line width=0.5pt] (-5.3,5)--(-5.2,4.8);
\draw [line width=0.5pt] (-5.3,3)--(-5.4,3.2);
\draw [line width=0.5pt] (-5.3,3)--(-5.2,3.2);
\draw (-7.35,4.5) node[anchor=north west] {$ \scriptstyle{\frac{2 l_\e}{ \tan \theta_0}}$};

\draw [line width=0.5pt] (4.7,4)--(4.7,3);
\draw [line width=0.5pt] (4.7,4)--(4.8,3.8);
\draw [line width=0.5pt] (4.7,4)--(4.6,3.8);
\draw [line width=0.5pt] (4.7,3)--(4.8,3.2);
\draw [line width=0.5pt] (4.7,3)--(4.6,3.2);
\draw (4.65,4) node[anchor=north west] {$\scriptstyle{2 \e\cos \theta_0}$};

\draw [line width=.5pt] (3.5,6.5)-- (3.5,5.5);
\draw [line width=.5pt] (3.5,5.5)-- (4.5,6.);
\draw [line width=.5pt] (4.5,6.)-- (3.5,6.5);
\draw [line width=.5pt] (3.5,5.5)-- (3.5,-0.5);
\draw [line width=.5pt] (4.5,6.)-- (4.5,0);
\draw [line width=.5pt] (3.5,5.5)-- (4.5,5.);
\draw [line width=.5pt] (4.5,5.)-- (3.5,4.5);
\draw [line width=.5pt] (3.5,4.5)-- (4.5,4.);
\draw [line width=.5pt] (4.5,4.)-- (3.5,3.5);
\draw [line width=.5pt] (3.5,3.5)-- (4.5,3.);
\draw [line width=.5pt] (4.5,3.)-- (3.5,2.5);
\draw [line width=.5pt] (3.5,2.5)-- (4.5,2.);
\draw [line width=.5pt] (4.5,2.)-- (3.5,1.5);
\draw [line width=.5pt] (3.5,1.5)-- (4.5,1.);
\draw [line width=.5pt] (4.5,1.)-- (3.5,0.5);
\draw [line width=.5pt] (3.5,0.5)-- (4.5,0.);
\draw [line width=.5pt] (4.5,0.)-- (3.5,-0.5);
\draw [line width=.5pt] (-7.,8.)-- (-4.,8.);
\draw [line width=.5pt] (-4.,10.)-- (-1.,10.);
\draw [line width=.5pt,dotted] (-4.,10.)-- (-4.,8.);
\draw [line width=.5pt] (1.,8.)-- (4.,8.);
\draw [line width=.5pt,dotted] (4.,8.)-- (4.010409845348733,9.590284748736297);
\draw [line width=.5pt] (4.010409845348733,9.590284748736297)-- (6.988609138924466,9.6205202745594);
\draw [line width=.5pt,color=red] (-5.,8.)-- (-3.,10.);
\draw [line width=.5pt,color=red] (3.4812881434444147,8.)-- (4.470981151606304,9.594960599561247);
\draw [line width=.5pt,color=red] (-3.,10.)-- (-1.,10.);
\draw [line width=.5pt,color=red] (-7.,8.)-- (-5.,8.);
\draw [line width=.5pt,color=red] (1.,8.)-- (3.4812881434444147,8.);
\draw [line width=.5pt,color=red] (4.470981151606304,9.594960599561247)-- (6.988609138924466,9.6205202745594);
\draw [line width=.5pt] (-4.,3.)-- (-1.,3.);
\draw [line width=.5pt] (-1.,3.)-- (-1.3108818319024143,3.2173683837453746);
\draw [line width=.5pt] (-1.,3.)-- (-1.322947219178344,2.8071452163637702);
\draw (-1,3.4) node[anchor=north west] {$\scriptstyle{\nu}$};
\draw [color=red](-4.2,-0.4) node[anchor=north west] {$\scriptstyle{\hat \chi_\varepsilon = 1}$};
\draw [color=red](-6.364146877036259,9) node[anchor=north west] {$\scriptstyle{\hat u_\varepsilon}$};
\draw (-5.5,11.3) node[anchor=north west] {$\scriptstyle{\lvert \left[ u \right] \odot \nu \rvert > \sqrt2 \sin \theta_0}$};
\draw (-3.,7.1) node[anchor=north west] {$\scriptstyle{\widehat{\mathbf T}_\varepsilon}$};

\draw (-10.8,7.9) node[anchor=north west] {$\scriptstyle{l_\e = \varepsilon \sqrt{2}^{-1} \lvert \left[ u \right] \odot \nu \rvert}$};
\draw [line width=.5pt] (-5.,7.2)-- (-3.,7.2);
\draw [line width=.5pt] (-3.,7.2)-- (-3.2,7.1);
\draw [line width=.5pt] (-3.,7.2)-- (-3.2,7.3);
\draw [line width=.5pt] (-5.,7.2)-- (-4.8,7.1);
\draw [line width=.5pt] (-5.,7.2)-- (-4.8,7.3);
\draw (2.,11) node[anchor=north west] {$\scriptstyle{\lvert \left[ u \right] \odot \nu \rvert \leq \sqrt2 \sin \theta_0}$};
\draw [color=red](1.9210034036337795,9) node[anchor=north west] {$\scriptstyle{\tilde u_\varepsilon}$};
\draw (2.1,7.1) node[anchor=north west] {$\scriptstyle{\widetilde{ \mathbf T}_\varepsilon}$};
\draw [color=red](2.8,-0.5) node[anchor=north west] {$\scriptstyle{\tilde \chi_\varepsilon = 1}$};

\draw [line width=.5pt] (3.5,7.)-- (4.5,7.);
\draw [line width=.5pt] (4.5,7.)-- (4.3,7.1);
\draw [line width=.5pt] (4.5,7.)-- (4.3,6.9);
\draw [line width=.5pt] (3.5,7.)-- (3.7,7.1);
\draw [line width=.5pt] (3.5,7.)-- (3.7,6.9);
\draw (4.6,7.7) node[anchor=north west] {$\scriptstyle{\varepsilon \sin \theta_0}$};
\end{tikzpicture}
\caption{Adaptation of the mesh according to the amplitude of the effective jump}
\label{fig:intro wrong recovery}
\end{figure}

As for the lower bound, the key argument allowing to recover the dissipative part of the energy is that a minimum scale for spatial oscillations is imposed. As this is explicitly visible in the one dimensional case, we briefly sketch it. Given a sequence $(u_\e, \chi_\e) \in X_\e(\O)$ with bounded energy
$$
\sup_{\e } \F_\e(u_\e,\chi_\e)  =: M < \infty,
$$
converging to $ (u,0) $ in $ L^1(\O) \times L^1(\O) $, and adapted to a subdivision $0 = x_0^\e < ... < x_{n_\e}^\e = 1 $ in the sense that $u_\e$ is affine and $\chi_\e$ is constant on $(x_i^\e, x_{i+1}^\e)$, with $\e \leq x_{i+1}^\e - x_i^\e \leq \omega(\e)$, 
one can find $m_\e \in \N$ points $0 \leq b^\e_j < c^\e_j \leq b_{j+1}^\e \leq 1$ among the subdivision points $\{ x_i^\e\}_i$ such that 
$$D_\e := \left\{ \chi_\e = 1 \right\} = \bigcup_{j=1}^{m_\e} \big( b_j^\e, c_j^\e \big) \quad \text{and} \quad 
M \geq \frac{1}{\e} \int_0^1 \chi_\e \, dx  \geq  m_\e .
$$
In particular, up to a subsequence (not relabeled), we can assume that $m_\e \equiv m $ is constant. We directly turn to the case $m \geq 1$. The idea is to modify $ u_\e$ on the subintervals where its variation is large, as it will be energetically favorable to create a jump at their end points.
Without loss of generality, we can also assume that  
$
 \chi_\e = \mathds{1}_{   {\left\{   \lvert u_\e' \rvert^2 \geq   \e  \lvert u_\e' \rvert^2 + 2 \e^{-1} \right\} }}
$
so that $D_\e \subset \left\{ u_\e' \neq 0 \right\}$. Hence, the $ SBV^2(\O)$ functions
$$
w_\e := \begin{cases}
		u_\e & \quad \text{in } \O \setminus D_\e, \\
		u_\e(b_j^\e) & \quad \text{in } (b_j^\e, c_j^\e) \text{ for all } j \in \llbracket 1,m \rrbracket,
		\end{cases} 
$$
satisfy $ w_\e' = (1- \chi_\e) u_\e'$, $  J_{w_\e} = \{ c_j^\e \}_{1\leq j \leq m}$ and $\F_\e(u_\e,\chi_\e) = \ds \int_0^1 \tfrac12 |w_\e'|^2 \, dx + \sum_{j=1}^{m} f_j^\e(l_j^\e) $, where 
$$
l_j^\e := c_j^\e - b_j^\e \geq \e, \quad f_j^\e(l_j^\e) := \frac{\e \lvert \left[ w_\e\right](c_j^\e)\rvert^2}{2 l_j^\e} + \frac{l_j^\e}{\e} \quad \text{and} \quad \left[ w_\e \right](c_j^\e)  = u_\e(c_j^\e) - u_\e(b_j^\e) \neq 0.
$$
As mentionned above, the fact that $l_j^\e$ is larger than $\e$ is fundamental, as it ensures that 
$
f_j^\e(l_j^\e) \geq \phi \left( \left[w_\e \right](c_j^\e) \right) 
$
for all $j$. Indeed if $\sqrt{2} \geq  \lvert \left[ w_\varepsilon \right](c_j^\varepsilon) \rvert  $, then $f_j^\e(l_j^\e) \geq f_j^\e(\e) = \lvert \left[ w_\e\right](c_j^\e)\rvert^2/2 + 1$ as $f_j^\e$ is non-decreasing on $\big[ \varepsilon \lvert \left[ w_\varepsilon \right](c_j^\varepsilon) \rvert / \sqrt{2} ,\infty \big)$. Otherwise, Young's inequality yields $f_j^\e(l_j^\e) \geq \sqrt{2} \lvert \left[ w_\varepsilon \right](c_j^\varepsilon) \rvert  $.
Hence, $\F_\e(u_\e,\chi_\e) \geq   \F_{\alpha,\beta}(w_\e,0)$. 
Noticing that
$$
\int_0^1 (\lvert w'_\e \rvert^2 + \left\lvert w_\e \right\rvert) \, dx + \sum_{j=1}^m \lvert \left[ w_\e \right] \rvert (c_j^\e) + \HH^0(J_{w_\e}) \lesssim \F_\e(u_\e,\chi_\e) + m + \norme{u_\e}_{L^1(\O)}
$$
is uniformly bounded, Ambrosio's compactness Theorem \cite[Theorem 4.7]{AFP} entails that $w_\e \rightharpoonup u $ weakly-* in $ BV(\O)$ as $ \e \searrow 0$ and we conclude by lower semi-continuity of $\F_{\alpha,\beta}(\cdot, 0)$ for the weak-* convergence in $BV(\O)$ (see \cite[Theorem 5.4]{AFP}).

Althought very simple, this argument cannot be generalized to the vectorial setting by means of a slicing method, as the lengths of the triangles' sections can be arbitrarily small. Hopefully, most of the hardwork for the vectorial case has already been made in the derivation of the lower bound for the jump part of the Griffith functional, in the regime of fracture. Indeed, using first a blow-up method, we can reduce to the case where $u$ is a step function taking two distinct values $u^+,u^- \in \R^2$ on $(0,1/2)\times (0,1)$ and $(1/2,1)\times(0,1)$ respectively, $u_\e \wto u $ weakly* in $ BD(\O)$ and $ (1 - \chi_\e) e(u_\e) \to 0 $ strongly in $ L^2(\O;\Msd)$. Then, because the mesh-size is of order $\e$, \cite[Proposition 2.3]{BB} entails that
$$
\sin \theta_0 \leq \liminf_{\e \searrow 0} \int_\O \frac{\chi_\e}{\e}=: M.
$$
Hence, up to a subsequence, using respectively the lower-semicontinuity of the norm for the weak*-convergence in $\M(\O;\Msd)$ and Cauchy-Schwarz' inequality, we get
$$
 \lvert \left[ u \right] \odot \nu_u \rvert \leq \liminf_{\e} \int_\O \lvert e(u_\e) \rvert \, dx = \liminf_\e \int_\O \chi_\e \lvert e(u_\e) \rvert \, dx 
\leq \sqrt{M} \liminf_\e  \norme{ \chi_\e \sqrt\e \lvert e(u_\e)\rvert}_{L^2(\O)},
$$
so that, using the fact that $M \geq \sin \theta_0$ once more, 
$$
\liminf_\e \F_\e(u_\e,\chi_\e) = \liminf_\e \int_\O  \chi_\e \left(\frac1\e + \frac\e2  \lvert e(u_\e) \rvert^2 \right)  \, dx \geq M + \frac{  \lvert \left[u\right] \odot \nu_u \rvert^2  }{M} \geq   \F_{\alpha,\beta}(u,0).
$$

\subsection{Organization of the paper}
The proof of Theorem \ref{thm:1} will be divided into five steps, corresponding to the five possible regimes for $\alpha$ and $\beta$. In section 2, we collect useful notation and preliminary results that will be useful in the subsequent sections. Sections \ref{sec:linear elasticity}, \ref{sec:trivial regime} and \ref{sec:brittle fracture} are devoted to showing the regime of linear elasticity ($\alpha=\infty$ or $\beta=\infty$, Proposition \ref{prop:linear elasticity}), the trivial regime ($\alpha=\beta=0$, Proposition \ref{prop:trivial regime}) and the regime of brittle fracture ($\alpha=0$ and $\beta \in (0,\infty)$, Proposition \ref{prop:brittle fracture}) respectively. The proofs of these regimes cause no particular difficulty, as the first two consist of trivial adaptations of the results of \cite[Theorems 5.1 and 4.1]{BIR} to the discrete setting, while the regime of brittle fracture is a direct consequence of \cite[Theorem 1.3]{BB}. Section \ref{sec:hencky plasticity} is devoted to the proof of the regime of Hencky plasticity ($\alpha \in (0,\infty)$ and $\beta=0$, Proposition \ref{prop:hencky plasticity}). Eventually, in section \ref{sec:intermediate}, we prove the intermediate regime in between plasticity and fracture ($\alpha,\beta \in (0,\infty)$), Proposition \ref{prop:intermediate}.

\section{Notation and preliminaries}

\subsection{Notation}\label{subsec:notation}
\noindent \textbf{Matrices.} If $a$ and $b \in \R^N$, with $N \in \N \setminus \{ 0 \}$, we write $a \cdot b = \sum_{i=1}^N a_i b_i$ for the Euclidean scalar product and $\left\lvert a \right\rvert = \sqrt{a \cdot a}$ for the corresponding norm. We use the notation $\mathbb S^{N-1}$ for the unit sphere in $\R^N$, and denote by $e_i  \in \R^N$ the $i^{\text{\it th}}$ vector of the canonical basis of $\R^N$. The space of symmetric $N \times N$ matrices is denoted by $\Ms$ and is endowed with the Frobenius scalar product $\xi : \eta = {\rm Tr} (\xi \eta)$ and the corresponding norm $\left\lvert \xi \right\rvert = \sqrt{\xi : \xi}$. 

\medskip

\noindent \textbf{Measures.} The Lebesgue and $k$-dimensional Hausdorff measures in $\R^N$ are respectively denoted by $\LL^N$ and $\HH^k$. If $X$ is a borel subset of $\R^N$ and $Y$ is an Euclidean space, we denote by $\M(X;Y)$ the space of $Y$-valued bounded Radon measures in $X$ which, according to the Riesz Representation Theorem, can be identified to the dual of $C_0(X;Y)$ (the closure of $C_c(X;Y)$ for the sup-norm in $X$). The weak-* topology of $\M(X;Y)$ is defined using this duality. The indication of the space $Y$ is omitted when $Y = \R$. For $\mu \in \M(X;Y)$, its total variation is denoted by $|\mu|$ and we denote by $\mu = \mu^a + \mu^s$ the Radon-Nikod\'ym decomposition of $\mu $ with respect to Lebesgue, where $\mu^a$ is absolutely continuous and $\mu^s$ is singular with respect to the Lebesgue measure $\LL^N$.

\medskip

\noindent \textbf{Functional spaces.}  We use standard notation for Lebesgue and Sobolev spaces. 
If $U$ is a bounded open subset of $\R^N$, we denote by $L^0(U;\R^m)$ the set of all $\LL^N$-measurable functions from $U$ to $\R^m$. We recall some properties regarding functions with values in a Banach space and refer to \cite{FL,Br,DMDSM} for details and proofs on this matter. If $Y$ is a Banach space and $T > 0$, we denote by $AC([0,T];Y)$ the space of absolutely continuous functions $f : [0,T] \to Y$. If $Y$ is the dual of a separable Banach space $X$, then every function $f \in AC([0,T];Y)$ is such that the weak-* limit 
$$\dot f(t) = \text{{\it w*-}}\lim_{s \to t} \frac{f(t) - f(s)}{t-s} \in Y$$
exists for $\LL^1$-a.e. $t \in [0,T]$, $\dot f :[0,T] \to Y$ is weakly-* measurable and $t \mapsto \Vert \dot f(t) \Vert_Y \in L^1([0,T])$. If $f : [0,T] \times \R \to \R$ is a function of two variables, time and spatial derivatives will be respectively denoted by $\dot f$ and $f'$.

We recall that a sequence $\{g_k\}_{k \in \N}$ in $L^0(U;\R^m)$ converges in measure to $g \in L^0(U;\R^m)$ if for all $\e > 0$, 
$$\LL^N \left( \left\{ x \in U: \, \left\lvert g_k(x) - g(x) \right\rvert > \e \right\} \right)\to 0.$$

\medskip

\noindent\textbf{Functions of bounded variation.} We refer to \cite{AFP} for an exhaustive treatment on that subject and just recall few notation. Let $U \subset \R^N$ be a bounded open set. A function $u \in L^1(U;\R^m)$ is a {\it function of bounded variation} in $U$, and we write $u \in BV(U;\R^m)$, if its distributional derivative $Du$ belongs to $\M(U;\mathbb M^{m \times N})$. We use standard notation for that space, referring to \cite{AFP} for details. We just recall that a function $u$ belongs to $SBV^2(U;\R^m)$ if $u \in SBV(U;\R^m)$ (the distributional derivative $Du$ has no Cantor part), its approximate gradient $\nabla u$ belongs to $L^2(U;\mathbb M^{m\times N})$ and its jump set $J_u$ satisfies $\HH^{N-1}(J_u)<\infty$.

If $U$ has Lipschitz boundary, every function $u \in BV(U;\R^m)$ has an inner trace on $\partial U$ (still denoted by $u$ and $\HH^{N-1}$-integrable on $\partial U$) and there exists a constant $C > 0$ depending only on $U$ such that 
\begin{equation}\label{eq:norme BV}
\frac1C \norme{u}_{BV(U;\R^m)} \leq \lvert Du \rvert (U) + \int_{\partial U} \left\lvert u \right\rvert \, d \HH^{N-1} \leq C \norme{u}_{BV(U;\R^m)}
\end{equation}
according to \cite[Proposition 2.4, Remark 2.5 (ii)]{Temam}.

\medskip

\noindent\textbf{(Generalized) functions of bounded deformation.} A function $u \in L^1(U;\R^N)$ is a {\it function of bounded deformation}, and we write $u \in BD(U)$, if its distributional symmetric gradient $Eu:=(Du+Du^T)/2$ belongs to $\M(U;\Ms)$. We refer to \cite{Suquet,Temam,ST1,ACDM,BCDM} for the main properties and notation of that space. The space $SBD^2(U)$ is made of all functions $u \in SBD(U)$ ($Eu$ has no Cantor part) such that the approximate symmetric gradient $e(u)$ (the absolutely continuous part of $Eu$ with respect to $\LL^N$) belongs to $L^2(U;\mathbb M^{N \times N}_{\rm sym})$ and its jump set $J_u$ satisfies $\HH^{N-1}(J_u)<\infty$.

\medskip

We now recall the definition and the main properties of the space of {\it generalized functions of bounded deformation} introduced in \cite{DM2}. We first need to introduce some notation. Let $\xi \in \mathbb S^{N-1}$, we denote by $\Pi_\xi:=\{y \in \R^N : \; y \cdot \xi=0\}$ the orthogonal space to $\xi$ and by $p_\xi$ the orthogonal projection onto $\Pi_\xi$. For every set $B \subset \R^N$, we define  for $\xi \in \mathbb S^{N-1}$ and $y \in \R^N$, 
$$B_y^\xi:=\{t \in \R : \; y+t\xi \in B\}, \quad B^\xi:=p_\xi(B)$$
and, for every (vector-valued) function $u:B \to \R^N$ and (scalar-valued) function $f:B \to \R$,
$$u_y^\xi(t):=u(y+t\xi)\cdot\xi, \quad f_y^\xi(t)=f(y+t\xi)\quad \text{ for all } y \in \R^N \text{ and all } t \in B_y^\xi.$$

\begin{defn}
Let $U \subset \R^N$ be a bounded open set and $u \in L^0(U;\R^N)$. Then, $u \in GBD(U)$ if there exists a nonnegative measure $\lambda \in \M(U)$ such that one of the following equivalent conditions holds true for every $\xi \in \mathbb S^{N-1}$:
\begin{enumerate}
\item for every $\tau \in C^1(\R)$ with $-\frac12 \leq \tau \leq \frac12$ and $0 \leq \tau' \leq 1$, the partial derivative $D_\xi(\tau(u \cdot \xi))=D(\tau(u\cdot\xi))\cdot \xi$ belongs to $\M(U)$, and
$$|D_\xi(\tau(u\cdot\xi))|(B) \leq \lambda(B) \quad \text{ for every Borel set }B \subset U;$$

\item $u_y^\xi \in BV_{\rm loc}(U_y^\xi)$ for $\HH^{N-1}$-a.e. $y \in U^\xi$, and$$\int_{\Pi_\xi} \left(|Du_y^\xi|(B_y^\xi \setminus J_{u_y^\xi}^1) + \HH^0(B_y^\xi \cap J_{u_y^\xi}^1) \right)d\HH^{N-1}(y)\leq \lambda(B)\quad \text{ for every Borel set }B \subset U,$$
where $J_{u_y^\xi}^1:=\{t \in J_{u_y^\xi} : \; |[u_y^\xi](t)|\geq 1\}$.
\end{enumerate}
The function $u$ belongs to $GSBD(U)$ if $u \in GBD(U)$ and $u_y^\xi \in SBV_{\rm loc}(U_y^\xi)$ for every $\xi \in \mathbb S^{N-1}$ and for $\HH^{N-1}$-a.e. $y \in U^\xi$.
\end{defn}

Every $u \in GBD(U)$ has an approximate symmetric gradient $e(u) \in L^1(U;\Ms)$ such that for  every $\xi \in \mathbb S^{N-1}$ and for $\HH^{N-1}$-a.e. $y \in U^\xi$,
$$e(u)(y+t\xi)\xi \cdot \xi=(u_y^\xi)'(t) \quad \text{ for $\LL^1$-a.e. $t \in U_y^\xi$.}$$
Moreover, the jump set $J_u$ of $u \in GBD(U)$, defined as the set of all $x_0 \in U$  for which there exist $(u^+(x_0),u^-(x_0),\nu_u(x_0)) \in \R^N \times \R^N \times  \mathbb S^{N-1}$ with $u^+(x_0) \neq u^-(x_0)$ such that the function
$$y\in B_1  \mapsto u_{x_0,\varrho}:=u(x_0+\varrho y)$$
converges in measure in $B_1$ as $\varrho \searrow 0$ to
$$y \in B_1 \mapsto
\begin{cases}
u^+(x_0) & \text{ if }y \cdot \nu_u(x_0)>0,\\
u^-(x_0) & \text{ if }y \cdot \nu_u(x_0)\leq 0,
\end{cases}$$
is countably $(\HH^{N-1},N-1)$-rectifiable. Finally, the energy space $GSBD^2(U)$ is defined as
$$GSBD^2(U):=\{u \in GSBD(U): \; e(u) \in L^2(U;\Ms), \, \HH^{N-1}(J_u)<\infty\}.$$

\medskip

\noindent\textbf{Homogenization and $H$-convergence} We refer to \cite{Allaire} for an exhaustive presentation of these notions and only recall minimal results. We denote, for fixed $\alpha,\beta >0$, the subset of fourth-order symmetric tensors
$$
\mathcal A_{\alpha,\beta} = \left\{ A \in \R^{N^4} \, : \, A_{ijkl}=A_{klij}=A_{jikl}, \, \alpha \left\lvert \xi \right\rvert^2 \leq A \xi:\xi \leq \beta \left\lvert \xi \right\rvert^2 \text{ for all } \xi \in \Ms \right\}.
$$
Let $\O$ be a bounded open set of $\R^N$. We say that $A_n \in L^\infty(\O;\mathcal A_{\alpha,\beta} )$ $H$-converges to $A \in L^\infty (\O;\mathcal A_{\alpha,\beta})$ if, for every $f \in H^{-1}(\O;\R^N)$, the solutions $u_n \in H^1_0(\O;\R^N)$ of the equilibrium equations 
$$\begin{cases}
 - {\rm div} \big( A_n e(u_n) \big) = f \quad \text{in } \O \\
 u_n = 0 \quad \text{on } \partial \O
 \end{cases}
$$
are such that $u_n \rightharpoonup u $ weakly in $ H^1_0(\O;\R^N) $ and $A_n e(u_n) \rightharpoonup A e(u) $ weakly in $ L^2(\O;\Ms)$ as $n \nearrow + \infty$, where $u \in H^1_0(\O;\R^N)$ is the solution of 
$$ \begin{cases}
- {\rm div} \big( A e(u) \big) = f \quad \text{in } \O \\
u = 0 \quad \text{on } \partial \O.
\end{cases}
$$
Given a volume fraction $\theta \in L^\infty(\O;[0,1])$ and $A,B \in L^\infty(\O;\mathcal A_{\alpha,\beta})$ with $A \leq B$ as quadratic forms on $\Ms$, the $\G$-closure set 
$$\G_\theta (A,B)$$
is defined as the set of all possible $H$-limits of $\chi_n A + (1 - \chi_n)B$ where $\chi_n \in L^\infty(\O;\{ 0,1 \})$ weakly-* converges to $\theta$ in $L^\infty(\O;[0,1])$.

\medskip

\noindent \textbf{Convex analysis.} We recall some definition and standard results from convex analysis (see \cite{Rockafellar}).
Let $f : \R^N \to [0,+ \infty]$ be a proper function ({\it i.e.} not identically $+ \infty$). The convex conjugate of $f$ is defined as 
$$f^* (x) = \sup_{y \in \R^N} \, \left\{ x \cdot y - f(y) \right\} $$
which turns out to be convex and lower semicontinuous. If $f$ is convex and finite, we define its recession function as 
$$f^\infty(x) = \lim_{t \nearrow + \infty} \, \frac{f(tx)}{t} \in [0, + \infty] $$
which is convex and positively $1$-homogeneous. If $f,g : \R^N \to [0,+ \infty]$ are proper convex functions, then their infimal convolution is defined as 
$$f \infconv g (x) = \inf_{y \in \R^N} \, \left\{ f(x-y) + g(y) \right\} $$
which is convex as well. It can be shown that
$$
f \infconv g = (f^* + g^*)^*.
$$
The indicator function of a set $C \subset \R^N$ is defined as $I_C = 0$ in $ C$ and $ + \infty$ otherwise. The convex conjugate $I_C^*$ of $I_C$ is called the support function of $C$.

\subsection{Domain of the $\Gamma$-limit}
We begin our analysis by identifying the domain of finiteness of the $\Gamma$-limit, according to the values of $\alpha$ and $\beta$.

\begin{prop}\label{prop:domains}
Let $\{ \e_k \}_{k\in \N}$ satisfying $\e_k \searrow 0$, $u\in L^1(\O;\R^2)$ and $\{ (u_k,\chi_k) \}_{k\in \N} \subset L^1(\O;\R^2) \times L^1(\O)$ be such that $M:=\limsup_k \F_{\e_k}(u_k,\chi_k) < \infty$ and $u_k \to u$ strongly in $L^1(\O;\R^2)$. Then, 
$$\chi_k \to 0 \text{ strongly in } L^1(\O)$$
and:
\begin{itemize}[leftmargin=4.cm,label=-]
\item if $\alpha = \infty$ or $\beta=\infty$, then $u \in H^1(\O;\R^2)$;
\item if $\beta=0<\alpha<\infty$, then $u \in BD(\O)$;
\item if $\alpha=0<\beta<\infty$, then $u \in L^1(\O;\R^2) \cap GSBD^2(\O)$;
\item if $\alpha$ and $\beta \in (0,\infty)$, then $u \in SBD^2(\O)$.
\end{itemize}
\end{prop}
\begin{rem}
As one can check in the following proof, this compactness result holds for general stiffness tensors $\mathbf A_0$ and $ \mathbf A_1$ satisfying \eqref{eq:rigidity tensors}, without additionally requiring that $\mathbf A_0$ is isotropic.
\end{rem}

\begin{proof} 
Up to a subsequence (not relabeled), we can assume that 
$$
\sup_k \F_{\e_k}(u_k,\chi_k) <\infty \quad \text{and} \quad \lim_k \F_{\e_k}(u_k,\chi_k) = M < \infty.
$$
By definition of $\F_{\e_k}$, we infer from the energy bound that $(u_k,\chi_k) \in X_{h_{\e_k}}(\O)$ for all $k$ and there exists $\mathbf T_k \in \T_{h_{\e_k}}(\O)$ such that $u_k$ is affine and $\chi_k$ is constant on each triangle $T \in \mathbf T_k$. In particular,
$$
\int_\O \chi_k \, dx \leq \frac{M}{\kappa} \e_k \to 0 \quad \text{as } k \nearrow \infty,
$$
so that 
$$\chi_k \to 0 \text{ strongly in } L^1(\O).$$

\medskip
Moreover, if $\alpha \in (0,\infty]$, arguing as in the proof of \cite[Lemma 2.3]{BIR}, we infer that there exists a constant $c=c(\alpha)\in (0,\infty)$ such that
$$
c\int_\O \lvert e(u_\e)\rvert \, dx - \frac{\LL^2(\O)}{c} \leq \int_\O \min \left( \sqrt{2a_0\kappa \eta_\e \e^{-1}} \lvert e(u_\e) \rvert \, ; \, \frac{a_1}{2} \lvert e(u_\e) \rvert^2 \right) \, dx \leq \F_\e(u_\e,\chi_\e)
$$
where we also used Young's inequality and \eqref{eq:rigidity tensors}, so that
\begin{equation*}
 \sup_k \int_\O \lvert e(u_k) \rvert \, dx   < \infty.
\end{equation*}
Therefore, we infer that 
\begin{equation}\label{eq:alpha>0}
u \in BD(\O) \, \text{ and } \, u_k \overset{*}{\rightharpoonup} u \text{ weakly* in }BD(\O), \, \text{ provided }0<\alpha\leq \infty.
\end{equation}

\medskip
On the other hand, introducing the displacements $v_k = (1 - \chi_k) u_k\in SBV^2(\O;\R^2)$ and the damaged sets 
$$
D_k = \underset{ {\scriptstyle T \in \mathbf T_k \, :  \, {\chi_k }_{|T} \equiv 1}}{\bigcup} T
$$
which satisfy $\nabla v_k = (1 - \chi_k) \nabla u_k $ and $J_{v_k} \subset \O \cap \partial D_k$, we infer by the coercivity property \eqref{eq:rigidity tensors} of $\mathbf A_1$ that 
\begin{equation}\label{eq:bound volume v_k}
\sup_{k} \int_\O \left\lvert e(v_k) \right\rvert^2  \, dx \leq \frac{2 M}{a_1} < + \infty.
\end{equation}
We then consider an exhaustion of $\O$ by a sequence of smooth open subsets $\{U_m\}_{m \in \N}$ satisfying $U_m \subset\subset U_{m+1} \subset\subset \O$ for all $m \in \N$ and $\bigcup_m U_m = \O$. In particular, for all $m \in \N$, there exists $k_m \in \N$ such that for all $k \geq k_m$, 
$$
U_m \cap \partial D_k \subset  \underset{ {\scriptstyle T \in \mathbf T_k, \, T \cap U_m \neq \emptyset}}{\bigcup} \partial T \subset \O.
$$
Noticing that $\HH^1(\partial T) \leq 6 \LL^2(T)/(h_{\e_k} \sin \theta_0)$ for all $T \in \mathbf T_k$, we infer that 
\begin{equation}\label{eq:bound jump set}
\HH^1 \big(J_{v_k} \cap U_m \big) \leq \frac{6}{h_{\e_k} \sin \theta_0} \int_\O \chi_k \, dx \leq \frac{6M}{ \sin \theta_0} \frac{\e_k}{h_{\e_k}} 
\end{equation}
for all $m \in \N$ and all $k \geq k_m$. In particular, if $\beta \in (0,\infty]$, as $u_k \to u$ strongly in $L^1(\O;\R^2)$, \cite[Theorem 11.3]{DM2} entails that $u_{|U_m} \in GSBD^2(U_m)$ and, up to a subsequence (not relabeled, depending on $m$),
$$
e(v_k)_{|U_m} \rightharpoonup e(u_{|U_m}) \quad \text{weakly in } L^2(U_m;\Msd)
$$
and
$$
\HH^1 \big(J_{u_{|U_m}} \big) \leq \liminf_k \HH^1(J_{v_k} \cap U_m)  .
$$
By uniformity of the bounds \eqref{eq:bound jump set} and \eqref{eq:bound volume v_k}, arguing as in the proof of \cite[Proposition 3.1]{BB} we infer that $u \in GSBD^2(\O)$ and, using a diagonal extraction argument, we can find a subsequence (not relabeled, independent of $m$) such that  
\begin{equation}\label{eq:beta>0}
e(v_k) \rightharpoonup e(u) \text{ weakly in } L^2(\O;\Msd) \, \text{ and } \, \HH^1(J_u) \leq \frac{6M}{ \beta \sin \theta_0} < \infty, \, \text{ provided } \beta \in (0,\infty].
\end{equation}

\medskip
Next, it remains to prove the ad hoc regularity of $u$ with respect to the corresponding regimes.
\begin{itemize}[leftmargin=*,label=-]
\item If $\alpha = \infty$, it follows from the lower bound inequality of \cite[Theorem 5.1]{BIR} that $u \in H^1(\O;\R^2)$. If $\beta = \infty$, we infer from \eqref{eq:beta>0} that $u \in GSBD^2(\O)$ and $\HH^1(J_u) = 0$. Therefore, one can actually check that $u \in H^1(\O;\R^2)$. Indeed, due to the Generalized area formula \cite[Theorem 2.91]{AFP}, we have for all $\xi \in \Sn^1$:
$$
0 = \HH^1 \left( J_u^\xi \right) = \int_{ J_u ^\xi} \left\lvert \nu_u \cdot \xi \right\rvert \, d \HH^1 = \int_{\Pi_\xi} \# \left( \big( J_u^\xi \big)^\xi_y \right) \, d \HH^1(y)
$$
where $ J_u ^\xi := \{ x \in J_u \, : \, \xi \cdot \left( u^+(x) - u^-(x) \right) \neq 0 \}$. In particular, \cite[Theorem 8.1]{DM2} entails that 
$$J^1_{u^\xi_y} := \left\{ x \in J_{u^\xi_y} \, : \, \left\lvert (u^\xi_y)^+ - (u^\xi_y)^- \right\rvert \geq 1 \right\} \subset J_{u^\xi_y} = \emptyset $$
are empty for $\HH^1$-a.e. $y \in \Pi_\xi$. Moreover, since $u \in L^1(\O)$, we get that $u^\xi_y \in L^1((\O)^\xi_y)$ for $\HH^1$-a.e. $y \in \Pi_\xi$ as well. Therefore, according to \cite[Definition 4.2]{DM2}, we infer that for $\HH^1$-a.e. $y \in \Pi_\xi$,
$$
u^\xi_y \in SBV_{\rm loc}( (\O)^\xi_y)  \quad \text{and} \quad  J_{u^\xi_y} = \emptyset = J^1_{u^\xi_y} .
$$
Thus, \cite[Definition 4.1 (b)]{DM2} entails 
$$
\int_{\Pi_\xi} \lvert D u^\xi_y \rvert \big( (\O)^\xi_y \big) \, d \HH^1 (y) < + \infty
$$
so that $u^\xi_y \in BV((\O)^\xi_y)$ for $\HH^1$-a.e. $y \in \Pi_\xi$. As $D \big(u^\xi_y \big) \in \M  \big((\O)^\xi_y \big)$ and $\lvert D^s \big(u^\xi_y \big) \rvert (I) = 0$ for all open set $I \subset \subset (\O)^\xi_y$, we infer that $D u^\xi_y $ has no Cantor part and no jump part in $(\O)^\xi_y$, so that $u^\xi_y \in W^{1,1} \big((\O)^\xi_y \big)$. Hence, \cite[Proposition 3.2 and Theorem 4.5]{ACDM} entail that $u \in LD(\O) \cap GSBD^2(\O)$. Finally, we deduce that
$$
u \in H^1(\O; \R^2) 
$$
thanks to Korn-Poincaré's inequality (see \cite[Theorem 1.1]{CCF}). 

\item If $\beta=0 < \alpha < \infty$, we infer from \eqref{eq:alpha>0} that $u \in BD(\O)$.
\item If $\alpha=0 < \beta < \infty$, we infer from \eqref{eq:beta>0} that $u \in L^1(\O;\R^2) \cap GSBD^2(\O)$.

\item Eventually, if $\alpha \in (0,\infty)$ and $\beta \in (0,\infty)$, we infer from \eqref{eq:alpha>0} and \eqref{eq:beta>0} that $u \in BD(\O)\cap GSBD^2(\O)$, so that $u \in SBD^2(\O)$. Indeed, by \cite[Definition 4.2]{DM2}, $u^\xi_y \in SBV_{\text{\rm loc}}(\O^\xi_y)$ for all $\xi \in \mathcal{S}^1$ and $\HH^1$-a.e. $y \in \Pi_\xi$, that is $u^\xi_y \in L^1_{\text{\rm loc}}(\O^\xi_y)$ and $D(u^\xi_y) \in \M(\O^\xi_y)$ has no Cantor part. On the other hand, since $u \in L^1(\O)$, so is $u^\xi_y \in L^1(\O^\xi_y)$ for all $\xi \in \mathcal{S}^1$ and $\HH^1$-a.e. $y \in \Pi_\xi$. Therefore, \cite[Proposition 4.7]{ACDM} entails that $u \in SBD^2(\O)$.
\end{itemize}
\end{proof}

\section{Linear elasticity} \label{sec:linear elasticity}
\begin{prop}\label{prop:linear elasticity}
If $\alpha=\infty$ or $\beta=\infty$, then the functional $\F_\e$ $\Gamma$-converges for the strong $L^1(\O;\R^2) \times L^1(\O)$-topology to the functional
$\F_{\alpha,\beta}:L^1(\O;\R^2) \times L^1(\O) \to [0,+\infty]$ defined by
$$\F_{\alpha,\beta}(u,\chi)=
\begin{cases}
\ds  \frac12 \int_\O \mathbf A_1 e(u):e(u)\, dx & \text{ if }
\begin{cases}
\chi=0\text{ a.e. in }\O,\\
u \in H^1(\O;\R^2),
\end{cases}\\
+\infty & \text{ otherwise.}
\end{cases}$$
\end{prop}

\begin{proof} {\it Lower bound.} Let $(u,\chi) \in L^1(\O;\R^2) \times L^1(\O)$ be such that $\F'_{\alpha,\beta}(u,\chi) <\infty$. According to Proposition \ref{prop:domains}, we infer that $u \in H^1(\O;\R^2)$ and $\chi=0$. If $\alpha=\infty$, as $X_{h_\e}(\O) \subset H^1(\O;\R^2) \times L^\infty(\O;\{0,1\})$, the result follows from \cite[Theorem 5.1]{BIR} which entails that 
$$\F'_{\alpha,\beta} (u,0)\geq \frac12   \int_\O \mathbf A_1 e(u):e(u)\, dx=\F_{\alpha,\beta}(u,0), \, \text{ for all }\beta \in [0,\infty] .$$
Next if $\beta=\infty$ and $\alpha \in [0,\infty]$ is arbitrary, for all $\delta >0$, one can find a subsequence $\e_k \to 0$ as $k \to \infty$ and $(u_k,\chi_k) \in X_{h_{\e_k}}(\O)$ such that $(u_k,\chi_k) \to (u,0)$ strongly in $L^1(\O;\R^2) \times L^1(\O)$ and the following limit exists and satisfies
$$\lim_k \F_{\e_k}(u_k,\chi_k)  \leq \F'_{\alpha,\beta}(u,0) + \delta <\infty.$$
Arguing as in the proof of Proposition \ref{prop:domains}, up to a further subsequence (not relabeled), $v_k= (1-\chi_k)u_k \in SBV^2(\O;\R^2)$ is such that $ e(v_k) \wto e(u) \quad \text{weakly in } L^2(\O;\Msd) $. By lower semicontinuity of the norm, 
$$\delta + \F'_{\alpha,\beta}(u,0)\geq \liminf_{k \to \infty} \frac12 \int_\O \mathbf A_1 e(v_k):e(v_k)\, dx \geq \frac12 \int_\O \mathbf A_1 e(u):e(u)\, dx=\F_{\alpha,\beta}(u,0)$$
and the conclusion follows by taking the limit as $\delta \searrow 0$.

\medskip
\noindent {\it Upper bound.} Let $u \in H^1(\O;\R^2)$. Since $\O$ has a Lipschitz boundary, we can extend $u$ and assume that $u \in H^1(\R^2;\R^2)$. By lower semicontinuity of $\F''_{\alpha,\beta}$ and density of $C^\infty_c (\R^2;\R^2)$ in $H^1(\R^2;\R^2)$, we can also assume without loss of generality that $u \in C^\infty_c(\R^2;\R^2)$. Then, for all $\e >0$, let us fix any triangulation $\mathbf T_\e \in \T_{h_\e}(\O)$ and let $u_\e \in H^1(\O;\R^2)$ be the Lagrange interpolation of the values of $u$ at the vertices of $\mathbf T_\e$. According to \cite[Theorem 3.1.5]{Ciarlet}, there exists a constant $C(\theta_0)>0$ such that $\norme{u_\e - u}_{H^1(T;\R^2)} \leq C \omega(h_\e) \norme{D^2u}_{L^2(T)}$ for all $\e > 0$ and $T \in \mathbf T_\e$. Hence, $(u_\e, 0) \in X_{h_\e}(\O)$ and $u_\e \to u$ strongly in $ H^1(\O;\R^2)$, so that
$$
\F_\e(u_\e,0) = \frac12 \int_\O \mathbf A_1 e(u_\e):e(u_\e)\, dx \to \frac12 \int_\O \mathbf A_1 e(u):e(u)\, dx,
$$
which shows that $\F''_{\alpha,\beta}(u,0) \leq \F_{\alpha,\beta}(u,0)$.
\end{proof}

\section{Trivial regime} \label{sec:trivial regime}
\begin{prop}\label{prop:trivial regime}
Assume that $\theta_0 \leq 45^\circ$. If $\alpha=\beta=0$, then the functional $\F_\e$ $\Gamma$-converges for the strong $L^1(\O;\R^2) \times L^1(\O)$-topology to the functional
$\F_{\alpha,\beta}:L^1(\O;\R^2) \times L^1(\O) \to [0,+\infty]$ defined by
$$\F_{\alpha,\beta}(u,\chi)=
\begin{cases}
0 & \text{ if $\chi=0$ a.e. in $\O$,}\\
+\infty & \text{ otherwise.}
\end{cases}$$
\end{prop}

\begin{proof}
According to Proposition \ref{prop:domains}, it is enough to identify the $\Gamma$-limit for $u \in L^1(\O;\R^2)$ and $\chi=0$. The proof of the lower bound is then straightforward, as $\F'_{\alpha,\beta}(u,0) \geq 0$. Let us prove the upper bound inequality. By a rescaling and translation argument, we can assume without loss of generality that $\O$ is the unit cube $Q:= (0,1)^2$. Next, arguing as in the proof of \cite[Theorem 4.1]{BIR}, we can assume that $u$ is a piecewise constant function of the form
$$u=\sum_{i\in \llbracket 0, N-1 \rrbracket^2} u_i \mathds{1}_{Q_i}$$
for some constant vectors $u_i \in \R^2$ and $Q_i:=\tfrac1N(i+Q)$ for all $i \in \llbracket 0, N-1 \rrbracket^2$. For all $\e >0$, we define the parameter $\delta_\e := \sqrt{\e \left( \eta_\e + h_\e \right)} >0$ which satisfies $\eta_\e \ll \delta_\e \ll \e \quad \text{and} \quad h_\e \ll \delta_\e \ll \e$. Assume that $\e>0$ is small enough so that $\delta_\e<(2N)^{-1}$ and define $N_\e := \left\lfloor (\delta_\e N)^{-1} \right\rfloor \geq 2 $ and $l_\e := (N_\e N)^{-1}$, so that $\delta_\e \leq l_\e \leq 2 \delta_\e$. Denoting by $Q_M := \tfrac{1}{M} Q$ the cube of side length $M^{-1} >0$, we then subdivide the cube $Q_N $ into $N_\e^2$ subcubes of side length $l_\e$, so that
$$
Q_N = \underset{j \in \llbracket 0,N_\e - 1 \rrbracket^2}{\bigcup} \left( \tfrac{j}{N_\e N} + Q_{N_\e N} \right) \quad \text{and} \quad Q = \underset{\begin{array}{c}
			{\scriptstyle i \in \llbracket 0,N - 1 \rrbracket^2 }\\
			{\scriptstyle j \in \llbracket 0,N_\e - 1 \rrbracket^2}
		  \end{array}}{\bigcup} \left( \tfrac{i}{N} + \tfrac{j }{N_\e N} + Q_{N_\e N} \right)
$$
up to a $\LL^2$-negligible set. These $(N_\e N)^2$ rectangles are then divided into two isoceles right triangles with edges of length $l_\e$ and $\sqrt 2 l_\e$. In particular, this defines a triangulation $\mathbf T_{\delta_\e} \in \T_{\delta_\e}(Q)$. 

\medskip
Let us now construct the recovery sequence. We first consider the cut-off function $\phi_N : \R^2 \to [0,1]$ defined as the Lagrange interpolation of the values $0$ at the vertices on the boundary $\partial Q_N$ and $1$ at the vertices inside $Q_N$, which we extend by $0$ outside $\bar Q_N$. One can check that $\phi_N$ is continuous, affine on each triangle $T \in \mathbf T_{\delta_\e}$ with $\norme{\nabla \phi_N}_{L^\infty} \leq C/\delta_\e$ for some constant $C(\theta_0)\in (0,\infty)$, and satisfies
$$
\phi_N \equiv 0 \quad \text{ in } \R^2 \setminus Q_N \quad \text{and} \quad \phi_N \equiv 1 \quad \text{ in } Q^\e_N := \left[ l_\e, \tfrac1N - l_\e \right]^2
$$
as illustrated in Figure \ref{fig:subdivision Q}. Defining the damaged set and the displacement 
$$
D_\e = \underset{i \in \llbracket 0,N-1 \rrbracket^2}{\bigcup} \left( \tfrac{i}{N} + Q_N \setminus Q^\e_N \right), \quad
u_\e = \underset{i \in \llbracket 0,N-1 \rrbracket^2}{\sum} u_i \phi_N \left( \cdot - \tfrac{i}{N} \right),
$$
as well as the characteristic function $\chi_\e = \mathds{1}_{D_\e}$, one can check that $u_\e \in C^0(Q;\R^2)$, $u_\e$ is affine and $\chi_\e$ is constant on each triangle $T \in \mathbf T_{\delta_\e}$. Let us further remark that $(u_\e,\chi_\e) \in X_{h_\e}(Q)$ is admissible. Indeed, as the mesh-size $\delta_\e \gg h_\e$, we can recursively subdivide the triangulation $\mathbf T_{\delta_\e}$ in $n_\e$ steps as illustrated in Figure \ref{fig:subdivision Q}, where
$$n_\e = \max \left\{ n \in \N \, : \, h_\e/\delta_\e \leq \sqrt 2^{-n} \right\},$$
in order to obtain an admissible triangulation $\mathbf T_\e \in \T_{h_\e}(Q)$. In particular, as $u_\e$ is still affine and $\chi_\e$ constant on each triangle of $\mathbf T_\e$, we infer that $(u_\e,\chi_\e) \in X_{h_\e}(Q)$. Moreover, one can check that
$$
\norme{\chi_\e}_{L^1(Q)} = \LL^2(D_\e) \leq 8 N \delta_\e  \to 0 \quad \text{and} \quad
\norme{u - u_\e}_{L^1(Q;\R^2)}  \leq \norme{u}_{L^\infty(Q;\R^2)} \, \LL^2 (D_\e) \to 0
$$
when $\e \searrow 0$. Finally, as $e(u_\e) = \underset{ i \in \llbracket 0, N-1 \rrbracket^2}{\sum} u_i \odot \nabla \phi_N \left( \cdot - \tfrac{i}{N} \right)
$ is supported in $D_\e $, we have
$$
\F_\e(u_\e,\chi_\e)  \leq C \big( \theta_0, a'_0, u, \kappa \big) \left( \frac{\eta_\e}{\delta^2_\e} + \frac{1}{\e} \right)\LL^2(D_\e) 
 \leq C \big( \theta_0, a'_0, u, \kappa \big) \left( \frac{\eta_\e}{\delta_\e} + \frac{\delta_\e}{\e} \right)  \to 0
$$
when $\e \searrow 0$, which completes the proof of the upper bound inequality.
\end{proof}

\begin{figure}[hbtp]
\centering
\begin{tikzpicture}[line cap=round,line join=round,>=triangle 45,x=0.5cm,y=0.5cm]
\clip(-4,-0.6) rectangle (20.319541019121996,6.4);
\fill[line width=1.pt,pattern=dots,fill opacity=0.4] (11.,5.) -- (11.,0.) -- (12,0) -- (12,5) -- cycle;
\fill[line width=1.pt,pattern=dots,fill opacity=0.4] (12.,5.) -- (12.,4.) -- (16.,4.) -- (16.,5.) -- cycle;
\fill[line width=1.pt,pattern=dots,fill opacity=0.4] (15.,4.) -- (15.,0.) -- (16.,0.) -- (16.,4.) -- cycle;
\fill[line width=1.pt,pattern=dots,fill opacity=0.4] (11.,1.) -- (15.,1.) -- (15.,0.) -- (11.,0.) -- cycle;
\fill[line width=1.pt,fill opacity=0.4] (12.,4.) -- (12.,1.) -- (15.,1.) -- (15.,4.) -- cycle;
\fill[line width=0.5pt, fill opacity=0.15] (4.,3.) -- (4.,4.) -- (5.,4.) -- cycle;
\draw [line width=.5pt] (0.,0.)-- (0.,5.);
\draw [line width=.5pt] (7.,0.)-- (7.,1.);
\draw [line width=0.5pt,opacity=0.4] (11,0.)-- (11,5.);
\draw [line width=.5pt] (0.,0.)-- (5.,0.);
\draw [line width=.5pt] (5.,0.)-- (5.,5.);
\draw [line width=.5pt] (5.,5.)-- (0.,5.);
\draw [line width=.5pt] (1.,5.)-- (1.,0.);
\draw [line width=.5pt] (2.,0.)-- (2.,5.);
\draw [line width=.5pt] (3.,5.)-- (3.,0.);
\draw [line width=.5pt] (4.,5.)-- (4.,0.);
\draw [line width=.5pt] (0.,4.)-- (5.,4.);
\draw [line width=.5pt] (0.,3.)-- (5.,3.);
\draw [line width=.5pt] (0.,2.)-- (5.,2.);
\draw [line width=.5pt] (0.,1.)-- (5.,1.);
\draw [line width=.5pt] (0.,4.)-- (1.,5.);
\draw [line width=.5pt] (0.,3.)-- (2.,5.);
\draw [line width=.5pt] (0.,2.)-- (3.,5.);
\draw [line width=.5pt] (0.,1.)-- (4.,5.);
\draw [line width=.5pt] (0.,0.)-- (5.,5.);
\draw [line width=.5pt] (1.,0.)-- (5.,4.);
\draw [line width=.5pt] (2.,0.)-- (5.,3.);
\draw [line width=.5pt] (3.,0.)-- (5.,2.);
\draw [line width=.5pt] (4.,0.)-- (5.,1.);
\draw (-1.5,3) node[anchor=north west] {$Q_N$};
\draw [line width=.5pt] (2,5.3)-- (3.,5.3);
\draw [line width=.5pt] (2,5.3)-- (2.2,5.4);
\draw [line width=.5pt] (2,5.3)-- (2.2,5.2);
\draw [line width=.5pt] (3.,5.3)-- (2.8,5.4);
\draw [line width=.5pt] (3.,5.3)-- (2.8,5.2);
\draw (2.,6.3) node[anchor=north west] {$l_\varepsilon \sim \delta_\e$};
\draw [line width=.5pt] (7.,0.)-- (8.,0.);
\draw [line width=.5pt] (8.,0.)-- (8.,1.);
\fill [opacity=0.15] (7,0)--(7,1)--(8,1)--cycle;
\draw [line width=.5pt] (8.,1.)-- (7.,1.);
\draw [line width=.5pt] (7.,0.)-- (8.,1.);
\draw [line width=.5pt] (7.,1.3)-- (8.,1.3);
\draw [line width=.5pt] (7.,1.3)-- (7.2,1.4);
\draw [line width=.5pt] (7.,1.3)-- (7.2,1.2);
\draw [line width=.5pt] (8.,1.3)-- (7.8,1.4);
\draw [line width=.5pt] (8.,1.3)-- (7.8,1.2);
\draw (7.15,2.3) node[anchor=north west] {$l_\varepsilon$};
\draw [line width=.5pt] (8.5,0.)-- (9.5,1.);
\draw [line width=.5pt] (9.25,0.9)-- (9.5,1.);
\draw [line width=.5pt] (9.4,0.75)-- (9.5,1.);
\draw [line width=.5pt] (8.6,0.25)-- (8.5,0.);
\draw [line width=.5pt] (8.75,0.1)-- (8.5,0.);
\draw (8.7,0.8) node[anchor=north west] {$\sqrt 2 l_\varepsilon$};
\draw [line width=0.5pt,opacity=0.4] (11.,0.)-- (16.,0.);
\draw [line width=0.5pt,opacity=0.4] (16.,0.)-- (16.,5.);
\draw [line width=0.5pt,opacity=0.4] (16.,5.)-- (11.,5.);
\draw [line width=0.5pt,opacity=0.4] (12.,5.)-- (12.,4.);
\draw [line width=0.5pt,opacity=0.4] (13.,5.)-- (13.,4.);
\draw [line width=0.5pt,opacity=0.4] (14.,5.)-- (14.,4.);
\draw [line width=0.5pt,opacity=0.4] (15.,5.)-- (15.,4.);
\draw [line width=0.5pt,opacity=0.4] (11.,4.)-- (12.,4.);
\draw [line width=0.5pt,opacity=0.4] (11.,3.)-- (12.,3.);
\draw [line width=0.5pt,opacity=0.4] (11.,2.)-- (12.,2.);
\draw [line width=0.5pt,opacity=0.4] (11.,1.)-- (12.,1.);
\draw [line width=0.5pt,opacity=0.4] (12.,0.)-- (12.,1.);
\draw [line width=0.5pt,opacity=0.4] (13.,0.)-- (13.,1.);
\draw [line width=0.5pt,opacity=0.4] (14.,0.)-- (14.,1.);
\draw [line width=0.5pt,opacity=0.4] (15.,0.)-- (15.,1.);
\draw [line width=0.5pt,opacity=0.4] (15.,1.)-- (16.,1.);
\draw [line width=0.5pt,opacity=0.4] (15.,2.)-- (16.,2.);
\draw [line width=0.5pt,opacity=0.4] (15.,3.)-- (16.,3.);
\draw [line width=0.5pt,opacity=0.4] (15.,4.)-- (16.,4.);
\draw [line width=0.5pt,opacity=0.4] (11.,4.)-- (12.,5.);
\draw [line width=0.5pt,opacity=0.4] (12.,4.)-- (13.,5.);
\draw [line width=0.5pt,opacity=0.4] (13.,4.)-- (14.,5.);
\draw [line width=0.5pt,opacity=0.4] (14.,4.)-- (15.,5.);
\draw [line width=0.5pt,opacity=0.4] (15.,4.)-- (16.,5.);
\draw [line width=0.5pt,opacity=0.4] (11.,3.)-- (12.,4.);
\draw [line width=0.5pt,opacity=0.4] (11.,2.)-- (12.,3.);
\draw [line width=0.5pt,opacity=0.4] (11.,1.)-- (12.,2.);
\draw [line width=0.5pt,opacity=0.4] (11.,0.)-- (12.,1.);
\draw [line width=0.5pt,opacity=0.4] (12.,0.)-- (13.,1.);
\draw [line width=0.5pt,opacity=0.4] (13.,0.)-- (14.,1.);
\draw [line width=0.5pt,opacity=0.4] (14.,0.)-- (15.,1.);
\draw [line width=0.5pt,opacity=0.4] (15.,0.)-- (16.,1.);
\draw [line width=0.5pt,opacity=0.4] (15.,1.)-- (16.,2.);
\draw [line width=0.5pt,opacity=0.4] (15.,2.)-- (16.,3.);
\draw [line width=0.5pt,opacity=0.4] (15.,3.)-- (16.,4.);
\draw (12.2,2.9) node[anchor=north west] {$\phi_N \equiv 1$};
\draw (12.5,6.5567700440460195) node[anchor=north west] {$\phi_N \equiv 0$};
\fill[fill=white,fill opacity=0.8] (13.,4.1) -- (13.,4.8) -- (14,4.8) -- (14,4.1) -- cycle;
\draw (12.75,5.05) node[anchor=north west,opacity=10] {$D_{\mathbf \varepsilon}$};
\draw[line width=.5pt,dash pattern=on 5pt off 2pt,smooth,samples=100,domain=5.1:6.5] plot(\x,{(((\x)-6.6)^(2)+1)});
\draw [line width=.5pt] (6.52,1)-- (6.4,1.2);
\draw [line width=.5pt] (6.52,1)-- (6.3,0.95);
\begin{scriptsize}
\fill (11.,5.) circle (1.0pt);
\fill (11.,0.) circle (1.0pt);
\fill (12.,4.) circle (1pt);
\fill (12.,1.) circle (1pt);
\fill (12.,5.) circle (1pt);
\fill (13.,5.) circle (1pt);
\fill (13.,4.) circle (1pt);
\fill (14.,5.) circle (1pt);
\fill (14.,4.) circle (1pt);
\fill (15.,5.) circle (1pt);
\fill (11.,4.) circle (1pt);
\fill (11.,3.) circle (1pt);
\fill (12.,3.) circle (1pt);
\fill (11.,2.) circle (1pt);
\fill (12.,2.) circle (1pt);
\fill (11.,1.) circle (1pt);
\fill (12.,0.) circle (1pt);
\fill (13.,0.) circle (1pt);
\fill (13.,1.) circle (1pt);
\fill (14.,0.) circle (1pt);
\fill (14.,1.) circle (1pt);
\fill (15.,0.) circle (1pt);
\fill (16.,1.) circle (1pt);
\fill (15.,2.) circle (1pt);
\fill (16.,2.) circle (1pt);
\fill (15.,3.) circle (1pt);
\fill (16.,3.) circle (1pt);
\fill (16.,4.) circle (1pt);
\fill (16.,5.) circle (1.0pt);
\fill (15.,4.) circle (1.0pt);
\fill (15.,1.) circle (1.0pt);
\fill (16.,0.) circle (1.0pt);
\end{scriptsize}
\end{tikzpicture}
\begin{tikzpicture}[line cap=round,line join=round,>=triangle 45,x=0.4cm,y=0.4cm]
\clip(-3.,-.2) rectangle (26.,6.6);
\draw [line width=.5pt] (0.,5.)-- (0.,0.);
\draw [line width=.5pt] (0.,0.)-- (5.,0.);
\fill [opacity=0.15] (0,0)--(0,5)--(5,5) --cycle;
\draw [line width=.5pt] (5.,0.)-- (5.,5.);
\draw [line width=.5pt] (5.,5.)-- (0.,5.);
\draw [line width=.5pt] (6.,5.)-- (6.,0.);
\draw [line width=.5pt] (6.,0.)-- (11.,0.);
\draw [line width=.5pt] (11.,0.)-- (11.,5.);
\draw [line width=.5pt] (11.,5.)-- (6.,5.);
\draw [line width=.5pt] (12.,5.)-- (12.,0.);
\draw [line width=.5pt] (12.,0.)-- (17.,0.);
\draw [line width=.5pt] (17.,0.)-- (17.,5.);
\draw [line width=.5pt] (17.,5.)-- (12.,5.);
\draw [line width=.5pt] (0.,0.)-- (5.,5.);
\draw [line width=.5pt] (6.,0.)-- (11.,5.);
\draw [line width=.5pt,color=A2] (6.,5.)-- (8.5,2.5);
\draw [line width=.5pt] (12.,0.)-- (17.,5.);
\draw [line width=.5pt,color=A2] (12.,5.)-- (14.5,2.5);
\draw [line width=.5pt,color=B5] (14.5,2.5)-- (14.5,5.);
\draw [line width=.5pt,color=B5] (14.5,2.5)-- (12.,2.5);
\begin{scope}[xshift=-1cm]
\draw [line width=.5pt] (20.,5.)-- (20.,0.);
\draw [line width=.5pt] (20.,0.)-- (25.,0.);
\draw [line width=.5pt] (25.,0.)-- (25.,1.5);
\draw [line width=.5pt] (25.,2.3)-- (25.,5.);
\draw [line width=.5pt] (25.,5.)-- (20.,5.);
\draw [line width=.5pt] (20.,0.)-- (25.,5.);
\draw [line width=.5pt,color=A2] (22.5,2.5)-- (20.,5.);
\draw [line width=.5pt,color=B5] (22.5,2.5)-- (20.,2.5);
\draw [line width=.5pt,color=B5] (22.5,2.5)-- (22.5,5.);
\draw [line width=.5pt,color=B4] (22.5,5.)-- (23.75,3.75);
\draw [line width=.5pt,color=C3] (23.75,3.75)-- (22.5,3.75);
\draw [line width=.5pt,color=C3] (23.75,3.75)-- (23.75,5.);
\draw [line width=.5pt,dotted] (23.75,5.)-- (24.375,4.375);
\draw [line width=.5pt,dotted] (23.75,5.)-- (23.125,4.375);
\draw [line width=.5pt,dotted] (22.5,3.75)-- (23.125,4.375);
\draw [line width=.5pt,dotted] (22.5,3.75)-- (23.125,3.125);
\draw [line width=.5pt,color=B4] (21.25,3.75)-- (22.5,5.);
\draw [line width=.5pt,color=C3] (21.25,3.75)-- (21.25,5.);
\draw [line width=.5pt,dotted] (21.25,5.)-- (20.,3.75);
\draw [line width=.5pt,dotted] (21.25,5.)-- (22.5,3.75);
\draw [line width=.5pt,dotted] (22.5,3.75)-- (20.,1.25);
\draw [line width=.5pt,color=B4] (21.25,3.75)-- (20.,2.5);
\draw [line width=.5pt,color=B4] (20.,2.5)-- (21.25,1.25);
\draw [line width=.5pt,color=C3] (22.5,3.75)-- (20.,3.75);
\draw [line width=.5pt,color=C3] (20.,1.25)-- (21.25,1.25);
\draw [line width=.5pt,color=C3] (21.25,3.75)-- (21.25,1.25);
\draw [line width=.5pt,dotted] (20.,3.75)-- (21.875,1.875);
\draw [line width=.5pt,dotted] (20.,1.25)-- (20.625,0.625);
\draw [line width=.5pt] (22.138354647388688,1.6713041585573125)-- (22.729440747688045,2.2922430922051227);
\draw [line width=.5pt] (22.729440747688045,2.2922430922051227)-- (22.55,2.25);
\draw [line width=.5pt] (22.729440747688045,2.2922430922051227)-- (22.699587914339595,2.113126092114408);
\draw [line width=.5pt] (22.138354647388688,1.6713041585573125)-- (22.17417804740683,1.8683328586570984);
\draw [line width=.5pt] (22.138354647388688,1.6713041585573125)-- (22.323442214149093,1.701156991905765);
\draw (22.3,2.6) node[anchor=north west] {$\scriptstyle{\sqrt 2 ^{-n_\varepsilon} l_\varepsilon \in [h_\e  , \sqrt{2} h_\e]}$};
\end{scope}
\fill (17.7,2.5) circle (.5pt);
\fill (17.5,2.5) circle (.5pt);
\fill (17.3,2.5) circle (.5pt);
\draw [line width=.5pt] (0.,5.3)-- (5.,5.3);
\draw [line width=.5pt] (0.,5.3)-- (0.2,5.2);
\draw [line width=.5pt] (0.,5.3)-- (0.2,5.4);
\draw [line width=.5pt] (5.,5.3)-- (4.8,5.2);
\draw [line width=.5pt] (5.,5.3)-- (4.8,5.4);
\draw [line width=.5pt] (6.7512814343666605,0.2209355579230366)-- (8.586179323202579,2.0416094321168132);
\draw [line width=.5pt] (6.7512814343666605,0.2209355579230366)-- (6.8,0.45);
\draw [line width=.5pt] (6.7512814343666605,0.2209355579230366)-- (6.9681141030799045,0.24682947797534774);
\draw [line width=.5pt] (8.586179323202579,2.0416094321168132)-- (8.340302729605124,1.98);
\draw [line width=.5pt] (8.586179323202579,2.0416094321168132)-- (8.55,1.8069891492300463);
\draw [line width=.5pt] (12,5.355541353845178)-- (14.5,5.358131238679478);
\draw [line width=.5pt] (14.5,5.358131238679478)-- (14.3,5.458131238679478);
\draw [line width=.5pt] (14.5,5.358131238679478)-- (14.3,5.258131238679478);
\draw [line width=.5pt] (12,5.355541353845178)-- (12.2,5.455541353845178);
\draw [line width=.5pt] (12.,5.355541353845178)-- (12.2,5.255541353845178);
\draw (2.0695478554310345,6.5) node[anchor=north west] {$l_\varepsilon$};
\draw (7.5,1.7) node[anchor=north west] {$\sqrt 2^{-1} l_\varepsilon$};
\draw (12.1,6.85) node[anchor=north west] {$\sqrt 2 ^{-2} l_\varepsilon$};
\draw (0.,4.2) node[anchor=north west] {$T \in \mathbf T_{\delta_\varepsilon}$};
\end{tikzpicture}
\caption{}
\label{fig:subdivision Q}
\end{figure}

\section{Brittle fracture} \label{sec:brittle fracture}
\begin{prop}\label{prop:brittle fracture}
Assume that $\theta_0 \leq 45^\circ - \text{\rm arctan}(1/2)$. If $\alpha=0$ and $\beta\in (0,\infty)$, then the functional $\F_\e$ $\Gamma$-converges for the strong $L^1(\O;\R^2) \times L^1(\O)$-topology to the functional
$\F_{\alpha,\beta}:L^1(\O;\R^2) \times L^1(\O) \to [0,+\infty]$ defined by
$$\F_{\alpha,\beta}(u,\chi)=
\begin{cases}
\ds  \frac12 \int_\O \mathbf A_1 e(u):e(u)\, dx + \beta\kappa \sin\theta_0 \HH^1(J_u)& \text{ if }
\begin{cases}
\chi=0\text{ a.e. in }\O,\\
u \in GSBD^2(\O)\cap L^1(\O;\R^2),
\end{cases}\\
+\infty & \text{ otherwise.}
\end{cases}$$
\end{prop}
\begin{rem}
Note that the same result holds in $L^0(\O;\R^2)\times L^0(\O)$ for the topology of convergence in measure, up to enlarging the domain of finiteness of the $\Gamma$-limit to 
$$
\left\{ (u,\chi) \in L^0(\O;\R^2)\times L^0(\O) \, : \, u \in GSBD^2(\O) , \, \chi=0 \right\}.
$$ 
The proof is almost exactly the same as below and is the direct consequence of \cite[Theorem 1.3]{BB}.
\end{rem}
\begin{proof}
First note that, according to \cite[Theorem 1.3]{BB}, the functional $\G_\e : L^0(\O;\R^2) \to [0,+\infty]$ defined by
$$
\G_\e ( u ) = \begin{cases}
					\ds \int_\O \min \left( \frac12 \mathbf A_1 e(u):e(u) \, ; \, \frac{\beta \kappa}{h_\e} \right) \, dx & \quad \text{if } u \in V_{h_\e}(\O), \\
					+ \infty & \quad \text{otherwise},
			  \end{cases}
$$
$\Gamma$-converges, for the topology of convergence in measure in $L^0(\O;\R^2)$, to $\G : L^0(\O;\R^2) \to [0,+ \infty]$ defined by
$$\G(u)=
\begin{cases}
\ds  \frac12 \int_\O \mathbf A_1 e(u):e(u)\, dx + \beta\kappa \sin\theta_0 \HH^1(J_u)& \quad \text{if }
u \in GSBD^2(\O), \\
+\infty & \quad \text{otherwise,}
\end{cases}$$
where $u \in V_{h_\e} \Leftrightarrow (u,0) \in X_{h_\e}(\O)$ ({\it i.e.} we only consider the finite elements on displacements and forget about the characteristic functions). Next, according to Proposition \ref{prop:domains}, it is enough to identify the $\Gamma$-limit for $\chi = 0$ and $u \in L^1(\O;\R^2) \cap  GSBD^2(\O)$. 

\medskip
\noindent
{\it Lower bound.} For all $(u_\e, \chi_\e) \in X_{h_\e}(\O)$ converging to $(u,0)$ in $L^1(\O;\R^2) \times L^1(\O)$, one can check that $u_\e \in V_{h_\e}(\O)$ converges in measure to $u$ and $\F_\e(u_\e,\chi_\e) \geq \G_\e(u_\e)$. Hence, passing to the infimum among such sequences leads to 
$$
\F'_{\alpha,\beta}(u,0) \geq \G(u) = \F_{\alpha,\beta}(u,0).
$$

\noindent {\it Upper bound.} Using the density result for $GSBD$ functions (see \cite[Theorem 1.1]{CCdensity}) as well as the lower semicontinuity of $\F_{0,\beta}''$ with respect to the convergence in $L^1(\O;\R^2)\times L^1(\O)$, we can assume without loss of generality that $u \in SBV^2(\O;\R^2) \cap L^\infty(\O;\R^2)$. Looking at the constructive proof for the recovery sequence in \cite[Proposition 3.11]{BB}, one can check that there exists $u_\e \in V_{h_\e}(\O)$ and a triangulation $\mathbf T_{h_\e} \in \T_{h_\e}(\O)$ such that $u_\e$ is affine on each of its triangle,
\begin{gather}
\underset{\e > 0}{\sup} \, \norme{u_\e}_{L^\infty(\O;\R^2)} \leq \norme{u}_{L^\infty(\O;\R^2)}, \quad u_\e \to u \quad \text{strongly in } L^1(\O;\R^2)\label{eq:norme sup}\\
\text{and } \G_\e(u_\e) \to \G(u) \quad \text{as } \e \searrow 0. \label{eq:G eps beta to G beta}
\end{gather}
(Beware that this requires that $\theta_0 \leq \Theta_0 := 45^\circ- \text{\rm arctan}(1/2)$ as the triangulation $\mathbf T_{h_\e}$ might have triangles with angles equal to $\Theta_0$.) Since $u_\e$ is affine on each triangle $T \in \mathbf T_{h_\e}$, the fact that the length of every edge of $T$ is larger than $h_\e$ together with \eqref{eq:norme sup} imply that there exists a constant $C(\theta_0)> 0$ such that
\begin{equation}\label{eq:borne eu eps}
\norme{e(u_\e)}_{L^\infty(T;\R^2)} \leq \frac{C \norme{u}_{L^\infty(\O;\R^2)}}{h_\e}.
\end{equation}
Introducing the characteristic functions $\chi_\e = \mathds{1}_{ \ds \left\{ \mathbf A_1 e(u_\e) : e(u_\e) \geq 2 \kappa \beta / h_\e  \right\} } $, one can check that $(u_\e, \chi_\e) \in X_{h_\e}(\O)$ and
$$\F_\e(u_\e, \chi_\e) = \G_\e (u_\e) + \frac12 \int_\O \eta_\e \chi_\e \mathbf A_0 e(u_\e) : e(u_\e) \, dx + \left( \frac{h_\e}{\e} - \beta  \right) \frac{\kappa}{h_\e}  \int_\O \chi_\e \, dx .
$$
On the one hand, \eqref{eq:G eps beta to G beta} entails that 
$$
M := \sup_{\e > 0} \, \frac{\kappa}{h_\e} \int_\O \chi_\e \, dx \leq \sup_{\e>0} \G_\e(u_\e) < \infty
$$
is bounded, so that $(u_\e, \chi_\e) \to (u,0) $ strongly in $ L^1(\O;\R^2) \times L^1(\O)$ when $\e \searrow 0$. On the other hand, using the growth condition of $\mathbf A_0$ together with \eqref{eq:borne eu eps}, we infer that there exists a constant $C'(a'_0;\kappa;\theta_0;u) >0$ such that 
$$
\F_\e(u_\e, \chi_\e) \leq \G_\e(u_\e) + C' M \left( \frac{\eta_\e}{h_\e} + \frac{h_\e}{\e} - \beta \right)  \to \G(u) = \F_{\alpha,\beta}(u,0) \quad \text{when } \e \searrow 0,
$$
where we also used the facts that $\lim_{\e \searrow 0} \, (\eta_\e / \e )= \alpha = 0$ and $\lim_{\e \searrow 0} \, ( h_\e / \e ) = \beta \in (0,+\infty)$.
\end{proof}

\section{Hencky Plasticity}\label{sec:hencky plasticity}
In this section, we further assume that the Hooke tensor of the damaged state is isotropic.
\begin{prop}\label{prop:hencky plasticity} Let us assume that $\theta_0 \leq \arctan (1/4)$ and that the stiffness tensor $\mathbf A_0$ satisfies \eqref{eq:isotropic Hooke Law A_0} with Lam\'e coefficients $\lambda,\mu >0$, {\it i.e.}
$$
\mathbf A_0 \xi = \lambda {\rm Tr}(\xi) I_2 + 2 \mu \xi , \quad \text{for all } \xi \in \Msd.
$$
If $\alpha \in (0, \infty)$ and $\beta = 0$, then the functional $ \F_\e$ $\Gamma$-converges for the strong $L^1(\O;\R^2) \times L^1(\O)$-topology to the functional
$$
 \F_{\alpha,\beta}(u,\chi) = 
\begin{cases}
\ds \int_\O \bar W_\alpha (e(u)) \, dx + \int_\O \bar W_\alpha^\infty \left( \frac{ d E^su}{d \left\lvert E^su \right\rvert} \right) \, d \left\vert   E^s u \right\rvert & \quad \text{if } u \in BD(\O) \text{ and } \chi = 0 \text{ a.e.},\\
+ \infty & \quad \text{otherwise,}
\end{cases}
$$
where $\bar W_\alpha$ and $\bar W_\alpha^\infty$ are defined in \eqref{eq:W alpha} and \eqref{eq:recession}.
\end{prop}

Although very similar to the one dimensional analysis in spirit, the vectorial case poses several mathematical challenges. On the one hand, our constructive proof cannot handle the case of a general stiffness tensor $\mathbf A_0$, as we rely on the explicit formula \eqref{eq:h isotropic case} to recover the $\Gamma$-limsup. On the other hand, as explained in the introduction, in the vectorial case we cannot directly use the candidates given by the Hashin-Shtrikman bound which fail to be admissible. Fortunately, as $\Gamma$-convergence only requires the convergence of the energies, we have more flexibility in the way we choose to laminate our approximating materials, so that a perfect immitation of $H$-convergence is not necessary. Finally, an additional step is needed in the vectorial setting, in order to merge in an admissible way the recovery sequences which are constructed locally.

\begin{proof}{\it Lower bound.} The proof of the lower bound inequality is the direct consequence of \cite[Theorem 3.1]{BIR} as it is easily checked that $X_{h_\e}(\O) \subset  H^1(\O;\R^2) \times L^\infty(\O; \{ 0,1 \})$.

\medskip
\noindent
{\it Upper bound.} By a rescaling and translation argument, we can assume without loss of generality that $\O= (0,1)^2$ is the unit cube of $\R^2$. Next owing to \cite[Corollary 1.10]{ARDPR} and \cite[Proposition 3.6]{BIR} together with the lower semi-continuity of $\F''_{\alpha,\beta}$ in $L^1(\O;\R^2) \times L^1(\O)$, we can assume that $u \in LD(\O)$. Then, due to the density result \cite[Theorem 3.2, Section II, 3.3]{Temam} and Reshetnyak's Continuity Theorem \cite[Theorem 2.39]{AFP} together with the lower semi-continuity of $\F''_{\alpha,\beta}$ again, we can further assume that $u \in C^\infty(\bar \O;\R^2)$. Thus dividing $\O$ into $n^2$ subcubes of side length $1/n$, each one being further divided into two isocele right triangles, we can consider the Lagrange interpolation $u_n$ of the values of $u$ at the vertices of the obtained triangulation (see Figure \ref{fig:triangulation plasticity}). In particular, since $\F_{\alpha,\beta}(u_n,0) \to \F_{\alpha,\beta}(u,0)$ as $n \nearrow \infty$ by standard approximation arguments (see \cite[Theorem 3.1.5]{Ciarlet}), we can eventually assume that $u = u_n \in H^1(\O;\R^2)$ for some $n \geq 1$.
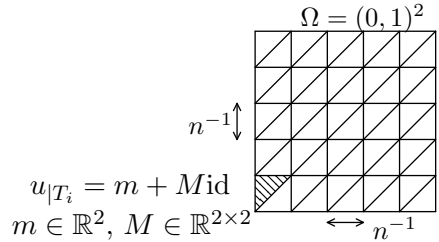
\begin{figure}[hbtp]
\begin{tikzpicture}[line cap=round,line join=round,>=triangle 45,x=0.6cm,y=0.6cm]
\clip(-10.5,-.7) rectangle (16,5.);
\begin{scope}[scale=0.8]
\draw [line width=.5pt] (0,0)-- (0,5);
\draw [line width=.5pt] (0.,0.)-- (5.,0.);
\draw [line width=.5pt] (5.,0.)-- (5.,5.);
\draw [line width=.5pt] (5.,5.)-- (0.,5.);
\draw [line width=.5pt] (0.,4.)-- (5.,4.);
\draw [line width=.5pt] (0.,3.)-- (5.,3.);
\draw [line width=.5pt] (0.,2.)-- (5.,2.);
\draw [line width=.5pt] (0.,1.)-- (5.,1.);
\draw [line width=.5pt] (1.,5.)-- (1.,0.);
\draw [line width=.5pt] (2.,5.)-- (2.,0.);
\draw [line width=.5pt] (3.,5.)-- (3.,0.);
\draw [line width=.5pt] (4.,5.)-- (4.,0.);
\draw [line width=.5pt] (0.,4.)-- (1.,5.);
\draw [line width=.5pt] (0.,3.)-- (2.,5.);
\draw [line width=.5pt] (0.,2.)-- (3.,5.);
\draw [line width=.5pt] (0.,1.)-- (4.,5.);
\draw [line width=.5pt] (0.,0.)-- (5.,5.);
\draw [line width=.5pt] (1.,0.)-- (5.,4.);
\draw [line width=.5pt] (2.,0.)-- (5.,3.);
\draw [line width=.5pt] (3.,0.)-- (5.,2.);
\draw [line width=.5pt] (4.,0.)-- (5.,1.);

\fill[line width=.5pt,pattern=north west lines,fill opacity=1] (0.,0.) -- (0.,1.) -- (1,1) -- cycle;

\draw (-6.5,1.3) node[anchor=north west]  {\large $u_{|{T_i}} = m + M {\rm id}$};
\draw (-7,0.3) node[anchor=north west] {\large $ m \in \R^2, \, M \in \R^{2 \times 2}$};

\draw (-2.1,3) node[anchor=north west] {$n^{-1}$};
\draw (3,0.1) node[anchor=north west] {$n^{-1}$};

\draw [line width=0.5pt] (-0.5,3)-- (-0.5,2);
\draw [line width=0.5pt] (-0.6,2.8)-- (-0.5,3);
\draw [line width=0.5pt] (-0.4,2.8)-- (-0.5,3);
\draw [line width=0.5pt] (-0.6,2.2)-- (-0.5,2);
\draw [line width=0.5pt] (-0.4,2.2)-- (-0.5,2);
\draw [line width=0.5pt] (2,-0.3)-- (3,-0.3);
\draw [line width=0.5pt] (2.2,-0.2)-- (2,-0.3);
\draw [line width=0.5pt] (2.2,-0.4)-- (2,-0.3);
\draw [line width=0.5pt] (2.8,-0.2)-- (3,-0.3);
\draw [line width=0.5pt] (2.8,-0.4)-- (3,-0.3);

\draw (1.,6.) node[anchor=north west] {$\O = (0,1)^2$};
\end{scope}
\end{tikzpicture}
\caption{Reduction to the case of a piecewise affine displacement.}
\label{fig:triangulation plasticity}
\end{figure}
Let 
$\{T_i\}_{1 \leq i \leq 2n^2}$
be an enumeration of the above triangles. The idea is to construct a local recovery sequence $(u_\e^i,\chi_\e^i)$ on each triangle $T_i$ in such a way that the local triangulation have all its vertices anchored to $\partial T_i$ in a specific way, allowing to glue them at the junctions of each triangle, thus realizing a global recovery sequence $(u_\e,\chi_\e) \in X_{h_\e}(\O)$. Indeed, applying Lemma \ref{lem:T plasticity} below to each triangle $T_i$ (up to a rescaling, rotation and translation argument), we obtain a local admissible triangulation $\mathbf T_\e^i \in \T_{h_\e}(T_i)
$ and its associated local recovery sequence $(u_\e^i, \chi_\e^i) \in X_{h_\e}(T_i)$. In particular, setting
\[
 u_\e := \sum_i \mathds{1}_{\ds T_i} u_\e^i \quad \text{and} \quad  \chi_\e := \sum_i \mathds{1}_{\ds T_i} \chi_\e^i ,
\]
one can check that $(u_\e,\chi_\e) \in X_{h_\e}(\O)$ is adapted to the global triangulation $\mathbf T_\e := \cup_i \mathbf T_\e^i \in \mathcal T_{h_\e}(\O)$. Eventually, by Fubini's Theorem and a rescaling argument, we infer that
\[ 
(u_\e,\chi_\e) \to (u,0) \quad \text{strongly in } L^1(\O;\R^2)\times L^1(\O) 
\]
and
\[
\F_\e(u_\e,\chi_\e) = \sum_i \E_\e^{T_i}(u_\e^i,\chi_\e^i) \to \sum_i \tfrac{1}{2n^2} \, \bar W_\alpha \left( e(u)_{|T_i} \right) = \F_{\alpha,\beta}(u,0)
\]
as $\e \searrow 0$.
\end{proof}

We are back to establishing the following result. 
\begin{lem}\label{lem:T plasticity}
Let $T := ABC $ be the isocele right triangle where $A=(0,0)$, $B=(1,0)$ and $C=(0,1)$. Let $v \in H^1(\R^2;\R^2)$ be an affine function. There exists an admissible triangulation $\mathbf T_\e \in \mathcal T_{h_\e}(T,\theta_0)$ and a recovery sequence $(v_\e,\chi_\e) \in X_{h_\e}(T)$ adapted to $\mathbf T_\e$, such that:\vspace{2mm}
\begin{itemize}[leftmargin=5mm,label=-]
\item the vertices of the triangulation are anchored to the segments $[AB]$, $[AC]$ and $[BC]$ at the nodes 
$$\bar x_{i,\e} = \frac{i}{ m_\e} \, e_1 , \quad \bar y_{i,\e} =  \frac{i}{ m_\e}  \, e_2 \quad \text{and} \quad \bar z_{i,\e} = e_1 +  \frac{i}{ m_\e}  \, (e_2 - e_1) $$
respectively, for $ i \in \llbracket 0, m_\e \rrbracket$ where $m_\e = \lfloor h_\e^{-1} \rfloor$;\vspace{2mm}
\item $v_\e \equiv v$ on $\partial T$ and $\chi_\e \equiv 1$ in each triangle in contact with $\partial T$;\vspace{2mm}
\item $(v_\e, \chi_\e) \to (v,0)$ strongly in $L^1(T;\R^2) \times L^1(T)$ as $\e \searrow 0$;\vspace{2mm}
\item $\E_\e^T(v_\e,\chi_\e) := \ds \int_T \left( \chi_\e  \tfrac\kappa\e + \tfrac12 \left( \chi_\e \eta_\e \mathbf A_0  + (1-\chi_\e) \mathbf A_1 \right) e(v_\e):e(v_\e) \right) \, dx \to \tfrac12 \bar W_{\alpha} (e(v))$ as $\e \searrow 0$.
\end{itemize}
\end{lem}

Note that the  gluing of recovery sequences is an issue inherent to the vectorial setting which did not occur in the scalar case. Yet, it is rather a technical difficulty than a substantial argument for the understanding of the proof. The core arguments lie in the construction of recovery sequence for an affine displacement in the whole plane.

\begin{proof}[Proof of Lemma \ref{lem:T plasticity}] The proof consists of two steps. First, we construct a recovery sequence $( v_\e,  \chi_\e) \in X_{h_\e}(T)$ with laminated structure, associated to a triangulation ${\mathbf T}_\e$. Then, we will carefully modify the recovery sequence and the triangulation, in order to ensure that the vertices of the new optimal triangulation $\mathbf{\tilde T}_\e$ are anchored in the above ad hoc way at the boundary $\partial T$ and the modified displacements satisfy $\tilde v_\e = v$ on $\partial T$.

\medskip
\noindent
\textbf{Step 1: lamination.} Let $\xi :=  e(v) \in \Msd$. By strict convexity of $\tau \in \Msd \mapsto \sqrt{2\kappa \alpha h(\tau)} + \frac12 \mathbf A_1 (\xi - \tau) : (\xi - \tau)$, there exists a unique $\tau \in \Msd$ such that 
\begin{equation}\label{eq:tau dans Wbar(xi)}
\bar W_\alpha(\xi) =  \sqrt{2 \kappa \alpha h(\tau)} + \frac12 \mathbf A_1 (\xi-\tau):(\xi - \tau).
\end{equation}
Let $\{ b_1,b_2\} \subset \mathcal{S}^1$ be an othonormal basis of $\R^2$ and $\tau_1, \tau_2 \in \R $ be such that 
$$
\tau = \tau_1 b_1 \otimes b_1 + \tau_2 b_2 \otimes b_2. 
$$
If $\tau = 0$, then $(v_\e, \chi_\e) \equiv (v,0) \in X_{h_\e}(T)$ directly gives the desired upper bound. We now assume that $\tau \neq 0$. Depending on the sign of ${\rm det} \,  \tau = \tau_1 \tau_2$, we must consider two different constructions. If ${\rm det} \, \tau \leq 0$, then one can find two vectors $a \in \R^2$ and $b \in \mathbb{S}^1$ such that $\tau = a \odot b$. In this case, we will construct a competitor material with a laminated structure in a single direction. More precisely, $T$ will be discretized in stripes (of widths at least $h_\e$) parallel to the direction $b$ with alternating layers of sound and damaged materials, {\it i.e.} where $\chi_\e$ is alternatively constant equal to $0$ and $1$ respectively. The displacement $v_\e$ will then be explicitly constructed to be affine along each stripe (see Figure \ref{fig:single lamination}). In this case, it is actually not necessary to assume that $\mathbf A_0$ is isotropic. 
Whereas if ${\rm det} \, \tau > 0$, we will rely on the isotropic character of $\mathbf A_0$ which provides an explicit formula for $I_\KK^*(\tau) = \sqrt{2 \kappa \alpha(\lambda + 2\mu)} (\lvert \tau_1 \rvert + \lvert \tau_2 \rvert )$. In this case, we will construct a laminated material with two directions of lamination, each one contributing to recover $\sqrt{2 \kappa \alpha(\lambda + 2\mu)} \lvert \tau_1 \rvert$ and $\sqrt{2 \kappa \alpha(\lambda + 2\mu)} \lvert \tau_2 \rvert$ respectively in the limit. In other words, $T$ will be discretized in stripes parallel to the direction $b_1$ and $b_2$, with alternating layers of sound and damaged materials in both directions, whose widths will be carrefully tuned and differ with respect to the direction of lamination, as illustrated in Figure \ref{fig:two laminations}.

\medskip
{\it Case 1.} Let us assume that ${\rm det} \big( \tau \big) = \tau_1 \tau_2 \leq 0$. In particular there exist $a \in \R^2$ and $b \in \mathbb{S}^1$ such that
$$
\tau = \tau_1 b_1 \otimes b_1 + \tau_2 b_2 \otimes b_2 = a \odot b.
$$ 
This is immediate if ${\rm det} (\tau) = 0$, while if ${\rm det}(\tau) <0$, then $\delta := \sqrt{ \lvert \tau_1 \rvert / \lvert \tau_2 \rvert} >0$ and $\delta \tau_2 + \tau_1 /\delta = 0$. Hence, one can check that $\tau = ( \delta b_1 + b_2) \odot \left( \frac{\tau_1}{\delta} b_1 + \tau_2 b_2 \right)$. Here, the construction of a recovery sequence relies on the fact that 
\begin{equation}\label{eq:lower bound h}
h(\tau) = \sup_{\eta \in \Msd} \left\{ 2 \tau:\eta - \sup_{k \in \mathbb S^1} \left\lvert \Pi_k \left( \mathbf A_0^{-\scriptstyle{\frac12}} \eta \right) \right\rvert^2 \right\} \geq \mathbf A_0 \tau:\tau ,
\end{equation}
where we used the fact that $\Pi_k ( \mathbf A_0^{-\scriptstyle{\frac12}} \eta ) \leq \mathbf A_0 \eta:\eta$. As will be detailed below, \eqref{eq:lower bound h} ensures that laminations in the single direction $b$ (which is not an eigen direction) prove to be sufficient. Let us remark that in the case where $\mathbf A_0$ is isotropic, $h(\tau) = \mathbf A_0 \tau : \tau$ is actually an equality (see \eqref{eq:h isotropic case}).

We must first introduce some notation, and we refer to Figure \ref{fig:single lamination} to ease the reading. Let $\gamma \in [0,\pi/4]$ be the (absolute) angle between $e_1 $ and $b$,
$$
\gamma = \left\lvert \arccos{ \left( e_1 \cdot b \right)} \right\rvert.
$$
Setting $l = \lvert \sin \gamma \rvert \geq 0$ and $L = \lvert \cos \gamma \rvert +l \geq 1$, we infer that there exists $\bar x \in \R^2$ such that 
$$
T \subset (0,1)^2 \subset \bar x + Q \quad \text{where} \quad Q = \{ x b + y b^\perp \, , \, (x,y) \in [0,L]^2 \}
$$
with $b^\perp = (- b\cdot e_2 , b \cdot e_1 ) \in \mathbb{S}^1$ as illustrated in Figure \ref{fig:single lamination}. Given a fixed parameter 
\begin{equation}\label{eq:choice for Theta}
\Theta \in (0,\infty)
\end{equation}
to be specified later on, we then define the following quantities:
\begin{equation}\label{eq:quantities cas 2}
 \theta_\e = \e \Theta \gg h_\e > 0, \quad m_\e = \left\lfloor \frac{ L \theta_\e}{h_\e} \right\rfloor \in \N \setminus \{0\}, \quad l_\e = \frac{L}{m_\e} >0 ,
\end{equation}
as well as the characteristic function
\begin{equation}\label{eq:chi eps plasticity cas 2}
 \chi_\e = \sum_{k=0}^{m_\e - 1} \mathds{1}_{\big( k l_\e, (k + \theta_\e) l_\e \big)} \big( ( \text{\rm id} - \bar x) \cdot b \big) \in L^1(\bar x + Q) 
\end{equation}
and the displacement
\begin{multline}\label{eq:u eps single lamination}
 v_\e =   v + a \otimes b \sum_{k=0}^{m_\e - 1} \Bigg( \frac{1 - \theta_\e}{\theta_\e} \big( \text{\rm id} - \bar x - k l_\e b  \big) \mathds{1}_{\big( k l_\e, (k + \theta_\e) l_\e \big)} \big( (\text{\rm id} - \bar x) \cdot b  \big) \\
  \hspace*{5cm} + \big( (k+1)l_\e b + \bar x - \text{\rm id} \big) \mathds{1}_{ \big[ (k + \theta_\e) l_\e, (k+1)l_\e \big)} \big( (\text{\rm id} - \bar x) \cdot b \big) \Bigg) .
\end{multline}
In particular, one can check that $v_\e$ is affine on all stripes $\bar x +  [kl_\e , (k+\theta_\e)l_\e] b + \R b^\perp$ and $\bar x + [(k+\theta_\e)l_\e , (k+1)l_\e] b + \R b^\perp$, for all $k \in \llbracket 0, m_\e - 1 \rrbracket$, and
\begin{equation}\label{eq:eu eps plasticity cas 2}
e(v_\e) = \xi + \tau \left( \frac{1 - \theta_\e}{\theta_\e} \chi_\e  - ( 1 - \chi_\e) \right) .
\end{equation}
Also note that, since $(  a \otimes b )  b^\perp = (b \cdot b^\perp) a = 0$, we have $v_\e = v $ on $ \bar x +  k l_\e b  + \R b^\perp $ and $v_\e$ is continuous across the interfaces $\bar x +  (k+\theta_\e)l_\e b + \R b^\perp$ for all $k \in \llbracket 0, m_\e - 1 \rrbracket$, with $v_\e \equiv v + (1 - \theta_\e) l_\e a  $ on $ \bar x +  (k+\theta_\e)l_\e b + \R b^\perp$. Thus $v_\e \in H^1(\bar x + Q; \R^2)$. Moreover, using \eqref{eq:quantities cas 2}, \eqref{eq:eu eps plasticity cas 2} together with the fact that $\lvert \tau \rvert = \lvert a \odot b \rvert \geq \lvert a \rvert / \sqrt{2} $, one can find a constant $C = C(\kappa, \alpha, \mathbf A_0,\tau) >0$ such that
\begin{equation}\label{eq:Linfty bound case 1}
\norme{v_\e - v}_{L^\infty(\bar x + Q;\R^2)} \leq  l_\e \sqrt{2}\lvert \tau \rvert \leq  \frac{C h_\e}{\e} \quad \text{and} \quad
\norme{e(v_\e) - \xi }_{L^\infty(\bar x + Q; \Msd)} \leq \frac{\lvert \tau \rvert}{\theta_\e} \leq \frac{C }{\e}.
\end{equation}
Next, noticing that $h_\e \leq \theta_\e l_\e  \leq 2 h_\e \leq \omega (h_\e)$ for $ \e > 0 $ small enough and $(1 - \theta_\e) l_\e \gg h_\e$, we have 
\begin{equation}\label{eq:widths triangulation case 1}
 n_\e = \left\lfloor \frac{(1 - \theta_\e)l_\e}{h_\e} \right\rfloor \gg 1 \quad \text{and} \quad h_\e \leq L_\e := \frac{(1 - \theta_\e)l_\e}{n_\e} \leq 2 h_\e \leq \omega (h_\e).
\end{equation}
Therefore, it is easily seen from Figure \ref{fig:single lamination} that we can define an admissible triangulation $\mathbf T_\e \in \T_\e ((0,1)^2) \subset \T_\e(T)$, composed of rectangles (parallel to $b$ and $b^\perp$) with edges of lengths $L / \lfloor \frac{L}{h_\e}\rfloor $, $\theta_\e l_\e$ or $L_\e$, each subdivided into two right subtriangles. Indeed, one can check that the hypothenuse of each triangle has length between $h_\e$ and $\sqrt{8} h_\e \leq \omega (h_\e)$. Moreover, $v_\e$ is affine and $\chi_\e$ is constant on each triangle of $\mathbf T_\e $, hence $(v_\e, \chi_\e) \in X_{h_\e}(T)$. 

On the one hand, $(v_\e, \chi_\e) \to (v,0) $ strongly in $L^1(T;\R^2) \times L^1(T; \R)$ when $ \e \searrow 0$. Indeed, using  \eqref{eq:quantities cas 2} together with the change of variables 
$$
\Phi = \big( (id - \bar x ) \cdot b, (id- \bar x) \cdot b^\perp  \big) : \R^2 \to \R^2
$$
which satisfies $\nabla \Phi \nabla \Phi^T = I_2$, one can check that
\begin{equation}\label{eq:kappa chi / eps}
\frac{\kappa}{\e} \int_T \chi_\e \, dx = \LL^2(T) \frac{\kappa \theta_\e}{\e} + o_{\e \searrow 0} (1) \to \frac{ \kappa \Theta}{2}
\end{equation}
and
$$
\begin{aligned}
\norme{v_\e - v}_{L^1(T; \R^2)} & \leq  \lvert a \otimes b \rvert L \sum_{k=0}^{m_\e - 1} \left( \int_{k l_\e}^{(k + \theta_\e) l_\e} \frac{1 - \theta_\e}{\theta_\e} ( s - k l_\e ) \, ds + \int_{(k + \theta_\e) l_\e}^{(k+1)l_\e} ( (k+1) l_\e - s \big) \, ds \right) \\
& =  \lvert a \otimes b \rvert \frac{ L m_\e}{2} \left( \frac{1 - \theta_\e }{\theta_\e} \lvert \theta_\e l_\e \rvert^2 + \lvert (1 - \theta_\e) l_\e \rvert^2 \right) = \lvert a \otimes b \rvert \frac{  L^2}{2} (1 - \theta_\e) l_\e \to 0 
\end{aligned}
$$
when $\e \searrow 0$. On the other hand, using \eqref{eq:eu eps plasticity cas 2}, \eqref{eq:kappa chi / eps} together with the facts that $\chi_\e \to 0$ strongly in $L^1(T)$ and $\eta_\e / \theta_\e \to   \alpha / \Theta$ when $\e \searrow 0$, we infer that
\begin{multline}\label{eq:volume cas 2}
\frac12 \int_T \left( \eta_\e \chi_\e \mathbf A_0 + (1 - \chi_\e) \mathbf A_1 \right) e(v_\e):e(v_\e) \, dx \\
= \frac{ \LL^2(T)}{2} \left( \eta_\e \theta_\e \frac{ \lvert 1 - \theta_\e \rvert^2}{\lvert \theta_\e \rvert^2} \mathbf A_0 \tau:\tau +\mathbf A_1 (\xi - \tau) : (\xi- \tau) \right) \, dx + o_{\e \searrow 0}(1) \\
 \to \frac14 \left( \frac{\alpha}{\Theta} \mathbf A_0 \tau:\tau  + \mathbf A_1 (\xi - \tau): ( \xi - \tau) \right) \quad \text{when } \e \searrow 0.
\end{multline}
Therefore, gathering \eqref{eq:kappa chi / eps} and \eqref{eq:volume cas 2} entails that
$$
\F_\e \big({v_\e}_{|T}, {\chi_\e}_{|T} \big) \to \LL^2(T) \left( \frac12  \mathbf A_1 (\xi-\tau) : (\xi - \tau) + \kappa \Theta + \frac{\alpha}{2\Theta} \mathbf A_0 \tau:\tau \right) \quad \text{when } \e \searrow 0.
$$
Finally choosing 
$$
\Theta = \frac{\sqrt{2 \alpha \kappa h(\tau)} + \sqrt{\Delta}}{2\kappa}
$$
in \eqref{eq:choice for Theta}, where $\Delta = 2 \kappa \alpha \left( h(\tau) - \mathbf A_0 \tau:\tau \right) \geq 0 $ according to \eqref{eq:lower bound h}, one can check that
$$
\F_\e \big({v_\e}_{|T}, {\chi_\e}_{|T} \big) \to \LL^2(T) \left( \frac12  \mathbf A_1 (\xi-\tau) : (\xi - \tau) + \sqrt{2 \alpha \kappa h(\tau)} \right) = \LL^2(T) \, \bar W_\alpha(\xi) \, \text{ when } \e \searrow 0.
$$
\begin{figure}[hbtp]
\begin{tikzpicture}[line cap=round,line join=round,>=triangle 45,x=.7cm,y=0.7cm]

\begin{scope}[scale=0.6]
\begin{scope}[xshift=150]

\begin{scope}[xshift=-150,yshift=-120]
\draw (-6.8,8.4) node[anchor=north west] {$e_1$};
\draw (-6.9,9.6) node[anchor=north west] {$b$};
\draw (-11.5,12.9) node[anchor=north west] {$e_2$};
\draw (-13.,12.75) node[anchor=north west] {$b^\perp$};
\draw [line width=.5pt] (-11.,8.)-- (-7.,8.);
\draw [line width=.5pt] (-7.,8.)-- (-7.35,8.2);
\draw [line width=.5pt] (-7.,8.)-- (-7.35,7.8);

\draw [line width=.5pt] (-11.,8.)-- (-11.,12.);
\draw [line width=.5pt] (-11.,12.)-- (-11.2,11.6);
\draw [line width=.5pt] (-11.,12.)-- (-10.8,11.6);

\draw [line width=.5pt] (-11.,8.)-- (-7.12261793264476,8.982806340919716);
\draw [line width=.5pt] (-7.12261793264476,8.982806340919716)-- (-7.45,9.05);
\draw [line width=.5pt] (-7.12261793264476,8.982806340919716)-- (-7.35,8.75);

\draw [line width=.5pt] (-11.,8.)-- (-11.980280425947944,11.878021439665247);
\draw [line width=.5pt] (-11.980280425947944,11.878021439665247)-- (-12.069090909090903,11.440330578512391);
\draw [line width=.5pt] (-11.980280425947944,11.878021439665247)-- (-11.688925619834706,11.5395041322314);
\end{scope}

\fill[line width=.5pt,fill opacity=0.5] (-11.,6.) -- (-11.,2.) -- (-7.,2.) -- cycle;
\fill[line width=.6pt,fill opacity=0.2] (-10.71,1.0318299881936266) -- (-5.9651948051948125,2.2612672176308557) -- (-7.2985281385281455,6.936591892955529) -- (-12.060432900432906,5.68) -- cycle;
\draw [line width=.5pt] (-10.71,1.0318299881936266)-- (-5.9651948051948125,2.2612672176308557);
\draw [line width=.5pt] (-5.9651948051948125,2.2612672176308557)-- (-7.2985281385281455,6.936591892955529);
\draw [line width=.5pt] (-7.2985281385281455,6.936591892955529)-- (-12.060432900432906,5.68);
\draw [line width=.5pt] (-12.060432900432906,5.68)-- (-10.71,1.0318299881936266);

\begin{scope}[xshift=-20,yshift=-10]
\draw [line width=.5pt] (-4.2,6.)-- (-3.4,6.);
\draw [line width=.5pt] (-3.4,6.)-- (-3.6,6.2);
\draw [line width=.5pt] (-3.4,6.)-- (-3.6,5.8);
\draw (-4.35,7.258348082147434) node[anchor=north west] {$\Phi$};
\end{scope}

\draw (-10,1.2) node[anchor=north west] {$\bar x + Q$};
\draw (-10.314657099192818,3.3957981827220314) node[anchor=north west] {$T$};
\draw (-13.4,1.4) node[anchor=north west] {$l$};
\draw (-12.9,3.4144578440719124) node[anchor=north west] {$L$};

\draw [line width=.5pt] (-12.7235329603582,5.506506212227982)-- (-11.294223251978725,0.8473278121857986);
\draw [line width=.5pt] (-11.294223251978725,0.8473278121857986)-- (-11.5,1.);
\draw [line width=.5pt] (-11.294223251978725,0.8473278121857986)-- (-11.231981973992896,1.083049616572006);
\draw [line width=.5pt] (-12.7235329603582,5.506506212227982)-- (-12.524,5.3);
\draw [line width=.5pt] (-12.7235329603582,5.506506212227982)-- (-12.8,5.2);

\draw [line width=.5pt] (-12.742095424103388,1.47845157952219)-- (-12.426533540435193,0.5874533197531668);
\draw [line width=.5pt] (-12.426533540435193,0.5874533197531668)-- (-12.336030464386289,0.8);
\draw [line width=.5pt] (-12.426533540435193,0.5874533197531668)-- (-12.621287012143432,0.7);
\draw [line width=.5pt] (-12.742095424103388,1.47845157952219)-- (-12.8,1.25);
\draw [line width=.5pt] (-12.742095424103388,1.47845157952219)-- (-12.55,1.3);

\begin{scope}[yshift=10,xshift=-25]
\fill[line width=.5pt,fill opacity=0.2] (-2.,1.) -- (2.958878439283476,0.9941704146639477) -- (2.9415199267744336,5.92865182619081) -- (-2.002881713835631,5.933906238136196) -- cycle;
\draw [line width=.5pt] (-2.,1.)-- (2.95,1);
\draw [line width=.5pt] (2.95,1)-- (2.95,5.93);
\draw [line width=.5pt] (2.95,5.93)-- (-2.,5.93);
\draw [line width=.5pt] (-2.,5.93)-- (-2.,1.);
\filldraw [line width=.5pt,pattern=north west lines,fill opacity=0.5] (-1.,1)-- (-1.,5.93)--(-0.7,5.93)--(-0.7,1);
\filldraw [line width=.5pt,pattern=north west lines,fill opacity=0.5] (0.,5.93)-- (0.,1)--(0.3,1)--(0.3,5.93);
\filldraw [line width=.5pt,pattern=north west lines,fill opacity=0.5] (1,1)-- (1,5.93)--(1.3,5.93)--(1.3,1);
\filldraw [line width=.5pt,pattern=north west lines,fill opacity=0.5] (2,5.93)-- (2,1)--(2.3,1)--(2.3,5.93);
\filldraw [line width=.5pt,pattern=north west lines,fill opacity=0.5] (-2.,1.)--(-2,5.93)--(-1.7,5.93)-- (-1.7,1);
\draw [line width=.5pt] (-1.,6.5)-- (0.,6.5);
\draw [line width=.5pt] (0.,6.5)-- (-0.2,6.6);
\draw [line width=.5pt] (0.,6.5)-- (-0.2,6.4);
\draw [line width=.5pt] (-1.,6.5)-- (-0.8,6.6);
\draw [line width=.5pt] (-1.,6.5)-- (-0.8,6.4);
\draw [line width=0.5pt,dotted] (-1,5.93)--(-1,6.5);
\draw [line width=0.5pt,dotted] (0,5.93)--(0,6.5);

\draw [line width=.5pt] (1.,6.5)-- (1.3,6.5);
\draw [line width=.5pt] (1.,6.5)-- (1.05,6.55);
\draw [line width=.5pt] (1.,6.5)-- (1.05,6.45);
\draw [line width=.5pt] (1.3,6.5)-- (1.25,6.55);
\draw [line width=.5pt] (1.3,6.5)-- (1.25,6.45);

\draw [line width=0.8pt,dotted] (-0.3,2)--(-0.2,0.4);
\draw [line width=0.8pt,dotted] (-0.8,2)--(-1.8,0.5);

\draw (-1,7.482264018346008) node[anchor=north west] {$l_\varepsilon$};
\draw (1.2,7.2) node[anchor=north west] {$\theta_\varepsilon \, l_\varepsilon $};
\draw (-3.5,0.7) node[anchor=north west] {$\chi_\varepsilon = 1$};
\draw (-0.7,0.5) node[anchor=north west] {$\chi_\varepsilon = 0$};
\end{scope}
\end{scope}

\begin{scope}[xshift=-280,yshift=220]
\fill[line width=.5pt,fill opacity=0.2] (6.,-2.) -- (14.,-2.5) -- (20.,-0.5) -- (12.,0.) -- cycle;
\fill[line width=.5pt,fill opacity=0.5] (6.,2.) -- (12.,6.) -- (20.,4.) -- (14.,0.) -- cycle;
\draw [line width=.5pt] (6.,-2.)-- (14.,-2.5);
\draw [line width=.5pt] (14.,-2.5)-- (20.,-0.5);
\draw [line width=.5pt] (20.,-0.5)-- (12.,0.);
\draw [line width=.5pt] (12.,0.)-- (6.,-2.);

\draw [line width=.5pt] (6.,2.)-- (12.,6.);
\draw [line width=.5pt] (12.,6.)-- (20.,4.);
\draw [line width=.5pt] (20.,4.)-- (14.,0.);
\draw [line width=.5pt] (14.,0.)-- (6.,2.);

\draw [line width=.5pt,dotted] (12.,6.)-- (12.,0.);
\draw [line width=.5pt,dotted] (14.,0.)-- (14.,-2.5);
\draw [line width=.5pt,dotted] (20.,4.)-- (20.,-0.5);
\draw [line width=.5pt,dotted] (6.,2.)-- (6.,-2.);

\filldraw [line width=.5pt,color=red,fill=red,fill opacity=0.7] (7.5,3.)-- (7.875,3.75)-- (15.875,1.75)-- (15.5,1.);
\filldraw [line width=.5pt,color=red,fill=red,fill opacity=0.7] (9.,4.)-- (9.375,4.75) -- (17.375,2.75) -- (17.,2.) ;
\filldraw [line width=.5pt,color=red,fill=red,fill opacity=0.7] (10.5,5.)-- (10.875,5.75) -- (18.875,3.75) -- (18.5,3.) ;
\filldraw [line width=.5pt,color=red,fill=red,fill opacity=0.7]  (6.,2.)-- (6.375,2.75) -- (14.375,0.75)--(14.,0.) ;

\filldraw [line width=.5pt,color=red,fill=red,fill opacity=0.4] (6.375,2.75) -- (7.5,3.) -- (15.5,1.) --(14.375,0.75)  ;
\filldraw [line width=.5pt,color=red,fill=red,fill opacity=0.4]  (7.875,3.75)-- (15.875,1.75) -- (17.,2) -- (9.,4.) ;
\filldraw [line width=.5pt,color=red,fill=red,fill opacity=0.4]  (9.375,4.75) -- (17.375,2.75)-- (18.5,3.) -- (10.5,5.) ;
\filldraw [line width=.5pt,color=red,fill=red,fill opacity=0.4]  (10.875,5.75) -- (18.875,3.75)  -- (20,4) --(12,6) ;

\draw [line width=.pt] (6.,2.)-- (6.375,2.75);
\draw [line width=.pt] (7.5,3.)-- (6.375,2.75);
\draw [line width=.pt] (7.5,3.)-- (7.875,3.75);
\draw [line width=.pt] (7.875,3.75)-- (9.,4.);
\draw [line width=.pt] (9.,4.)-- (9.375,4.75);
\draw [line width=.pt] (9.375,4.75)-- (10.5,5.);
\draw [line width=.pt] (10.5,5.)-- (10.875,5.75);
\draw [line width=.pt] (10.875,5.75)-- (12.,6.);
\draw [line width=.pt] (14.,0.)-- (14.375,0.75);
\draw [line width=.pt] (14.375,0.75)-- (15.5,1.);
\draw [line width=.pt] (15.5,1.)-- (15.875,1.75);
\draw [line width=.pt] (15.875,1.75)-- (17.,2.);
\draw [line width=.pt] (17.,2.)-- (17.375,2.75);
\draw [line width=.pt] (17.375,2.75)-- (18.5,3.);
\draw [line width=.pt] (18.5,3.)-- (18.875,3.75);
\draw [line width=.pt] (18.875,3.75)-- (20.,4.);

\draw (8,6.) node[anchor=north west] {${\color{red}{v_\varepsilon}}$};
\draw (0,2.5) node[anchor=north west] {$v= m + M {\rm id}$};
\end{scope}
\end{scope}
\end{tikzpicture}
\caption{Recovery sequence using a single direction of lamination, when ${\rm det}(\tau) \leq 0$.}
\label{fig:single lamination}
\end{figure}
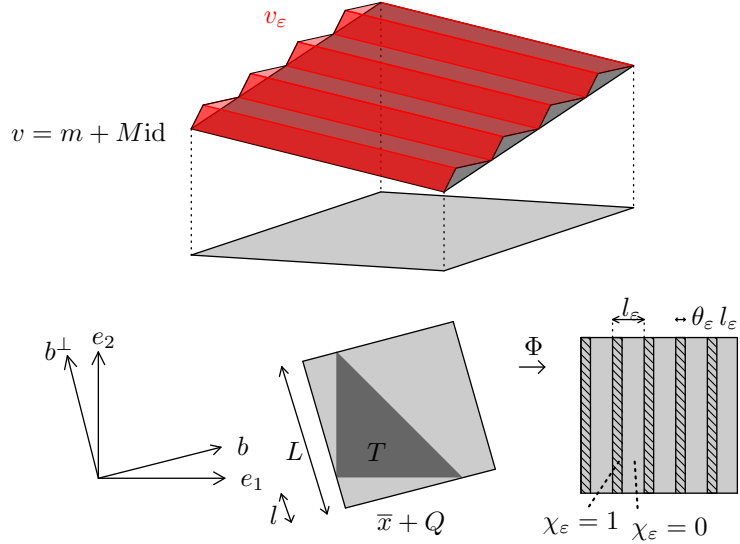

{\it Case 2.} Let us assume that ${\rm det} ( \tau ) = \tau_1 \tau_2 > 0$. As mentionned previously, in this case we rely on the isotropic character of $\mathbf A_0$ so that, setting $a := \lambda + 2 \mu$, one has $I_\KK^*(\tau) = \sqrt{2 \alpha\kappa a} ( \lvert \tau_1 \rvert  + \lvert \tau_2 \rvert )$. As before, we start by setting some notation and refer to Figure \ref{fig:two laminations} to help the reading. Let $\gamma := \lvert \arccos{ \left( e_1 \cdot b_1 \right)} \rvert \in [0,\pi /4]$ be the (absolute) angle between $e_1 $ and $b_1$. Setting $l = \lvert \sin \gamma \rvert \geq 0$ and $L = \lvert \cos \gamma \rvert +l \geq 1$, there exists $\bar x \in \R^2$ such that 
$$
T \subset (0,1)^2 \subset \bar x + Q \quad \text{where} \quad Q = \{ x b_1 + y b_2 \, , \, (x,y) \in [0,L]^2 \}.
$$
We then define the change of variables
\begin{equation}\label{eq:change of variables}
\Phi = \big( (\text{\rm id} - \bar x ) \cdot b_1, (\text{\rm id} - \bar x) \cdot b_2 \big) : \R^2 \to \R^2
\end{equation}
which satisfies $\nabla \Phi \nabla \Phi^T = I_2$. We introduce the quantities $\theta_{1,\e}$ and $\theta_{2,\e} \in (0,1)$, corresponding to the optimal proportion of weak material in the directions $b_1$ and $b_2$ respectively:
\begin{equation}\label{eq:quantities cas 1}
 \theta_{i,\e} = \frac{\e}{2 \kappa} \sqrt{2 \kappa \alpha a} \lvert \tau_i \rvert \gg h_\e > 0,
\end{equation}
as well as the subintervals lengths
\[
l_{i,\e} = \frac{L}{m_{i,\e}} >0  \quad \text{where} \quad m_{i,\e} = \left\lfloor \frac{L \theta_{i,\e}}{h_\e} \right\rfloor \in \N \setminus \{0\}
 \]
for all $\e >0$ and for $i \in \{ 1, 2 \}$. We then define the characteristic functions
\[
 \chi_{\e}^{(i)} = \sum_{k=0}^{m_{i,\e} - 1} \mathds{1}_{\big( k l_{i,\e}, (k + \theta_{i,\e}) l_{i,\e} \big)}  \in L^1([0,L]) \quad \text{for } i \in \{1,2\}
 \]
and
\begin{equation}\label{eq:chi eps plasticity}
 \ds \chi_\e = \chi_{\e}^{(1)} \big( b_1 \cdot ( \text{\rm id} - \bar x) \big) + \chi_{\e}^{(2)} \big( b_2 \cdot ( \text{\rm id} - \bar x) \big) \big( 1 - \chi_{\e}^{(1)} \big( b_1 \cdot ( \text{\rm id} - \bar x) \big) \big)  \in L^1( \bar x + Q),
\end{equation}
as well as the displacements
\begin{multline}\label{eq:u eps deux toles}
 v_\e =   v + \sum_{i=1}^2   \tau_i b_i \otimes b_i \sum_{k=0}^{m_{i,\e} - 1} \Bigg( \frac{1 - \theta_{i,\e}}{\theta_{i,\e}} \big( \text{\rm id} - \bar x  - k l_{i,\e} b_i \big) \mathds{1}_{\big( k l_{i,\e}, (k + \theta_{i,\e}) l_{i,\e} \big)} \big( b_i \cdot ( \text{\rm id} - \bar x) \big) \\
  \hspace*{5cm} + \big( (k+1)l_{i,\e} b_i - (\text{\rm id} - \bar x) \big) \mathds{1}_{ \big[ (k + \theta_{i,\e}) l_{i,\e}, (k+1)l_{i,\e} \big)} \big( b_i \cdot ( \text{\rm id} - \bar x) \big) \Bigg) .
\end{multline}
In particular, one can check that $v_\e = v$ on $\left\{ k l_{1,\e} b_1 + k' l_{2,\e} b_2\, : \,  (k,k') \in \llbracket 0,m_{1,\e}\rrbracket \times \llbracket 0, m_{2,\e} \rrbracket \right\}$ and $v_\e$ is piecewise affine. Using the fact that $b_1 \cdot b_2 = 0$, which ensures the continuity of the displacement across the interfaces $\bar x + (k + \theta_{i,\e})l_{i,\e} b_i  + \R b_{3-i} $, we infer that $v_\e \in \mathcal{C}^0(\bar x + Q; \R^2)$. Hence $v_\e \in H^1(\bar x + Q; \R^2)$, and we have
\begin{equation}\label{eq:eu eps plasticity cas 1}
e(v_\e) = \xi + \sum_{i=1}^2   \tau_i b_i \otimes b_i \left( \frac{1 - \theta_{i,\e}}{\theta_{i,\e}} \chi_\e^{(i)} \big( b_i \cdot ( \text{\rm id} - \bar x) \big) - \bigg( 1 - \chi_\e^{(i)} \big( b_i \cdot ( \text{\rm id} - \bar x) \big) \bigg) \right) .
\end{equation}
Moreover, using \eqref{eq:quantities cas 1}, we infer that there exists a constant $C = C(\kappa,\alpha,a,\tau) >0$ such that
\begin{eqnarray}
&\ds \norme{v_\e - v}_{L^\infty( \bar x + Q;\R^2)} \leq \sum_{i=1}^2 \lvert \tau_i\rvert (1 - \theta_{i,\e}) l_{i,\e} \leq C \frac{h_\e}{\e}, \label{seq:Linfty control plasticity} \\
&\ds \norme{e(v_\e) - \xi}_{L^\infty(\bar x + Q; \Msd)} \leq \sum_{i=1}^{2} \frac{\lvert \tau_i \rvert}{\theta_{i,\e}} \leq  \frac{C}{\e}.
 \label{seq:Linfty control plasticity strain}
\end{eqnarray}
Also note that for $i \in \{ 1, 2 \}$ and $ \e > 0 $ small enough, 
\begin{empheq}[left=\empheqlbrace]{align}\label{eq:widths triangulation case 2}
		&\ds h_\e \leq \theta_{i,\e} l_{i,\e} \leq 2 h_\e \leq \omega (h_\e), \\ 
		&\ds h_\e \leq L_{i,\e} := \frac{(1 - \theta_{i,\e})l_{i,\e}}{n_{i,\e}} \leq 2 h_\e \leq \omega (h_\e), 
\end{empheq}
where $ n_{i,\e} := \left\lfloor (1 - \theta_{i,\e})l_{i,\e} / h_\e \right\rfloor \gg 1 $, since $(1 - \theta_{i,\e}) l_{i,\e} \gg h_\e$. We can thus define an admissible triangulation $\mathbf T_\e$ composed of rectangles parallel to $b_1$ and $b_2$, with edges of lengths $\theta_{i,\e} l_{i,\e}$ or $L_{i,\e}$ for $i \in \{1,2 \}$, each one being subdivided into two right triangles, as illustrated in Figure \ref{fig:two laminations}. In particular, one can check that the hypothenuse of each triangle has length in between $h_\e$ and $\sqrt{8} h_\e \leq \omega (h_\e)$. Therefore, $\mathbf T_\e \in \T_\e ((0,1)^2) \subset \T_\e (T)$ is an admissible triangulation for $T$ and $(v_\e, \chi_\e) \in X_{h_\e}(T)$. We will now show that $(v_\e,\chi_\e)$ is a recovery sequence for $(v,0)$ in $T$. 
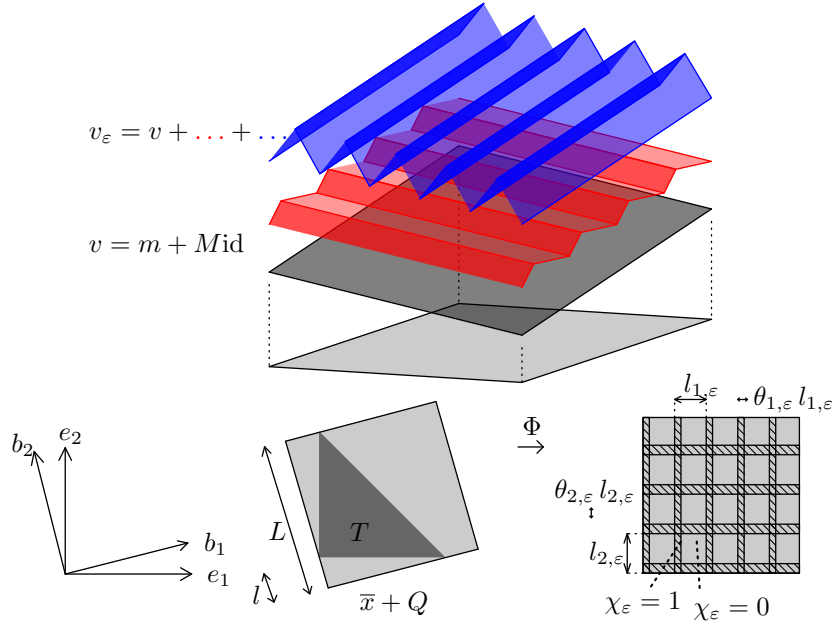
\begin{figure}[hbtp]
\begin{tikzpicture}[line cap=round,line join=round,>=triangle 45,x=.7cm,y=0.7cm]
\begin{scope}[scale=0.6]

\begin{scope}[xshift=150]

\begin{scope}[xshift=-200,yshift=-130]
\draw (-6.8,8.4) node[anchor=north west] {$e_1$};
\draw (-6.9,9.6) node[anchor=north west] {$b_1$};
\draw (-11.5,12.9) node[anchor=north west] {$e_2$};
\draw (-13.,12.75) node[anchor=north west] {$b_2$};
\draw [line width=.5pt] (-11.,8.)-- (-7.,8.);
\draw [line width=.5pt] (-7.,8.)-- (-7.35,8.2);
\draw [line width=.5pt] (-7.,8.)-- (-7.35,7.8);

\draw [line width=.5pt] (-11.,8.)-- (-11.,12.);
\draw [line width=.5pt] (-11.,12.)-- (-11.2,11.6);
\draw [line width=.5pt] (-11.,12.)-- (-10.8,11.6);

\draw [line width=.5pt] (-11.,8.)-- (-7.12261793264476,8.982806340919716);
\draw [line width=.5pt] (-7.12261793264476,8.982806340919716)-- (-7.45,9.05);
\draw [line width=.5pt] (-7.12261793264476,8.982806340919716)-- (-7.35,8.75);

\draw [line width=.5pt] (-11.,8.)-- (-11.980280425947944,11.878021439665247);
\draw [line width=.5pt] (-11.980280425947944,11.878021439665247)-- (-12.069090909090903,11.440330578512391);
\draw [line width=.5pt] (-11.980280425947944,11.878021439665247)-- (-11.688925619834706,11.5395041322314);
\end{scope}

\begin{scope}[xshift=-40]

\begin{scope}[xshift=-10,yshift=-10]
\draw [line width=.5pt] (-4.2,6.)-- (-3.4,6.);
\draw [line width=.5pt] (-3.4,6.)-- (-3.6,6.2);
\draw [line width=.5pt] (-3.4,6.)-- (-3.6,5.8);
\draw (-4.35,7.258348082147434) node[anchor=north west] {$\Phi$};
\end{scope}

\fill[line width=.5pt,fill opacity=0.5] (-11.,6.) -- (-11.,2.) -- (-7.,2.) -- cycle;
\fill[line width=.6pt,fill opacity=0.2] (-10.71,1.0318299881936266) -- (-5.9651948051948125,2.2612672176308557) -- (-7.2985281385281455,6.936591892955529) -- (-12.060432900432906,5.68) -- cycle;
\draw [line width=.5pt] (-10.71,1.0318299881936266)-- (-5.9651948051948125,2.2612672176308557);
\draw [line width=.5pt] (-5.9651948051948125,2.2612672176308557)-- (-7.2985281385281455,6.936591892955529);
\draw [line width=.5pt] (-7.2985281385281455,6.936591892955529)-- (-12.060432900432906,5.68);
\draw [line width=.5pt] (-12.060432900432906,5.68)-- (-10.71,1.0318299881936266);

\draw (-10,1.2) node[anchor=north west] {$\bar x + Q$};
\draw (-10.314657099192818,3.3957981827220314) node[anchor=north west] {$T$};
\draw (-13.4,1.4) node[anchor=north west] {$l$};
\draw (-12.9,3.4144578440719124) node[anchor=north west] {$L$};

\draw [line width=.5pt] (-12.7235329603582,5.506506212227982)-- (-11.294223251978725,0.8473278121857986);
\draw [line width=.5pt] (-11.294223251978725,0.8473278121857986)-- (-11.5,1.);
\draw [line width=.5pt] (-11.294223251978725,0.8473278121857986)-- (-11.231981973992896,1.083049616572006);
\draw [line width=.5pt] (-12.7235329603582,5.506506212227982)-- (-12.524,5.3);
\draw [line width=.5pt] (-12.7235329603582,5.506506212227982)-- (-12.8,5.2);

\draw [line width=.5pt] (-12.742095424103388,1.47845157952219)-- (-12.426533540435193,0.5874533197531668);
\draw [line width=.5pt] (-12.426533540435193,0.5874533197531668)-- (-12.336030464386289,0.8);
\draw [line width=.5pt] (-12.426533540435193,0.5874533197531668)-- (-12.621287012143432,0.7);
\draw [line width=.5pt] (-12.742095424103388,1.47845157952219)-- (-12.8,1.25);
\draw [line width=.5pt] (-12.742095424103388,1.47845157952219)-- (-12.55,1.3);
\end{scope}

\begin{scope}[yshift=10,xshift=-15]
\fill[line width=.5pt,fill opacity=0.2] (-2.,1.) -- (2.95,1) -- (2.95,5.93) -- (-2.,5.93) -- cycle;
\draw [line width=.5pt] (-2.,1.)-- (2.95,1);
\draw [line width=.5pt] (2.95,1)-- (2.95,5.93);
\draw [line width=.5pt] (2.95,5.93)-- (-2.,5.93);
\draw [line width=.5pt] (-2.,5.93)-- (-2.,1.);

\draw (-4,2.3) node[anchor=north west] {$l_{2,\varepsilon}$};
\draw (-5.1,4.25) node[anchor=north west] {$\theta_{2,\varepsilon} \, l_{2,\varepsilon} $};
\draw [line width=0.5pt,dotted] (-2,1)--(-2.5,1);
\draw [line width=0.5pt,dotted] (-2,2.25)--(-2.5,2.25);
\draw [line width=0.5pt] (-2.5,1)--(-2.5,2.25);
\draw [line width=0.5pt] (-2.5,1)--(-2.6,1.3);
\draw [line width=0.5pt] (-2.5,1)--(-2.4,1.3);
\draw [line width=0.5pt] (-2.5,2.25)--(-2.6,1.95);
\draw [line width=0.5pt] (-2.5,2.25)--(-2.4,1.95);

\begin{scope}[xshift=57,yshift=35,rotate=90]
\draw [line width=.5pt] (1.,6.5)-- (1.3,6.5);
\draw [line width=.5pt] (1.,6.5)-- (1.05,6.55);
\draw [line width=.5pt] (1.,6.5)-- (1.05,6.45);
\draw [line width=.5pt] (1.3,6.5)-- (1.25,6.55);
\draw [line width=.5pt] (1.3,6.5)-- (1.25,6.45);
\end{scope}

\filldraw [line width=.5pt,pattern=north west lines,fill opacity=0.5] (-2.,1)-- (2.95,1)--(2.95,1.3)--(-2,1.3);
\filldraw [line width=.5pt,pattern=north west lines,fill opacity=0.5] (-2.,2.25)-- (2.95,2.25)--(2.95,2.55)--(-2,2.55);
\filldraw [line width=.5pt,pattern=north west lines,fill opacity=0.5] (-2.,3.5)-- (2.95,3.5)--(2.95,3.8)--(-2,3.8);
\filldraw [line width=.5pt,pattern=north west lines,fill opacity=0.5] (-2.,4.75)-- (2.95,4.75)--(2.95,5.05)--(-2,5.05);

\filldraw [line width=.5pt,pattern=north west lines,fill opacity=0.5] (-1.,1)-- (-1.,5.93)--(-0.8,5.93)--(-0.8,1);
\filldraw [line width=.5pt,pattern=north west lines,fill opacity=0.5] (0.,5.93)-- (0.,1)--(0.2,1)--(0.2,5.93);
\filldraw [line width=.5pt,pattern=north west lines,fill opacity=0.5] (1,1)-- (1,5.93)--(1.2,5.93)--(1.2,1);
\filldraw [line width=.5pt,pattern=north west lines,fill opacity=0.5] (2,5.93)-- (2,1)--(2.2,1)--(2.2,5.93);
\filldraw [line width=.5pt,pattern=north west lines,fill opacity=0.5] (-2.,1.)--(-2,5.93)--(-1.8,5.93)-- (-1.8,1);
\draw [line width=.5pt] (-1.,6.5)-- (0.,6.5);
\draw [line width=.5pt] (0.,6.5)-- (-0.2,6.6);
\draw [line width=.5pt] (0.,6.5)-- (-0.2,6.4);
\draw [line width=.5pt] (-1.,6.5)-- (-0.8,6.6);
\draw [line width=.5pt] (-1.,6.5)-- (-0.8,6.4);
\draw [line width=0.5pt,dotted] (-1,5.93)--(-1,6.5);
\draw [line width=0.5pt,dotted] (0,5.93)--(0,6.5);

\draw [line width=.5pt] (1.,6.5)-- (1.3,6.5);
\draw [line width=.5pt] (1.,6.5)-- (1.05,6.55);
\draw [line width=.5pt] (1.,6.5)-- (1.05,6.45);
\draw [line width=.5pt] (1.3,6.5)-- (1.25,6.55);
\draw [line width=.5pt] (1.3,6.5)-- (1.25,6.45);

\draw [line width=0.8pt,dotted] (-0.3,2)--(-0.2,0.4);
\draw [line width=0.8pt,dotted] (-0.8,2)--(-1.8,0.5);

\draw (-1,7.6) node[anchor=north west] {$l_{1,\varepsilon}$};
\draw (1.2,7.2) node[anchor=north west] {$\theta_{1,\varepsilon} \, l_{1,\varepsilon} $};
\draw (-3.5,0.7) node[anchor=north west] {$\chi_\varepsilon = 1$};
\draw (-0.7,0.5) node[anchor=north west] {$\chi_\varepsilon = 0$};
\end{scope}
\end{scope}

\begin{scope}[xshift=-260,yshift=200]
\fill[line width=.5pt,fill opacity=0.2] (6.,-2.) -- (14.,-2.5) -- (20.,-0.5) -- (12.,0.) -- cycle;
\draw [line width=.5pt] (6.,-2.)-- (14.,-2.5);
\draw [line width=.5pt] (14.,-2.5)-- (20.,-0.5);
\draw [line width=.5pt] (20.,-0.5)-- (12.,0.);
\draw [line width=.5pt] (12.,0.)-- (6.,-2.);

\begin{scope}[yshift=-20]
\fill[line width=.5pt,fill opacity=0.5] (6.,2.) -- (12.,6.) -- (20.,4.) -- (14.,0.) -- cycle;
\draw [line width=.5pt] (6.,2.)-- (12.,6.);
\draw [line width=.5pt] (12.,6.)-- (20.,4.);
\draw [line width=.5pt] (20.,4.)-- (14.,0.);
\draw [line width=.5pt] (14.,0.)-- (6.,2.);
\end{scope}

\draw [line width=.5pt,dotted] (12.,4.5)-- (12.,0.);
\draw [line width=.5pt,dotted] (14.,-1.4)-- (14.,-2.5);
\draw [line width=.5pt,dotted] (20.,2.6)-- (20.,-0.5);
\draw [line width=.5pt,dotted] (6.,0.6)-- (6.,-2.);

\begin{scope}[yshift=10]
\filldraw [line width=.5pt,color=red,fill=red,fill opacity=0.7] (7.5,3.)-- (7.875,3.75)-- (15.875,1.75)-- (15.5,1.);
\filldraw [line width=.5pt,color=red,fill=red,fill opacity=0.7] (9.,4.)-- (9.375,4.75) -- (17.375,2.75) -- (17.,2.) ;
\filldraw [line width=.5pt,color=red,fill=red,fill opacity=0.7] (10.5,5.)-- (10.875,5.75) -- (18.875,3.75) -- (18.5,3.) ;
\filldraw [line width=.5pt,color=red,fill=red,fill opacity=0.7]  (6.,2.)-- (6.375,2.75) -- (14.375,0.75)--(14.,0.) ;

\filldraw [line width=.5pt,color=red,fill=red,fill opacity=0.4] (6.375,2.75) -- (7.5,3.) -- (15.5,1.) --(14.375,0.75)  ;
\filldraw [line width=.5pt,color=red,fill=red,fill opacity=0.4]  (7.875,3.75)-- (15.875,1.75) -- (17.,2) -- (9.,4.) ;
\filldraw [line width=.5pt,color=red,fill=red,fill opacity=0.4]  (9.375,4.75) -- (17.375,2.75)-- (18.5,3.) -- (10.5,5.) ;
\filldraw [line width=.5pt,color=red,fill=red,fill opacity=0.4]  (10.875,5.75) -- (18.875,3.75)  -- (20,4) --(12,6) ;
\end{scope}

\begin{scope}[yshift=50]
\filldraw [line width=.5pt,color=blue,fill=blue,fill opacity=0.5] (13.181898947551968,1.4207394377228377)  -- (19.160133564513558,5.406229182363898)  -- (20,4) -- (14.,0.) ;

\filldraw [line width=.5pt,color=blue,fill=blue,fill opacity=0.8]   (12.359505350588046,0.4101236623529884)--(18.36521962642836,4.413933179579864)  -- (19.160133564513558,5.406229182363898) --  (13.181898947551968,1.4207394377228377) ;

\filldraw [line width=.5pt,color=blue,fill=blue,fill opacity=0.5]  (12.359505350588046,0.4101236623529884)--(18.36521962642836,4.413933179579864)  -- (17.569626646813987,5.803855911788791)--(11.591392029852397,1.8183661671477305);

\filldraw [line width=.5pt,color=blue,fill=blue,fill opacity=0.8]   (17.569626646813987,5.803855911788791)--(11.591392029852397,1.8183661671477305)-- (10.769025860682943,0.8077435348292642)-- (16.77403366719961,4.811082072507042);

\filldraw [line width=.5pt,color=blue,fill=blue,fill opacity=0.5]   (10.769025860682943,0.8077435348292642)-- (16.77403366719961,4.811082072507042)-- (15.979061766299909,6.201497131917311)--(10.000827149338319,2.21600738727625);

\filldraw [line width=.5pt,color=blue,fill=blue,fill opacity=0.8]  (10.000827149338319,2.21600738727625) --  (15.979061766299909,6.201497131917311)-- (15.184089865400207,5.213174981073778) --(9.174056286392393,1.206485928401902);

\filldraw [line width=.5pt,color=blue,fill=blue,fill opacity=0.5]   (15.184089865400207,5.213174981073778) --(9.174056286392393,1.206485928401902) -- (8.408055572489173,2.6142002814885363)-- (14.386290189450763,6.599690026129597);

\filldraw [line width=.5pt,color=blue,fill=blue,fill opacity=0.8]  (8.408055572489173,2.6142002814885363)-- (14.386290189450763,6.599690026129597) -- (13.58849051350132,5.592203618804298) --  (7.60013460748718,1.599966348128205);

\filldraw [line width=.5pt,color=blue,fill=blue,fill opacity=0.5]   (13.58849051350132,5.592203618804298) --  (7.60013460748718,1.599966348128205)  --  (6.81601063978907,3.012211514663562)-- (12.79424525675066,6.997701259304622);

\filldraw [line width=.5pt,color=blue,fill=blue,fill opacity=0.8]   (6.81601063978907,3.012211514663562)-- (12.79424525675066,6.997701259304622)  --  (12.,6.) -- (6,2) ;
\end{scope}

\draw (0,6.) node[anchor=north west] {$v_\varepsilon = v + {\color{red}{\dots}} + {\color{blue}{\dots}}$};
\draw (0,2.5) node[anchor=north west] {$v= m + M {\rm id}$};
\end{scope}
\end{scope}
\end{tikzpicture}
\caption{Illustration of the recovery sequence in the case of two orthogonal directions of lamination, when ${\rm det}(\tau) > 0$.}
\label{fig:two laminations}
\end{figure}
We first prove that
\begin{equation}\label{eq:u eps chi eps to 0 plasticity}
(v_\e, \chi_\e) \to (v,0) \quad \text{strongly in } L^1(\bar x + Q;\R^2) \times L^1(\bar x + Q) \quad \text{when } \e \searrow 0.
\end{equation}
Indeed, using the change of variables \eqref{eq:change of variables}, we deduce from \eqref{eq:quantities cas 1} that
$$
\int_{\bar x + Q} \chi_\e \, dx = \int_0^L \int_0^L \big( \chi_\e^{(1)}(s) +  \chi_\e^{(2)}(t)(1 - \chi_\e^{(1)}(s)) \big) \, ds dt = L^2 \big( \theta_{1,\e} +  \theta_{2,\e}(1 - \theta_{1,\e}) \big) \to 0 
$$
and
\begin{align*}
\norme{v_\e - v}_{L^1(\bar x + Q; \R^2)}&  \leq  \sum_{i=1}^2 \lvert \tau_i \rvert \frac{ L m_{i,\e}}{2} \left( \frac{1 - \theta_{i,\e} }{ \theta_{i,\e}} \lvert \theta_{i,\e} l_{i,\e} \rvert^2 +  \lvert (1 - \theta_{i,\e} ) l_{i,\e} \rvert^2 \right) \\
&= \sum_{i=1}^2 \lvert \tau_i \rvert \frac{L^2 (1 - \theta_{i,\e}) }{2} l_{i,\e} \to 0 .
\end{align*}
Next, we prove that
\begin{equation}\label{eq:CV energie plasticity cas 1}
\int_T \left( \chi_\e \frac\kappa\e + \frac12 \big(\eta_\e \chi_\e \mathbf A_0 + (1-\chi_\e) \mathbf A_1 \big) e(v_\e):e(v_\e) \right) \, dx  \to \LL^2(T) \bar W_\alpha (\xi)  \quad \text{when } \e \searrow 0.
\end{equation}
On the one hand, note that for all $0 \leq s < s' \leq L$ and $ \lvert t \rvert \leq \min \left( L - s' \, ; \, s \right)$,
\begin{equation}\label{eq:almost translation invariance}
\int_{s + t}^{s' + t} \chi_\e^{(i)}(r) \, dr =   o_{\e \searrow 0}(\e) + (s' - s) \theta_{i,\e} \quad \text{for $i \in \{ 1, 2 \}$.}
\end{equation}
Thus, using the change of variables \eqref{eq:change of variables} together with Fubini's Theorem, we have
\begin{align*}
2 \int_T \chi_\e \, dx  + o_{\e \searrow 0}(\e) & = \int_{(0,1)^2} \chi_\e \, dx  = \left( \theta_{1,\e} + \big(1 - \theta_{1,\e} \big) \theta_{2,\e} \right)  + o_{\e \searrow 0 }(\e).
\end{align*}
Thus, we infer by \eqref{eq:quantities cas 1} that
\begin{equation}\label{eq:CV volume cas 1 plasticity}
\frac{\kappa}{\e} \int_T \chi_\e(x) \, dx \to \frac14   \sqrt{2 \kappa \alpha a}(\lvert \tau_1 \rvert + \lvert \tau_2 \rvert)  \quad \text{when } \e \searrow 0.
\end{equation}
On the other hand, recalling \eqref{eq:chi eps plasticity} together with \eqref{eq:eu eps plasticity cas 1}, we have 
$$
\int_T \big( \chi_\e \eta_\e \mathbf A_0 + (1 - \chi_\e) \mathbf A_1 \big) e(v_\e):e(v_\e) \, dx = T_\e^{(1)} + T_\e^{(2)} + T_\e^{(3)} + T_\e^{(4)}
$$
where:
$$
\begin{aligned}
T_\e^{(1)} &:= \ds \int_T \chi_\e^{(1)}  {\scriptstyle{( (x-\bar x ) \cdot b_1 )}} \chi_\e^{(2)}  {\scriptstyle{( (x-\bar x ) \cdot b_2 )}} \eta_\e \mathbf A_0 \bigg( \underset{=: M_\e^{(1)}}{\underbrace{   \xi + \sum_{i=1}^2 \tau_i \frac{1- \theta_{i,\e}}{\theta_{i,\e}} b_i \otimes b_i   }}\bigg):M_\e^{(1)} \, dx, \\
T_\e^{(2)} &:= \ds \int_T  \chi_\e^{(1)}  {\scriptstyle{((x - \bar x) \cdot b_1 )}}  ( 1 - \chi_\e^{(2)}  {\scriptstyle{( (x - \bar x ) \cdot b_2 )}} )   \eta_\e \mathbf A_0 \bigg( \underset{=: M_\e^{(2)}}{\underbrace{  \xi + \tau_1 \frac{1 - \theta_{1,\e}}{\theta_{1,\e}} b_1 \otimes b_1- \tau_2 b_2 \otimes b_2  }}\bigg):M_\e^{(2)}  \, dx , \\
T_\e^{(3)} & := \ds \int_T  ( 1 - \chi_\e^{(1)} {\scriptstyle{( (x - \bar x) \cdot b_1 ) }} ) \chi_\e^{(2)} {\scriptstyle{ ( (x - \bar x ) \cdot b_2 ) }}    \eta_\e \mathbf A_0 \bigg( \underset{=: M_\e^{(3)}}{\underbrace{ \xi - \tau_1  b_1 \otimes b_1 + \tau_2 \frac{1 - \theta_{2,\e}}{\theta_{2,\e}} b_2 \otimes b_2}} \bigg):  M_\e^{(3)}  \, dx 
\end{aligned}
$$
and 
$$  T_\e^{(4)} := \ds \int_T (1 - \chi_\e) \mathbf A_1 e(v_\e):e(v_\e) \, dx \underset{\e}{\to} \frac12 \mathbf A_1 \big( \xi - \tau \big) : \big( \xi - \tau \big).
$$
Using the fact that $\eta_\e/\theta_{i,\e} = \sqrt{2 \kappa \alpha a^{-1} }/\lvert \tau_i \rvert \, + o_{\e \searrow 0} \big( 1 \big)$, the change of variables \eqref{eq:change of variables} and \eqref{eq:almost translation invariance}, we infer that:

\vspace{2mm}
\begin{itemize}[leftmargin=5mm]
\item[$\mathbf{1.}$] $T_\e^{(1)} \to 0$ when $\e \searrow 0$. Indeed, one can check that $\theta_{i,\e} M_\e^{(1)} = \tau_i ( b_1 \otimes b_1 + b_2 \otimes b_2) + o_{\e \searrow 0} \big( 1 \big)$ and 
$$
 \iint_{\Phi ( T )} \chi_\e^{(1)}\chi_\e^{(2)} \, ds dt \leq \theta_{1,\e} \theta_{2,\e} L^2,
$$
so that $\ds T_\e^{(1)} \leq \theta_{1,\e} \theta_{2,\e} L^2 \, \eta_\e \mathbf A_0 M_\e^{(1)} :  M_\e^{(1)}  = \eta_\e    L^2 \mathbf A_0 \big(\theta_{1,\e} M_\e^{(1)}\big) : \big(\theta_{2,\e} M_\e^{(1)} \big)  \underset{\e}{\to} 0 $. 
  
\item[$\mathbf{2.}$] $T_\e^{(2)} + T_\e^{(3)}  \to \tfrac12 \sqrt{2 \kappa \alpha a} (\lvert \tau_1 \rvert + \lvert \tau_2 \rvert)$ when $\e \searrow 0$. Indeed, as $\theta_{1,\e} M_\e^{(2)} = \tau_1 b_1 \otimes b_1 + o_{\e \searrow 0} (1)$ and
$$
 \iint_{\Phi ( T )} \chi_\e^{(1)} (1 - \chi_\e^{(2)}) \, ds dt = \theta_{1,\e}(1 - \theta_{2,\e}) \left( \tfrac12 + o_{\e \searrow 0} (1) \right), 
$$
by isotropy of $\mathbf A_0$ we infer that $ T_\e^{(2)}  =  \frac{|\tau_1|^2}{2}  \frac{\eta_\e}{\theta_{1,\e}} \mathbf A_0 ( b_1 \otimes b_1) : ( b_1 \otimes b_1  ) + o_{\e \searrow 0} (1) \underset{\e}{\to} \frac12 \sqrt{2 \kappa \alpha a} \lvert \tau_1 \rvert $. Similarily, $  T_\e^{(3)} \underset{\e}{\to} \frac12 \sqrt{2 \kappa \alpha a} \lvert \tau_2 \rvert$.
\end{itemize}
\noindent
Therefore, we infer that
\begin{equation*}
\frac12 \int_T \big( \chi_\e \eta_\e \mathbf A_0 + (1 - \chi_\e) \mathbf A_1 \big) e(v_\e):e(v_\e) \, dx \to \frac14 \sqrt{2 \kappa \alpha h(\tau)} + \frac14 \mathbf A_1 \big( \xi - \tau \big) : \big( \xi - \tau \big) 
\end{equation*}
which, together with \eqref{eq:CV volume cas 1 plasticity}, yields \eqref{eq:CV energie plasticity cas 1}.

\bigskip
\textbf{Step 2: adaptation close to the boundary to allow gluing.} We now slightly modify the recovery sequence and the triangulation in a neighborhood of $\partial T$, such that the first two points of Lemma \ref{lem:T plasticity} are satisfied. This step is technical and quite involved. We treat both cases 1 and 2 at once, and denote by $\mathbf T_\e \in \mathcal{T}_\e(T)$ the triangulation associated to $(v_\e, \chi_\e)$ constructed above, which can be generically discribed as follows: there exists two unit orthogonal vectors $b,b^\perp \in \mathbb S^1$ with $b^\perp = (- b\cdot e_2, b \cdot e_1)$ and 
$$\gamma := \lvert {\rm arccos } (b \cdot e_1) \rvert \in [0,\tfrac{\pi}{4}],$$
such that $\mathbf T_\e $ is composed by rectangles (each one subdivided into two right triangles) parallel to $b$ and $b^\perp$, with side lengths in between $h_\e$ and $d_\e h_\e$ where $\sqrt{2} \geq d_\e \geq 1$ and $ d_\e \to 1$ as $\e \searrow 0$ (see \eqref{eq:widths triangulation case 1} and \eqref{eq:widths triangulation case 2}). We then introduce the vertices on $\partial T$ where the modified triangulation will be anchored, which are defined for all $i \in \llbracket 0, m_\e \rrbracket$, where $m_\e = \lfloor h_\e^{-1} \rfloor$, by:
$$
\bar x_{i,\e} = \frac{i}{m_\e} e_1, \quad \bar y_{i,\e} = \frac{i}{m_\e} e_2 \quad \text{and} \quad \bar z_{i,\e} = e_1 + \frac{i}{m_\e} (e_2-e_1)
$$
as illustrated in Figure \ref{fig:adaptation triangulation T}. Then, we introduce the closed triangle 
$$T^{\rm in}_\e = \left\{ x \in T \, : \, {\rm dist}(x,\partial T) \geq \frac{1}{m_\e} \right\} \subset \subset T$$
and define the triangulation $\mathbf T^{\rm in}_\e$ composed of all the triangles of $\mathbf T_\e$ which are contained in $T^{\rm in}_\e$ (up to switching the diagonal that subdivides a rectangle into two right triangles, so that each time three vertices at distance less than $d_\e h_\e$ from each other are inside $T^{\rm in}_\e$,  either the whole rectangle they are part of is included in $T^{\rm in}_\e$, or the triangle they form is added to $\mathbf  T^{\rm in}_\e$), as illustrated in Figure \ref{fig:adaptation triangulation T}.
\begin{figure}[hbtp]
\begin{tikzpicture}[line cap=round,line join=round,>=triangle 45,x=5.5cm,y=5.5cm]
\draw[line width=.5pt] (0.,1.) -- (0.,0.) -- (1.,0.) -- cycle;
\fill[line width=.5pt,fill opacity=0.15] (0.05,0.8792902994331602) -- (0.05,0.05) -- (0.8792902994331602,0.05) -- cycle;
\fill[line width=.5pt,color=red,pattern=north east lines,pattern color=red] (0.05,0.8792902994331602) -- (0.1,0.8292902994331602) -- (0.09980417919340398,0.05000038345935337) -- (0.05,0.05) -- cycle;
\fill[line width=.5pt,color=green,pattern=north east lines,pattern color=green] (0.1,0.1) -- (0.09980417919340398,0.05000038345935337) -- (0.8792902994331602,0.05) -- (0.8292902994331602,0.1) -- cycle;
\fill[line width=.5pt,pattern=north east lines] (0.1,0.8292902994331602) -- (0.1,0.7585799137110342) -- (0.7585799137110342,0.1) -- (0.8292902994331602,0.1) -- cycle;
\draw [line width=.5pt] (0.5,0.5)-- (0.46445924265117566,0.4648310567819846);
\draw [line width=.5pt] (0.4,0.05)-- (0.4,0.1);

\draw (-0.2086459882060116,0.7162920505372126) node[anchor=north west] {$T$};
\fill [white,fill opacity=0.5] (0.35,0.32) circle (10pt);
\draw (0.2868993028405091,0.37027972450823576) node[anchor=north west] {$T^{\rm in}_\varepsilon$};
\draw (0.77,0.6) node[anchor=north west] {$m_\varepsilon^{-1}$};
\draw (0.15,-0.10787801246648096) node[anchor=north west] {$d_\varepsilon \, h_\varepsilon$};
\draw [line width=.5pt,dotted] (0.38376763342403847,0.06596904026986557)-- (0.2514982203898729,-0.1081732685980446);
\draw [line width=.5pt] (0.4,0.1)-- (0.39168653337426623,0.08824585244998599);
\draw [line width=.5pt] (0.4,0.1)-- (0.4079891932748392,0.08838029634542365);
\draw [line width=.5pt] (0.4,0.05)-- (0.4075987516600748,0.058316292008567736);
\draw [line width=.5pt] (0.4,0.05)-- (0.39159064545473604,0.05812107120118555);
\draw [line width=.5pt,dotted] (0.4967499323525812,0.4767355671021258)-- (0.7811191734677155,0.5272682816177744);

\draw [line width=.5pt] (0.5,0.5)-- (0.48,0.495);
\draw [line width=.5pt] (0.5,0.5)-- (0.495,0.48);
\draw [line width=.5pt] (0.46445924265117566,0.46483105678198455)-- (0.48438048875699417,0.46983105678198455);
\draw [line width=.5pt] (0.46445924265117566,0.46483105678198455)-- (0.46983105678198455,0.48438048875699417);
\draw [color=green](0.6,-0.05919286106541889) node[anchor=north west] {$S_{\bar x}$};
\draw (0.76,0.38) node[anchor=north west] {$S_{\bar z}$};
\draw [color=red](-0.25,0.35984719206515103) node[anchor=north west] {$S_{\bar y}$};
\draw [line width=.5pt,dotted,color=red] (-0.11692837670808932,0.3318617113026523)-- (0.07545996505005505,0.5033706981413107);
\draw [line width=.5pt,dotted] (0.765970059887426,0.31993065134865867)-- (0.57656448311778,0.32142203384290785);
\draw [line width=.5pt,dotted,color=green] (0.5564407975835196,0.0775710178950593)-- (0.6325013047902278,-0.06261893656436582);
\draw [line width=.5pt] (0.,0.95)-- (0.05,0.95);
\draw [line width=.5pt] (0.,0.9)-- (0.05,0.95);
\draw [line width=.5pt] (0.,0.05)-- (0.05,0.);
\draw [line width=.5pt] (0.95,0.05)-- (0.95,0.);
\draw [line width=.5pt] (0.95,0.05)-- (0.9,0.);
\draw (-0.43,0.02) node[anchor=north west] {$\begin{aligned} \bar x_{0,\varepsilon} &= (0,0)\\ &=\bar y_{0,\varepsilon} \end{aligned}$};
\draw (-0.15,0.55) node[anchor=north west] {$\bar y_{i,\varepsilon}$};
\draw (-0.36,1.12) node[anchor=north west] {$\bar y_{m_\varepsilon,\varepsilon} = (0,1) = \bar z_{m_\varepsilon,\varepsilon}$};
\draw (1.,0.06078126203005548) node[anchor=north west] {$\begin{aligned}(1,0) &= \bar z_{0,\varepsilon}  \\
																				&= \bar x_{m_\varepsilon,\varepsilon} \end{aligned}$};
\draw (0.45,0.65) node[anchor=north west] {$\bar z_{i,\varepsilon}$};
\draw (0.4,-0.012246465071537614) node[anchor=north west] {$\bar x_{i,\varepsilon}$};
\draw [line width=.5pt] (-0.035,0.2)-- (-0.035,0.15);
\draw [line width=.5pt] (-0.035,0.15)-- (-0.0425,0.16);
\draw [line width=.5pt] (-0.035,0.15)-- (-0.0275,0.16);
\draw [line width=.5pt] (-0.035,0.2)-- (-0.0425,0.19);
\draw [line width=.5pt] (-0.035,0.2)-- (-0.0275,0.19);
\draw (-0.25,0.2294405365265919) node[anchor=north west] {$m_\varepsilon^{-1}$};
\draw [line width=.5pt] (0.3345750055964303,0.7277256713858388)-- (0.3810171081355779,0.680372939385139);
\draw [line width=.5pt] (0.3810171081355779,0.680372939385139)-- (0.37965116394325,0.6940323813084178);
\draw [line width=.5pt] (0.3810171081355779,0.680372939385139)-- (0.3669023514815232,0.6844707719621226);
\draw [line width=.5pt] (0.3345750055964303,0.7277256713858388)-- (0.34003878236574175,0.7108790263471283);
\draw [line width=.5pt] (0.3345750055964303,0.7277256713858388)-- (0.3518769653659166,0.7236278388088552);
\draw (0.35,0.8014910654890712) node[anchor=north west] {$\sqrt{2} m_\varepsilon^{-1}$};
\fill  (0.,1.) circle (1pt);
\fill  (0.,0.) circle (1pt);
\fill  (1.,0.) circle (1pt);
\fill  (0.1,0.) circle (1pt);
\fill  (0.,0.1) circle (1pt);
\fill  (0.9,0.) circle (1pt);
\fill  (0.9,0.1) circle (1pt);
\fill   (0.,0.9) circle (1pt);
\fill  (0.1,0.9) circle (1pt);
\fill   (0.,0.2) circle (1pt);
\fill   (0.,0.3) circle (1pt);
\fill   (0.,0.4) circle (1pt);
\fill   (0.,0.5) circle (1pt);
\fill   (0.,0.6) circle (1pt);
\fill   (0.,0.7) circle (1pt);
\fill   (0.,0.8) circle (1pt);
\fill  (0.2,0.8) circle (1pt);
\fill  (0.3,0.7) circle (1pt);
\fill  (0.4,0.6) circle (1pt);
\fill  (0.5,0.5) circle (1pt);
\fill  (0.6,0.4) circle (1pt);
\fill  (0.7,0.3) circle (1pt);
\fill  (0.8,0.2) circle (1pt);
\fill  (0.,0.05) circle (1pt);
\fill  (0.,0.15) circle (1pt);
\fill  (0.,0.25) circle (1pt);
\fill  (0.,0.35) circle (1pt);
\fill  (0.05,0.) circle (1pt);
\fill  (0.15,0.) circle (1pt);
\fill   (0.2,0.) circle (1pt);
\fill  (0.25,0.) circle (1pt);
\fill   (0.3,0.) circle (1pt);
\fill  (0.35,0.) circle (1pt);
\fill    (0.4,0) circle (1pt);
\fill  (0.45,0.) circle (1pt);
\fill   (0.5,0.) circle (1pt);
\fill  (0.55,0.) circle (1pt);
\fill   (0.6,0.) circle (1pt);
\fill  (0.,0.45) circle (1pt);
\fill  (0.,0.55) circle (1pt);
\fill  (0.,0.65) circle (1pt);
\fill  (0.,0.75) circle (1pt);
\fill  (0.,0.85) circle (1pt);
\fill  (0.,0.95) circle (1pt);
\fill(0.05,0.95) circle (1pt);
\fill(0.15,0.85) circle (1pt);
\fill(0.25,0.75) circle (1pt);
\fill(0.35,0.65) circle (1pt);
\fill(0.45,0.55) circle (1pt);
\fill(0.55,0.45) circle (1pt);
\fill(0.65,0.35) circle (1pt);
\fill(0.75,0.25) circle (1pt);
\fill  (0.5,0.5) circle (1pt);
\fill(0.85,0.15) circle (1pt);
\fill(0.95,0.05) circle (1pt);
\fill  (0.95,0.) circle (1pt);
\fill  (0.85,0.) circle (1pt);
\fill   (0.8,0.) circle (1pt);
\fill  (0.75,0.) circle (1pt);
\fill   (0.7,0.) circle (1pt);
\fill  (0.65,0.) circle (1pt);

\draw [line width=.pt] (0.8,0.8) -- (1.,0.8);
\draw [line width=.pt] (0.8,0.8) -- (0.9909022526850197,0.8596349722879523);
\draw (0.95,0.8) arc[start angle=0,end angle=13.5,radius=0.2]; 
\draw (1,0.87) node[anchor=north west] {$\gamma \in [0,\frac{\pi}{4}]$};

\draw [line width=1.pt,color=blue] (0.13995020450636275,0.05314513997560981)-- (0.07911991248884448,0.07735773269707912);
\draw [line width=1.pt,color=blue] (0.07911991248884448,0.07735773269707912)-- (0.05203846957519554,0.1640502934517487);
\draw [line width=1.pt,color=blue] (0.05203846957519554,0.1640502934517487)-- (0.08302795176346858,0.23536349037668552);
\draw [line width=1.pt,color=blue] (0.08302795176346858,0.23536349037668552)-- (0.05456682539202146,0.32647266557722354);
\draw [line width=1.pt,color=blue] (0.05456682539202146,0.32647266557722354)-- (0.10099242061921858,0.3828513543872846);
\draw [line width=1.pt,color=blue] (0.10099242061921858,0.3828513543872846)-- (0.05598855941136322,0.5269167937509449);

\draw [line width=1.pt,color=blue] (0.05598855941136322,0.5269167937509449)-- (0.08806127936084297,0.5775268507471532);
\draw [line width=1.pt,color=blue] (0.08806127936084297,0.5775268507471532)-- (0.05899808030546658,0.6705633687003167);
\draw [line width=1.pt,color=blue] (0.13995020450636275,0.05314513997560981)-- (0.2919195741573003,0.10061806929040569);
\draw [line width=1.pt,color=blue] (0.2919195741573003,0.10061806929040569)-- (0.37116476291679035,0.06258918935425603);
\draw [line width=1.pt,color=blue] (0.05899808030546658,0.6705633687003167)-- (0.09269168955717065,0.7383822956941709);
\draw [line width=1.pt,color=blue] (0.09269168955717065,0.7383822956941709)-- (0.08181385655403697,0.7732041920573088);

\draw [line width=1.pt,color=blue] (0.08181385655403697,0.7732041920573088)-- (0.15406787984577788,0.7575552879219803);

\draw [line width=1.pt,color=blue] (0.15406787984577788,0.7575552879219803)-- (0.20720942833240594,0.7168623489796634);
\draw [line width=1.pt,color=blue] (0.37116476291679035,0.06258918935425603)-- (0.40799992065512247,0.07409593471766557);
\draw [line width=1.pt,color=blue] (0.40799992065512247,0.07409593471766557)-- (0.4794874081155688,0.059211419672018734);

\draw [line width=1.pt,color=blue] (0.20720942833240594,0.7168623489796634)-- (0.23627262738778232,0.6238258310265);
\draw [line width=1.pt,color=blue] (0.23627262738778232,0.6238258310265)-- (0.3087206498868862,0.6058664708752461);
\draw [line width=1.pt,color=blue] (0.3087206498868862,0.6058664708752461)-- (0.36677433667097203,0.5631994214390925);
\draw [line width=1.pt,color=blue] (0.36677433667097203,0.5631994214390925)-- (0.3944732234382344,0.4745303120580204);
\draw [line width=1.pt,color=blue] (0.3944732234382344,0.4745303120580204)-- (0.4625496555718905,0.45392032232059554);

\draw [line width=1.pt,color=blue] (0.4794874081155688,0.059211419672018734)-- (0.5763941610576789,0.08948362151844394);
\draw [line width=1.pt,color=blue] (0.5763941610576789,0.08948362151844394)-- (0.6340791494283251,0.05750353817490293);
\draw [line width=1.pt,color=blue] (0.6340791494283251,0.05750353817490293)-- (0.6926422169910847,0.07579775347138229);
\draw [line width=1.pt,color=blue] (0.6926422169910847,0.07579775347138229)-- (0.7465764422929175,0.05060556323148212);

\draw [line width=1.pt,color=blue] (0.47946301650048917,0.39977761749709434)-- (0.5344652071282602,0.37638565137678553);
\draw [line width=1.pt,color=blue] (0.5344652071282602,0.37638565137678553)-- (0.5520065256569655,0.3202327415368017);

\draw [line width=1.pt,color=blue] (0.5520065256569655,0.3202327415368017)-- (0.6214894010624671,0.30357069147272703);
\draw [line width=1.pt,color=blue] (0.6214894010624671,0.30357069147272703)-- (0.6376510361333741,0.25183439607861147);
\draw [line width=1.pt,color=blue] (0.6376510361333741,0.25183439607861147)-- (0.6919195741573003,0.22557201079542802);
\draw [line width=1.pt,color=blue] (0.7203807005287475,0.13446283559489014)-- (0.6919195741573003,0.22557201079542802);
\draw [line width=1.pt,color=blue] (0.7203807005287475,0.13446283559489014)-- (0.7960586516594033,0.10810348129188985);
\draw [line width=1.pt,color=blue] (0.8463551828613287,0.08177493051255723)-- (0.7465764422929175,0.05060556323148212);
\draw [line width=1.pt,color=blue] (0.8463551828613287,0.08177493051255723)-- (0.7960586516594033,0.10810348129188985);
\draw [line width=1.pt,color=blue] (0.4625496555718905,0.45392032232059554)-- (0.47946301650048917,0.39977761749709434);

\begin{scope}[opacity=0.5]
\draw [line width=.pt] (0.08302795176346858,0.23536349037668552)-- (0.13995020450636275,0.05314513997560981);
\draw [line width=.pt] (0.10099242061921858,0.3828513543872846)-- (0.1982941912685987,0.07137091774825537);
\draw [line width=.pt] (0.08806127936084297,0.5775268507471532)-- (0.24191957415730034,0.08499882660227794);
\draw [line width=.pt] (0.09269168955717065,0.7383822956941709)-- (0.30978857262818305,0.043416197126446655);
\draw [line width=.pt] (0.08181385655403697,0.7732041920573088)-- (0.1431900468426442,0.7923771842851182);
\draw [line width=.pt] (0.1431900468426442,0.7923771842851182)-- (0.38175689081735464,0.028681886317677194);
\draw [line width=.pt] (0.07115281592861776,0.6316538778055052)-- (0.21936416395555708,0.677952858084852);
\draw [line width=.pt] (0.21936416395555708,0.677952858084852)-- (0.2802595235154391,0.696975646075784);
\draw [line width=.pt] (0.05899808030546658,0.6705633687003167)-- (0.20720942833240594,0.7168623489796634);
\draw [line width=.pt] (0.09269168955717065,0.7383822956941709)-- (0.15406787984577788,0.7575552879219803);
\draw [line width=.pt] (0.23627262738778232,0.6238258310265)-- (0.4185920485556868,0.04018863168108674);
\draw [line width=.pt] (0.08806127936084297,0.5775268507471532)-- (0.29716798694766433,0.642848619017432);
\draw [line width=.pt] (0.05598855941136322,0.5269167937509449)-- (0.3087206498868862,0.6058664708752461);
\draw [line width=.pt] (0.2802595235154391,0.696975646075784)-- (0.49401408795489254,0.01369666881771333);
\draw [line width=.pt] (0.3944732234382344,0.4745303120580204)-- (0.5344666836447851,0.026386112618184243);
\draw [line width=.pt] (0.47946301650048917,0.39977761749709434)-- (0.5906247242434024,0.04392903391817498);
\draw [line width=.pt] (0.07329353385195628,0.4715204637683568)-- (0.36677433667097203,0.5631994214390925);
\draw [line width=.pt] (0.08444968578281031,0.435807618550407)-- (0.3779304886018261,0.5274865762211427);
\draw [line width=.pt] (0.10099242061921858,0.3828513543872846)-- (0.3944732234382344,0.4745303120580204);
\draw [line width=.pt] (0.05456682539202146,0.32647266557722354)-- (0.4625496555718905,0.45392032232059554);
\draw [line width=.pt] (0.07148018632062009,0.27232996075372234)-- (0.47946301650048917,0.39977761749709434);
\draw [line width=.pt] (0.5520065256569655,0.3202327415368017)-- (0.6340791494283251,0.05750353817490293);
\draw [line width=.pt] (0.49371797130129236,0.013656832071749775)-- (0.8343900042828823,0.1200776152756963);
\draw [line width=.pt] (0.7046073955695312,0.03749506870824322)-- (0.7465764422929175,0.05060556323148212);
\draw [line width=.pt] (0.8343900042828823,0.1200776152756963)-- (0.8463551828613287,0.08177493051255723);
\draw [line width=.pt] (0.7582470978717011,0.013245699466594003)-- (0.8580258384401123,0.04441506674766911);
\draw [line width=.pt] (0.8343900042828823,0.1200776152756963)-- (0.8580258384401123,0.04441506674766911);
\draw [line width=.pt] (0.8196944858166332,0.03244093276386266)-- (0.7712359605731154,0.18756537192873765);
\draw [line width=.pt] (0.38175689081735464,0.028681886317677194)-- (0.7818280884736797,0.1536580688921588);
\draw [line width=.pt] (0.6796200828567605,0.26494489060185034)-- (0.7582470978717011,0.013245699466594003);
\draw [line width=.pt] (0.6105695932197253,0.33852695683328105)-- (0.7046073955695312,0.03749506870824322);
\draw [line width=.pt] (0.30978857262818305,0.043416197126446655)-- (0.7712359605731154,0.18756537192873765);
\draw [line width=.pt] (0.2919195741573003,0.10061806929040569)-- (0.6919195741573003,0.22557201079542802);
\draw [line width=.pt] (0.07911991248884448,0.07735773269707912)-- (0.6796200828567605,0.26494489060185034);
\draw [line width=.pt] (0.06295827741793729,0.1290940280911947)-- (0.6214894010624671,0.30357069147272703);
\draw [line width=.pt] (0.05203846957519554,0.1640502934517487)-- (0.6105695932197253,0.33852695683328105);
\draw [line width=.pt] (0.08302795176346858,0.23536349037668552)-- (0.5344652071282602,0.37638565137678553);
\end{scope}
\end{tikzpicture}
\caption{}
\label{fig:adaptation triangulation T}
\end{figure}
Let us also define the set of vertices $V_\e$ which compose its boundary $\Gamma_\e = \partial \left( \cup_{T_\e' \in \mathbf T^{\rm in}_\e} T_\e' \right)$:
$$
V_\e = \left\{ v \in T^{\rm in}_\e \, : \, v \text{ is a vertex of a triangle of } \mathbf T_\e \right\} \cap \Gamma_\e.
$$ 
In particular, one can check that ${\rm dist}(v,\partial T^{\rm in}_\e ) \leq  d_\e h_\e$ for all $v \in V_\e$.
Indeed, if $v \in V_\e$ is such that ${\rm dist}(x,\partial T^{\rm in}_\e) >d_\e h_\e$, then $v$ has four neighboring vertices of $\mathbf T_\e$ included in $T^{\rm in}_\e$, $v_j \in \bar B_{d_\e h_\e}(v) \subset T^{\rm in}_\e$ as illustrated in Figure \ref{fig:V_eps}, which is impossible by definition of $V_\e$ and $\mathbf T^{\rm in}_\e$.
Moreover, since all vertex of $ V_\e$ is at distance at most $\sqrt2 d_\e h_\e$ from its closest neighbor in $ V_\e$ and because $\Gamma_\e$ is connected, when one slides a rectangle $R$ of sides lengths $d_\e h_\e$ and $2 d_\e h_\e$ along the boundary $\partial T^{\rm in}_\e$ while remaining inside $T^{\rm in}_\e$  (and outside the right angle of $T^{\rm in}_\e$, as illustrated in Figure \ref{fig:V_eps}), the rectangle always contains at least one vertex of $V_\e$. Arguing in a similar way, one can also check that there exists at least one vertex of $ V_\e$ in each acute corner of $T^{\rm in}_\e$, and in each side of its obtuse corner, {\it i.e.}
$$
V_\e \cap C_\e^{(i)} \neq \emptyset \quad \text{for } i \in \llbracket 1, 4 \rrbracket,
$$
where $ C_\e^{(i)}$ are illustrated in Figure \ref{fig:V_eps}. In particular, we infer that 
\begin{equation}\label{eq:admissible lengths}
h_\e \leq m_\e^{-1} \leq {\rm dist}(\bar v,V_\e) \leq 2 \sqrt{2} (d_\e h_\e + m_\e^{-1}) \leq 6 h_\e \quad \text{for } \e>0 \text{ small enough,}
\end{equation}
for all $\bar v \in \bar V_\e := \{ \bar x_{i,\e}, \bar y_{i,\e}, \bar z_{i,\e} \}_{i \in \llbracket 0, m_\e \rrbracket}$.  

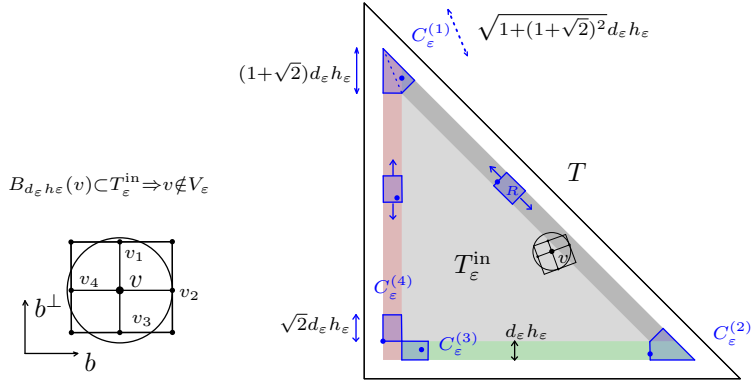
\begin{figure}[hbtp]
\begin{tikzpicture}[line cap=round,line join=round,>=triangle 45,x=.7cm,y=.7cm]
\draw[line width=.5pt] (-0.9260405493689347,0.9177754150343975) -- (0.9975154297806933,0.9248865277114756) -- (1.,-0.8) -- (-0.9189294366918566,-0.803113852818507) -- cycle;
\draw [line width=.5pt] (-0.9260405493689347,0.9177754150343975)-- (0.9975154297806933,0.9248865277114756);
\draw [line width=.5pt] (0.9975154297806933,0.9248865277114756)-- (1.,-0.8);
\draw [line width=.5pt] (1.,-0.8)-- (-0.9189294366918566,-0.803113852818507);
\draw [line width=.5pt] (-0.9189294366918566,-0.803113852818507)-- (-0.9260405493689347,0.9177754150343975);
\draw [line width=.5pt] (0.,0.9211988551429702)-- (0.,-0.801622703137993);
\draw [line width=.5pt] (-0.9222480889762307,0.)-- (0.99884765974833,0.);
\draw [line width=.5pt] (0.,0.) circle (0.99884765974833);
\draw (-2.3,2.4) node[anchor=north west] {$\scriptstyle B_{d_\varepsilon h\varepsilon}(v) \subset T^{\rm in}_\e \Rightarrow v \notin V_\e$};
\draw (0.,0.45) node[anchor=north west] {$ v$};
\draw (-0.0899789028551332,0.9342090021363277) node[anchor=north west] {$\scriptstyle v_1$};
\draw (0.95,0.2) node[anchor=north west] {$\scriptstyle v_2$};
\draw (0.05,-0.3) node[anchor=north west] {$\scriptstyle v_3$};
\draw (-0.95,0.45) node[anchor=north west] {$\scriptstyle v_4$};
\draw (-0.85,-1.) node[anchor=north west] {$b$};
\draw (-1.8,0.1) node[anchor=north west] {$b^\perp$};
\draw [line width=.5pt] (-1.8,-1.2)-- (-0.8,-1.2);
\draw [line width=.5pt] (-0.8,-1.2)-- (-0.879818316967927,-1.1728917120265692);
\draw [line width=.5pt] (-0.8,-1.2)-- (-0.8691516479523098,-1.2475583951358893);
\draw [line width=.5pt] (-1.8,-1.2)-- (-1.8,-0.2);
\draw [line width=.5pt] (-1.8,-0.2)-- (-1.8398185283734714,-0.28044707105326533);
\draw [line width=.5pt] (-1.8,-0.2)-- (-1.7459681575053432,-0.28691394734515413);

\fill(-0.9260405493689347,0.9177754150343975) circle (1.0pt);
\fill(0.9975154297806933,0.9248865277114756) circle (1.0pt);
\fill(1.,-0.8) circle (1.0pt);
\fill(-0.9189294366918566,-0.803113852818507) circle (1.0pt);
\fill(0.,0.9211988551429702) circle (1.0pt);
\fill(0.,-0.801622703137993) circle (1.0pt);
\fill(-0.9222480889762307,0.) circle (1.0pt);
\fill (0.99884765974833,0.) circle (1.0pt);
\fill  (0.,0.) circle (1.5pt);
\end{tikzpicture}
\begin{tikzpicture}[line cap=round,line join=round,>=triangle 45,x=5.0cm,y=5.0cm]
\draw[line width=.5pt] (0.,1.) -- (0.,0.) -- (1.,0.) -- cycle;
\fill[line width=.pt,fill opacity=0.15] (0.05,0.8792902994331602) -- (0.05,0.05) -- (0.8792902994331602,0.05) -- cycle;
\fill[line width=.pt,fill=red,fill opacity=0.15] (0.05,0.8792902994331602) -- (0.1,0.8292902994331602) -- (0.09980417919340398,0.05000038345935337) -- (0.05,0.05) -- cycle;
\fill[line width=.pt,fill=green,fill opacity=0.15] (0.1,0.1) -- (0.09980417919340398,0.05000038345935337) -- (0.8792902994331602,0.05) -- (0.8292902994331602,0.1) -- cycle;
\fill[line width=.pt,fill opacity=0.15] (0.1,0.8292902994331602) -- (0.1,0.7585799137110342) -- (0.7585799137110342,0.1) -- (0.8292902994331602,0.1) -- cycle;
\fill[line width=0.5pt,color=blue,fill=blue,fill opacity=0.15] (0.050246873107573864,0.760300640808397) -- (0.1,0.76) -- (0.13398797354371808,0.7939879735437181) -- (0.05055407702197197,0.8784110969835932) -- cycle;
\fill[line width=0.5pt,color=blue,fill=blue,fill opacity=0.15] (0.759364206441339,0.09852682462452662) -- (0.7605632880987034,0.04996401750126662) -- (0.8786728313491008,0.04996401750126662) -- (0.79473711533359,0.1338997335167777) -- cycle;
\fill[line width=0.5pt,color=blue,fill=blue,fill opacity=0.15] (0.0498784930253531,0.17010961501875332) -- (0.09907309451459545,0.17010961501875332) -- (0.1,0.1) -- (0.049949418349947504,0.0999999488298812) -- cycle;
\fill[line width=0.5pt,color=blue,fill=blue,fill opacity=0.15] (0.1,0.1) -- (0.09999981053851656,0.050097329534453415) -- (0.16975499320603563,0.05023311483806997) -- (0.1702368927243943,0.09978978255341087) -- cycle;
\fill[line width=0.5pt,color=blue,fill=blue,fill opacity=0.15] (0.05047654451612099,0.5397851753287766) -- (0.050843031639766015,0.4697861347125781) -- (0.10072907738090946,0.4700473181981338) -- (0.1,0.54) -- cycle;
\fill[line width=0.5pt,color=blue,fill=blue,fill opacity=0.15] (0.3785597673446907,0.5489184650462958) -- (0.4287245316957276,0.5000973997403755) -- (0.39423111092155,0.46465461876140435) -- (0.3443405845416286,0.5137574698725076) -- cycle;

\draw [line width=.pt] (0.,1.)-- (0.,0.);
\draw [line width=.pt] (0.,0.)-- (1.,0.);
\draw [line width=.pt] (1.,0.)-- (0.,1.);
\draw [line width=.5pt] (0.4,0.05)-- (0.4,0.1);
\draw [line width=.5pt] (0.4,0.1)-- (0.39168653337426623,0.08824585244998599);
\draw [line width=.5pt] (0.4,0.1)-- (0.4079891932748392,0.08838029634542365);
\draw [line width=.5pt] (0.4,0.05)-- (0.4075987516600748,0.058316292008567736);
\draw [line width=.5pt] (0.4,0.05)-- (0.39159064545473604,0.05812107120118555);
\begin{scope}[xshift=-60,yshift=10,blue]
\draw [line width=.5pt] (0.4,0.03)-- (0.4,0.1);
\draw [line width=.5pt] (0.4,0.1)-- (0.39168653337426623,0.08824585244998599);
\draw [line width=.5pt] (0.4,0.1)-- (0.4079891932748392,0.08838029634542365);
\draw [line width=.5pt] (0.4,0.03)-- (0.4075987516600748,0.038316292008567736);
\draw [line width=.5pt] (0.4,0.03)-- (0.39159064545473604,0.03812107120118555);
\end{scope}
\begin{scope}[xshift=29,yshift=18,blue,rotate=45]
\draw [line width=.5pt] (0.4,0.06)-- (0.4,0.1);
\draw [line width=.5pt] (0.4,0.06)-- (0.4075987516600748,0.068316292008567736);
\draw [line width=.5pt] (0.4,0.06)-- (0.39159064545473604,0.06812107120118555);
\end{scope}
\begin{scope}[xshift=17.5,yshift=30,blue,rotate=45]
\draw [line width=.5pt] (0.4,0.06)-- (0.4,0.1);
\draw [line width=.5pt] (0.4,0.1)-- (0.39168653337426623,0.08824585244998599);
\draw [line width=.5pt] (0.4,0.1)-- (0.4079891932748392,0.08838029634542365);
\end{scope}
\begin{scope}[xshift=-46,yshift=52,blue]
\draw [line width=.5pt] (0.4,0.06)-- (0.4,0.1);
\draw [line width=.5pt] (0.4,0.06)-- (0.4075987516600748,0.068316292008567736);
\draw [line width=.5pt] (0.4,0.06)-- (0.39159064545473604,0.06812107120118555);
\end{scope}
\begin{scope}[xshift=-46,yshift=68.5,blue]
\draw [line width=.5pt] (0.4,0.06)-- (0.4,0.1);
\draw [line width=.5pt] (0.4,0.1)-- (0.39168653337426623,0.08824585244998599);
\draw [line width=.5pt] (0.4,0.1)-- (0.4079891932748392,0.08838029634542365);
\end{scope}
\draw (0.21394296479083336,0.3643041763464595) node[anchor=north west] {$T^{\rm in}_\varepsilon$};
\draw (0.35,0.1705955270184695) node[anchor=north west] {$\scriptstyle d_\varepsilon  h_\varepsilon$};
\draw (-0.25,0.2) node[anchor=north west] {$\scriptstyle \sqrt 2 d_\varepsilon h_\varepsilon$};
\draw [line width=.5pt,color=blue] (0.050246873107573864,0.760300640808397)-- (0.1,0.76);
\draw [line width=.5pt,color=blue] (0.1,0.76)-- (0.13398797354371808,0.7939879735437181);
\draw [line width=.5pt,color=blue] (0.13398797354371808,0.7939879735437181)-- (0.05055407702197197,0.8784110969835932);
\draw [line width=.5pt,color=blue] (0.05055407702197197,0.8784110969835932)-- (0.050246873107573864,0.760300640808397);
\draw [line width=.5pt,dotted,color=blue] (0.05055407702197197,0.8784110969835932)-- (0.1,0.76);
\draw [line width=.5pt,color=blue] (-0.02,0.88)-- (-0.02,0.7595026307566123);
\draw [line width=.5pt,color=blue] (-0.02,0.7595026307566123)-- (-0.02505739526209852,0.7723287643040692);
\draw [line width=.5pt,color=blue] (-0.02,0.7595026307566123)-- (-0.014210475982627369,0.7717344125627283);
\draw [line width=.5pt,color=blue] (-0.02,0.88)-- (-0.024635317383801356,0.8667481107503986);
\draw [line width=.5pt,color=blue] (-0.02,0.88)-- (-0.015342916665414263,0.8663002842097534);
\draw [line width=.7pt,dotted,color=blue] (0.22312550610469958,0.9819682159875782)-- (0.27417972996736306,0.8597056250777013);
\draw [line width=.5pt,color=blue] (0.27417972996736306,0.8597056250777013)-- (0.2739021696711589,0.8699513573387057);
\draw [line width=.5pt,color=blue] (0.27417972996736306,0.8597056250777013)-- (0.26617545813777604,0.866928992338548);
\draw [line width=.5pt,color=blue] (0.22312550610469958,0.9819682159875782)-- (0.22355767338982954,0.96939213195474);
\draw [line width=.5pt,color=blue] (0.22312550610469958,0.9819682159875782)-- (0.231790826737489,0.9727937690491082);
\draw [line width=.5pt,color=blue] (0.759364206441339,0.09852682462452662)-- (0.7605632880987034,0.04996401750126662);
\draw [line width=.5pt,color=blue] (0.7605632880987034,0.04996401750126662)-- (0.8786728313491008,0.04996401750126662);
\draw [line width=.5pt,color=blue] (0.8786728313491008,0.04996401750126662)-- (0.79473711533359,0.1338997335167777);
\draw [line width=.5pt,color=blue] (0.79473711533359,0.1338997335167777)-- (0.759364206441339,0.09852682462452662);
\draw (-0.36,0.8682749843438573) node[anchor=north west] {$\scriptstyle (1+\sqrt2) d_\varepsilon h_\varepsilon$};
\draw (0.2746821175462195,0.9979612834702235) node[anchor=north west] {$\scriptstyle{ \sqrt{1 + (1 + \sqrt2)^2}d_\varepsilon h_\varepsilon}$};
\draw [line width=.5pt,color=blue] (0.0498784930253531,0.17010961501875332)-- (0.09907309451459545,0.17010961501875332);
\draw [line width=.5pt,color=blue] (0.09907309451459545,0.17010961501875332)-- (0.1,0.1);
\draw [line width=.5pt,color=blue] (0.1,0.1)-- (0.049949418349947504,0.0999999488298812);
\draw [line width=.5pt,color=blue] (0.049949418349947504,0.0999999488298812)-- (0.0498784930253531,0.17010961501875332);
\draw [line width=.5pt,color=blue] (0.1,0.1)-- (0.09999981053851656,0.050097329534453415);
\draw [line width=.5pt,color=blue] (0.09999981053851656,0.050097329534453415)-- (0.16975499320603563,0.05023311483806997);
\draw [line width=.5pt,color=blue] (0.16975499320603563,0.05023311483806997)-- (0.1702368927243943,0.09978978255341087);
\draw [line width=.5pt,color=blue] (0.1702368927243943,0.09978978255341087)-- (0.1,0.1);
\draw [color=blue](0.1,0.97) node[anchor=north west] {$\scriptstyle C_\varepsilon^{(1)}$};
\draw [color=blue](0.9017728297775038,0.17716192191094374) node[anchor=north west] {$\scriptstyle C_\varepsilon^{(2)}$};
\draw [line width=.5pt,color=blue] (0.05047654451612099,0.5397851753287766)-- (0.050843031639766015,0.4697861347125781);
\draw [line width=.5pt,color=blue] (0.050843031639766015,0.4697861347125781)-- (0.10072907738090946,0.4700473181981338);
\draw [line width=.5pt,color=blue] (0.10072907738090946,0.4700473181981338)-- (0.1,0.54);
\draw [line width=.5pt,color=blue] (0.1,0.54)-- (0.05047654451612099,0.5397851753287766);
\draw [line width=.5pt,color=blue] (0.3785597673446907,0.5489184650462958)-- (0.4287245316957276,0.5000973997403755);
\draw [line width=.5pt,color=blue] (0.4287245316957276,0.5000973997403755)-- (0.39423111092155,0.46465461876140435);
\draw [line width=.5pt,color=blue] (0.39423111092155,0.46465461876140435)-- (0.3443405845416286,0.5137574698725076);
\draw [line width=.5pt,color=blue] (0.3443405845416286,0.5137574698725076)-- (0.3785597673446907,0.5489184650462958);
\draw [color=blue](0.,0.3) node[anchor=north west] {$\scriptstyle C_\varepsilon^{(4)}$};
\draw [color=blue](0.1729029967128697,0.15) node[anchor=north west] {$\scriptstyle C_\varepsilon^{(3)}$};
\draw (0.5159971298446457,0.5990527937524136) node[anchor=north west] {$T$};
\draw [line width=.pt] (0.49881545648490955,0.3388668764372225) circle (0.050174249656008446);
\draw [line width=.pt] (0.49881545648490955,0.3388668764372225)-- (0.545681114398746,0.3562912877128799);
\draw [line width=.pt] (0.49881545648490955,0.3388668764372225)-- (0.48846348714598037,0.36702423303911025);
\draw [line width=.pt] (0.49881545648490955,0.3388668764372225)-- (0.4613296144767788,0.32490938207249287);
\draw [line width=.pt] (0.49881545648490955,0.3388668764372225)-- (0.5157976796737873,0.2918391814526377);
\draw [line width=.pt] (0.45079912155962176,0.35342321213292527)-- (0.5354916518488883,0.3840066258484935);
\draw [line width=.pt] (0.5354916518488883,0.3840066258484935)-- (0.562925430392585,0.309386748209637);
\draw [line width=.pt] (0.562925430392585,0.309386748209637)-- (0.47877562778615756,0.2780543748987329);
\draw [line width=.pt] (0.47877562778615756,0.2780543748987329)-- (0.45079912155962176,0.35342321213292527);
\draw (0.49,0.35445458400774815) node[anchor=north west] {$\scriptstyle v$};
\draw [color=blue](0.35,0.5366720422739083) node[anchor=north west] {$\scriptscriptstyle R$};

\begin{scope}[blue]
\fill(0.760165270087012,0.06608374697477139) circle (1.pt);
\fill(0.1523650486680179,0.07751335126039864) circle (1pt);
\fill(0.08722745505664807,0.48136643165089643) circle (1pt);
\fill(0.04994941834994752,0.0999999488298812) circle (1.pt);
\fill(0.3541751974541631,0.523862760204653) circle (1.pt);
\fill(0.1,0.8) circle (1.pt);
\end{scope}

\fill(0.5157976796737873,0.2918391814526377) circle (.5pt);
\fill(0.4613296144767788,0.32490938207249287) circle (.5pt);
\fill(0.48846348714598037,0.36702423303911025) circle (.5pt);
\fill(0.545681114398746,0.3562912877128799) circle (.5pt);
\fill (0.49881545648490955,0.3388668764372225) circle (1pt);
\end{tikzpicture}
\caption{Some illustrations of the properties of $V_\e$.}
\label{fig:V_eps}
\end{figure}

\medskip

The following procedure then consists of connecting the vertices $V_\e$ to the anchors $\bar V_\e$ in an admissible way. To do so, we first add the five triangles at the extremities of $T$ to $\mathbf T^{\rm in}_\e$, whose vertices are respectively $\left\{ \bar x_{0,\e}, \bar x_{1,\e}, \bar y_{1,\e} \right\}$, $\left\{ \bar y_{m_\e -1,\e}, \bar y_{m_\e,\e}, \bar z_{m_\e -1,\e} \right\}$, $\left\{ \bar y_{m_\e -2,\e}, \bar y_{m_\e-1,\e}, \bar z_{m_\e -1,\e} \right\}$, $\left\{ \bar x_{m_\e - 1,\e}, \bar x_{m_\e,\e}, \bar z_{ 1,\e} \right\}$ and $\left\{ \bar x_{m_\e - 2,\e}, \bar x_{m_\e-1,\e}, \bar z_{ 1,\e} \right\}$. We then introduce three sectors:
$$
S_{\bar x} := \left\{ x \in T^{\rm in}_\e \, : \, m_\e^{-1} + d_\e h_\e  \leq x \cdot e_2  \text{ and } x \cdot e_1 \leq m_\e^{-1} + d_\e h_\e \right\},
$$
$$
S_{\bar y} := \left\{ x \in T^{\rm in}_\e \, : \, x\cdot e_1 \leq m_\e^{-1} + d_\e h_\e \right\},
$$
and
$$
S_{\bar z} := \left\{ x \in T^{\rm in}_\e \, : \, m_\e^{-1} + d_\e h_\e  \leq x \cdot e_2 , \,  m_\e^{-1} + d_\e h_\e \leq x \cdot e_1 \text{ and } {\rm dist}(x,\partial T^{\rm in}_\e) < d_\e h_\e \right\}
$$
as illustrated in Figure \ref{fig:adaptation triangulation T}, as well as the corresponding vertices:
$$
V_{\bar x,\e} := V_\e \cap S_{\bar x}, \quad V_{\bar y,\e} = V_\e \cap S_{\bar y} \setminus V_{\bar x,\e}, \quad V_{\bar z,\e} = V_\e \cap S_{\bar z} \setminus (V_{\bar x,\e} \cup V_{\bar y,\e}).
$$
One can check that, except for three eventual pairs of vertices, the vertices of $V_{\bar y,\e}$ are well separated vertically, the ones of $V_{\bar x,\e}$ are well separated horizontally, and the vertices of $V_{\bar z,\e}$ are well separated along the diagonal direction (see Figure \ref{fig:well separated}). 
\begin{enumerate}[leftmargin=*,label=(\roman*)]
\item If $v,v' \in V_{\bar y,\e}$ are such that $v \neq v'$ and $(v - v') \cdot e_2 = 0$, then either $v \cdot e_2 > 1 - (1 + \sqrt{2})m_\e^{-1} -2 d_\e h_\e$ or $v \cdot e_2 < m_\e^{-1} + d_\e h_\e$. Moreover, there exists at most one such pair of vertices for each case, which are respectively the highest and lowest vertices of $V_{\bar y,\e}$. 
\item If $v,v' \in V_{\bar x,\e}$ are such that $v \neq v'$ and $(v - v') \cdot e_1 = 0$, then $v \cdot e_1 >  1 - (1 + \sqrt{2})m_\e^{-1} -2 d_\e h_\e$. Moreover, there is at most one such pair of vertices, which are the vertices the most on the right of $V_{\bar x,\e}$.
\item For all $v,v' \in V_{\bar z,\e}$, $(v - v') \cdot (e_2-e_1) \neq 0$.
\end{enumerate}
Indeed, let $v,v' \in S_{\bar x}$ be such that $v\cdot e_1 < v' \cdot e_1$ and $v\cdot e_2 = v' \cdot e_2$. By construction, the segment $[v,v']$ is a common edge of two right triangles $T_1,T_2$ such that exactly one of them (say $T_1$) is included in the inner triangle $T_\e$. Because $T_1$ and $T_2$ are right triangles, we have $T_1 \subset S_Y$ and $T_2 \subset \{ x \in T \, : \, m_\e^{-1}\leq x \cdot e_1 \leq m_\e^{-1} + d_\e h_\e \}$. Hence, either $v\cdot e_2 > 1 - (1 + \sqrt{2})m_\e^{-1} -2 d_\e h_\e$ or $v \cdot e_2 < m_\e^{-1} + d_\e h_\e$, otherwise $T_2 \subset T^{\rm in}_\e$. Moreover, by definition of $V_\e$, one can check that for each case, there is actually not enough room to have two such pairs of vertices. Indeed, in the case where $v \cdot e_2 > 1 - (1 + \sqrt{2})m_\e^{-1} -2 d_\e h_\e$, if there exists another pair $v_*,v_*'\in S_{\bar y}$ such that $v_* \cdot e_1 < v_*' \cdot e_1$ and $v_* \cdot e_2 = v_*' \cdot e_2 >   1 - (1 + \sqrt{2})m_\e^{-1} -2 d_\e h_\e$, then necessarily $(v-v_*) \cdot e_2 = 0$ (otherwise the whole triangle formed by $v, v'$ and $v_*$ is included in $ S_{\bar y}$ which contradicts the fact that $v' \in V_\e$). Hence, $2 h_\e \leq (v'-v) \cdot e_1 + (v_*' - v_*) \cdot e_1 \leq d_\e h_\e \leq \sqrt2 h_\e$, which is impossible. The remaining cases can be treated similarily and we omit the proof.

In the case there exists $v,v' \in V_{\bar y,\e}$ such that $v \cdot e_1 <  v' \cdot e_1$, $(v - v') \cdot e_2 = 0$, and $v \cdot e_2 < m_\e^{-1} + d_\e h_\e$, then we remove $v'$ from $V_{\bar y,\e}$ and add it to $V_{\bar x,\e}$, and one can check that $(ii)$ remains valid (otherwise $v' \notin V_\e$). Next, in the case there exists $v,v' \in V_{\bar y,\e}$ such that $v \cdot e_1 <  v' \cdot e_1$, $(v - v') \cdot e_2 = 0$, and $v \cdot e_2 > 1 - (1 + \sqrt{2})m_\e^{-1} -2 d_\e h_\e$, then we remove $v'$ from $V_{\bar y,\e}$ and add it to $V_{\bar z,\e}$. By doing so, one can check that $(iii)$ remains valid, otherwise $v' \notin V_\e$. Finally, if there exists $v,v' \in V_{\bar x,\e}$ are such that $v \cdot e_2 <  v' \cdot e_2$, $(v - v') \cdot e_1 = 0$ and $v \cdot e_1 >  1 - (1 + \sqrt{2})m_\e^{-1} -2 d_\e h_\e$, then we remove $v'$ from $V_{\bar x,\e}$ and add it to $V_{\bar z,\e}$. By doing so, one can also check that $(iii)$ remains valid, otherwise $v' \notin V_\e$ (see Figure \ref{fig:well separated}).

\begin{figure}[hbtp]
\begin{tikzpicture}[line cap=round,line join=round,>=triangle 45,x=1.0cm,y=1.0cm]
\filldraw [line width=0.5pt,color=red,fill=red,fill opacity=0.15] (0.,4.) -- (0.5,3.4995955752898946) -- (0.5,0.) -- (0.,0.) -- cycle;
\filldraw [line width=0.5pt,color=green,fill=green,fill opacity=0.15] (0.5,0.5) -- (0.5,0.) -- (3.9967672171536877,0.) -- (3.4971713150094765,0.5) -- cycle;
\filldraw [line width=0.5pt,fill opacity=0.15] (0.5,3.4995955752898946) -- (0.5,2.7922028226483517) -- (2.7903502741570203,0.5) -- (3.4971713150094765,0.5) -- cycle;
\filldraw [line width=0.5pt,color=blue,fill=blue,fill opacity=0.1] (0.7875289187562333,3.2118340889344315) -- (0.503851832562187,2.924094088176796) -- (0.15437885461986184,3.2738497371347086) -- (0.43357750193984973,3.56607179914909) -- cycle;
\filldraw [line width=0.5pt,color=blue,fill=blue,fill opacity=0.1] (2.879891986368259,0.4979862376410919) -- (2.878383280142858,0.) -- (3.3822911594265848,0.) -- (3.380782453201185,0.49647753141569156) -- cycle;
\filldraw [line width=0.5pt,color=blue,fill=blue,fill opacity=0.1] (0.,0.6) -- (0.5,0.6) -- (0.5,0.1) -- (0.,0.1) -- cycle;
\draw [line width=0.5pt] (0.,-0.2136308295209641)-- (0.5023088950562918,-0.2136308295209641);
\draw [line width=0.5pt] (0.,-0.2136308295209641)-- (0.07565414229770648,-0.17836249777928582);
\draw [line width=0.5pt] (0.,-0.2136308295209641)-- (0.07565414229770648,-0.2564431201307321);
\draw [line width=0.5pt] (0.5023088950562918,-0.2136308295209641)-- (0.43753824361320215,-0.17921175704510826);
\draw [line width=0.5pt] (0.5023088950562918,-0.2136308295209641)-- (0.44195146436095206,-0.2520298993829821);
\draw (-0.05,-0.23) node[anchor=north west] {${\scriptscriptstyle{d_\varepsilon  h_\varepsilon}}$};
\draw [line width=0.5pt,dotted] (0.5,0.4961108708228137)-- (0.,0.4961682030175763);
\draw [line width=0.5pt,dotted] (0.3976313550052226,0.)-- (0.4028979032503602,3.596776213014192);
\draw [line width=0.5pt,dotted] (0.,3.4283534610016226)-- (3.4255826804366674,0.);

\draw [line width=0.5pt] (0.40206209959228406,3.0259661537131946-0.3)-- (0.,3.0266031264770965-0.3);
\draw [line width=0.5pt] (0.,3.0266031264770965-0.3)-- (0.05,3.07-0.3);
\draw [line width=0.5pt] (0.,3.0266031264770965 -0.3)-- (0.05,2.97-0.3);
\draw [line width=0.5pt] (0.40206209959228406,3.0259661537131946-0.3)-- (0.35,3.07-0.3);
\draw [line width=0.5pt] (0.40206209959228406,3.0259661537131946-0.3)-- (0.35,2.97-0.3);
\draw (-0.05,3.05-0.3) node[anchor=north west] {${\scriptscriptstyle{h_\varepsilon}}$};
\draw [line width=0.5pt] (0.017108488483926737,4.062822142929176)-- (0.40384678733824564,3.679315640444116);
\draw [line width=0.5pt] (0.40384678733824564,3.679315640444116)-- (0.39291187320241955,3.7230571983559924);
\draw [line width=0.5pt] (0.40384678733824564,3.679315640444116)-- (0.3600256511410914,3.687311304811071);
\draw [line width=0.5pt] (0.017108488483926737,4.062822142929176)-- (0.03474752311912302,4.011357441233875);
\draw [line width=0.5pt] (0.017108488483926737,4.062822142929176)-- (0.07657138492340157,4.048778791269282);
\draw (0.0,4.25) node[anchor=north west] {${\scriptscriptstyle{\sqrt2 h_\varepsilon > d_\varepsilon  h_\varepsilon}}$};
\draw [line width=0.5pt] (0.846225881906347,3.2708676920931485)-- (0.4911838653209333,3.6229261623206988);
\draw [line width=0.5pt] (0.4911838653209333,3.6229261623206988)-- (0.5196617721857049,3.542105667220553);
\draw [line width=0.5pt] (0.4911838653209333,3.6229261623206988)-- (0.5810491107988853,3.5964773099922267);
\draw [line width=0.5pt] (0.846225881906347,3.2708676920931485)-- (0.8327371991129249,3.3360196018763038);
\draw [line width=0.5pt] (0.846225881906347,3.2708676920931485)-- (0.7713498604997444,3.2860327690055717);
\draw (0.5,3.8) node[anchor=north west] {${\scriptscriptstyle{d_\varepsilon  h_\varepsilon}}$};
\draw (0.75,3.35) node[anchor=north west] {${\scriptscriptstyle{v}}$};
\draw (0.45,3.15) node[anchor=north west] {${\scriptscriptstyle{v'}}$};
\draw [line width=0.5pt,dash pattern=on 1pt off 1pt,color=blue] (0.5642385217633922,3.2952443290220805)-- (1.662636712845795,3.5970020738249358);
\draw (1.5,4.6) node[anchor=north west] {$\begin{aligned}  & {\scriptstyle{\Rightarrow T_2  \subset T_\varepsilon}} \\ & {\scriptstyle{\Rightarrow v' \notin V_\varepsilon \, : \,  \text{impossible} }}\end{aligned}$};
\draw (-0.1,0.9) node[anchor=north west] {${\scriptscriptstyle{v}}$};
\draw (0.4,1) node[anchor=north west] {${\scriptscriptstyle{v'}}$};
\draw (-2.,1.5) node[anchor=north west] {$\begin{aligned}  & {\scriptstyle{\Rightarrow T_2  \subset T_\varepsilon}} \\ & {\scriptstyle{ \Rightarrow v' \notin V_\varepsilon }}\end{aligned}$};
\draw [line width=.5pt,dotted] (3.,0.)-- (3.,0.5);
\draw [line width=.5pt] (3.05,-0.2)-- (4.,-0.2);
\draw [line width=.5pt] (3.05,-0.2)-- (3.15,-0.15);
\draw [line width=.5pt] (3.05,-0.2)-- (3.15,-0.25);
\draw [line width=.5pt] (4.,-0.2)-- (3.9,-0.15);
\draw [line width=.5pt] (4.,-0.2)-- (3.9,-0.25);
\draw (3.1,-0.276478172428172) node[anchor=north west] {${\scriptscriptstyle{2 d_\varepsilon  h_\varepsilon}}$};
\draw (2.5,0.6) node[anchor=north west] {${\scriptscriptstyle{v'}}$};
\draw (2.5,0.25) node[anchor=north west] {${\scriptscriptstyle{v}}$};
\draw (3.5,2.5) node[anchor=north west] {$\begin{aligned}  & {\scriptstyle{\Rightarrow T_2  \subset T_\varepsilon}} \\ & {\scriptstyle{\Rightarrow v' \notin V_\varepsilon}}\end{aligned}$};
\draw [line width=.5pt,dash pattern=on 1pt off 1pt,color=blue] (0.09512964070727835,0.25140219790202656)-- (-0.6417866682411005,0.5539257352597806);
\draw [line width=0.5pt,dash pattern=on 1pt off 1pt,color=blue] (3.2987763311881255,0.37551441835648974)-- (3.8883093783468285,1.4382253059978307);
\draw (-0.1,1.900920802210346) node[anchor=north west] {${\scriptscriptstyle{S_{\bar y}}}$};
\draw (1.701492554517906,0.45) node[anchor=north west] {${\scriptscriptstyle{S_{\bar x}}}$};
\draw (2.1,1.55) node[anchor=north west] {${\scriptscriptstyle{S_{\bar z}}}$};
\filldraw [blue] (0.7875289187562333,3.2118340889344315) circle (1.pt);
\filldraw [blue] (0.503851832562187,2.924094088176796) circle (1.pt);
\filldraw [blue] (2.8797077409816048,0.4371714379419707) circle (1.pt);
\filldraw [blue] (2.8784438912726324,0.020006219877134357) circle (1pt);
\filldraw [blue] (0.059578183783078525,0.6) circle (1pt);
\filldraw [blue] (0.46184721538130064,0.6) circle (1.pt);

\begin{scope}[xshift=180]
\draw [line width=.pt] (3.,0.)-- (3.,0.41);
\draw [line width=.pt] (0.5,0.4961108708228137)-- (0.,0.4961682030175763);
\draw [line width=.pt] (0.3976313550052226,0.)-- (0.4028979032503602,3.596776213014192);
\draw [line width=.pt] (0.4,3.)-- (3.4255826804366674,0.);
\draw [line width=.pt] (0.,3.2)-- (0.4,3.2);

\fill [line width=0.5pt,fill=red,fill opacity=0.15] (0.,4.) -- (0.5,3.4995955752898946) -- (0.5,0.) -- (0.,0.) -- cycle;
\fill [line width=0.5pt,fill=green,fill opacity=0.15] (0.5,0.5) -- (0.5,0.) -- (3.9967672171536877,0.) -- (3.4971713150094765,0.5) -- cycle;
\fill [line width=0.5pt,fill opacity=0.15] (0.5,3.4995955752898946) -- (0.5,2.7922028226483517) -- (2.7903502741570203,0.5) -- (3.4971713150094765,0.5) -- cycle;
\draw [green, dash pattern=on 1pt off 1pt] (0.46184721538130064,0.4)--(0.7,0.6);
\draw (.5,0.85) node[anchor=north west] {{\color{green}{ ${\scriptscriptstyle{ \text{Added to } V_{\bar x,\e}}}$}}};
\draw [dash pattern=on 1pt off 1pt] (0.46184721538130064,3.3)--(2.1,2.8);
\draw [dash pattern=on 1pt off 1pt] (3.1,0.4371714379419707)--(2.7,2.6);
\draw (2.1,3) node[anchor=north west] { ${\scriptscriptstyle{ \text{Added to } V_{\bar z,\e}}}$};
\draw [line width=.pt] (0.7875289187562333,3.2118340889344315) -- (0.503851832562187,2.924094088176796);
\draw [line width=.pt] (2.93,0.5) -- (2.93+0.28,0.5+0.28);

\filldraw  (3.1,0.4371714379419707) circle (1.pt);
\filldraw [green] (3.1,0.020006219877134357) circle (1pt);
\draw [blue,line width=0.5pt] (3.1,0.4371714379419707)--(3.1,0.020006219877134357);

\filldraw [red] (0.059578183783078525,3.3) circle (1pt);
\filldraw  (0.46184721538130064,3.3) circle (1.pt);
\draw [blue,line width=0.5pt] (0.059578183783078525,3.3)--(0.46184721538130064,3.3);

\filldraw [red] (0.059578183783078525,0.4) circle (1pt);
\filldraw [green] (0.46184721538130064,0.4) circle (1.pt);
\draw [blue,line width=0.5pt] (0.059578183783078525,0.4)--(0.46184721538130064,0.4);
\end{scope}
\end{tikzpicture}
\caption{}
\label{fig:well separated}
\end{figure}

We next proceed in three steps, in the spirit of \cite[Appendix]{CDM}. First, for $i \in \llbracket 1, m_\e - 2 \rrbracket$, $\bar y_{i,\e}$ is connected to:
\begin{enumerate}[leftmargin=*,label=(\roman*)]
\item the lowest vertex $v \in V_{\bar y,\e} $ such that $\bar y _{i,\e} \cdot e_2 \leq v \cdot e_2 < \bar y _{i+1,\e} \cdot e_2$, provided it exists;
\item each vertex $v \in V_{\bar y,\e}$ such that $\bar y _{i-1,\e} \cdot e_2 \leq v \cdot e_2 < \bar y _{i,\e} \cdot e_2$ provided there exist such vertices, otherwise $\bar y_{i,\e}$ is connected to the highest vertex $v \in V_{\bar y,\e}$ such that $v \cdot e_2 < \bar y_{i,\e} \cdot e_2$ provided it exists. 
\end{enumerate}
Next, denoting by $v_0 $ the lowest vertex of $V_{\bar y ,\e}$, for $i \in \llbracket 1, m_\e - 2 \rrbracket$, we connect $\bar x_{i,\e}$ to the vertices $\{ v_0\} \cup V_{\bar x,\e} $ in the same manner as above, where we replace the scalar product with $e_2$ by the scalar product with $e_1$, "lowest" by "the most on the left" and "highest" by "the most on the right". Finally, denoting by $v_1$ the vertex of $V_{\bar x ,\e}$ the most on the right and $v_2$ the highest vertex of $ V_{\bar y,\e} $, we connect $\bar z_{1,\e}$ to $v_1$. Then, for $i \in \llbracket 1, m_\e - 1 \rrbracket$, we connect $\bar z_{i,\e}$ to the vertices $ \{ v_1, v_2 \} \cup V_{\bar z,\e} $ in the same manner as above, where we replace the scalar product with $e_2$ by the scalar product with $e_2 - e_1$, "lowest" by "the lowest most on the right" and "highest" by "the highest most on the left". We claim that all the triangles $\tilde T_\e$ created through this process are admissible, provided $\theta_0 < \Theta_0 = \arctan (1/4)$ (independently of $\gamma \in [0,\pi /4]$). Since the proof involves extensive calculations, we only provide the main arguments and refer to Figure \ref{fig:GLUING} for further details.
\begin{itemize}[leftmargin=*]
\item First, since $m_\e^{-1} \sim h_\e \sim d_\e h_\e$ and two succesive vertices $v,v' \in V_\e$ are such that either $\lvert v - v' \rvert \sim h_\e $ and $v - v'$ is parallel to $b$ or $b^\perp$, or $\lvert v - v' \rvert \sim \sqrt2 h_\e $ and $v - v'$ is parallel to $b - b^\perp$ or $b + b^\perp$, we infer that there are at most two vertices of $V_\e$ in each rectangle of $S_{\bar y}$ and $S_{\bar x}$ delimited by the heigths of $\bar y_{i,\e}$ and $\bar y_{i+1,\e}$, and the abscissae of $\bar x_{i,\e}$ and $\bar x_{i+1,\e}$ respectively. Arguing similarly along the diagonal in $S_{\bar z}$, a more accurate study shows that the triangles we have constructed are of five different kinds; then, studying every type of triangle, we may establish that their angles are all asymptotically greater than $\Theta_0 = \arctan (1/4)$, and the edges of the triangles have lengths in between $m_\e^{-1} $ (which is larger than $h_\e$) and $  6 h_\e$, for $\e>0$ small enough.
\item Second, it remains to control the triangles created near the acute corners of $T^{\rm in}_\e$ and near the right corner of $T^{\rm in}_\e$. The latest causes no difficulty if there exists a vertex $v \in V_\e$ in the rectangle delimited by $\bar x_{1,\e}$, $\bar x_{2,\e}$, $\bar y_{1,\e}$ and $\bar y_{2,\e}$. Otherwise, we may slightly rearrange the triangulation there, as illustrated in Figure \ref{fig:GLUING}, in a way that ensures the admissibility of the modified triangles, based on the fact that $C_\e^{(3)} \cap V_\e \neq \emptyset$ and $C_\e^{(4)} \cap V_\e \neq \emptyset$. As for the acute corners, two situations may occur according to the position of the vertices in $ V_\e \cap C_\e^{(1)}$ and $V_\e \cap C_\e^{(2)}$, which are illustrated in Figure \ref{fig:GLUING}. More precisely, in this example there exists $v \in V_\e \cap C_\e^{(2)}$ at distance less than $1$ from the acute corner of $T^{\rm in}_\e$, in which case the admissibility of the triangles poses to particular difficulty. Whereas in the other acute corner of $T^{\rm in}_\e$, the vertex $v \in V_\e \cap  C_\e^{(1)}$ is at distance strictly greater than $1$ from the acute corner. In this case, we may slitghly rearrange the triangulation near this corner in a way that ensures the admissibility of the modified triangles, provided $\theta_0 < \Theta_0 = \arctan (1/4)$.
\end{itemize}
Therefore, provided $\theta_0 < \Theta_0 = \arctan (1/4)$, the global triangulation we construct in this way is in $\mathcal{T}_{h_\e}(T,\theta_0)$ for $\e>0 $ sufficiently small.

\definecolor{qqqqff}{rgb}{0.,0.,1.}
\definecolor{uuuuuu}{rgb}{0.26666666666666666,0.26666666666666666,0.26666666666666666}
\definecolor{qqzzqq}{rgb}{0.,0.6,0.}
\definecolor{qqccqq}{rgb}{0.,0.8,0.}
\definecolor{ffqqqq}{rgb}{1.,0.,0.}
\definecolor{zzttqq}{rgb}{0.6,0.2,0.}

\begin{figure}[hbtp]
\begin{tikzpicture}[line cap=round,line join=round,>=triangle 45,x=8.0cm,y=8.0cm]

\begin{scope}[scale=0.7,yshift=80,xshift=15]
\draw[line width=.5pt] (0.,1.) -- (0.,0.) -- (1.,0.) -- cycle;
\fill[line width=.5pt,color=ffqqqq,fill=ffqqqq,fill opacity=0.15] (0.05,0.8792902994331602) -- (0.1,0.8292902994331602) -- (0.09980417919340398,0.05000038345935337) -- (0.05,0.05) -- cycle;
\fill[line width=.pt,color=qqccqq,fill=qqccqq,fill opacity=0.15] (0.1,0.1) -- (0.09980417919340398,0.05000038345935337) -- (0.8792902994331602,0.05) -- (0.8292902994331602,0.1) -- cycle;
\fill[line width=.pt,fill=black,fill opacity=0.15] (0.1,0.8292902994331602) -- (0.1,0.7585799137110342) -- (0.7585799137110342,0.1) -- (0.8292902994331602,0.1) -- cycle;

\draw [line width=.5pt] (0.,0.95)-- (0.05,0.95);
\draw [line width=.5pt] (0.,0.05)-- (0.05,0.);
\draw [line width=.5pt] (0.95,0.05)-- (0.95,0.);
\draw [line width=.5pt] (-0.03695146270062737,0.2003464727801436)-- (-0.03695146270062737,0.1517208437990187);
\draw [line width=.5pt] (-0.03695146270062737,0.1517208437990187)-- (-0.03387389234680213,0.15638287360456785);
\draw [line width=.5pt] (-0.03695146270062737,0.1517208437990187)-- (-0.04051752109316883,0.15648353464617945);
\draw [line width=.5pt] (-0.03695146270062737,0.2003464727801436)-- (-0.03951091067705266,0.19624464608276843);
\draw [line width=.5pt] (-0.03695146270062737,0.2003464727801436)-- (-0.03437719755486021,0.1960433239995452);
\draw (-0.2,0.25) node[anchor=north west] {$\scriptstyle m_\varepsilon^{-1}$};
\draw [line width=.5pt] (0.3345750055964303,0.7277256713858388)-- (0.3810171081355779,0.680372939385139);
\draw [line width=.5pt] (0.3810171081355779,0.680372939385139)-- (0.37965116394325,0.6940323813084178);
\draw [line width=.5pt] (0.3810171081355779,0.680372939385139)-- (0.3669023514815232,0.6844707719621226);
\draw [line width=.5pt] (0.3345750055964303,0.7277256713858388)-- (0.34003878236574175,0.7108790263471283);
\draw [line width=.5pt] (0.3345750055964303,0.7277256713858388)-- (0.3518769653659166,0.7236278388088552);
\draw (0.37383078760727495,0.7637293893474378) node[anchor=north west] {$\scriptstyle \sqrt{2} m_\varepsilon^{-1}$};
\draw [line width=.5pt,color=qqqqff] (0.13995020450636275,0.05314513997560981)-- (0.07911991248884448,0.07735773269707912);
\draw [line width=.5pt,color=qqqqff] (0.07911991248884448,0.07735773269707912)-- (0.05203846957519554,0.1640502934517487);
\draw [line width=.5pt,color=qqqqff] (0.05203846957519554,0.1640502934517487)-- (0.08302795176346858,0.23536349037668552);
\draw [line width=.5pt,color=qqqqff] (0.08302795176346858,0.23536349037668552)-- (0.05456682539202146,0.32647266557722354);
\draw [line width=.5pt,color=qqqqff] (0.05456682539202146,0.32647266557722354)-- (0.10099242061921858,0.3828513543872846);
\draw [line width=.5pt,color=qqqqff] (0.10099242061921858,0.3828513543872846)-- (0.05598855941136322,0.5269167937509449);
\draw [line width=.5pt,color=qqqqff] (0.05598855941136322,0.5269167937509449)-- (0.08806127936084297,0.5775268507471532);
\draw [line width=.5pt,color=qqqqff] (0.08806127936084297,0.5775268507471532)-- (0.05899808030546658,0.6705633687003167);
\draw [line width=.5pt,color=qqqqff] (0.13995020450636275,0.05314513997560981)-- (0.2919195741573003,0.10061806929040569);
\draw [line width=.5pt,color=qqqqff] (0.2919195741573003,0.10061806929040569)-- (0.37116476291679035,0.06258918935425603);
\draw [line width=.5pt,color=qqqqff] (0.05899808030546658,0.6705633687003167)-- (0.09269168955717065,0.7383822956941709);
\draw [line width=.5pt,color=qqqqff] (0.09269168955717065,0.7383822956941709)-- (0.08181385655403697,0.7732041920573088);
\draw [line width=.5pt,color=qqqqff] (0.08181385655403697,0.7732041920573088)-- (0.15406787984577788,0.7575552879219803);
\draw [line width=.5pt,color=qqqqff] (0.15406787984577788,0.7575552879219803)-- (0.20720942833240594,0.7168623489796634);
\draw [line width=.5pt,color=qqqqff] (0.37116476291679035,0.06258918935425603)-- (0.40799992065512247,0.07409593471766557);
\draw [line width=.5pt,color=qqqqff] (0.40799992065512247,0.07409593471766557)-- (0.4794874081155688,0.059211419672018734);
\draw [line width=.5pt,color=qqqqff] (0.20720942833240594,0.7168623489796634)-- (0.23627262738778232,0.6238258310265);
\draw [line width=.5pt,color=qqqqff] (0.23627262738778232,0.6238258310265)-- (0.3087206498868862,0.6058664708752461);
\draw [line width=.5pt,color=qqqqff] (0.3087206498868862,0.6058664708752461)-- (0.36677433667097203,0.5631994214390925);
\draw [line width=.5pt,color=qqqqff] (0.36677433667097203,0.5631994214390925)-- (0.3944732234382344,0.4745303120580204);
\draw [line width=.5pt,color=qqqqff] (0.3944732234382344,0.4745303120580204)-- (0.4625496555718905,0.45392032232059554);
\draw [line width=.5pt,color=qqqqff] (0.4794874081155688,0.059211419672018734)-- (0.5763941610576789,0.08948362151844394);
\draw [line width=.5pt,color=qqqqff] (0.5763941610576789,0.08948362151844394)-- (0.6340791494283251,0.05750353817490293);
\draw [line width=.5pt,color=qqqqff] (0.6340791494283251,0.05750353817490293)-- (0.6926422169910847,0.07579775347138229);
\draw [line width=.5pt,color=qqqqff] (0.6926422169910847,0.07579775347138229)-- (0.7465764422929175,0.05060556323148212);
\draw [line width=.5pt,color=qqqqff] (0.47946301650048917,0.39977761749709434)-- (0.5344652071282602,0.37638565137678553);
\draw [line width=.5pt,color=qqqqff] (0.5344652071282602,0.37638565137678553)-- (0.5520065256569655,0.3202327415368017);
\draw [line width=.5pt,color=qqqqff] (0.5520065256569655,0.3202327415368017)-- (0.6214894010624671,0.30357069147272703);
\draw [line width=.5pt,color=qqqqff] (0.6214894010624671,0.30357069147272703)-- (0.6376510361333741,0.25183439607861147);
\draw [line width=.5pt,color=qqqqff] (0.6376510361333741,0.25183439607861147)-- (0.6919195741573003,0.22557201079542802);
\draw [line width=.5pt,color=qqqqff] (0.7203807005287475,0.13446283559489014)-- (0.6919195741573003,0.22557201079542802);
\draw [line width=.5pt,color=qqqqff] (0.7203807005287475,0.13446283559489014)-- (0.7960586516594033,0.10810348129188985);
\draw [line width=.5pt,color=qqqqff] (0.8463551828613287,0.08177493051255723)-- (0.7465764422929175,0.05060556323148212);
\draw [line width=.5pt,color=qqqqff] (0.8463551828613287,0.08177493051255723)-- (0.7960586516594033,0.10810348129188985);
\draw [line width=.5pt,color=qqqqff] (0.4625496555718905,0.45392032232059554)-- (0.47946301650048917,0.39977761749709434);
\draw [line width=.5pt,color=ffqqqq] (0.,0.05)-- (0.07911991248884448,0.07735773269707912);
\draw [line width=.5pt,color=ffqqqq] (0.,0.1)-- (0.06295827741793729,0.1290940280911947);
\draw [line width=.5pt,color=ffqqqq] (0.,0.1)-- (0.07911991248884448,0.07735773269707912);
\draw [line width=.5pt,color=ffqqqq] (0.,0.15)-- (0.05203846957519554,0.1640502934517487);
\draw [line width=.5pt,color=ffqqqq] (0.,0.15)-- (0.06295827741793729,0.1290940280911947);
\draw [line width=.5pt,color=ffqqqq] (0.,0.2)-- (0.08302795176346858,0.23536349037668552);
\draw [line width=.5pt,color=ffqqqq] (0.,0.2)-- (0.05203846957519554,0.1640502934517487);
\draw [line width=.5pt,color=ffqqqq] (0.,0.25)-- (0.07148018632062009,0.27232996075372234);
\draw [line width=.5pt,color=ffqqqq] (0.,0.25)-- (0.08302795176346858,0.23536349037668552);
\draw [line width=.5pt,color=ffqqqq] (0.,0.3)-- (0.05456682539202146,0.32647266557722354);
\draw [line width=.5pt,color=ffqqqq] (0.,0.3)-- (0.07148018632062009,0.27232996075372234);
\draw [line width=.5pt,color=ffqqqq] (0.,0.35)-- (0.10099242061921858,0.3828513543872846);
\draw [line width=.5pt,color=ffqqqq] (0.,0.35)-- (0.05456682539202146,0.32647266557722354);
\draw [line width=.5pt,color=ffqqqq] (0.,0.4)-- (0.08444968578281031,0.435807618550407);
\draw [line width=.5pt,color=ffqqqq] (0.,0.4)-- (0.10099242061921858,0.3828513543872846);
\draw [line width=.5pt,color=ffqqqq] (0.,0.45)-- (0.07329353385195628,0.4715204637683568);
\draw [line width=.5pt,color=ffqqqq] (0.,0.45)-- (0.08444968578281031,0.435807618550407);
\draw [line width=.5pt,color=ffqqqq] (0.,0.5)-- (0.05598855941136322,0.5269167937509449);
\draw [line width=.5pt,color=ffqqqq] (0.,0.5)-- (0.07329353385195628,0.4715204637683568);
\draw [line width=.5pt,color=ffqqqq] (0.,0.55)-- (0.08806127936084297,0.5775268507471532);
\draw [line width=.5pt,color=ffqqqq] (0.,0.55)-- (0.05598855941136322,0.5269167937509449);
\draw [line width=.5pt,color=ffqqqq] (0.,0.6)-- (0.07115281592861776,0.6316538778055052);
\draw [line width=.5pt,color=ffqqqq] (0.,0.6)-- (0.08806127936084297,0.5775268507471532);
\draw [line width=.5pt,color=ffqqqq] (0.,0.65)-- (0.05899808030546658,0.6705633687003167);
\draw [line width=.5pt,color=ffqqqq] (0.,0.65)-- (0.07115281592861776,0.6316538778055052);
\draw [line width=.5pt,color=ffqqqq] (0.,0.7)-- (0.09269168955717065,0.7383822956941709);
\draw [line width=.5pt,color=ffqqqq] (0.,0.7)-- (0.05899808030546658,0.6705633687003167);
\draw [line width=.5pt,color=ffqqqq] (0.,0.75)-- (0.08181385655403697,0.7732041920573088);
\draw [line width=.5pt,color=ffqqqq] (0.,0.75)-- (0.09269168955717065,0.7383822956941709);
\draw [line width=.5pt,color=ffqqqq] (0.,0.8000680841009774)-- (0.08181385655403697,0.7732041920573088);
\draw [line width=.5pt,color=ffqqqq] (0.,0.85)-- (0.08181385655403697,0.7732041920573088);
\draw [line width=.5pt,color=qqzzqq] (0.05,0.)-- (0.07911991248884448,0.07735773269707912);
\draw [line width=.5pt,color=qqzzqq] (0.1,0.)-- (0.13995020450636275,0.05314513997560981);
\draw [line width=.5pt,color=qqzzqq] (0.15,0.)-- (0.1982941912685987,0.07137091774825537);
\draw [line width=.5pt,color=qqzzqq] (0.15,0.)-- (0.13995020450636275,0.05314513997560981);
\draw [line width=.5pt,color=qqzzqq] (0.2,0.)-- (0.24191957415730034,0.08499882660227794);
\draw [line width=.5pt,color=qqzzqq] (0.2,0.)-- (0.1982941912685987,0.07137091774825537);
\draw [line width=.5pt,color=qqzzqq] (0.25,0.)-- (0.2919195741573003,0.10061806929040569);
\draw [line width=.5pt,color=qqzzqq] (0.25,0.)-- (0.24191957415730034,0.08499882660227794);
\draw [line width=.5pt,color=qqzzqq] (0.3,0.)-- (0.2919195741573003,0.10061806929040569);
\draw [line width=.5pt,color=qqzzqq] (0.35,0.)-- (0.37116476291679035,0.06258918935425603);
\draw [line width=.5pt,color=qqzzqq] (0.35,0.)-- (0.2919195741573003,0.10061806929040569);
\draw [line width=.5pt,color=qqzzqq] (0.3998489369208939,-4.894250669973278E-5)-- (0.40799992065512247,0.07409593471766557);
\draw [line width=.5pt,color=qqzzqq] (0.3998489369208939,-4.894250669973278E-5)-- (0.37116476291679035,0.06258918935425603);
\draw [line width=.5pt,color=qqzzqq] (0.45,0.)-- (0.4794874081155688,0.059211419672018734);
\draw [line width=.5pt,color=qqzzqq] (0.45,0.)-- (0.40799992065512247,0.07409593471766557);
\draw [line width=.5pt,color=qqzzqq] (0.4993911564772603,0.)-- (0.5202361204590616,0.0719407002184532);
\draw [line width=.5pt,color=qqzzqq] (0.4993911564772603,0.)-- (0.4794874081155688,0.059211419672018734);
\draw [line width=.5pt,color=qqzzqq] (0.55,0.)-- (0.5763941610576789,0.08948362151844394);
\draw [line width=.5pt,color=qqzzqq] (0.55,0.)-- (0.5202361204590616,0.0719407002184532);
\draw [line width=.5pt,color=qqzzqq] (0.65,0.)-- (0.6926422169910847,0.07579775347138229);
\draw [line width=.5pt,color=qqzzqq] (0.65,0.)-- (0.6340791494283251,0.05750353817490293);
\draw [line width=.5pt,color=qqzzqq] (0.7,0.)-- (0.7465764422929175,0.05060556323148212);
\draw [line width=.5pt,color=qqzzqq] (0.7,0.)-- (0.6926422169910847,0.07579775347138229);
\draw [line width=.5pt,color=qqzzqq] (0.75,0.)-- (0.7465764422929175,0.05060556323148212);
\draw [line width=.5pt,color=qqzzqq] (0.8,0.)-- (0.8080238302378496,0.06980079652875078);
\draw [line width=.5pt,color=qqzzqq] (0.8,0.)-- (0.7465764422929175,0.05060556323148212);
\draw [line width=.5pt,color=qqzzqq] (0.85,0.)-- (0.8463551828613287,0.08177493051255723);
\draw [line width=.5pt,color=qqzzqq] (0.85,0.)-- (0.8080238302378496,0.06980079652875078);
\draw [line width=.5pt,color=qqzzqq] (0.9,0.)-- (0.8463551828613287,0.08177493051255723);
\draw [line width=.5pt] (0.95,0.05)-- (0.8463551828613287,0.08177493051255723);
\draw [line width=.5pt] (0.9,0.1)-- (0.8463551828613287,0.08177493051255723);
\draw [line width=.5pt] (0.85,0.15)-- (0.7960586516594033,0.10810348129188985);
\draw [line width=.5pt] (0.85,0.15)-- (0.8463551828613287,0.08177493051255723);
\draw [line width=.5pt] (0.8,0.2)-- (0.7203807005287475,0.13446283559489014);
\draw [line width=.5pt] (0.8,0.2)-- (0.7960586516594033,0.10810348129188985);
\draw [line width=.5pt] (0.75,0.25)-- (0.6919195741573003,0.22557201079542802);
\draw [line width=.5pt] (0.75,0.25)-- (0.7203807005287475,0.13446283559489014);
\draw [line width=.5pt] (0.7,0.3)-- (0.6376510361333741,0.25183439607861147);
\draw [line width=.5pt] (0.7,0.3)-- (0.6919195741573003,0.22557201079542802);
\draw [line width=.5pt] (0.65,0.35)-- (0.5520065256569655,0.3202327415368017);
\draw [line width=.5pt] (0.65,0.35)-- (0.6214894010624671,0.30357069147272703);
\draw [line width=.5pt] (0.65,0.35)-- (0.6376510361333741,0.25183439607861147);
\draw [line width=.5pt] (0.6,0.4)-- (0.5344652071282602,0.37638565137678553);
\draw [line width=.5pt] (0.6,0.4)-- (0.5520065256569655,0.3202327415368017);
\draw [line width=.5pt] (0.55,0.45)-- (0.47946301650048917,0.39977761749709434);
\draw [line width=.5pt] (0.55,0.45)-- (0.5344652071282602,0.37638565137678553);
\draw [line width=.5pt] (0.5,0.5)-- (0.3944732234382344,0.4745303120580204);
\draw [line width=.5pt] (0.5,0.5)-- (0.4625496555718905,0.45392032232059554);
\draw [line width=.5pt] (0.5,0.5)-- (0.47946301650048917,0.39977761749709434);
\draw [line width=.5pt] (0.45,0.55)-- (0.3779304886018261,0.5274865762211427);
\draw [line width=.5pt] (0.45,0.55)-- (0.3944732234382344,0.4745303120580204);
\draw [line width=.5pt] (0.4,0.6)-- (0.3087206498868862,0.6058664708752461);
\draw [line width=.5pt] (0.4,0.6)-- (0.36677433667097203,0.5631994214390925);
\draw [line width=.5pt] (0.4,0.6)-- (0.3779304886018261,0.5274865762211427);
\draw [line width=.5pt] (0.35,0.65)-- (0.23627262738778232,0.6238258310265);
\draw [line width=.5pt] (0.35,0.65)-- (0.3087206498868862,0.6058664708752461);
\draw [line width=.5pt] (0.3,0.7)-- (0.21936416395555708,0.677952858084852);
\draw [line width=.5pt] (0.3,0.7)-- (0.23627262738778232,0.6238258310265);
\draw [line width=.5pt] (0.25,0.75)-- (0.20720942833240594,0.7168623489796634);
\draw [line width=.5pt] (0.25,0.75)-- (0.21936416395555708,0.677952858084852);
\draw [line width=.5pt] (0.2,0.8)-- (0.15406787984577788,0.7575552879219803);
\draw [line width=.5pt] (0.2,0.8)-- (0.20720942833240594,0.7168623489796634);
\draw [line width=.5pt] (0.1,0.)-- (0.07911991248884448,0.07735773269707912);
\draw [line width=.5pt] (0.15,0.85)-- (0.08181385655403697,0.7732041920573088);
\draw [line width=.5pt] (0.15,0.85)-- (0.15406787984577788,0.7575552879219803);
\draw [line width=.5pt] (0.1,0.9)-- (0.08181385655403697,0.7732041920573088);
\draw [line width=.5pt] (0.,0.9)-- (0.05,0.95);
\draw [line width=.5pt] (0.,0.9)-- (0.05,0.8792902994331602);
\draw [line width=.5pt] (0.05,0.8792902994331602)-- (0.05,0.95);
\draw [line width=.5pt] (0.05,0.8792902994331602)-- (0.1,0.9);
\draw [line width=.5pt] (0.,0.85)-- (0.05,0.8792902994331602);
\draw [line width=.5pt] (0.05,0.8792902994331602)-- (0.08181385655403697,0.7732041920573088);

\fill [black] (0.,1.) circle (1.5pt);
\fill [black] (0.,0.) circle (1.5pt);
\fill[black] (1.,0.) circle (1.5pt);
\fill[qqzzqq] (0.1,0.) circle (1.5pt);
\fill[ffqqqq] (0.,0.1) circle (1.5pt);
\fill[qqzzqq] (0.9,0.) circle (1.5pt);
\fill[black] (0.9,0.1) circle (1.5pt);
\fill[black] (0.,0.9) circle (1.5pt);
\fill[black] (0.1,0.9) circle (1.5pt);
\fill[ffqqqq] (0.,0.2) circle (1.5pt);
\fill[ffqqqq] (0.,0.3) circle (1.5pt);
\fill[ffqqqq] (0.,0.4) circle (1.5pt);
\fill[ffqqqq] (0.,0.5) circle (1.5pt);
\fill[ffqqqq] (0.,0.6) circle (1.5pt);
\fill[ffqqqq] (0.,0.7) circle (1.5pt);
\fill[ffqqqq] (0.,0.800068084109774) circle (1.5pt);
\fill[black] (0.2,0.8) circle (1.5pt);
\fill[black] (0.3,0.7) circle (1.5pt);
\fill[black] (0.4,0.6) circle (1.5pt);
\fill[black] (0.5,0.5) circle (1.0pt);
\fill[black] (0.6,0.4) circle (1.5pt);
\fill[black] (0.7,0.3) circle (1.5pt);
\fill[black] (0.8,0.2) circle (1.5pt);
\fill[ffqqqq] (0.,0.05) circle (1.5pt);
\fill[ffqqqq] (0.,0.15) circle (1.5pt);
\fill[ffqqqq] (0.,0.25) circle (1.5pt);
\fill[ffqqqq] (0.,0.35) circle (1.5pt);
\fill[qqzzqq] (0.05,0.) circle (1.5pt);
\fill[qqzzqq] (0.15,0.) circle (1.5pt);
\fill[qqzzqq] (0.2,0.) circle (1.5pt);
\fill[qqzzqq] (0.25,0.) circle (1.5pt);
\fill[qqzzqq] (0.3,0.) circle (1.5pt);
\fill[qqzzqq] (0.35,0.) circle (1.5pt);
\fill[qqzzqq] (0.3998489369208939,-4.894250669973278E-5) circle (1.5pt);
\fill[qqzzqq] (0.45,0.) circle (1.5pt);
\fill[qqzzqq] (0.55,0.) circle (1.5pt);
\fill[ffqqqq] (0.,0.45) circle (1.5pt);
\fill[ffqqqq] (0.,0.55) circle (1.5pt);
\fill[ffqqqq] (0.,0.65) circle (1.5pt);
\fill[ffqqqq] (0.,0.75) circle (1.5pt);
\fill[ffqqqq] (0.,0.85) circle (1.5pt);
\fill[black] (0.,0.95) circle (1.5pt);
\fill[black] (0.05,0.95) circle (1.5pt);
\fill[black] (0.15,0.85) circle (1.5pt);
\fill[black] (0.25,0.75) circle (1.5pt);
\fill[black] (0.35,0.65) circle (1.5pt);
\fill[black] (0.45,0.55) circle (1.5pt);
\fill[black] (0.55,0.45) circle (1.5pt);
\fill[black] (0.65,0.35) circle (1.5pt);
\fill[black] (0.75,0.25) circle (1.5pt);
\fill[uuuuuu] (0.5,0.5) circle (1.5pt);
\fill[black] (0.85,0.15) circle (1.5pt);
\fill[black] (0.95,0.05) circle (1.5pt);
\fill[black] (0.95,0.) circle (1.5pt);
\fill[qqzzqq] (0.85,0.) circle (1.5pt);
\fill[qqzzqq] (0.8,0.) circle (1.5pt);
\fill[qqzzqq] (0.75,0.) circle (1.5pt);
\fill[qqzzqq] (0.7,0.) circle (1.5pt);
\fill[qqzzqq] (0.65,0.) circle (1.5pt);
\fill[qqqqff] (0.05456682539202146,0.32647266557722354) circle (1.5pt);
\fill[qqqqff] (0.07148018632062009,0.27232996075372234) circle (1.5pt);
\fill[qqqqff] (0.08302795176346858,0.23536349037668552) circle (1.5pt);
\fill[qqqqff] (0.13995020450636275,0.05314513997560981) circle (1.5pt);
\fill[qqqqff] (0.10099242061921858,0.3828513543872846) circle (1.5pt);
\fill[qqqqff] (0.08444968578281031,0.435807618550407) circle (1.5pt);
\fill[qqqqff] (0.07329353385195628,0.4715204637683568) circle (1.5pt);
\fill[qqqqff] (0.05598855941136322,0.5269167937509449) circle (1.5pt);
\fill[qqqqff] (0.05203846957519554,0.1640502934517487) circle (1.5pt);
\fill[qqqqff] (0.06295827741793729,0.1290940280911947) circle (1.5pt);
\fill[ffqqqq] (0.07911991248884448,0.07735773269707912) circle (1.5pt);
\fill[qqqqff] (0.05899808030546658,0.6705633687003167) circle (1.5pt);
\fill[qqqqff] (0.07115281592861776,0.6316538778055052) circle (1.5pt);
\fill[qqqqff] (0.08806127936084297,0.5775268507471532) circle (1.5pt);
\fill[qqqqff] (0.1982941912685987,0.07137091774825537) circle (1.5pt);
\fill[qqqqff] (0.24191957415730034,0.08499882660227794) circle (1.5pt);
\fill[qqqqff] (0.2919195741573003,0.10061806929040569) circle (1.5pt);
\fill[qqqqff] (0.37116476291679035,0.06258918935425603) circle (1.5pt);
\fill[ffqqqq] (0.08181385655403697,0.7732041920573088) circle (1.5pt);
\fill[qqqqff] (0.09269168955717065,0.7383822956941709) circle (1.5pt);
\fill[qqqqff] (0.15406787984577788,0.7575552879219803) circle (1.5pt);
\fill[qqqqff] (0.20720942833240594,0.7168623489796634) circle (1.5pt);
\fill[qqqqff] (0.40799992065512247,0.07409593471766557) circle (1.5pt);
\fill[qqqqff] (0.4794874081155688,0.059211419672018734) circle (1.50pt);
\fill[qqqqff] (0.21936416395555708,0.677952858084852) circle (1.50pt);
\fill[qqqqff] (0.23627262738778232,0.6238258310265) circle (1.50pt);
\fill[qqqqff] (0.3087206498868862,0.6058664708752461) circle (1.50pt);
\fill[qqqqff] (0.36677433667097203,0.5631994214390925) circle (1.50pt);
\fill[qqqqff] (0.3779304886018261,0.5274865762211427) circle (1.50pt);
\fill[qqqqff] (0.3944732234382344,0.4745303120580204) circle (1.50pt);
\fill[qqqqff] (0.4625496555718905,0.45392032232059554) circle (1.50pt);
\fill[qqqqff] (0.5202361204590616,0.0719407002184532) circle (1.50pt);
\fill[qqqqff] (0.5763941610576789,0.08948362151844394) circle (1.50pt);
\fill[qqqqff] (0.6340791494283251,0.05750353817490293) circle (1.50pt);
\fill[qqqqff] (0.6926422169910847,0.07579775347138229) circle (1.50pt);
\fill[qqqqff] (0.7465764422929175,0.05060556323148212) circle (1.50pt);
\fill[qqqqff] (0.8080238302378496,0.06980079652875078) circle (1.50pt);
\fill[qqqqff] (0.7960586516594033,0.10810348129188985) circle (1.50pt);
\fill[qqzzqq] (0.8463551828613287,0.08177493051255723) circle (1.5pt);
\fill[qqqqff] (0.47946301650048917,0.39977761749709434) circle (1.50pt);
\fill[qqqqff] (0.5344652071282602,0.37638565137678553) circle (1.50pt);
\fill[qqqqff] (0.5520065256569655,0.3202327415368017) circle (1.50pt);
\fill[qqqqff] (0.6214894010624671,0.30357069147272703) circle (1.50pt);
\fill[qqqqff] (0.6376510361333741,0.25183439607861147) circle (1.50pt);
\fill[qqqqff] (0.6919195741573003,0.22557201079542802) circle (1.50pt);
\fill[qqqqff] (0.7203807005287475,0.13446283559489014) circle (1.50pt);
\fill[qqzzqq] (0.4993911564772603,0.) circle (1.5pt);
\fill[black] (0.05,0.8792902994331602) circle (1.5pt);

\filldraw [fill=green] (0.8463551828613287,0.08177493051255723) circle (1.5pt);
\filldraw [fill=red] (0.08181385655403697,0.7732041920573088) circle (1.5pt);
\filldraw [fill=red] (0.07911991248884448,0.07735773269707912) circle (1.5pt);

\draw [line width=.5pt] (0.9,0.)-- (0.95,0.05);
\draw [line width=.5pt,color=qqzzqq] (0.6,0.)-- (0.5763941610576789,0.08948362151844394);
\draw [line width=.5pt,color=qqzzqq] (0.6,0.)-- (0.6340791494283251,0.05750353817490293);

\draw [color=ffqqqq](0.05,0.16) node[anchor=north west] {$\scriptstyle v_0$};
\draw [color=qqzzqq](0.9063300227332884,0.19829205720332474) node[anchor=north west] {$\scriptstyle v_1$};
\draw [color=ffqqqq](0.1,0.75) node[anchor=north west] {$\scriptstyle v_2$};
\end{scope}

\draw [line width=.5pt] (-0.8,1.05)-- (-0.8,0.95);
\draw [line width=.5pt] (-0.8,0.95)-- (-0.8,0.85);
\draw [line width=.5pt] (-0.8,0.85)-- (-0.8,0.75);
\draw [line width=.5pt] (-0.8,0.95)-- (-0.7,0.95);
\draw [line width=.5pt] (-0.7,0.95)-- (-0.6,0.95);
\draw [line width=.5pt] (-0.7,0.95)-- (-0.7,0.75);
\draw [line width=.5pt] (-0.7,0.95)-- (-0.7,1.05);
\draw [line width=.5pt] (-0.7,1.05)-- (-0.6,1.05);
\draw [line width=.5pt] (-0.6,1.05)-- (-0.6,0.95);
\draw [line width=.5pt] (-0.6,0.95)-- (-0.6,0.75);
\draw [line width=.5pt] (-0.6,0.75)-- (-0.7,0.75);
\draw [line width=.5pt] (-0.7,0.85)-- (-0.6,0.85);
\draw [line width=.5pt] (-0.8,0.95)-- (-0.7,0.95);
\draw [line width=.5pt] (-0.7,0.95)-- (-0.6,0.85);
\draw [line width=.5pt] (-0.6,0.85)-- (-0.8,0.95);

\fill[line width=.pt,fill=black,fill opacity=0.15] (-0.8,0.95) -- (-0.7,0.95) -- (-0.6,0.85) -- cycle;
\draw[line width=1.pt,color=ffqqqq] (-0.8,0.95) -- (-0.6524643078284141,0.970236151849565) -- (-0.6345455899166497,0.8316018606374934) -- cycle;
\draw[line width=1.pt,color=ffqqqq] (-0.8,0.6) -- (-0.658937451895758,0.5886693646823111) -- (-0.6124491499134396,0.5001321534160827) -- cycle;
\fill[line width=.pt,fill=black,pattern=north east lines,pattern color=black] (-0.6,0.5) -- (-0.6,0.6) -- (-0.6124491499134396,0.6) -- (-0.6124491499134396,0.5001321534160827) -- cycle;
\draw[line width=1.pt,color=ffqqqq] (-0.45,0.6) -- (-0.3506398504808524,0.5876538043736991) -- (-0.3038484647002904,0.4992763965324718) -- cycle;
\draw[line width=1.pt,color=ffqqqq] (-0.8,0.3) -- (-0.6178171429181265,0.28218526737637645) -- (-0.8,0.2) -- cycle;
\draw[line width=1.pt,dotted,color=ffqqqq] (-0.45,0.95) -- (-0.30414132892276446,0.9707097701013675) -- (-0.28485018583178084,0.8321642879024874) -- cycle;
\draw[line width=1.pt,color=ffqqqq] (-0.45,0.95) -- (-0.28485018583178084,0.8321642879024874) -- (-0.45,0.85) -- cycle;
\filldraw[line width=.pt,fill=black,fill opacity=0.10000000149011612] (-0.45,0.95) -- (-0.35,0.75) -- (-0.45,0.85) -- cycle;
\filldraw[line width=.pt,fill=black,fill opacity=0.10000000149011612] (-0.45,0.95) -- (-0.25,0.75) -- (-0.45,0.85) -- cycle;
\filldraw[line width=.pt,fill=black,fill opacity=0.10000000149011612] (-0.7,0.2) -- (-0.6,0.2) -- (-0.6,0.3) -- (-0.7,0.3) -- cycle;

\draw (-1.02,0.9833167028014623) node[anchor=north west] {$\scriptstyle V_\varepsilon \cap Q = \emptyset$};
\draw (-0.85,1) node[anchor=north west] {$\bar v$};
\draw (-0.6655491627417824,1.03) node[anchor=north west] {$v$};
\draw (-0.65,0.85) node[anchor=north west] {$v'$};
\draw [shift={(-0.8,0.95)},line width=.5pt]  plot[domain=-0.46364760900080615:0.,variable=\t]({1.*0.06512854122103169*cos(\t r)+0.*0.06512854122103169*sin(\t r)},{0.*0.06512854122103169*cos(\t r)+1.*0.06512854122103169*sin(\t r)});
\draw [shift={(-0.6,0.85)},line width=.5pt]  plot[domain=2.356194490192345:2.677945044588988,variable=\t]({1.*0.060766028310588166*cos(\t r)+0.*0.060766028310588166*sin(\t r)},{0.*0.060766028310588166*cos(\t r)+1.*0.060766028310588166*sin(\t r)});
\draw (-0.8412190135050034,1.15) node[anchor=north west] {$\scriptstyle \arctan\tfrac12$};
\draw (-0.634440960002462,0.9357394515530904) node[anchor=north west] {$\scriptstyle \arctan \tfrac13$};
\draw [color=ffqqqq](-0.77,0.9) node[anchor=north west] {$\scriptstyle \tilde T_\varepsilon$};
\draw (-1.05,0.913780720207688) node[anchor=north west] {$\scriptstyle \lvert v - v' \rvert \sim \sqrt2 h_\varepsilon$};

\draw [line width=.5pt] (-0.8,0.7)-- (-0.8,0.4);
\draw [line width=.5pt] (-0.7,0.7)-- (-0.7,0.4);
\draw [line width=.5pt] (-0.6,0.7)-- (-0.6,0.4);
\draw [line width=.5pt] (-0.8,0.6)-- (-0.6,0.6);
\draw [line width=.5pt] (-0.7,0.5)-- (-0.6,0.5);
\draw (-1.022,0.58) node[anchor=north west] {$\scriptstyle \# V_\varepsilon \cap Q = 2$};
\draw (-1.025,0.5441420758934133) node[anchor=north west] {$\scriptstyle \lvert v - v' \rvert \sim h_\varepsilon$};
\draw (-0.85,0.63) node[anchor=north west] {$\bar v$};
\draw [line width=.5pt] (-0.6124491499134396,0.663877912637784)-- (-0.6124491499134396,0.4332376133000251);
\draw [line width=0.5pt] (-0.6124491499134396,0.5001321534160827)-- (-0.6957277767483788,0.6583957603340425);
\draw [line width=.5pt] (-0.6,0.5)-- (-0.6,0.6);
\draw [line width=.5pt] (-0.6,0.6)-- (-0.6124491499134396,0.6);
\draw [line width=.5pt] (-0.6124491499134396,0.6)-- (-0.6124491499134396,0.5001321534160827);
\draw [line width=.5pt] (-0.6124491499134396,0.5001321534160827)-- (-0.6,0.5);
\draw [shift={(-0.6124491499134396,0.5001321534160827)},line width=.5pt]  plot[domain=1.5707963267948966:2.055185147081925,variable=\t]({1.*0.1637457592217012*cos(\t r)+0.*0.1637457592217012*sin(\t r)},{0.*0.1637457592217012*cos(\t r)+1.*0.1637457592217012*sin(\t r)});
\draw (-0.67,0.7106624552627153) node[anchor=north west] {$\scriptstyle \gamma$};
\draw (-0.66,0.64) node[anchor=north west] {$v$};
\draw (-0.6,0.55) node[anchor=north west] {$v'$};

\draw [line width=.5pt,dotted] (-0.6124491499134396,0.5001321534160827)-- (-0.7012644488954771,0.45417736403857295);
\draw [line width=.5pt,dotted] (-0.7472184065637304,0.5429910556036844)-- (-0.658937451895758,0.5886693646823111);
\draw [line width=.5pt,dotted] (-0.7472184065637304,0.5429910556036844)-- (-0.7012644488954771,0.45417736403857295);
\draw [line width=.pt] (-0.6124491499134396,0.46548624631340096)-- (-0.6,0.46574961527338965);
\draw [line width=.pt] (-0.6,0.46574961527338965)-- (-0.6021887199764262,0.46785656695329997);
\draw [line width=.pt] (-0.6,0.46574961527338965)-- (-0.6019253510164374,0.46377434807347384);
\draw [line width=.pt] (-0.6124491499134396,0.46548624631340096)-- (-0.6081145215761738,0.4686466738332663);
\draw [line width=.pt] (-0.6124491499134396,0.46548624631340096)-- (-0.6086412594961513,0.4629842411935075);
\draw (-0.6,0.4892452475299072) node[anchor=north west] {$\scriptstyle 1 - \cos \gamma$};
\draw [shift={(-0.8,0.6)},line width=.5pt]  plot[domain=5.7938893778396:6.203033902931456,variable=\t]({1.*0.03328028368804163*cos(\t r)+0.*0.03328028368804163*sin(\t r)},{0.*0.03328028368804163*cos(\t r)+1.*0.03328028368804163*sin(\t r)});
\draw [shift={(-0.6124491499134396,0.5001321534160827)},line width=.5pt]  plot[domain=2.051247419212382:2.6522967242498066,variable=\t]({1.*0.024669421756758923*cos(\t r)+0.*0.024669421756758923*sin(\t r)},{0.*0.024669421756758923*cos(\t r)+1.*0.024669421756758923*sin(\t r)});
\draw (-.695,0.445) node[anchor=north west] {$ \scriptstyle\arctan \tfrac{2 - \cos \gamma}{\cos \gamma} - \gamma $};
\draw (-1.1,0.69) node[anchor=north west] {$\scriptstyle{ \geq  \arctan\tfrac{(\cos \gamma)^2}{2}} $};

\draw (-0.68,1.0949402538072581) node[anchor=north west] {$\scriptstyle 1$};

\draw [line width=.5pt] (-0.45,0.6)-- (-0.25,0.6);
\draw [line width=.5pt] (-0.25,0.6)-- (-0.25,0.5);
\draw [line width=.5pt] (-0.25,0.5)-- (-0.35,0.5);
\draw [line width=.5pt] (-0.35,0.5)-- (-0.35,0.6);
\draw [line width=.pt] (-0.34983741603467217,0.47957862199204093)-- (-0.3038484647002904,0.47923926850791415);
\draw [line width=.pt] (-0.3038484647002904,0.47923926850791415)-- (-0.31250853278073143,0.4836508638015617);
\draw [line width=.pt] (-0.3038484647002904,0.47923926850791415)-- (-0.3101330583918443,0.4751670266983934);
\draw [line width=.pt] (-0.34983741603467217,0.47957862199204093)-- (-0.3407116677144845,0.48575413836698145);
\draw [line width=.pt] (-0.34983741603467217,0.47957862199204093)-- (-0.3414178799311827,0.47516095511650835);
\draw [shift={(-0.45,0.6)},line width=.5pt]  plot[domain=5.67976321654871:6.159561933498742,variable=\t]({1.*0.04802243963202228*cos(\t r)+0.*0.04802243963202228*sin(\t r)},{0.*0.04802243963202228*cos(\t r)+1.*0.04802243963202228*sin(\t r)});
\draw [shift={(-0.3038484647002904,0.4992763965324718)},line width=.5pt]  plot[domain=2.062613342255034:2.538170562958916,variable=\t]({1.*0.04126402963068424*cos(\t r)+0.*0.04126402963068424*sin(\t r)},{0.*0.04126402963068424*cos(\t r)+1.*0.04126402963068424*sin(\t r)});
\draw (-0.36,0.4819256704147731) node[anchor=north west] {$\scriptstyle \sin \gamma$};
\draw (-0.31603935549412404,0.5679307015175993) node[anchor=north west] {$ \scriptstyle \arctan(1 + \sin \gamma) - \gamma $};
\draw (-0.5411163517845009,0.68) node[anchor=north west] {$ \scriptstyle \arctan \tfrac{1}{1+\sin \gamma} - \arctan(1-\cos \gamma) $};

\draw (-1.,0.30625581965155346) node[anchor=north west] {$\scriptstyle \# V_\varepsilon \cap Q = 1$};
\draw [line width=.pt] (-0.8,0.3)-- (-0.8,0.1);
\draw [line width=.pt] (-0.8,0.3)-- (-0.6,0.3);
\draw [line width=.pt] (-0.6,0.3)-- (-0.6,0.2);
\draw [line width=.pt] (-0.6,0.2)-- (-0.7,0.2);
\draw [line width=.pt] (-0.7,0.2)-- (-0.7,0.3);
\draw [shift={(-0.8,0.3)},line width=.5pt]  plot[domain=4.71238898038469:6.124397888600517,variable=\t]({1.*0.027800308754739322*cos(\t r)+0.*0.027800308754739322*sin(\t r)},{0.*0.027800308754739322*cos(\t r)+1.*0.027800308754739322*sin(\t r)});
\draw [shift={(-0.8,0.2)},line width=.5pt]  plot[domain=0.42378007991997063:1.5707963267948966,variable=\t]({1.*0.044293900208140106*cos(\t r)+0.*0.044293900208140106*sin(\t r)},{0.*0.044293900208140106*cos(\t r)+1.*0.044293900208140106*sin(\t r)});
\draw [shift={(-0.6178171429181265,0.28218526737637645)},line width=.5pt]  plot[domain=3.0441176461193744:3.565372733509764,variable=\t]({1.*0.039350340058558586*cos(\t r)+0.*0.039350340058558586*sin(\t r)},{0.*0.039350340058558586*cos(\t r)+1.*0.039350340058558586*sin(\t r)});
\draw (-0.7552139824021764,0.17267353730035523) node[anchor=north west] {$\scriptstyle \geq \arctan \tfrac12$};

\draw [line width=.5pt] (-0.45,0.95)-- (-0.45,0.85);
\draw [line width=.5pt] (-0.45,0.95)-- (-0.45,1.05);
\draw [line width=.5pt] (-0.45,0.85)-- (-0.45,0.75);
\draw [line width=.5pt] (-0.45,0.95)-- (-0.25,0.95);
\draw [line width=.5pt] (-0.35,0.95)-- (-0.35,0.75);
\draw [line width=.5pt] (-0.35,0.75)-- (-0.25,0.75);
\draw [line width=.5pt] (-0.25,0.95)-- (-0.25,0.75);
\draw [line width=.5pt] (-0.35,0.85)-- (-0.25,0.85);
\draw [line width=.5pt,dotted,color=ffqqqq] (-0.45,0.95)-- (-0.30414132892276446,0.9707097701013675);
\draw [line width=.5pt,dotted,color=ffqqqq] (-0.30414132892276446,0.9707097701013675)-- (-0.28485018583178084,0.8321642879024874);
\draw [line width=.5pt,color=ffqqqq] (-0.28485018583178084,0.8321642879024874)-- (-0.45,0.95);
\draw [line width=.5pt,color=ffqqqq] (-0.45,0.95)-- (-0.28485018583178084,0.8321642879024874);
\draw [line width=.5pt,color=ffqqqq] (-0.28485018583178084,0.8321642879024874)-- (-0.45,0.85);
\draw [line width=.5pt,color=ffqqqq] (-0.45,0.85)-- (-0.45,0.95);

\draw [shift={(-0.45,0.95)},line width=.5pt]  plot[domain=4.71238898038469:5.176036589385496,variable=\t]({1.*0.06873098422955504*cos(\t r)+0.*0.06873098422955504*sin(\t r)},{0.*0.06873098422955504*cos(\t r)+1.*0.06873098422955504*sin(\t r)});
\draw [shift={(-0.35,0.75)},line width=.5pt]  plot[domain=2.0344439357957027:2.356194490192345,variable=\t]({1.*0.05434244016595504*cos(\t r)+0.*0.05434244016595504*sin(\t r)},{0.*0.05434244016595504*cos(\t r)+1.*0.05434244016595504*sin(\t r)});
\draw [shift={(-0.25,0.75)},line width=.5pt]  plot[domain=2.3561944901923453:2.677945044588987,variable=\t]({1.*0.0806749251325863*cos(\t r)+0.*0.0806749251325863*sin(\t r)},{0.*0.0806749251325863*cos(\t r)+1.*0.0806749251325863*sin(\t r)});
\draw (-0.31,0.9467188172257917) node[anchor=north west] {$\scriptstyle \sqrt2$};
\draw (-0.5173277261603147,1.05) node[anchor=north west] {$\scriptstyle \arctan \tfrac12$};
\draw (-0.4459618492877562,0.76) node[anchor=north west] {$\scriptstyle \arctan \tfrac13$};
\draw (-0.2757816813608859,0.8350952662199959) node[anchor=north west] {$\scriptstyle \arctan \tfrac13$};
\draw (-0.6472502199539469,0.25684867412439794) node[anchor=north west] {$\scriptstyle Q$};

\fill[ffqqqq] (-0.8,0.85) circle (1.5pt);
\fill[ffqqqq] (-0.8,0.6) circle (1.5pt);
\fill[qqqqff] (-0.6524643078284141,0.970236151849565) circle (1.5pt);
\fill[qqqqff] (-0.6345455899166497,0.8316018606374934) circle (1.5pt);
\fill[qqqqff] (-0.658937451895758,0.5886693646823111) circle (1.5pt);
\fill[qqqqff] (-0.6124491499134396,0.5001321534160827) circle (1.5pt);
\fill[qqqqff] (-0.7012644488954771,0.45417736403857295) circle (1.5pt);
\fill[ffqqqq] (-0.45,0.6) circle (1.5pt);
\fill[qqqqff] (-0.3506398504808524,0.5876538043736991) circle (1.5pt);
\fill[qqqqff] (-0.3038484647002904,0.4992763965324718) circle (1.5pt);
\fill[ffqqqq] (-0.8,0.3) circle (1.5pt);
\fill[ffqqqq] (-0.8,0.2) circle (1.5pt);
\fill[qqqqff] (-0.6178171429181265,0.28218526737637645) circle (1.5pt);
\fill[ffqqqq] (-0.45,0.95) circle (1.5pt);
\fill[ffqqqq] (-0.45,0.85) circle (1.5pt);
\fill[qqqqff] (-0.30414132892276446,0.9707097701013675) circle (1.5pt);
\fill[qqqqff] (-0.28485018583178084,0.8321642879024874) circle (1.5pt);
\fill[ffqqqq] (-0.8,0.95) circle (1.5pt);

\begin{scope}[scale=0.6,xshift=45,yshift=40]
\fill (-0.55,-0.3) circle (1.5pt);
\fill (-0.45,-0.3) circle (1.5pt);
\fill (-0.55,-0.2) circle (1.5pt);
\fill (-0.55,-0.1) circle (1.5pt);
\fill (-0.35,-0.3) circle (1.5pt);
\fill (-0.25,-0.3) circle (1.5pt);
\fill (-0.45,-0.2) circle (1.5pt);
\draw [fill=blue] (-0.21051600305385818,-0.10113272373409347) circle (2pt);
\draw [fill=blue] (-0.3517988055217461,0.03943621975207706) circle (2pt);
\fill (-0.55,0.) circle (1.5pt);
\draw [line width=.pt] (-0.55,0.15)-- (-0.55,-0.3);
\draw [line width=.pt] (-0.55,-0.3)-- (0.,-0.3);
\draw [line width=.pt] (-0.55,-0.2)-- (-0.45,-0.3);
\draw [line width=.pt] (-0.45,-0.2)-- (-0.45088439121096696,0.07223526161355277);
\draw [line width=.pt] (-0.45,-0.2)-- (-0.05,-0.2);
\draw [line width=.pt] (-0.35,-0.2)-- (-0.3520483889007433,0.07265787715617661);
\filldraw [line width=.5pt,color=blue, fill opacity=0.15] (-0.3507512611340923,-0.10000079284884116)-- (-0.35,-0.2)-- (-0.21,-0.2)-- (-0.21075126113409226,-0.10000079284884116)-- cycle;
\draw [line width=.pt] (-0.45,-0.2)-- (-0.45,-0.3);
\draw [color=blue](-0.18008483544488665,0.0) node[anchor=north west] {$C_\varepsilon^{(3)}$};
\draw [line width=.5pt,dotted] (-0.21,-0.2)-- (-0.20994657168192243,-0.3);
\draw (-0.12,-0.17) node[anchor=north west] {$\scriptstyle{  \arctan{ \tfrac{2}{6 - \sqrt2}}}$};
\draw [line width=.pt] (-0.55,-0.2)-- (-0.45,-0.2);
\draw [line width=.pt] (-0.45,-0.2)-- (-0.35,-0.3);
\draw [line width=.pt] (-0.45,-0.2)-- (-0.55,-0.1);
\filldraw [line width=.5pt,color=qqzzqq, fill opacity=0.15] (-0.35,-0.3)-- (-0.25,-0.3)-- (-0.21051600305385818,-0.10113272373409347)--  cycle;
\draw [line width=.pt] (-0.45,-0.2)-- (-0.21051600305385818,-0.10113272373409347);
\filldraw [line width=.5pt,color=blue, fill opacity=0.15] (-0.45032485941472994,-0.10000105534753098)-- (-0.3507512611340923,-0.10000079284884116)-- (-0.3517988055217461,0.03943621975207706)-- (-0.4505247431606776,0.039998801961282415) -- cycle;
\filldraw [line width=.5pt,color=red,fill opacity=0.15] (-0.55,-0.1)-- (-0.55,0.)-- (-0.3517988055217461,0.03943621975207706)-- cycle;
\draw [line width=.pt] (-0.45,-0.2)-- (-0.3517988055217461,0.03943621975207706);
\draw [line width=.5pt,color=blue] (-0.3517988055217461,0.03943621975207706)-- (-0.21051600305385818,-0.10113272373409347);
\draw [color=blue](-0.45,0.2) node[anchor=north west] {$C_\varepsilon^{(4)}$};
\draw [shift={(-0.21051600305385818,-0.10113272373409347)},line width=.5pt]  plot[domain=4.100729127945055:4.5163933644903755,variable=\t]({1.*0.09221468778930107*cos(\t r)+0.*0.09221468778930107*sin(\t r)},{0.*0.09221468778930107*cos(\t r)+1.*0.09221468778930107*sin(\t r)});
\end{scope}
\end{tikzpicture}
\caption{Illustration of the gluing procedure.}
\label{fig:GLUING}
\end{figure}
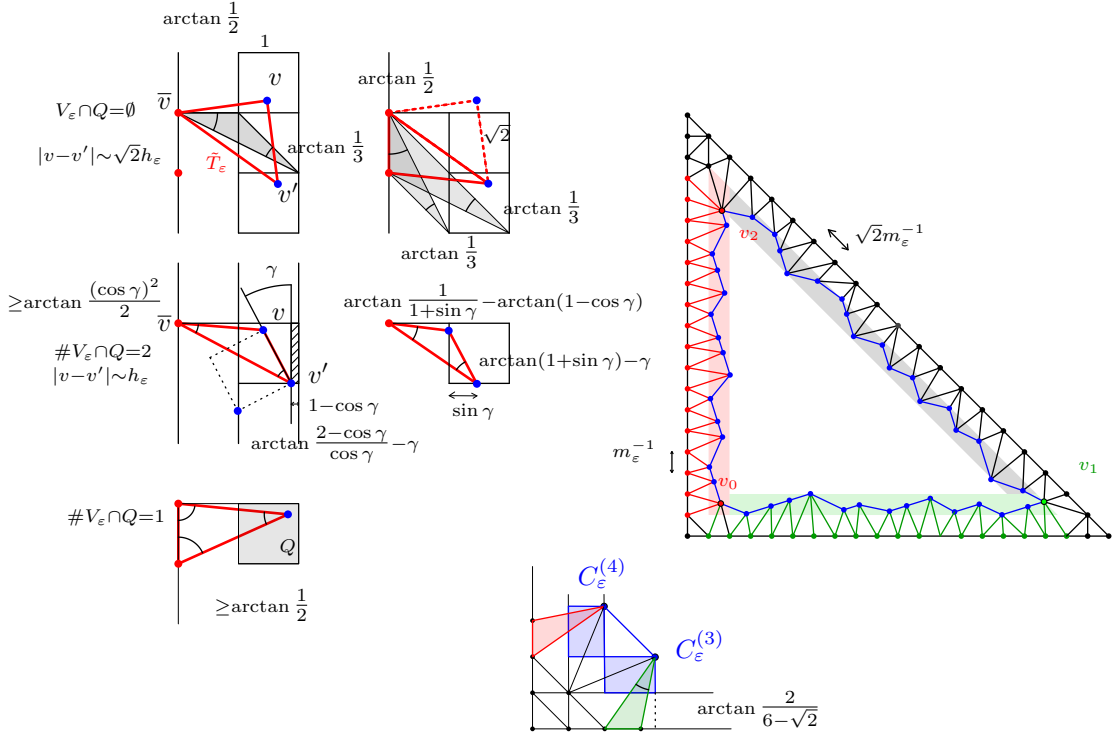

\medskip
Then, we introduce the modified characteristic function $\tilde \chi_\e$ and displacement $\tilde v_\e$ defined by $\tilde \chi_\e \equiv 1$ on each triangle $\tilde T_\e$ created through the above process while $\tilde \chi_\e = \chi_\e$ on each triangle of $\mathbf{T}^{\rm in}_\e$, and $\tilde{v}_\e$ is the Lagrange interpolation of the values of $v$ at the anchored vertices of $\bar V_\e$ and the values of $v_\e$ at the vertices of $V_\e$, while $\tilde v_\e = v_\e$ in $\cup_{T_\e' \in \mathbf T^{\rm in}_\e} T_\e'$. In particular, $\tilde v_\e = v$ on $\partial T$. Next, we add all the triangles $\tilde T_\e$ created through the above process to $\mathbf T^{\rm in}_\e$, thus defining an admissible triangulation $\mathbf{\tilde T}_\e \in X_{h_\e}(T)$ adapted to $(\tilde v_\e, \tilde \chi_\e) \in X_{h_\e}(T)$. Moreover, setting
$$
R_\e :=   \left\{ x \in T \, : \, {\rm dist}(x,\partial T) \leq m_\e^{-1} + d_\e h_\e \right\} ,
$$
we infer that $\norme{\tilde \chi_\e }_{L^1(T)} \leq \LL^2 \left(  R_\e \right) +  \norme{\chi_\e }_{L^1(T)} \leq 8 (m_\e^{-1} + d_\e h_\e) + \norme{ \chi_\e  }_{L^1(T)} \to 0$ and
\begin{align*}
 \norme{\tilde v_\e - v}_{L^1(T;\R^2)}& = \norme{v_\e - v}_{L^1(T^{\rm in}_\e;\R^2)} +  \int_{R_\e} \lvert \tilde v_\e - v \rvert \, dx \\
 &\leq \norme{v_\e - v}_{L^1(T;\R^2)} 
 + (3 \norme{v}_{L^\infty(T;\R^2)} + \tfrac{Ch_\e}{\e} ) \LL^2 \left( R_\e \right) \to 0
\end{align*}
when $\e \searrow 0$, where we used \eqref{eq:Linfty bound case 1}, \eqref{seq:Linfty control plasticity} together with the fact that $\norme{\tilde v_\e}_{L^\infty(T;\R^2)} \leq \norme{v}_{L^\infty(T;\R^2)} + \norme{v_\e}_{L^\infty(T;\R^2)}$ by convexity. Finally, noticing that
\begin{align*}
\left\lvert \F_\e (\tilde v_\e, \tilde \chi_\e) - \F_\e(v_\e, \chi_\e) \right\rvert &\leq \sum_{\tilde T_\e} \int_{\tilde T_\e} \tfrac{\eta_\e}{2} \mathbf A_0 e(\tilde v_\e):e(\tilde v_\e) \, dx  \\
& \hfill + \int_{R_\e} \left( \tfrac{\kappa}{\e} (1 - \chi_\e) + \tfrac{1}{2} \left( \eta_\e \chi_\e \mathbf A_0 + (1-\chi_\e) \mathbf A_1 \right) e(v_\e):e(v_\e) \right) \, dx\\
 &=  \sum_{\tilde T_\e} \int_{\tilde T_\e}  \tfrac{\eta_\e}{2} \mathbf A_0 e(\tilde v_\e):e(\tilde v_\e) \, dx + \LL^2 \left(  R_\e \right)  \tfrac{\kappa}{\e} + \F_\e({v_\e}_{|R_\e}, {\chi_\e}_{|R_\e})
\end{align*}
and $\LL^2 \left(  R_\e \right)  \tfrac{\kappa}{\e} + \F_\e({v_\e}_{|R_\e}, {\chi_\e}_{|R_\e}) \leq 8 \kappa (m_\e^{-1} + d_\e h_\e) \e^{-1} + \F_\e({v_\e}_{|R_\e}, {\chi_\e}_{|R_\e}) \to 0$ when $\e \searrow$, since $m_\e^{-1} \sim d_\e h_\e \sim h_\e \ll \e$ and $\F_\e(v_\e, \chi_\e)$ converges to $\F_{\alpha,\beta}(v,0)$, the crucial point is to establish that
$$
\sum_{\tilde T_\e} \int_{\tilde T_\e}  \eta_\e \mathbf A_0 e(\tilde v_\e):e(\tilde v_\e) \, dx \to 0 \quad \text{when } \e \searrow 0.
$$
This is due from the fact that $\tilde v_\e$ has uniformly bounded slopes along the edges of each triangle $\tilde T_\e$. Indeed, let us denote its vertices by $a,b,c \in \R^2$, such that $\tilde v_\e(a) = v(a)$, $\tilde v_\e (b) = v_\e(b)$ and either $\tilde v_\e(c) = v(c)$ or $\tilde v_\e(c) = v_\e(c)$. Then, using \eqref{eq:Linfty bound case 1}, \eqref{seq:Linfty control plasticity}, we get
$$
\left\lvert \nabla \tilde v_{\e | {\tilde T_\e}} \frac{a-b}{\lvert a-b \rvert} \right\rvert \leq \left\lvert \frac{v(a) - v(b)}{\lvert a-b \rvert} \right\rvert + \left\lvert \frac{v(b) - v_\e(b)}{\lvert a-b \rvert} \right\rvert \leq \left\lvert \nabla  v  \right\rvert +  \frac{C}{\e}, 
$$
and similarily for $a$ and $c$. Therefore, using \cite[Remark 3.5]{CDM}, we infer that
$$
\lvert \nabla \tilde v_{\e | {\tilde T_\e}} \rvert \leq \left( \left\lvert \nabla  v  \right\rvert +  \frac{C}{\e} \right) \frac{\sqrt{5}}{\sin \theta_0},
$$ 
hence 
$$
\sum_{\tilde T_\e} \int_{\tilde T_\e}  \eta_\e \mathbf A_0 e(\tilde v_\e):e(\tilde v_\e) \, dx 
\leq C \LL^2(R_\e) \eta_\e \left( 1 +  \frac{1}{\e} \right)^2  \leq C \frac{ h_\e}{\e} \frac{\eta_\e}{\e}  \to 0 \quad \text{when } \e \searrow 0,
$$
for a constant $C(\theta_0,\kappa, \alpha,\mathbf A_0, \tau,\nabla v) \in (0,\infty)$, thus concluding the proof of Lemma \ref{lem:T plasticity}.
\end{proof}

\section{In between plasticity and brittle fracture} \label{sec:intermediate}
\begin{prop}\label{prop:intermediate}
Assume that $\theta_0 \leq \arctan (1/4)$. If $\alpha \in (0,\infty)$ and $\beta\in (0,\infty)$, then the functional $\F_\e$ $\Gamma$-converges for the strong $L^1(\O;\R^2) \times L^1(\O)$-topology to the functional
$\F_{\alpha,\beta}:L^1(\O;\R^2) \times L^1(\O) \to [0,+\infty]$ defined by
$$\F_{\alpha,\beta}(u,\chi)=
\begin{cases}
\ds  \frac12 \int_\O \mathbf A_1 e(u):e(u)\, dx + \int_{J_u} \phi_{\alpha,\beta} \left( \lvert \sqrt{\mathbf A_0} \left[ u \right] \odot \nu_u \rvert \right) \, d\HH^1 & \text{ if }
\begin{cases}
\chi=0\text{ a.e. in }\O,\\
u \in SBD^2(\O),
\end{cases}\\
+\infty & \text{ otherwise,}
\end{cases}$$
where $\big[u\big] := u^+ - u^-$ and $\phi_{\alpha,\beta} $ is defined in \eqref{eq:density intermediate}.
\end{prop}

\begin{proof}
{\it Lower bound.} As seen in Proposition \ref{prop:domains}, we can assume that $\chi = 0$ and $u \in SBD^2(\O)$. Let $(u_\e, \chi_\e) \in X_{h_\e}(\O;\R^2)$ be such that $(u_\e , \chi_\e) \to (u,0) $ in $ L^1(\O;\R^2) \times L^1(\O) $ as $ \e \searrow 0$ and, up to a subsequence (not relabeled), the following limit exists and is finite
$$
\lim_{\e \searrow 0} \F_\e(u_\e,\chi_\e)  < + \infty.
$$
By definition of the finite element set $X_{h_\e}(\O;\R^2)$, there exists an admissible triangulation $\mathbf T^\e \in \mathcal T_{\e}(\O)$ such that $u_\e$ and $\chi_\e$ are respectively affine and constant on each $T \in \mathbf T^\e$. Without loss of generality, we can assume that  
$$
 \ds \chi_\e = \mathds{1}_{ \left\{  \textstyle{\frac12 \mathbf A_1 e(u_\e):e(u_\e)  \geq  \frac12 \eta_\e \mathbf A_0 e(u_\e):e(u_\e) + \frac{\kappa}{\e} }\right\} }
$$
and
$$
\F_\e(u_\e,\chi_\e) =  \int_\O \min \, \left( \textstyle{\frac12 \eta_\e \mathbf A_0 e(u_\e):e(u_\e) + \frac{\kappa}{\e} \, ; \,  \frac12 \mathbf A_1 e(u_\e):e(u_\e)} \right) \, dx .
$$
We introduce the Radon measures 
$$ \lambda_\e:=\min \, \left( \textstyle{\frac12 \eta_\e \mathbf A_0 e(u_\e):e(u_\e) + \frac{\kappa}{\e} \, ; \,  \frac12 \mathbf A_1 e(u_\e):e(u_\e) } \right)   \LL^2\res \O,$$
which, as they are uniformly bounded in $\M(\O)$, weakly* converge to some nonnegative measure $\lambda \in \M(\O)$ (up to a subsequence, not relabeled). In particular, it is enough to show that
\begin{equation}\label{eq:Lebesgue}
\frac{d\lambda}{d\LL^2}(x_0) \geq {\textstyle{ \frac12}} \mathbf A_1 e(u)(x_0):e(u)(x_0) \quad \text{for } \LL^2\text{-a.e. } x_0 \in \O,
\end{equation}
and
\begin{equation}\label{eq:Hausdorff}
\frac{d\lambda}{d\HH^1\res J_u}(x_0) \geq \phi_{\alpha,\beta} \left( \lvert \sqrt{\mathbf A_0} \left[ u \right](x_0) \odot \nu_u(x_0) \rvert \right) \quad \text{for } \HH^1\text{-a.e. } x_0 \in J_u.
\end{equation}
Indeed, $\LL^2\res \O$ and $\HH^1 \res J_u$ being mutually singular, it follows from Radon-Nikod\'ym and Besicovitch Theorems together with the lower semicontinuity of weak* convergence in $\M(\O)$ along open sets, that
$$\lim_{\e \searrow 0} \F_\e(u_\e, \chi_\e) \geq \lambda(\O) \geq \int_\O \frac12 \mathbf A_1 e(u):e(u)\, dx + \int_{J_u} \phi_{\alpha,\beta} \left( \lvert \sqrt{\mathbf A_0} \left[ u \right] \odot \nu_u \rvert \right)\, d \HH^1=\F_{\alpha,\beta}(u,0).$$

\bigskip
\noindent
\textbf{Step 1.} We first establish the lower bound for the bulk energy \eqref{eq:Lebesgue}. Let $x_0 \in \O$ be such that
$$\frac{d \lambda}{d \LL^2}(x_0)=\lim_{\varrho  \searrow 0} \frac{\lambda\big(\bar{B_\varrho (x_0)}\big)}{\pi \varrho^2}$$
exists and is finite, and 
$$\lim_{\varrho  \searrow 0} \frac{1}{\varrho^2}\int_{B_\varrho(x_0)}|e(u) (y)- e(u)(x_0)|^2\, dy = 0.$$
According to Besicovitch and Lebesgue differentiation Theorems, $\LL^2$-almost every point $x_0$ in $\O$ satisfies these properties. We next consider a sequence of radii $\varrho_j  \searrow  0$ as $j \to  \infty$, such that $\lambda(\partial B_{\varrho_j}(x_0))=0$ for all $j \in \N$. Recalling \eqref{eq:beta>0} in the proof of Proposition \ref{prop:domains}, one can check that $v_\e := (1 - \chi_\e) u_\e \in SBV^2(\O)$ converges in measure to $u$ and
\begin{equation*}
e(v_\e) = (1 - \chi_\e) e(u_\e) \wto e(u) \quad \text{weakly in } L^2(\O;\Msd).
\end{equation*} 
Hence
\begin{eqnarray*}
\lambda(B_{\varrho_j}(x_0))= \lim_{\e \searrow 0} \lambda_\e(B_{\varrho_j}(x_0)) & \geq & \liminf_{\e \searrow 0}  \ds \int_{B_{\varrho_j}(x_0)} \textstyle{\frac12 \mathbf A_1 e(v_\e):e(v_\e)} \, dx  \\
 & \geq  & \int_{B_{\varrho_j}(x_0)} \textstyle{\frac12 \mathbf A_1 e(u) : e(u) } \, dx
\end{eqnarray*}
for all $j \in \N$. Thus, by the choice of the point $x_0$, dividing the previous inequality by $\pi\varrho_j^2$ and passing to the limit as $j \to \infty$ implies that
$$
\frac{d\lambda}{d\LL^2}(x_0) = \underset{j \to \infty}{\lim} \, \frac{\lambda(B_{\varrho_j}(x_0))}{\pi \varrho_j^2}  \geq \lim_{j \to \infty}   \int_{B_{\varrho_j}(x_0)}  \frac{ \mathbf A_1 e(u) : e(u)}{2 \pi \varrho_j^2} \, dx = {\textstyle{ \frac12}} \mathbf A_1 e(u)(x_0): e(u)(x_0).
$$

\medskip
\textbf{Step 2.} We next pass to the lower bound inequality for the singular part of the energy \eqref{eq:Hausdorff}. As mentionned in the Introduction, the difficult part of the proof was actually done in \cite{BB}, so that the following substantially relies on Cauchy-Schwarz' inequality and the lower-semicontinuity of the norm for the weak*-convergence in $\M(\O;\Msd)$.

\medskip

Let $x_0 \in J_u$ be such that 
$$ \frac{d \lambda}{d \HH^1 \res J_u}(x_0)=\lim_{\varrho \searrow 0} \, \frac{\lambda\big(\bar{B_\varrho (x_0)}\big)}{\HH^1\big(J_u \cap \bar{B_\varrho (x_0)}\big)}$$
exists and is finite, and
$$\lim_{\varrho \searrow 0} \, \frac{\HH^1( J_u \cap B_\varrho (x_0))}{2 \varrho}=1.$$
According to the Besicovitch differentiation Theorem and the  countably $(\HH^1,1)$-rectifiability of $J_u$ (see \cite[Theorem 2.83]{AFP}), it follows that $\HH^1$-almost every point $x_0$ in $J_u$ fulfills these conditions. By definition of the jump set $J_u$, there exist $\nu :=\nu_u(x_0) \in \mathbb S^1$ and $u^\pm(x_0) \in \R^2$ with $u^+(x_0) \neq u^-(x_0)$ such that the function $u_{x_0,\varrho} := u( x_0 + \varrho \, \cdot )$ converges in $L^1(B_1(0))$ to the jump function
$$\bar{u} : y \in B_1(0) \mapsto 
\begin{cases}
u^+(x_0) & \text{ if } y \cdot \nu > 0, \\
u^-(x_0) & \text{ if } y \cdot \nu < 0,
\end{cases} $$
as $\varrho \searrow 0 $. The point $x_0 \in J_u$ being fixed, we will write
$$
\left[u \right] \odot \nu := \left[u \right](x_0) \odot \nu_u(x_0)
$$
and we will sometimes omit to write the dependence in $x_0$ in the rest of the proof. As before, we consider a sequence of radii $\{\varrho_j\}_{j \in \N}$ such that $\varrho_j \searrow 0$ and $\lambda(\partial B_{\varrho_j}(x_0))=0$ for all $j \in \N$. Hence, writing $B:= B_1(0)$, we have    
$$  \begin{cases}
\ds \lim_{j \to \infty}\lim_{\e \searrow 0} u_\e (x_0 + \varrho_j \, \cdot)  =\lim_{j \to \infty} u_{x_0,\varrho_j} =   \bar{u} \quad \text{ in } L^1(B), \\
\ds\lim_{j \to \infty}\lim_{\e \searrow 0}  \frac{\lambda_\e (B_{\varrho_j}(x_0))}{2\varrho_j} =\lim_{j \to \infty}\frac{ \lambda (B_{\varrho_j}(x_0))}{2\varrho_j}  = \frac{d \lambda}{d \HH^1 \res J_u}(x_0), \\
\ds\lim_{j \to \infty}\lim_{\e \searrow 0} \frac{\eta_\e}{\varrho_j} =\lim_{j \to \infty}\lim_{\e \searrow 0} \frac{\e}{\varrho_j} = \lim_{j \to \infty}\lim_{\e \searrow 0} \frac{\omega(h_\e)}{\varrho_j}= 0. 
\end{cases}   $$
Therefore, there exists a subsequence $\{\e_j\}_{j \in \N}$ such that $\e_j \searrow  0$ as $j \nearrow \infty$ and
\begin{subequations}\label{eq:beforeSections}
	\begin{empheq}[left=\empheqlbrace]{align}
	& u_j := u_{\e_j}(x_0 + \varrho_j \, \cdot) \to \bar{u} \quad\text{ in } L^1(B), \label{seq:CVL1} \\
	& \frac{ \lambda_{\e_j}( B_{\varrho_j}(x_0))}{2\varrho_j} \to \frac{d \lambda}{d \HH^1 \res J_u}(x_0), \label{seq:2h(x_0)}\\
	& \eta'_j := \frac{\eta_{\e_j}}{\varrho_j} \to 0, \quad \e'_j := \frac{\e_j}{\varrho_j} \to 0, \quad \frac{\omega(h_{\e_j})}{\varrho_j}\to 0. \label{seq:eps'}
	\end{empheq}
\end{subequations}
Using a change of variables, we get that
\begin{multline}\label{eq:liminf}
2\frac{d \lambda}{d \HH^1 \res J_u}(x_0)  =  \lim_{j \to \infty}\frac{1}{\varrho_j}  \int_{B_{\varrho_j}(x_0)} \min \, \left( \frac{ \eta_{\e_j}}{2} \mathbf A_0 e(u_{\e_j}):e(u_{\e_j}) + \frac{\kappa}{\e_j} \, ; \,  \frac12 \mathbf A_1 e(u_{\e_j}):e(u_{\e_j}) \right)  \, dx  \\
=  \lim_{j \to \infty} \int_{B}  \min \, \left( \frac{\eta'_j}{2} \mathbf A_0 e(u_j):e(u_j) + \frac{\kappa}{\e'_j} \, ; \,  \frac{1}{2 \varrho_j} \mathbf A_1 e(u_j):e(u_j) \right)   \, dy  \\ 
=  \lim_{j \to \infty} \int_{B}  \left(  \left( \frac{\eta'_j}{2} \mathbf A_0 e(u_j):e(u_j) + \frac{\kappa}{\e'_j} \right) \chi_j +  \frac{1}{2 \varrho_j} \mathbf A_1 e(u_j):e(u_j) (1 - \chi_j) \right)  \, dy  
\end{multline}
where 
$$
\chi_j := \chi_{\e_j}(x_0 + \varrho_j \cdot) = \ds \mathds{1}_{\ds \left\{ \tfrac{1}{2 \varrho_j} \mathbf A_1 e(u_j):e(u_j) \geq  \tfrac{\eta'_j}{2} \mathbf A_0 e(u_j):e(u_j) + \tfrac{\kappa}{\e'_j} \right\} }.
$$
As before, one can check that $u_j \overset{*}{\wto} \bar u $ weakly* in $ BD(B)$, so that 
$$
\sqrt{ \mathbf A_0} e(u_j) \, \LL^2 \res B \overset{*}{\wto} \sqrt{ \mathbf A_0} E (\bar u) = \sqrt{ \mathbf A_0} \left[ u \right] \odot \nu \,  \HH^1 \res B^{\nu} 
$$
weakly* in $\M(B)$. In particular, by lower semi-continuity of the norm for the weak* convergence of measures, we get that
\begin{equation}\label{eq:lsc}
2 \lvert \sqrt{ \mathbf A_0} \left[ u \right]  \odot \nu  \rvert = \lvert \sqrt{ \mathbf A_0} E(\bar u) \rvert (B) \leq \liminf_{j \nearrow \infty} \int_B \left\lvert \sqrt{ \mathbf A_0} e(u_j) \right\rvert \, dx.
\end{equation}
On the one hand, following the proof of \cite[Proposition 3.5]{BB}, one can check that
\begin{equation}\label{eq:morceau fracture}
l(x_0) := \liminf_{j \nearrow \infty} \int_B \frac{\chi_j}{\e'_j}   \, dx \geq 2 \beta  \sin \theta_0.
\end{equation}
Up to a subsequence (not relabeled), we can assume that the above limit exists. On the other hand, \eqref{eq:liminf} entails that $(1-\chi_j) e(u_j) \to 0$ strongly in $L^2(B;\Msd)$, so that Cauchy-Schwarz' inequality together with \eqref{eq:morceau fracture} yield 
\begin{multline*}
\liminf_{j \nearrow \infty} \int_B \left\lvert \sqrt{ \mathbf A_0} e(u_j) \right\rvert \, dx = \liminf_{j \nearrow \infty}  \int_B \chi_j \left\lvert \sqrt{ \mathbf A_0} e(u_j) \right\rvert \, dx \\
\leq  \sqrt{ \frac{2}{\alpha}l(x_0)} \, \liminf_{j \nearrow \infty}  \left( \int_B \frac{\eta'_j \chi_j }{2}  \mathbf A_0 e(u_j):e(u_j)  \, dx \right)^\frac12 .
\end{multline*}
Using \eqref{eq:lsc}, we infer that
\begin{equation}\label{eq:morceau plastique}
\int_B \frac{\eta'_j \chi_j }{2}  \mathbf A_0 e(u_j):e(u_j)  \, dx \geq \frac{2\alpha}{ l(x_0) }  \mathbf A_0 \left[ u \right] \odot \nu : \left[ u \right] \odot \nu.
\end{equation}
Therefore, gathering \eqref{eq:morceau plastique} and \eqref{eq:morceau fracture}, \eqref{eq:liminf} entails in turn that
$$
2\frac{d \lambda}{d \HH^1 \res J_u}(x_0) \geq \kappa l(x_0) + \frac{2 \alpha}{l(x_0)} \mathbf A_0 \left[ u \right]  \odot \nu : \left[ u \right] \odot \nu \geq 2 \phi_{\alpha,\beta} \left( \lvert \sqrt{ \mathbf A_0} \left[ u \right] \odot \nu  \rvert \right),
$$
which concludes the proof of the lower bound for the jump part \eqref{eq:Hausdorff}. Indeed, either $\lvert \sqrt{ \mathbf A_0} \left[ u \right] \odot \nu  \rvert \geq \sqrt{\frac{2  \kappa}{\alpha}} \beta \sin \theta_0$ and Young's inequality entails that 
$$
\frac{d \lambda}{d \HH^1 \res J_u}(x_0) \geq \sqrt{2 \alpha \kappa} \, \lvert \sqrt{ \mathbf A_0} \left[ u \right] \odot \nu  \rvert ,
$$
or $\sqrt{ \frac{2 \alpha}{\kappa}} \lvert \sqrt{ \mathbf A_0} \left[ u \right] \odot \nu  \rvert \leq 2 \beta \sin \theta_0 \leq  l(x_0)
$ and
$$
\frac{d \lambda}{d \HH^1 \res J_u}(x_0) \geq \kappa \beta \sin \theta_0 + \frac{ \alpha}{2 \beta \sin \theta_0} \mathbf A_0 \left[ u \right] \odot \nu : \left[ u \right] \odot \nu
$$
since $l \in (0,+\infty) \mapsto \kappa l + \frac{2 \alpha}{l} \mathbf A_0 \left[ u \right] \odot \nu : \left[ u \right] \odot \nu $ is non-decreasing on $\left( \sqrt{ \frac{2 \alpha}{\kappa}} \lvert \sqrt{ \mathbf A_0} \left[ u \right] \odot \nu  \rvert , + \infty \right)$.

\bigskip
\noindent
{\it Upper bound.} We now turn to the proof of the upper bound inequality which, as mentionned in the introduction, represents the most delicate part. We can assume that $\F_{\alpha,\beta}(u,\chi)<\infty$, and thus that $\chi = 0$ and $u \in SBD^2(\O)$. Using the density result for $SBD^2$ functions (see \cite[Theorem 1.1]{Crismale_density}) as well as the lower semicontinuity of $\F_{\alpha,\beta}''( \cdot, 0)$ with respect to the $L^1(\O;\R^2)$-convergence, we can further assume without loss of generality that $u \in SBV^2(\O;\R^2) \cap L^\infty(\O;\R^2)$. Indeed, $\exists u_k \in SBV^2(\O) \cap L^\infty(\O;\R^2)$ such that $u_k \to u$ strongly in $BD(\O)$, $eu_k \to eu$ strongly in $L^2(\O;\Msd)$ and $\HH^1(J_{u_k} \Delta J_u) \to 0$. In particular, $\left\lvert E^j (u_k) - E^j (u) \right\rvert (\O) \to 0$. Therefore, since $\phi_{\alpha,\beta}$ is Lipschitz continuous and $\phi_{\alpha,\beta} \leq c (1 + {\rm id })$ for some constant $c(\alpha,\beta,\kappa,\theta_0)>0$, we infer that
\begin{multline*}
\left\lvert \int_{J_{u_k}} \phi_{\alpha,\beta} \left( \lvert \sqrt{\mathbf A_0} \left[ u_k \right] \odot \nu_{u_k} \rvert \right) \, d \HH^1 -  \int_{J_{u}} \phi_{\alpha,\beta} \left( \lvert \sqrt{\mathbf A_0} \left[ u \right] \odot \nu_{u} \rvert \right) \, d \HH^1 \right\rvert \\
\leq \text{\rm Lip} (\phi_{\alpha,\beta})   \sqrt{a'_0}  \int_{J_{u_k} \cap J_u} \lvert \left[ u_k - u \right] \odot \nu \rvert \, d \HH^1 \\
 \hfill + c \biggl( \HH^1 \left( J_{u_k} \Delta J_u \right) +  \sqrt{a'_0}  \Bigl( \lvert E^j(u_k) \rvert (\O \setminus J_u) +  \lvert E^j(u) \rvert (\O \setminus J_{u_k})  \Bigr) \biggr) \\
\to 0 \text{ when } k \nearrow \infty.
\end{multline*}
Hence, $\F_{\alpha,\beta}(u_k,0) \to \F_{\alpha,\beta}(u,0)$ as $k \nearrow \infty$. Next, applying \cite[Lemma 4.2]{CDM} to both components of $u \in SBV^2(\O)\cap L^\infty(\O;\R^2)$, we can find an extension $v \in SBV^2(\O';\R^2) \cap L^\infty(\O';\R^2)$ for some rectangle $\O' \supset \supset \O$, such that
$$
{v}_{|\O} = u , \quad \norme{v}_{L^\infty(\O';\R^2)} \leq \sqrt{2} \norme{u}_{L^\infty(\O;\R^2)} \quad \text{and} \quad \HH^1(\partial \O \cap J_{v} ) = 0.
$$
Then, owing to the density result in $SBV$ (see \cite[Theorem 3.1 and Remark 3.5]{CT}), we can find a sequence $\{v_k\}_{k \in \N}$ in $SBV^2(\O';\R^2) \cap L^\infty(\O';\R^2)$ as well as $N_k$ disjoint closed segments $L^k_1,\ldots,L_{N_k}^k  \subset \bar{\O'}$ such that:
$$\bar{J_{v_k}} = \bigcup_{i=1}^{N_k} L^k_i  , \quad \HH^1( \bar{J_{v_k}} \setminus J_{v_k}) = 0, \quad v_k \in W^{2,\infty}(\O' \setminus \bar{J_{v_k}};\R^2),$$
$$
v_k \to v \, \text{ strongly in } L^1(\O';\R^2), \quad \nabla v_k \to  \nabla v \, \text{ strongly in } L^2(\O'; \mathbb M^{2 \times 2})
$$ 
and
$$
\limsup_{k \nearrow \infty} \ds \int_{\bar{A} \cap J_{v_k} } \phi_{\alpha,\beta} \left( \lvert \sqrt{\mathbf A_0} \left[ v_k \right] \odot \nu_{v_k} \rvert \right) \, d \HH^1 \leq  \ds \int_{\bar{A} \cap J_{v} } \phi_{\alpha,\beta} \left( \lvert \sqrt{\mathbf A_0} \left[ v \right] \odot \nu_{v} \rvert \right) \, d \HH^1 
$$
for all open subset $ A \subset \subset \O'$. Especially, since $\HH^1(\partial \O \cap J_v) = 0$, we obtain that
$$
\limsup_k \int_\O \frac12 \mathbf A_1 e(v_k):e(v_k) \, dx + \int_{\bar \O \cap J_{v_k}} \phi_{\alpha,\beta} \left( \lvert \sqrt{\mathbf A_0} \left[v_k \right] \odot \nu_{v_k} \rvert \right) \, d \HH^1 \leq \F_{\alpha,\beta} \left( v_{|\O} , 0\right) = \F_{\alpha,\beta}(u,0)
$$
and, by lower semicontinuity of $\F_{\alpha,\beta} ''(\cdot , 0)$ in $L^1(\O;\R^2)$ once more, we infer that the proof of the upper bound inequality follows from Lemma \ref{lem:upper bound in between} below.
\end{proof}

We are back to establishing the following Lemma.

\begin{lem}\label{lem:upper bound in between}
Let $v \in SBV^2(\O';\R^2) \cap L^\infty(\O';\R^2)$ be such that
$$ \bar{J_v} =  \bigcup_{i=1}^{N} L_i  , \quad \HH^1( \bar{J_{v}} \setminus J_{v}) = 0, \quad v \in W^{2,\infty}(\O' \setminus \bar{J_{v}};\R^2),$$
for some pairwise disjoint closed segments $L_1,\ldots, L_N \subset \bar{\O'} $. Then, 
$$ \F_{\alpha,\beta}''(v_{|\O},0) \leq  \int_\O \frac12 \mathbf A_1 e(v) : e(v) \, dx +  \int_{J_v \cap \bar{\O}} \phi_{\alpha,\beta} \left( \lvert \sqrt{\mathbf A_0} \left[v \right] \odot \nu_{v} \rvert \right) \, d \HH^1.$$
\end{lem}

\begin{proof} Since $\O\subset \subset \O'$, then $d:={\rm dist}(\O,\R^2 \setminus \O')>0$. For all $\delta \in (0,d)$, let us consider the open sets
$$\O_\delta:=\{x \in \O' : \; {\rm dist}(x,\R^2\setminus \O')>\delta\}$$
which satisfy $\O \subset\subset \O_\delta \subset\subset \O'$. We introduce a cut-off function $\phi_\delta \in C^\infty_c(\R^2;[0,1])$ which is supported in $\O'$ and such that $\phi_\delta=1$ in $\O_{\delta}$, $\phi_\delta=0$ in $\R^2 \setminus \O_{\frac\delta2}$. We next introduce the function $ \bar{v} := \phi_\delta v \in SBV^2(\R^2;\R^2) \cap L^\infty(\R^2;\R^2)$. We remark that 
\begin{equation}\label{eq:Sv}
 \bar{v} \in W^{2,\infty} \left(\R^2 \setminus  \cup_{i=1}^N {L_i} ;\R^2 \right), \quad J_{\bar{v}} \subset J_v, \quad \text{and} \quad  J_v \setminus J_{\bar{v}} \subset J_v \setminus \O_\delta.
\end{equation}

\begin{figure}
\begin{tikzpicture}[line cap=round,line join=round,>=triangle 45,x=0.6cm,y=0.6cm]
\clip(-5,-2.5) rectangle (23.889655285978645,11.750551517717625);

\draw [line width=1.pt,color=red] (0.,0.)-- (0.,9.);

\fill [color=black,fill=black,fill opacity=0.1] (0,-1) -- (3,-1) -- (3,10) -- (0,10) -- cycle;
\fill [color=black,fill=black,fill opacity=0.05] (0,-1) -- (-3,-1) -- (-3,10) -- (0,10) -- cycle;
\draw [color=black](1.5,0) node[anchor=north west] {$R_i^+$};
\draw [color=black](-2.7,0) node[anchor=north west] {$R_i^-$};

\draw [fill=black] (0.,5.) circle(1pt);
\draw [color=black](0,5) node[anchor=north west] {$x$};
\draw [color=black](1.7,6.4) node[anchor=north west] {$h_\e  \max \left( \sin \theta;\sqrt{\frac{\alpha}{2 \kappa}} \frac{ \left\lvert \sqrt{\mathbf A_0} \left[ \bar v \right](x) \odot \nu_i \right\rvert}{\beta}  \right)$};
 
\draw[line width=1.pt,dash pattern=on 2pt off 2pt,smooth,samples=100,domain=0.4:2] plot(\x,{3.17*((\x)-0.4)^(0.5)});
\draw[line width=1.pt,dash pattern=on 2pt off 2pt,smooth,samples=100,domain=-2:-0.4] plot(\x,{3.17*(abs(-(\x)-0.4))^(0.5)});

\draw[line width=1.pt,dash pattern=on 2pt off 2pt,smooth,samples=100,domain=-2:-0.4] plot(\x,{(-0.15*((\x)-0.4)^(4)+9+2.7*((\x)+2)*((\x)+0.4))});
\draw[line width=1.pt,dash pattern=on 2pt off 2pt,smooth,samples=100,domain=0.4:2] plot(\x,{(-0.15*(-(\x)-0.4)^(4)+9+2.7*(-(\x)+2)*(-(\x)+0.4))});

\draw [line width=1.pt,color=red] (6.5599042738052935,0.27995213690264675)-- (18.,6.);

\fill [color=black,fill=black,fill opacity=0.1] (5.5599042738052935,-0.22) -- (6.5599042738052935,-2.12) -- (20.,4.5) -- (19.,6.5) -- cycle;
\fill [color=black,fill=black,fill opacity=0.05] (5.5599042738052935,-0.22) -- (4.5599042738052935,1.78) -- (18.,8.5) -- (19.,6.5) -- cycle;
\draw [color=black](16.5,7.5) node[anchor=north west] {$R_j^-$};
\draw [color=black](17.5,5) node[anchor=north west] {$R_j^+$};

\draw [line width=1.pt,dash pattern=on 2pt off 2pt,domain=42:192] plot ({12 + 1.643823555378385*cos(\x)}, {3 + 1.643823555378385*sin(\x)});
\draw [line width=1.pt,dash pattern=on 2pt off 2pt,domain=224:371] plot ({12 + 1.643823555378385*cos(\x)}, {3 + 1.643823555378385*sin(\x)});

\draw [line width=1.pt,dash pattern=on 2pt off 2pt] (6.371966934802747,0.6724858182339979)-- (10.40868365491987,2.6774267253409376);
\draw [line width=1.pt,dash pattern=on 2pt off 2pt] (6.771286399333424,-0.13292123734481923)-- (10.808003119450545,1.8720196697621212);
\draw [line width=1.pt,dash pattern=on 2pt off 2pt] (13.205468543691232,4.09214281693733)-- (17.808737057116463,6.388335857913015);
\draw [line width=1.pt,dash pattern=on 2pt off 2pt] (13.60478800822191,3.2867357613585133)-- (18.20805652164714,5.582928802334198);

\draw [line width=0.7pt] (-1.7,5)-- (1.7,5);
\draw [line width=0.7pt] (-1.7,5)-- (-1.5,5.1);
\draw [line width=0.7pt] (-1.7,5)-- (-1.5,4.9);
\draw [line width=0.7pt] (1.7,5)-- (1.5,5.1);
\draw [line width=0.7pt] (1.7,5)-- (1.5,4.9);

\draw [color=red](-0.4,-0.3) node[anchor=north west] {$L_i$};
\draw (5.4,0.4) node[anchor=north west,color=red] {$L_j$};

\draw [color=red,line width=1pt] (-2-0.3,8.3)-- (-1.2-0.3,8.3);
\draw [color=red](-1.3-0.3,8.8) node[anchor=north west] {$\nu_i$};
\draw [color=red,line width=1pt](-2-0.3,8.3)--(-2-0.3,9.1);
\draw [color=red](-2-0.3,9.8) node[anchor=north west] {$\nu_i^\perp$};
\draw [color=red,line width=0.8pt](-2.1-0.3,8.9)--(-2-0.3,9.1);
\draw [color=red,line width=0.8pt](-1.9-0.3,8.9)--(-2-0.3,9.1);
\draw [color=red,line width=0.8pt] (-1.4-0.3,8.4)-- (-1.2-0.3,8.3);
\draw [color=red,line width=0.8pt] (-1.4-0.3,8.2)-- (-1.2-0.3,8.3);

\end{tikzpicture}
\caption{}
\label{fig:R_i pm}
\end{figure}
\noindent
Before going into the construction of an optimal mesh and displacement, let us remark that the amplitude of the jump of $\bar v$ is Lischitz-continuous along each segment $L_i$. 

\medskip
Indeed, since $\bar v \in W^{1,\infty}(\R^2 \setminus \cup_{i=1}^N L_i;\R^2)$, we infer that $\bar v_{|U}$ is Lipschitz continuous with constant
\begin{equation}\label{eq:bar v lip}
\text{\rm Lip}\left(\bar v_{|U} \right) \leq \norme{\nabla \bar v}_{L^\infty(\R^2 \setminus \bar{J_v})} = : \text{\rm Lip}(\bar v) <+ \infty 
\end{equation}
on all open set with Lipschitz boundary $U \subset \R^2 \setminus \cup_{i=1}^N L_i$. In particular, all $x \in \R^2 \setminus \cup_{i=1}^N L_i$ is a Lebesgue point of $\bar v$, so that its discontinuity set is included in the segments $L_i$:
$$
S_{\bar v} \subset \cup_{i=1}^N L_i.
$$ 
Moreover, one can actually check that 
$$
S_{\bar v} = J_{\bar v}.
$$
Indeed, since $\bar v_{|R_i^\pm}$ is Lipschitz-continuous on each open rectangle 
$$
R_i^\pm = \left\{ x + s \nu_i^\perp \pm t \nu_i \, : \, x \in L_i, \lvert s \rvert < r, 0 < t < r \right\} \subset \R^2 \setminus \bar{J_v}
$$
where 
$$
\nu_i := {\nu_v}_{| L_i} \in \mathcal{S}^1 \quad \text{and} \quad 2 r =  \min \big\{ \text{\rm dist}(L_i , L_j) \, : \, 1 \leq i <j  \leq N \big\} > 0 
$$
(see Figure \ref{fig:R_i pm}), it admits a unique Lipschitz-continuous extension $\bar v^\pm$ onto $L_i \cup R_i^\pm$. Thus, for all $x \in L_i$,
$$
\lim_{\varrho \searrow 0} \frac{1}{\varrho^2} \int_{B^\pm_\varrho (x,\nu_i))} \left\lvert  \bar v (y) - \bar v^\pm(x) \right\rvert \, dy = 0
$$
where $B^\pm_\varrho (x,\nu_i) = \left\{ y \in B_\varrho(x) \, : \, \pm (y - x)\cdot \nu_i > 0 \right\}$. Hence, either $\bar v^+(x) = \bar v^-(x)$ and $x \notin S_{\bar v}$ is a Lebesgue point of $\bar v$, or $\bar v^+ \neq \bar v^-$ and $x \in J_{\bar v}$. In particular, \eqref{eq:bar v lip} entails that for all $i \in \llbracket 1,N \rrbracket$ and all $x,x' \in L_i $,
\begin{multline*}
\left\lvert \bar v^\pm (x) - \bar v^\pm(x') \right\rvert  = \lim_{\varrho \searrow 0} \frac{2}{\pi \varrho^2} \left\lvert \int_{B^\pm_\varrho(x,\nu_i)} \bar v(y)  \, dy - \int_{B^\pm_\varrho(x',\nu_i)} \bar v(z) \, dz \right\rvert \\
 \leq  \limsup_{\varrho \searrow 0} \Biggl( \left\lvert \bar v (x \pm \varrho \nu_i) - \bar v (x' \pm \varrho \nu_i) \right\rvert + \frac{2}{\pi \varrho^2}  \int_{B^\pm_\varrho(x,\nu_i)} \left\lvert \bar v(y) - \bar v(x \pm \varrho \nu_i) \right\rvert \, dy   \\
   + \frac{2}{\pi \varrho^2} \int_{B^\pm_\varrho(x',\nu_i)}  \left\lvert \bar v(z) - \bar v (x' \pm \varrho \nu_i) \right\rvert \, dz  \Biggr) \\
 \leq \text{\rm Lip}(\bar v) \left( \left\lvert x - x' \right\rvert + \lim_{\varrho \searrow 0} 4  \varrho \right) = \text{\rm Lip}(\bar v) \left\lvert x - x' \right\rvert.
\end{multline*}
Therefore, for all $i \in \llbracket 1,N \rrbracket$ and all $x,x' \in L_i $,
\begin{equation}\label{eq:almost lip continuity of the jump}
\left\lvert  \left[ \bar v \right](x) - \left[ \bar v \right] (x') \right\rvert \leq 2 \, \text{\rm Lip} ( \bar v) \left\lvert x - x' \right\rvert,
\end{equation}
where $\left[ \bar v \right] (x) := \bar v^+(x) - \bar v^-(x) $ for all $x \in L_i$.

\bigskip
\noindent
\textbf{Step 1: construction of the mesh.} Unlike the optimal mesh constructed in \cite[Appendix A]{CDM} (which was optimal for recovering the length of the cracks), here one must take into account the amplitude of the jump in order to capture the "cohesive" part of the limit energy. More precisely, the larger the amplitude of the jump is, further away from the jump set must the interpolation points be chosen (see Figure \ref{fig:recovery cohesive fracture}). We proceed independently on each segment $L_i$ and construct a local admissible triangulation covering $L_i$, which we eventually glue together by means of a background admissible triangulation. In the first step, contrary to \cite{CDM} where the optimal triangles were the smallest isocele ones parallel to the jump set with angles $\theta_0$ (and $\pi - \theta_0$) and height $h_\e \sin \theta_0$ exactly, here one cannot directly use the angle $\theta_0$ but has to introduce a fixed angle $\theta > \theta_0$ to ensure the admissibility of the local mesh. Finally, in the gluing step, as in \cite{CDM}, one introduces some universal angle $\Theta_0 = \arctan (1/4)$ which a posteriori imposes an upper bound for $\theta_0$.

\bigskip
Let us fix some angle 
$$
\theta_0 < \theta \leq \frac{\pi}{3}.
$$
In all what follows, we will denote by 
\begin{equation}\label{eq:notation constant}
C \in (0,+\infty)
\end{equation}
a generic constant, which can change from line to line, depending only of $\alpha$, $\beta$, $\kappa$, $a'_0, \theta, \norme{\nabla \bar v}_{L^\infty(\R^2 \setminus \bar{J_v})}$ and $ \norme{v}_{L^\infty(\R^2 \setminus \bar{J_v})}$. For all $i \in \llbracket 1,N \rrbracket$, we introduce the amplitude function
\begin{equation}\label{eq:l_i(x)}
l_i : x \in  L_i \mapsto \max \left( \sin \theta \, , \,\sqrt{\frac{\alpha}{2 \kappa}} \frac{ \left\lvert \sqrt{\mathbf A_0} \left[ \bar v \right](x) \odot \nu_i \right\rvert}{\beta} \right)
\end{equation}
which is taken equal to zero if $x \not \in J_{\bar v}$, as well as the incremental points $a^{i,j}_\e \in L_i$ defined inductively by
$$
a^{i,1}_\e = \text{\rm argmin} \left\{ x \cdot \nu_i^\perp \, : \, x \in L_i \right\}
$$
and
\begin{equation}\label{eq:a_i,j}
a^{i,j+1}_\e =	a^{i,j}_\e +  h_\e \dfrac{ l_i \left( a^{i,j}_\e \right)}{\tan \theta}  \nu_i^\perp
\end{equation}
as long as the latest remains in $L_i$. Letting $M^i_\e \in \N $ be the number of these incremental points, one can check that
\begin{equation}\label{eq:M_i}
1 \leq M^i_\e \leq \frac{ \HH^1(L_i)}{h_\e \cos \theta} + 1.
\end{equation}
For later convenience, we also introduce the two extreme incremental points
$$
a^{i,0}_\e := a^{i,1}_\e -  h_\e \dfrac{ l_i \left( a^{i,1}_\e \right)}{\tan \theta}  \nu_i^\perp \quad \text{and} \quad a^{i,M^i_\e +1}_\e := a^{i,M^i_\e}_\e +  h_\e \dfrac{ l_i \left( a^{i,M^i_\e}_\e \right)}{\tan \theta}  \nu_i^\perp 
$$
so that
$$
L_i \subset \subset \left[ a^{i,0}_\e, a^{i,M^i_\e + 1}_\e \right] \subset \subset L_i \cup \left( \O' \setminus \bar{J_v} \right).
$$
We next define the interpolation vertices $ \left\{  G^{i,j}_\e \right\}_{j \in \llbracket 1,M^i_\e \rrbracket}, \left\{ D^{i,j}_\e \right\}_{j \in \llbracket 1,M^i_\e \rrbracket} \subset \O' \setminus \bar{ J_v}$ defined by
\begin{equation}\label{eq:G_i,j D_i,j}
G^{i,j}_\e = a^{i,j}_\e -h_\e \frac{ l_i \left( a^{i,j}_\e \right)}{2} \nu_i \quad \text{and} \quad  D^{i,j}_\e = a^{i,j}_\e +h_\e \frac{ l_i \left( a^{i,j}_\e \right)}{2} \nu_i
\end{equation}
for $ j \in \llbracket 1,M^i_\e \rrbracket$, as well as the additional extreme vertices
$$
G^{i,0}_\e := a^{i,1}_\e -h_\e \frac{  l_i \left( a^{i,1}_\e \right)}{\tan \theta}  \nu_i^\perp  -h_\e  \frac{l_i \left( a^{i,1}_\e \right)}{2} \nu_i  \in \O' \setminus \bar J_v$$
and 
$$
\begin{cases}
G^{i,M^i_\e +1}_\e := a^{i,M^i_\e}_\e +h_\e \dfrac{  l_i \big( a^{i,M^i_\e}_\e \big)}{\tan \theta}  \nu_i^\perp  -h_\e \dfrac{ l_i \big( a^{i,M^i_\e}_\e \big)}{2} \nu_i  \in \O' \setminus \bar{ J_v}, \vspace{2mm}\\
D^{i, M^i_\e +1}_\e := a^{i,M^i_\e}_\e +h_\e \dfrac{  l_i \big( a^{i,M^i_\e}_\e \big)}{\tan \theta}  \nu_i^\perp +h_\e  \dfrac{ l_i \big( a^{i,M^i_\e}_\e \big)}{2} \nu_i  \in \O' \setminus \bar{ J_v}.
\end{cases}
$$
Note that, up to considering $-\nu_i$ instead of $\nu_i$, for $\e >0$ small enough one has 
$$
\left\{ G^{i,j}_\e \right\}_{0 \leq j \leq M^i_\e + 1} \subset R_i^- \quad \text{and} \quad \left\{ D^{i,j}_\e \right\}_{1 \leq j \leq M^i_\e + 1} \subset R_i^+.
$$
We finally define the local triangulation $\mathbf T^i_\e$ composed of the $M^i_\e$ following triangles: 
$$
T^{i,j}_\e = \begin{cases}
				 G^{i,j+1}_\e D^{i,j}_\e G^{i,j-1}_\e & \text{if } j \text{ is odd,} \vspace{2mm}\\
				 D^{i,j+1}_\e G^{i,j}_\e D^{i,j-1}_\e & \text{if } j \text{ is even,}
			\end{cases}
$$
for all $j \in \llbracket 1,M^i_\e \rrbracket$ (see Figure \ref{fig:recovery cohesive fracture}). In particular, one can check that the damage sets
$$
D^i_\e := \ds \bigcup_{j=1}^{M^i_\e} T^{i,j}_\e \subset \subset \O'
$$
are pairwise disjoint for $\e >0$ small enough. For the sake of curiosity, also note that for all $x \in L_i \setminus \{ a^{i,j}_\e \, : \, j \in \llbracket 1,M^i_\e \rrbracket \}$, the section orthogonal to $L_i$ passing through $x$ crosses exactly two triangles of $\mathbf T^i_\e$. 


\begin{figure}
\begin{tikzpicture}[line cap=round,line join=round,>=triangle 45,x=0.7cm,y=0.7cm]
\clip(-3.,-2.0162159721530024) rectangle (23.889655285978645,11.750551517717625);

\fill[line width=1.pt,fill=black,fill opacity=0.05] (-0.4,-0.6) -- (0.4,0.) -- (-0.4397624149076536,0.5986481627201515) -- cycle;
\fill[line width=1.pt,fill=black,fill opacity=0.05] (0.4,0.) -- (-0.4397624149076536,0.5986481627201515) -- (0.5523510208782617,1.2373197540262435) -- cycle;
\fill[line width=1.pt,fill=black,fill opacity=0.1] (-0.4397624149076536,0.5986481627201515) -- (-0.8147077332685259,2.041410429296884) -- (0.5523510208782617,1.2373197540262435) -- cycle;
\draw (-2,1.5) node[anchor=north west] {$T^{i,j}_\e$};

\fill[line width=1.pt,fill=black,fill opacity=0.05] (1.4525814181032777,3.252273883358846) -- (0.5523510208782617,1.2373197540262435) -- (-0.8147077332685259,2.041410429296884) -- cycle;
\fill[line width=1.pt,fill=black,fill opacity=0.05] (-1.5588763939683166,5.400080065219799) -- (-0.8147077332685259,2.041410429296884) -- (1.4525814181032777,3.252273883358846) -- cycle;
\fill[line width=1.pt,fill=black,fill opacity=0.05] (1.4525814181032777,3.252273883358846) -- (-1.5588763939683166,5.400080065219799) -- (0.6780132573158881,7.805093582435323) -- cycle;
\fill[line width=1.pt,fill=black,fill opacity=0.05] (-0.4221311887869993,8.835576022839954) -- (-1.5588763939683166,5.400080065219799) -- (0.6780132573158881,7.805093582435323) -- cycle;
\fill[line width=1.pt,fill=black,fill opacity=0.05] (-0.4221311887869993,8.835576022839954) -- (0.41972494914674985,9.438851463383562) -- (0.6780132573158881,7.805093582435323) -- cycle;
\fill[line width=1.pt,fill=black,fill opacity=0.05] (5.794532710020824,0.38775420614486034) -- (6.771286399333424,-0.13292123734481923) -- (6.953552179569636,0.9586747518091671) -- cycle;
\fill[line width=1.pt,fill=black,fill opacity=0.05] (6.771286399333424,-0.13292123734481923) -- (7.957475332908588,0.462475073846904) -- (6.953552179569636,0.9586747518091671) -- cycle;
\fill[line width=1.pt,fill=black,fill opacity=0.05] (6.953552179569636,0.9586747518091671) -- (8.167363245994464,1.5678821294257212) -- (7.957475332908588,0.462475073846904) -- cycle;
\fill[line width=1.pt,fill=black,fill opacity=0.05] (8.167363245994464,1.5678821294257212) -- (9.353552179569629,2.1494673741926142) -- (9.175890088141692,1.057871385038627) -- cycle;
\fill[line width=1.pt,fill=black,fill opacity=0.05] (9.353552179569629,2.1494673741926142) -- (10.359013021407204,3.1929170435044254) -- (10.348267955292027,1.6486640074220738) -- cycle;
\fill[line width=1.pt,fill=black,fill opacity=0.05] (10.359013021407204,3.1929170435044254) -- (13.542703550797997,4.260361697550868) -- (12.533698306781957,1.4333426245687462) -- cycle;
\fill[line width=1.pt,fill=black,fill opacity=0.05] (13.542703550797997,4.260361697550868) -- (14.766658344666336,4.870892339409538) -- (14.584476106223814,3.7754217866029576) -- cycle;
\fill[line width=1.pt,fill=black,fill opacity=0.05] (14.766658344666336,4.870892339409538) -- (15.97838879013843,5.475325256654477) -- (15.788509261124037,4.376015157014175) -- cycle;
\fill[line width=1.pt,fill=black,fill opacity=0.05] (7.957475332908588,0.462475073846904) -- (9.175890088141692,1.057871385038627) -- (8.167363245994464,1.5678821294257212) -- cycle;
\fill[line width=1.pt,fill=black,fill opacity=0.05] (9.175890088141692,1.057871385038627) -- (10.348267955292027,1.6486640074220738) -- (9.353552179569629,2.1494673741926142) -- cycle;
\fill[line width=1.pt,fill=black,fill opacity=0.05] (10.348267955292027,1.6486640074220738) -- (12.533698306781957,1.4333426245687462) -- (10.359013021407204,3.1929170435044254) -- cycle;
\fill[line width=1.pt,fill=black,fill opacity=0.05] (12.533698306781957,1.4333426245687462) -- (14.584476106223814,3.7754217866029576) -- (13.542703550797997,4.260361697550868) -- cycle;
\fill[line width=1.pt,fill=black,fill opacity=0.05] (14.584476106223814,3.7754217866029576) -- (15.788509261124037,4.376015157014175) -- (14.766658344666336,4.870892339409538) -- cycle;
\fill[line width=1.pt,fill=black,fill opacity=0.05] (15.97838879013843,5.475325256654477) -- (16.98150033692288,4.966570348172506) -- (15.788509261124037,4.376015157014175) -- cycle;
\fill[line width=1.pt,fill=black,fill opacity=0.05] (15.97838879013843,5.475325256654477) -- (17.17583454931682,6.072632715724544) -- (16.98150033692288,4.966570348172506) -- cycle;
\fill[line width=1.pt,fill=black,fill opacity=0.05] (16.98150033692288,4.966570348172506) -- (18.20805652164714,5.582928802334198) -- (17.17583454931682,6.072632715724544) -- cycle;
\fill[line width=1.pt,fill=black,fill opacity=0.05] (17.17583454931682,6.072632715724544) -- (18.400597953780096,6.6846219423158155) -- (18.20805652164714,5.582928802334198) -- cycle;
\fill[line width=1.pt,fill=black,fill opacity=0.05] (5.608914896965494,8.398641524978338) -- (6.2,4.) -- (7.919987243363377,5.997219125468108) -- cycle;

\draw [line width=1.pt,color=red] (0.,0.)-- (0.,9.);

\draw[line width=1.pt,dash pattern=on 2pt off 2pt,smooth,samples=100,domain=0.4:2] plot(\x,{3.17*((\x)-0.4)^(0.5)});
\draw[line width=1.pt,dash pattern=on 2pt off 2pt,smooth,samples=100,domain=-2:-0.4] plot(\x,{3.17*(abs(-(\x)-0.4))^(0.5)});

\draw [line width=0.3pt] (-0.4,0.)-- (0.4,0.);
\draw [line width=0.3pt] (-0.4397624149076536,0.5986481627201515)-- (0.43784669862560316,0.5990982914336569);
\draw [line width=0.3pt] (-0.5517297058752677,1.2347941696371818)-- (0.5523510208782617,1.2373197540262435);
\draw [line width=0.3pt] (-0.8147077332685259,2.041410429296884)-- (0.8143221634768217,2.040461219568319);
\draw [line width=0.3pt] (-1.4529669878949816,3.2528694970222642)-- (1.4525814181032777,3.252273883358846);
\draw [line width=0.3pt] (-1.5588763939683166,5.400080065219799)-- (1.5671230869018764,5.389130855129186);

\draw[line width=1.pt,dash pattern=on 2pt off 2pt,smooth,samples=100,domain=-2:-0.4] plot(\x,{(-0.15*((\x)-0.4)^(4)+9+2.7*((\x)+2)*((\x)+0.4))});
\draw[line width=1.pt,dash pattern=on 2pt off 2pt,smooth,samples=100,domain=0.4:2] plot(\x,{(-0.15*(-(\x)-0.4)^(4)+9+2.7*(-(\x)+2)*(-(\x)+0.4))});

\draw [line width=0.3pt] (-0.67740087813073,7.807280636672901)-- (0.6780132573158881,7.805093582435323);
\draw [line width=0.3pt] (-0.4221311887869993,8.835576022839954)-- (0.41619615396659965,8.838155491586887);

\draw [line width=0.7pt] (-0.4,-0.6)-- (0.4,0.);
\draw [line width=0.7pt] (0.4,0.)-- (-0.4397624149076536,0.5986481627201515);
\draw [line width=0.7pt] (-0.4397624149076536,0.5986481627201515)-- (-0.4,-0.6);
\draw [line width=0.7pt] (-0.4397624149076536,0.5986481627201515)-- (0.5523510208782617,1.2373197540262435);
\draw [line width=0.7pt] (0.5523510208782617,1.2373197540262435)-- (0.4,0.);
\draw [line width=0.7pt] (-0.4397624149076536,0.5986481627201515)-- (-0.8147077332685259,2.041410429296884);
\draw [line width=0.7pt] (1.4525814181032777,3.252273883358846)-- (0.5523510208782617,1.2373197540262435);
\draw [line width=0.7pt] (0.5523510208782617,1.2373197540262435)-- (-0.8147077332685259,2.041410429296884);
\draw [line width=0.7pt] (-0.8147077332685259,2.041410429296884)-- (1.4525814181032777,3.252273883358846);
\draw [line width=0.7pt] (-1.5588763939683166,5.400080065219799)-- (-0.8147077332685259,2.041410429296884);
\draw [line width=0.7pt] (1.4525814181032777,3.252273883358846)-- (-1.5588763939683166,5.400080065219799);
\draw [line width=0.7pt] (-1.5588763939683166,5.400080065219799)-- (0.6780132573158881,7.805093582435323);
\draw [line width=0.7pt] (0.6780132573158881,7.805093582435323)-- (1.4525814181032777,3.252273883358846);
\draw [line width=0.7pt] (-0.4221311887869993,8.835576022839954)-- (-1.5588763939683166,5.400080065219799);
\draw [line width=0.7pt] (0.6780132573158881,7.805093582435323)-- (-0.4221311887869993,8.835576022839954);
\draw [line width=0.7pt] (-0.4221311887869993,8.835576022839954)-- (0.41972494914674985,9.438851463383562);
\draw [line width=0.7pt] (0.41972494914674985,9.438851463383562)-- (0.6780132573158881,7.805093582435323);

\draw [line width=1.pt,color=red] (6.5599042738052935,0.27995213690264675)-- (18.,6.);

\draw [line width=1.pt,dash pattern=on 2pt off 2pt,domain=42:192] plot ({12 + 1.643823555378385*cos(\x)}, {3 + 1.643823555378385*sin(\x)});
\draw [line width=1.pt,dash pattern=on 2pt off 2pt,domain=224:371] plot ({12 + 1.643823555378385*cos(\x)}, {3 + 1.643823555378385*sin(\x)});

\draw [line width=1.pt,dash pattern=on 2pt off 2pt] (6.371966934802747,0.6724858182339979)-- (10.40868365491987,2.6774267253409376);
\draw [line width=1.pt,dash pattern=on 2pt off 2pt] (6.771286399333424,-0.13292123734481923)-- (10.808003119450545,1.8720196697621212);
\draw [line width=1.pt,dash pattern=on 2pt off 2pt] (13.205468543691232,4.09214281693733)-- (17.808737057116463,6.388335857913015);
\draw [line width=1.pt,dash pattern=on 2pt off 2pt] (13.60478800822191,3.2867357613585133)-- (18.20805652164714,5.582928802334198);

\draw [line width=0.3pt] (10.359013021407204,3.1929170435044254)-- (11.164452015990445,1.5758433697642413);
\draw [line width=0.3pt] (11.059125378237253,4.351510058789538)-- (12.533698306781957,1.4333426245687462);
\draw [line width=0.3pt] (13.542703550797997,4.260361697550868)-- (13.948671296822372,3.458270924230137);
\draw [line width=0.3pt] (14.185427970935399,4.5682673496943424)-- (14.584476106223814,3.7754217866029576);
\draw [line width=0.3pt] (14.766658344666336,4.870892339409538)-- (15.177038118692211,4.071002364926434);
\draw [line width=0.3pt] (15.36553000395857,5.169620282933558)-- (15.788509261124037,4.376015157014175);
\draw [line width=0.3pt] (15.97838879013843,5.475325256654477)-- (16.37676860827669,4.669449488572588);
\draw [line width=0.3pt] (16.573244375094465,5.772049910143182)-- (16.98150033692288,4.966570348172506);
\draw [line width=0.3pt] (17.17583454931682,6.072632715724544)-- (17.583727103258465,5.271502071175783);
\draw [line width=0.3pt] (6.953552179569636,0.9586747518091671)-- (7.352871644100313,0.15326769623034997);
\draw [line width=0.3pt] (6.371966934802747,0.6724858182339979)-- (6.771286399333424,-0.13292123734481923);
\draw [line width=0.3pt] (7.558155868377911,1.2678821294257212)-- (7.957475332908588,0.462475073846904);
\draw [line width=0.3pt] (8.776570623611015,1.8632784406174447)-- (9.175890088141692,1.057871385038627);
\draw [line width=0.3pt] (8.167363245994464,1.5678821294257212)-- (8.566682710525141,0.7624750738469037);
\draw [line width=0.3pt] (9.353552179569629,2.1494673741926142)-- (9.752871644100306,1.3440603186137965);
\draw [line width=0.3pt] (9.94894849076135,2.4540710630008915)-- (10.348267955292027,1.6486640074220738);

\draw [line width=0.7pt] (5.794532710020824,0.38775420614486034)-- (6.771286399333424,-0.13292123734481923);
\draw [line width=0.7pt] (6.771286399333424,-0.13292123734481923)-- (6.953552179569636,0.9586747518091671);
\draw [line width=0.7pt] (6.953552179569636,0.9586747518091671)-- (5.794532710020824,0.38775420614486034);
\draw [line width=0.7pt] (6.771286399333424,-0.13292123734481923)-- (7.957475332908588,0.462475073846904);
\draw [line width=0.7pt] (7.957475332908588,0.462475073846904)-- (6.953552179569636,0.9586747518091671);

\draw [line width=0.7pt] (6.953552179569636,0.9586747518091671)-- (8.167363245994464,1.5678821294257212);
\draw [line width=0.7pt] (8.167363245994464,1.5678821294257212)-- (7.957475332908588,0.462475073846904);
\draw [line width=0.7pt] (8.167363245994464,1.5678821294257212)-- (9.353552179569629,2.1494673741926142);
\draw [line width=0.7pt] (9.353552179569629,2.1494673741926142)-- (9.175890088141692,1.057871385038627);
\draw [line width=0.7pt] (9.175890088141692,1.057871385038627)-- (8.167363245994464,1.5678821294257212);
\draw [line width=0.7pt] (9.353552179569629,2.1494673741926142)-- (10.359013021407204,3.1929170435044254);
\draw [line width=0.7pt] (10.348267955292027,1.6486640074220738)-- (9.353552179569629,2.1494673741926142);
\draw [line width=0.7pt] (10.359013021407204,3.1929170435044254)-- (13.542703550797997,4.260361697550868);
\draw [line width=0.7pt] (13.542703550797997,4.260361697550868)-- (14.766658344666336,4.870892339409538);
\draw [line width=0.7pt] (14.766658344666336,4.870892339409538)-- (15.97838879013843,5.475325256654477);
\draw [line width=0.7pt] (7.957475332908588,0.462475073846904)-- (9.175890088141692,1.057871385038627);
\draw [line width=0.7pt] (9.175890088141692,1.057871385038627)-- (10.348267955292027,1.6486640074220738);
\draw [line width=0.7pt] (10.348267955292027,1.6486640074220738)-- (12.533698306781957,1.4333426245687462);
\draw [line width=0.7pt] (12.533698306781957,1.4333426245687462)-- (10.359013021407204,3.1929170435044254);
\draw [line width=0.7pt] (10.359013021407204,3.1929170435044254)-- (10.348267955292027,1.6486640074220738);
\draw [line width=0.7pt] (12.533698306781957,1.4333426245687462)-- (14.584476106223814,3.7754217866029576);
\draw [line width=0.7pt] (14.584476106223814,3.7754217866029576)-- (13.542703550797997,4.260361697550868);
\draw [line width=0.7pt] (13.542703550797997,4.260361697550868)-- (12.533698306781957,1.4333426245687462);
\draw [line width=0.7pt] (14.584476106223814,3.7754217866029576)-- (15.788509261124037,4.376015157014175);
\draw [line width=0.7pt] (15.788509261124037,4.376015157014175)-- (14.766658344666336,4.870892339409538);
\draw [line width=0.7pt] (14.766658344666336,4.870892339409538)-- (14.584476106223814,3.7754217866029576);
\draw [line width=0.7pt] (16.98150033692288,4.966570348172506)-- (15.788509261124037,4.376015157014175);
\draw [line width=0.7pt] (15.788509261124037,4.376015157014175)-- (15.97838879013843,5.475325256654477);
\draw [line width=0.7pt] (15.97838879013843,5.475325256654477)-- (17.17583454931682,6.072632715724544);
\draw [line width=0.7pt] (16.98150033692288,4.966570348172506)-- (15.97838879013843,5.475325256654477);
\draw [line width=0.7pt] (16.98150033692288,4.966570348172506)-- (18.20805652164714,5.582928802334198);
\draw [line width=0.7pt] (17.17583454931682,6.072632715724544)-- (16.98150033692288,4.966570348172506);
\draw [line width=0.7pt] (17.17583454931682,6.072632715724544)-- (18.400597953780096,6.6846219423158155);
\draw [line width=0.7pt] (18.400597953780096,6.6846219423158155)-- (18.20805652164714,5.582928802334198);
\draw [line width=0.7pt] (18.20805652164714,5.582928802334198)-- (17.17583454931682,6.072632715724544);

\draw [line width=1pt,color=red] (7.,9.)-- (7.,4.);

\draw [line width=0.3pt] (6.075615369349821,5.998824996518861)-- (7.919987243363377,5.997219125468108);
\draw (10.,6.5) node[anchor=north west] {$h_\varepsilon \, l_i \big( a_\varepsilon^{i,j}\big)$};
\draw [line width=0.3pt] (5.608914896965494,8.398641524978338)-- (8.385671575096966,8.398641524978338);
\draw (10.2,8.9) node[anchor=north west] {$h_\varepsilon \, l_i \big( a_\varepsilon^{i,j+1}\big) $};
\draw [line width=0.3pt] (6.2,4.)-- (7.8,4.);
\draw [line width=0.7pt] (5.608914896965494,8.398641524978338)-- (6.2,4.);
\draw [line width=0.7pt] (6.2,4.)-- (7.919987243363377,5.997219125468108);
\draw [line width=0.7pt] (7.919987243363377,5.997219125468108)-- (5.608914896965494,8.398641524978338);

\draw (2.8,7.9) node[anchor=north west] {$h_\varepsilon \, \frac{l_i \big(a^{i,j}_\varepsilon \big)}{\tan \theta}$};
\draw (2.8,5.5) node[anchor=north west] {$h_\varepsilon \, \frac{l_i \big(a^{i,j-1}_\varepsilon \big)}{\tan \theta}$};
\draw [line width=0.5pt] (5.303815221499847,8.144941594436245)-- (5.303815221499847,6.226666138918378);
\draw [line width=0.5pt] (5.303815221499847,6.226666138918378)-- (5.4,6.4);
\draw [line width=0.5pt] (5.303815221499847,6.226666138918378)-- (5.2,6.4);
\draw [line width=0.5pt] (5.303815221499847,8.144941594436245)-- (5.2,7.94);
\draw [line width=0.5pt] (5.303815221499847,8.144941594436245)-- (5.4,7.94);

\draw [line width=0.5pt] (5.7,5.76)-- (5.7,4.);
\draw [line width=0.5pt] (5.7,5.76)-- (5.6,5.56);
\draw [line width=0.5pt] (5.7,5.76)-- (5.8,5.56);
\draw [line width=0.5pt] (5.7,4.)-- (5.6,4.2);
\draw [line width=0.5pt] (5.7,4.)-- (5.8,4.2);

\draw (11,0) node[anchor=north west] {${\left\lvert l_i \big( a^{i,j}_\varepsilon \big) - l_i \big( a^{i,j+1}_\varepsilon \big) \right\rvert \leq C h_\varepsilon}$};

\draw [line width=1.pt,color=red] (7.,4.)-- (7.,3.5);

\draw[line width=1pt,dash pattern=on 2pt off 2pt,color=blue,smooth,samples=100,domain=7.800000284756093:8.4] plot(\x,{(5.6*((\x)-7.8)^(0.5)+4)});
\draw[line width=1.pt,dash pattern=on 2pt off 2pt,color=blue,smooth,samples=100,domain=5.58:6.2] plot(\x,{(5.6*(abs((\x)-6.2))^(0.5)+4)});

\draw [line width=0.5pt,dotted] (1.065444622709632,1.4392908745516793)-- (3.7,6.5);
\draw [line width=0.5pt,dotted,opacity=0.2] (3.7,6.5)-- (4.63,8.3);
\draw [line width=0.5pt,dotted] (4.63,8.3)-- (5.0119917487619885,9.00521023537731);

\draw [line width=0.5pt,dotted] (1.0782882651358743,1.3509063352237427)-- (4.615745235631819,3.2);

\draw (5.6,10.3) node[anchor=north west] {\textbf{Close-up:}};
\fill [fill=white, fill opacity=0.7] (7.18,5.55)circle (0.28cm);
\draw (6.63,6.) node[anchor=north west] {$a^{i,j}_\varepsilon$};
\fill [fill=white, fill opacity=0.7] (8.23,6.5)circle (0.34cm);
\draw (7.63,6.9) node[anchor=north west] {$D^{i,j}_\varepsilon$};
\fill [fill=white, fill opacity=0.7] (6.43,6.5)circle (0.34cm);
\draw (5.83,6.9) node[anchor=north west] {$G^{i,j}_\varepsilon$};

\draw [line width=0.5pt] (9.22-0.8,6.)-- (10.71-0.8,6.);
\draw [line width=0.5pt] (10.71-0.8,6.)-- (10.51-0.8,6.1);
\draw [line width=0.5pt] (10.71-0.8,6.)-- (10.51-0.8,5.9);
\draw [line width=0.5pt] (9.22-0.8,6.)-- (9.42-0.8,6.1);
\draw [line width=0.5pt] (9.22-0.8,6.)-- (9.42-0.8,5.9);

\draw [line width=0.5pt] (9.-0.3,8.4)-- (10.71-0.5,8.4);
\draw [line width=0.5pt] (10.71-0.5,8.4)-- (10.51-0.5,8.5);
\draw [line width=0.5pt] (10.71-0.5,8.4)-- (10.51-0.5,8.3);
\draw [line width=0.5pt] (9.-0.3,8.4)-- (9.2-0.3,8.5);
\draw [line width=0.5pt] (9.-0.3,8.4)-- (9.2-0.3,8.3);

\draw [color=red](-0.4,-0.47696684326043215) node[anchor=north west] {$L_i$};

\draw [color=red,line width=1pt] (-2,9.)-- (-1.2,9.);
\draw [color=red](-1.3,9.5) node[anchor=north west] {$\nu_i$};
\draw [color=red,line width=1pt](-2,9)--(-2,9.8);
\draw [color=red](-2,10.5) node[anchor=north west] {$\nu_i^\perp$};
\draw [color=red,line width=0.8pt](-2.1,9.6)--(-2,9.8);
\draw [color=red,line width=0.8pt](-1.9,9.6)--(-2,9.8);
\draw [color=red,line width=0.8pt] (-1.4,9.1)-- (-1.2,9.);
\draw [color=red,line width=0.8pt] (-1.4,8.9)-- (-1.2,9.);

\draw [line width=0.5pt,dash pattern=on 1.5pt off 1.5pt] (4.98+0.5,0.09+0.2)-- (5.45+0.5,-0.8+0.2);
\draw [line width=0.5pt,dash pattern=on 1.pt off 1.pt] (5.45+0.5,-0.8+0.2)-- (5.45+0.5,-0.58+0.2);
\draw [line width=0.5pt,dash pattern=on 1.pt off 1.pt] (5.45+0.5,-0.8+0.2)-- (5.25+0.5,-0.65+0.2);
\draw [line width=0.5pt,dash pattern=on 1.pt off 1.pt] (4.98+0.5,0.09+0.2)-- (5.18+0.5,-0.06+0.2);
\draw [line width=0.5pt,dash pattern=on 1.pt off 1.pt] (4.98+0.5,0.09+0.2)-- (4.98+0.5, -0.15+0.2);
\draw (2.8,0.) node[anchor=north west] {$  \geq h_\varepsilon \, \sin \theta$};

\begin{scriptsize}
\draw [fill=blue] (6.075615369349821,5.998824996518861) circle (1.pt);
\draw [fill=blue] (7.919987243363377,5.997219125468108) circle (1.pt);
\draw [fill=blue] (7.,5.998020146618993) circle (1.pt);
\draw [fill=blue] (7.,8.398641524978338) circle (1.pt);
\draw [fill=blue] (5.608914896965494,8.398641524978338) circle (1.pt);
\draw [fill=blue] (8.385671575096966,8.398641524978338) circle (1.pt);
\draw [fill=blue] (6.2,4.) circle (1.pt);
\draw [fill=blue] (7.8,4.) circle (1.pt);
\draw [fill=blue] (7.,4.) circle (1.pt);
\end{scriptsize}
\end{tikzpicture}
\caption{}
\label{fig:recovery cohesive fracture}
\end{figure}

\medskip
We next show that the local triangulation $\mathbf T^i_\e$ is admissible (for $\e >0$ small enough), in the sense that all its triangles have edges of length larger than $h_\e$, while their angles are all greater or equal to $\theta_0$. As we consider a fixed $i \in \llbracket 1, N \rrbracket$, we will momentarily omit to write the dependences in $i$ in what follows. We only detail the calculations for triangles $T^j_\e$ with $j$ odd, as the remaining ones can be treated in a similar way. Let $j \in \llbracket 1,M_\e \rrbracket$, $j$ odd, and let us denote the angles and the (rescaled) lengths of the edges of the triangle $T^j_\e = G^{j+1}_\e D^{j}_\e G^{j-1}_\e$ by:
$$
\alpha_{j,\e} := \yhwidehat{\scriptstyle{D^j_\e G^{j-1}_\e G^{j+1}_\e}}, \, \beta_{j,\e}:= \yhwidehat{\scriptstyle{G^{j-1}_\e G^{j+1}_\e D^{j}_\e}}, \, \gamma_{j,\e} := \yhwidehat{\scriptstyle{G^{j+1}_\e D^j_\e G^{j-1}_\e}} \in (0,\pi )
$$
and
$$
L_{\alpha_{j,\e}} := \frac{\left\lvert D^j_\e - G^{j+1}_\e \right\rvert}{h_\e}, \quad L_{\beta_{j,\e}} := \frac{\left\lvert G^{j-1}_\e - D^j_\e \right\rvert}{h_\e}, \quad L_{\gamma_{j,\e}} := \frac{\left\lvert G_\e^{j+1} - G_\e^{j-1} \right\rvert}{h_\e}.
$$
On the one hand, using Pythagore's Theorem together with the fact that $0< \theta\leq \pi / 3$, one can check that for all $\e > 0$,
$$
L_{\alpha_{j,\e}}^2 = \left\lvert  \frac{l(a^j_\e)}{\tan \theta} \right\rvert^2 + \left\lvert  \frac{ l(a^j_\e) + l(a^{j+1}_\e)}{2} \right\rvert^2 \geq   \lvert \cos \theta \rvert^2 + \lvert \sin \theta \rvert^2  = 1,
$$
similarily
$$
L_{\beta_{j,\e}}^2 = \left\lvert  \frac{l(a^{j-1}_\e)}{\tan \theta} \right\rvert^2 + \left\lvert  \frac{ l(a^{j-1}_\e) + l(a^{j}_\e)}{2} \right\rvert^2  \geq 1,
$$
and
$$
L_{\gamma_{j,\e}}^2 = \left\lvert  \frac{l(a^{j-1}_\e) + l(a^j_\e)}{\tan \theta} \right\rvert^2 + \left\lvert  \frac{ l(a^{j-1}_\e) - l(a^{j+1}_\e)}{2} \right\rvert^2   \geq \left\lvert 2  \cos \theta \right\rvert^2 \geq 1.
$$
On the other hand, \eqref{eq:bar v lip} entails that for all $\e >0$,
$$
\left\lvert  l(a^{j+1}_\e) -  l(a^{j}_\e) \right\rvert \leq \sqrt{\frac{a'_0}{2}} \left\lvert \left[ \bar v \right](a^{j+1}_\e) - \left[ \bar v \right](a^{j}_\e) \right\rvert 
 \leq C h_\e.
$$ 
Therefore, Al-Kashi's Theorem entails in turn that for all $\e \searrow 0$, 
\begin{multline*}
2 \left\lvert \cos \alpha_{j,\e} - \cos \theta \right\rvert = \left\lvert \frac{ L_{\gamma_{j,\e}}}{L_{\beta_{j,\e}}} + \frac{L_{\beta_{j,\e}}^2 - L^2_{\alpha_{j,\e}}}{L_{\gamma_{j,\e}}L_{\beta_{j,\e}}} - 2\cos \theta \right\rvert \\
\leq \left\lvert \frac{ l(a^{j-1}_\e) + l(a^j_\e) }{\tan (\theta) L_{\beta_{j,\e}}} - 2\cos \theta \right\rvert + C h_\e + C h_\e \left( L_{\beta_{j,\e}}+ L_{\alpha_{j,\e}} \right) \\
\leq 2 \cos \theta \left\lvert  \frac{l(a^{j-1}_\e) + l(a^j_\e)}{2 \sin \theta}  -  L_{\beta_{j,\e}} \right\rvert + C h_\e  \leq C h_\e.
\end{multline*}
One can proceed in the same manner to obtain the same estimate for $\beta_{j,\e}$. Therefore, we have
\begin{equation}\label{eq:cos}
\lvert \cos \alpha_{j,\e} - \cos \theta \rvert + \lvert \cos \beta_{j,\e} - \cos \theta \rvert \leq C h_\e .
\end{equation}
In particular, by continuity of the arccosinus, \eqref{eq:cos} entails that the angles $\alpha_{j,\e}$, $\beta_{j,\e}$ and $\gamma_{j,\e}$ converge (uniformly with respect to $j$) to $\theta$, $\theta$ and $\pi - 2 \theta$ respectively. In other words, for all $\eta > 0$, there exists $\e_\eta >0$ such that for all $\e_\eta > \e > 0$ and all odd $j \in \llbracket 1,M_\e  \rrbracket $,
$$
\lvert  \alpha_{j,\e} -  \theta \rvert + \lvert  \beta_{j,\e} -  \theta \rvert \leq \eta.
$$
By choosing $\eta := \frac{ \theta - \theta_0}{2} > 0$, we infer that for all $\e_\eta = \e(\theta,\theta_0) > \e >0$ and all odd $j \in \llbracket 1,M_\e \rrbracket $, 
$$
 \frac{3 \theta - \theta_0}{2} = \theta + \eta \geq \alpha_{j,\e} \geq   \theta -  \eta \geq \theta_0,
$$
similarily
$$
 \frac{3 \theta - \theta_0}{2} \geq \beta_{j,\e} \geq \theta_0
$$
and
$$
\gamma_{j,\e} = \pi - \alpha_{j,\e} - \beta_{j,\e} \geq \pi - 3 \theta + \theta_0 \geq \theta_0
$$
since we chose $\theta_0 < \theta < \pi / 3$. One can treat all the remaining triangles of $\mathbf T^i_\e$ in a similar way.

It remains to glue together these $N$ local triangulations, in an admissible way. Up to reducing the threshold angle $\Theta_0 \leq 45^\circ - \text{\rm arctan}(1/2)$ introduced in \cite[Appendix A]{CDM} to $\Theta_0 = \arctan (1/4)$, one can slighlty adapt their construction while handling for the orientation of the local triangulations (non constant along the jumpset), in the spirit of the gluing performed in the proof of Proposition \ref{prop:hencky plasticity} (with less constraints as there, we had to connect the local triangulations to a set of nodes at distance $h_\e$ from the local triangulations, whereas here the segments $L_i$ are disjoint and well separated at scale $1$). In the end, one must choose $\theta_0 \leq \Theta_0$ in order for the whole triangulation to be admissible.
%
%
%
%

\bigskip
\noindent
\textbf{Step 2: recovery sequence and convergence of the energies.} We define the damage set 
$$
\ds D_\e :=  \bigcup_{i=1}^N D^i_\e \subset \subset \O'
$$
and the characteristic function $\chi_\e := \mathds{1}_{D_\e} \in L^\infty(\R^2;\lbrace 0,1 \rbrace)$, while $\bar v_\e \in H^1(\O';\R^2)$ is the Lagrange interpolation of the values of $\bar v$ at the vertices of the triangulation $\mathbf T_\e$. In particular, 
$$
(\bar v_\e, \chi_\e) \in X_{h_\e} (\O',\omega,\theta_0).
$$

\medskip
\noindent
\underline{\textit{Step 2.a.}} First, we show that
\begin{equation}\label{eq:App_v_eps}
\begin{cases}
\chi_\e\to 0 & \text{ strongly in } L^1(\O'), \\
\bar v_\e\to \bar v & \text{ strongly in } L^2(\O';\R^2), \\
e(\bar v_\e) \mathds{1}_{\O' \setminus D_\e}\to e(\bar v)&  \text{ strongly in } L^2(\O';\mathbb M^{2 \times 2}_{\rm sym}).
\end{cases}
\end{equation}
Indeed, one has
$$
\norme{\chi_\e}_{L^1(\O')} = \LL^2(D_\e) =\sum_{ i=1}^N \LL^2(D^i_\e) = \sum_{i=1}^N \sum_{j=1}^{M^i_\e} \LL^2(T^{i,j}_\e) .
$$
Considering the two adjacent subtriangles obtained by subdividing $T^{i,j}_\e$ along the section parallel to $\nu_i$ passing through $a^{i,j}_\e$ (see Figure \ref{fig:recovery cohesive fracture}), we infer that for all $j \in \llbracket 1, M^i_\e \rrbracket$:
\begin{align*}
\frac{ \LL^2(T^{i,j}_\e)}{ h_\e } & \leq \frac{  l_i \left(a^{i,j}_\e \right) + \max \Big(  l_i \left(a^{i,j-1}_\e \right) ,  l_i \left( a^{i,j+1}_\e \right) \Big) }{4}   \frac{ h_\e \left(l_i \left(a^{i,j}_\e \right) + l_i \left( a^{i,j-1}_\e \right)  \right) }{\tan \theta}  \\
& = \frac{  l_i \left(a^{i,j}_\e \right) + \max \Big(  l_i \left(a^{i,j-1}_\e \right) ,  l_i \left( a^{i,j+1}_\e \right) \Big) }{4} \HH^1 \left( \left[ a^{i,j-1}_\e , a^{i,j+1}_\e \right] \right) \\
& = \int_{\left[ a^{i,j-1}_\e, a^{i,j+1}_\e \right]} \frac{l_i(x)}{2} \, d \HH^1(x) + R^{i,j}_\e
\end{align*}
where, using \eqref{eq:bar v lip} and \eqref{eq:almost lip continuity of the jump},
\begin{multline*}
\lvert R^{i,j}_\e \rvert \leq \frac14 \int_{\left[ a^{i,j-1}_\e, a^{i,j+1}_\e \right]} \left\lvert  l_i \left(a^{i,j}_\e \right) + \max \Big(  l_i \left(a^{i,j-1}_\e \right) ,  l_i \left( a^{i,j+1}_\e \right) \Big) - 2 l_i(x) \right\rvert \, d \HH^1(x) \\
\leq C h_\e \HH^1 \left( \left[ a^{i,j-1}_\e , a^{i,j+1}_\e \right] \right) .
\end{multline*}
Thus, taking into account the overlapping of the segments $ \left[ a^{i,j-1}_\e , a^{i,j+1}_\e \right]$, we infer that
\begin{equation}\label{eq:cv chi_eps div eps}
\frac{\norme{\chi_\e}_{L^1(\O')}}{h_\e} \leq \sum_{i=1}^N \int_{L_i} l_i(x) \, d \HH^1 (x) +  C h_\e \HH^1 \left( \bar{J_v} \right)
\end{equation}
which in particular proves the first convergence of \eqref{eq:App_v_eps}. Next, since every triangle $T \in \mathbf T_\e \setminus \bigcup_{i} \mathbf T^i_\e$ is contained in $\R^2 \setminus \bar{J_v}$ and $\bar v \in W^{2,\infty} ( \R^2 \setminus\bar{J_v}; \R^2 )$, we infer that
\begin{equation}\label{eq:LagrangeEstimate}
\norme{ \bar v_\e - \bar v }_{H^1(T;\R^2)} \leq C h_\e  \norme{D^2 \bar v}_{L^2(T)}
\end{equation}
for all $\e >0$ and $T \in \mathbf T_\e \setminus \bigcup_{i} \mathbf T^i_\e$ (see e.g. \cite[Theorem 3.1.5]{Ciarlet}). In particular, as $\norme{\bar v_\e }_{L^\infty(T;\R^2)} \leq \norme{\bar v}_{L^\infty(T;\R^2)}$ for all $T \in \mathbf T_\e$, we get that 
$$
\norme{ \bar v_\e - \bar v }^2_{L^2(\O';\R^2)}  \leq 4 \norme{\bar v}^2_{L^\infty(\R^2;\R^2)}  \LL^2(D_\e) + C^2 h_\e^2   \norme{D^2 \bar v}^2_{L^2 ( \R^2 \setminus \bar{J_v}) }  \to 0.
$$
On the other hand, as $e(\bar v) \in L^2(\O';\Msd)$ and $\LL^2(D_\e) \to 0$, \eqref{eq:LagrangeEstimate} entails that
$$\norme{ e( \bar v_\e) \mathds{1}_{\O' \setminus D_\e} - e(\bar v) }_{L^2(\O';\mathbb M^{2 \times 2}_{\rm sym})}^2 = \int_{\O' \cap D_\e} \lvert e(\bar v) \rvert ^2 \, dx + \int_{\O' \setminus D_\e} \lvert e( \bar v_\e) - e (\bar v) \rvert ^2 \, dx \to 0,$$
which concludes the proof of \eqref{eq:App_v_eps}.

\bigskip
\noindent
\underline{\textit{Step 2.b.}} Next, we show that 
\begin{equation}\label{eq:BadUpperBound}
\F_{\alpha,\beta}''(v_{|\O},0) \leq  \int_\O \frac12 \mathbf A_1 e(v):e(v) \, dx +  \int_{J_v \cap \bar \O} \phi_{\alpha,\beta}\left( \lvert \sqrt{\mathbf A_0} \left[v \right] \odot \nu_{v} \rvert \right) \, d   \HH^1 .
\end{equation}
Indeed, \eqref{eq:App_v_eps} entails the convergence of the bulk parts:
\begin{equation}\label{eq:bulk part upper bound}
\int_{\O'}\frac12  \mathbf A_1 e(\bar v_\e) : e(\bar v_\e) (1 - \chi_\e)  \, dx \to  \int_{\O'} \frac12 \mathbf A_1 e(\bar v) : e(\bar v) \, dx.
\end{equation}
Besides, on the one hand \eqref{eq:cv chi_eps div eps} implies that
\begin{equation}\label{eq:bornesup1}
\limsup_{\e \searrow 0} \frac{\kappa}{\e} \int_{\O'} \chi_\e \, dx \leq \kappa \beta \sum_{i=1}^N \int_{L_i} l_i(x) \, d \HH^1 (x). 
\end{equation}
On the other hand, one can check that
\begin{equation}\label{eq:bornesup2}
\limsup_{\e \searrow 0} \int_{\O'} \frac{\eta_\e \chi_\e}{2} \mathbf A_0 e(\bar v_\e) : e(\bar v_\e) \, dx \leq \frac{\alpha}{2 \beta} \sum_{i=1}^N \int_{L_i}  \frac{\mathbf A_0 \left[ \bar v \right] \odot \nu_i : \left[ \bar v \right] \odot \nu_i}{l_i(x)} \, d \HH^1(x) .
\end{equation}
Indeed, let us fix some $i \in \llbracket 1,N \rrbracket$ and momentarily omit to write the dependences in $i$. By definition of the Lagrange interpolation, we have that for all odd $j \in \llbracket 1, M_\e \rrbracket$ the gradient of $\bar v_\e$ on $T^{j}_\e$, denoted by $N_\e := \nabla {\bar v_\e}_{|T^{j}_\e}$, satisfies 
$$
 \left[ \bar v \right]^{D^j_\e}_{G^{j+1}_\e} := \bar v (D^j_\e) -\bar v (G^{j+1}_\e) = N_\e \left( D^j_\e - G^{j+1}_\e  \right) = h_\e  \frac{l(a^{j+1}_\e) + l(a^j_\e)}{2}   N_\e \nu - h_\e \frac{l(a^j_\e)}{\tan \theta} N_\e \nu^\perp
$$
and
\begin{multline*}
\left[ \bar v \right]^{G^{j+1}_\e}_{G^{j-1}_\e} := \bar v (G^{j+1}_\e) - \bar v (G^{j-1}_\e)  = N_\e \left( G^{j+1}_\e - G^{j-1}_\e  \right) \\  = h_\e \frac{ l(a^{j-1}_\e) - l(a^{j+1}_\e) }{2} N_\e \nu + h_\e \frac{ l(a^j_\e) + l(a^{j-1}_\e)}{\tan \theta } N_\e \nu^\perp,
\end{multline*}
or, equivalently,
$$
N_\e \nu = \frac{2}{ l(a^j_\e) + l(a^{j+1}_\e) }  \left( \frac{ \left[ \bar v \right]^{D^j_\e}_{G^{j+1}_\e}}{h_\e } + \frac{ l(a^j_\e)}{\tan \theta } N_\e \nu^\perp \right) 
$$
and
$$
N_\e \nu^\perp = \frac{\tan \theta}{l(a^j_\e) + l(a^{j-1}_\e)} \left( \frac{ \left[ \bar v \right]^{G^{j+1}_\e}_{G^{j-1}_\e}}{h_\e } -  \frac{l(a^{j-1}_\e) - l(a^{j+1}_\e)}{2}  N_\e \nu \right) .
$$
Therefore, using the "one-sided" Lipschitz-continuity of $\bar v$, \eqref{eq:bar v lip} and \eqref{eq:almost lip continuity of the jump}, we infer that
$$
N_\e \nu = \frac{ \left[ \bar v \right](a^j_\e)}{h_\e l(a^j_\e)} + \mathcal{O}_\e(1) +\mathcal{O}_\e(1)  N_\e \nu^\perp \quad \text{and} \quad
N_\e \nu^\perp  = \mathcal{O}_\e(1) -  h_\e \mathcal{O}_\e(1) N_\e \nu
$$
implying in turn 
$$
N_\e \nu = \frac{ \left[ \bar v \right](a^j_\e)}{h_\e l(a^j_\e)} + \mathcal{O}_\e(1) \quad \text{and} \quad N_\e \nu^\perp = \mathcal{O}_\e(1),
$$
so that
$$
N_\e = \left( N_\e \nu \right) \otimes \nu + \left( N_\e \nu^\perp \right) \otimes \nu^\perp = \frac{ \left[ \bar v \right](a^j_\e) \otimes \nu}{h_\e l(a^j_\e)} + \mathcal{O}_\e(1),
$$
where $\mathcal O_\e(1)$ generically denots an $L^\infty(\O';\R^k)$-function, $k \geq 1$, which is uniformly bounded in $\e$, $i$ and $j$ by a constant as in \eqref{eq:notation constant}:
$$
\sup_{i,j,\e} \norme{ \mathcal{O}_\e (1)}_{L^\infty(\O';\R^k)} \leq C.
$$
Thus, we obtain that
$$
\mathbf A_0 e(\bar v_\e)_{|T^j_\e} : e(\bar v_\e)_{|T^j_\e} \leq \frac{\mathbf A_0  \left[ \bar v \right](a^j_\e) \odot \nu:  \left[ \bar v \right](a^j_\e) \odot \nu}{\lvert h_\e l(a^j_\e) \rvert^2} + C \left( 1 + \frac{1}{h_\e} \right).
$$
Since
$$
 \LL^2 \left( T^j_\e \right)   \leq \left( \frac{ h_\e l(a^j_\e)}{2} + C \lvert h_\e \rvert^2 \right) \HH^1  \left( \left[ a^{j-1}_\e, a^{j+1}_\e \right] \right),
$$
we get 
$$
\frac{\eta_\e  }{2} \int_{T^j_\e} \mathbf A_0 e(\bar v_\e) : e(\bar v_\e) \, dx \leq 
\left(  \frac{\eta_\e}{h_\e} \frac{ \mathbf A_0  \left[ \bar v \right](a^j_\e) \odot \nu : \left[ \bar v \right](a^j_\e) \odot \nu}{4 l(a^j_\e)} + C \eta_\e \right) \HH^1  \left( \left[ a^{j-1}_\e, a^{j+1}_\e \right] \right).
$$
We can proceed similarily for all even $j \in \llbracket 1, M_\e \rrbracket$. Eventually, taking into account the overlapping of the segments $\left[ a^{j-1}_\e, a^{j+1}_\e \right]$ and \eqref{eq:bar v lip} together with \eqref{eq:almost lip continuity of the jump} once more, we infer that
$$
\frac{\eta_\e  }{2} \int_{\O'} \mathbf A_0 e(\bar v_\e) : e(\bar v_\e) \leq  \frac{\eta_\e}{h_\e} \sum_{i=1}^N \int_{L_i}  \frac{ \mathbf A_0  \left[ \bar v \right](x) \odot \nu_i : \left[ \bar v \right](x) \odot \nu_i}{2 l_i(x)} \, d \HH^1(x)  + C \eta_\e  \HH^1  \left( \bar{J_v} \right)
$$
which yields \eqref{eq:bornesup2} as we pass to the limit when $\e \searrow 0$. Gathering \eqref{eq:bornesup1} and \eqref{eq:bornesup2}, we obtain the following upper bound for the surface parts:
\begin{equation}\label{eq:cohesive part upper bound}
\limsup_{\e \searrow 0} \int_{\O'} \chi_\e \left( \frac{\kappa}{\e} + \frac{\eta_\e}{2} \mathbf A_0 e(\bar v_\e) : e(\bar v_\e) \right) \, dx \leq \int_{\bar{J_v}} \phi_{\alpha,\beta} \left( \sqrt{\mathbf A_0} \left[ \bar v \right] \odot \nu_v \right) \, d \HH^1,
\end{equation}
as one can check that 
$$
 \kappa \beta l_i + \frac{\alpha}{\beta}  \frac{ \mathbf A_0  \left[ \bar v \right]\odot \nu_i : \left[ \bar v \right] \odot \nu_i}{2 l_i} = \phi_{\alpha,\beta} \left( \sqrt{\mathbf A_0} \left[ \bar v \right] \odot \nu_i \right) \quad \text{on } L_i
$$
for all $i \in \llbracket 1, N \rrbracket$ by the Definition \eqref{eq:l_i(x)} of the amplitude function $l_i$. Now gathering \eqref{eq:bulk part upper bound} and \eqref{eq:cohesive part upper bound}, we first obtain a "bad" upper bound:
\begin{multline*}
\limsup_{\e \searrow 0} \int_{\O'} \left( \frac{1 - \chi_\e}{2} \mathbf A_1 e(\bar v_\e) : e(\bar v_\e) + \chi_\e \left( \frac{\kappa}{\e} + \frac{\eta_\e}{2} \mathbf A_0 e(\bar v_\e) : e(\bar v_\e) \right) \right) \, dx \\
 \leq \int_{\O'} \frac12 \mathbf A_1 e(\bar v) : e(\bar v) \, dx + \int_{\bar{J_v}} \phi_{\alpha,\beta} \left( \sqrt{\mathbf A_0} \left[ \bar v \right] \odot \nu_v \right) \, d \HH^1.
\end{multline*}
Then, after decomposing the above integrals over $\O' \setminus \bar \O$ and $\bar \O$, and applying the lower bound estimate applied to the open bounded set with Lipschitz boundary $\O' \setminus \bar \O$, we have 
\begin{multline*}
\int_{\O'} \frac12 \mathbf A_1 e(\bar v) : e(\bar v) \, dx + \int_{\bar{J_v}} \phi_{\alpha,\beta} \left( \sqrt{\mathbf A_0} \left[ \bar v \right] \odot \nu_v \right) \, d \HH^1 \\
 \geq \F_{\alpha,\beta} ''({\bar v}_{|\O}, 0) +  \int_{\O' \setminus \bar \O} \frac12 \mathbf A_1 e(\bar v) : e(\bar v) \, dx +  \int_{J_{\bar v}  \setminus \bar \O} \phi_{\alpha,\beta} \left( \sqrt{\mathbf A_0} \left[ \bar v \right] \odot \nu_{ v} \right) \, d \HH^1.
\end{multline*}
Using \eqref{eq:Sv} together with the facts that $\bar{J_v} \cap \bar{\O} = J_v \cap \bar{\O}$ and $\bar v _{|\O_\delta} = v _{|\O_\delta}$ by construction, we deduce that
\begin{eqnarray*}
\int_{\O} \frac12 \mathbf A_1 e( v) : e( v) \, dx + \int_{J_v \cap \bar{\O}} \phi_{\alpha,\beta} \left( \sqrt{\mathbf A_0} \left[  v \right] \odot \nu_v \right) \, d \HH^1 + \int_{J_v \setminus \left( \bar{\O} \cup J_{\bar v} \right)} \phi_{\alpha,\beta} \left( \sqrt{\mathbf A_0} \left[ \bar  v \right] \odot \nu_v \right) \, d \HH^1 \\
\geq \F_{\alpha,\beta} ''( v_{|\O},0) .
\end{eqnarray*}
Finally, using \eqref{eq:Sv} once more together with the monotone convergence Theorem, we can pass to the limit as $\delta \searrow 0^+$:
$$
\int_{J_v \setminus \left( \bar{\O} \cup J_{\bar v} \right)} \phi_{\alpha,\beta} \left( \sqrt{\mathbf A_0} \left[ \bar  v \right] \odot \nu_v \right) \, d \HH^1 \leq C \HH^1 ( J_v \setminus \O_\delta ) \to C \HH^1 ( J_v \setminus \O' )=0,
$$
which settles \eqref{eq:BadUpperBound} and completes the proof of Lemma \ref{lem:upper bound in between}.
\end{proof}

\section*{Acknowledgements}
EB acknowledges financial support from the European Research Council (ERC) under the European Union's Horizon 2020 research and innovation programme (Grant Agreement n$^\circ$864066).
\footnote{ Views and opinions expressed are however those of the authors only and do not necessarily reflect those of the European Union or the European Research Council Executive Agency. Neither the European Union nor the granting authority can be held responsible for them.}

\end{document}